\documentclass[11pt]{article}

\usepackage{mathrsfs}
\usepackage{enumerate}
\usepackage[section]{algorithm}
\usepackage{algorithmic}
\usepackage{multirow}
\usepackage{xcolor}
\usepackage{url}

\usepackage{graphicx}
\usepackage{amsmath,amsfonts,amsthm,amsbsy,amssymb}
\usepackage{fixmath}
\usepackage{amssymb,latexsym}

\usepackage{multirow}

\usepackage{hyperref}
\usepackage{cleveref}

\def\bd{\pmb{d}}
\def\be{\pmb{e}}

\def\bp{\pmb{p}}

\def\bw{\pmb{w}}
\def\bx{\pmb{x}}
\def\by{\pmb{y}}

\def\bbI{\mathbb{I}}

\def\bbP{\mathbb{P}}
\def\bbR{\mathbb{R}}

\def\bbZ{\mathbb{Z}}

\def\scrA{\mathscr{A}}

\def\scrD{\mathscr{D}}
\def\scrG{\mathscr{G}}
\def\scrH{\mathscr{H}}

\def\scrM{\mathscr{M}}

\def\scrR{\mathscr{R}}

\def\cP{{\cal P}}

\def\cR{{\cal R}}
\def\cX{{\cal X}}
\def\cY{{\cal Y}}

\def\wtd{\widetilde}
\def\what{\widehat}

\usepackage{accents}
\newcommand\munderbar[1]{%
  \underaccent{\bar}{#1}}

\newcommand\STM[2]{{\rm St}(#1,#2)}
\newcommand\STMnbr{{\rm St}_{\delta}(k,n)}

\DeclareMathOperator{\diag}{diag}
\DeclareMathOperator{\dist}{dist}
\DeclareMathOperator{\eig}{eig}
\DeclareMathOperator{\grad}{grad}
\DeclareMathOperator*{\opt}{opt}
\DeclareMathOperator{\rank}{rank}
\DeclareMathOperator{\sym}{sym}
\DeclareMathOperator{\tr}{tr}

\DeclareMathOperator{\F}{F}
\DeclareMathOperator{\HH}{H}
\DeclareMathOperator{\T}{T}

\DeclareMathOperator{\KKT}{KKT}
\DeclareMathOperator{\NEPv}{NEPv}

\def\scrR{\mathscr{R}}

\def\Blue{\textcolor{blue}}

\newtheorem{theorem}{Theorem}[section]
\newtheorem{lemma}{Lemma}[section]
\newtheorem{corollary}{Corollary}[section]

\theoremstyle{definition}
\newtheorem{definition}{Definition}[section]
\newtheorem{remark}{Remark}[section]
\newtheorem{example}{Example}[section]

\allowdisplaybreaks

\numberwithin{equation}{section}

\title{A Theory of the NEPv Approach for Optimization \\
         On the Stiefel Manifold}

\author{Ren-Cang Li%
\thanks{Department of Mathematics, University of Texas at Arlington, Arlington, TX 76019-0408, USA.
        Supported in part by NSF DMS-2009689 and DMS-2407692.
        Email: {\tt rcli@uta.edu}.}
}

\date{
      October 19, 2024 \\
      September 17, 2025 \\
      April 29, 2026 \\
      October 5, 2026
}

\begin{document}

\maketitle

\begin{abstract}
The NEPv approach has been increasingly used lately for optimization on the Stiefel manifold arising
from machine learning. General speaking, the approach first turns the first order optimality condition
into a nonlinear eigenvalue problem with eigenvector dependency (NEPv) and then solve the nonlinear problem
via some variations of the self-consistent-field (SCF) iteration. The difficulty, however, lies in designing
a proper SCF iteration so that a maximizer is found at the end.
Currently, each use of the approach is very much individualized, especially in its convergence analysis phase
to show that the approach does work or otherwise. Related, the NPDo approach is recently
proposed for the sum of coupled traces and it seeks to
turn  the first order optimality condition into a nonlinear polar decomposition
with orthonormal polar factor dependency (NPDo).  In this paper, two unifying frameworks are established,
one for each approach.
Each framework is built upon
a basic assumption, under which globally convergence to a stationary point is guaranteed
and during the SCF iterative process that leads to the stationary point, the objective function increases monotonically.
Also the notion of atomic function for each approach is proposed, and the atomic functions include commonly used matrix traces of linear and quadratic forms  as special ones. It is shown
that the basic assumptions of the approaches are satisfied by their respective atomic functions and, more importantly, by convex compositions
of their respective atomic functions. Together they provide a large collection of objectives
for which either one of approaches or both
are guaranteed to work, respectively.


\end{abstract}

\clearpage
{\scriptsize
\tableofcontents
}

\clearpage
\section{Introduction}\label{sec:intro}
Optimization on the Stiefel and Grassmann manifold is  constrained optimization
with orthogonality constraints, and optimization problems as such can be and often are handled
by the method of Lagrange multipliers.
In a milestone paper, Edelman, Ariasz, and Smith \cite{edas:1999} in 1999 advocated to treat orthogonality constraints
from the geometrical point of view and established a unifying framework to adapt standard optimization techniques, such as Newton's method and
conjugate gradient methods, for better
understanding and computational efficiency. Since then, there have been a long list of research articles on
optimization on matrix manifolds seeking the benefit of the view and  extending most generic optimization techniques
such as gradient descent/ascent methods, trusted region methods, and many others, to optimization on matrix manifolds \cite{abms:2008}.
Most conveniently, there are software toolboxes
{\tt manopt} \cite{manopt} and {\tt STOP} \cite{galc:2018,wagl:2021}  for optimization on manifolds that have been made available online
to allow anyone to try out.

By and large, aforementioned progresses, while successful, are about skillful adaptations of classical optimization techniques
for optimization on Riemannian manifolds
(see \cite{abms:2008,manopt,edas:1999,galc:2018,wagl:2021,weyi:2013} and references therein), following
the geometrical point of view \cite{edas:1999}. Recently, we witnessed several  optimization problems on the Stiefel and Grassmann manifolds emerging from
data science applications. Prominent examples include the orthogonal linear discriminant analysis (OLDA)
and several others that will be listed momentarily in Table~\ref{tbl:obj-funs}. In those problems, matrices of large/huge sizes may be involved
and objective functions are made from one or more matrix traces to serve various modeling objectives for underlying applications.
Apart from the trend of adapting generic optimization techniques, efforts and progresses have been made along a different route of designing
customized optimization methods through taking advantage of structures in objective
functions  and leveraging mature numerical linear algebra (NLA)
techniques and software packages so as to gain even more efficiency
(see \cite{wazl:2023,zhys:2020,zhli:2014a,zhli:2014b,zhwb:2022} and references therein). This new route is
the NEPv approach, where NEPv stands for {\em nonlinear eigenvalue problem with eigenvector dependency\/} coined by \cite{cazb:2018},
and has been successfully demonstrated on several machine learning applications in these papers, where theoretical analysis seems to be much
individualized.
The goal of this paper is to establish a unifying framework that streamlines the NEPv approach among these papers
and guides new applications of the approach to  emerging optimization on Riemannian manifolds from
data science and other disciplines. In addition, we will also establish another unifying framework for the NPDo
approach, where NPDo stands for {\em nonlinear polar decomposition with orthonormal polar factor  dependency}, along the line of
\cite{wazl:2022a}.

A maximization problem on the Stiefel manifold in its generality takes the form
\begin{equation}\label{eq:main-opt}
\max_{P^{\T}P=I_k}f(P),
\end{equation}
where $P\in\bbR^{n\times k}$ with $1\le k\le n$ (usually $k\ll n$), $I_k$ is the $k\times k$ identity matrix, and
objective function $f(P)$ is defined on some neighborhood of the Stiefel manifold
\begin{equation}\label{eq:STM}
\STM{k}{n}=\{P\in\bbR^{n\times k}\,:\,P^{\T}P=I_k\}\subset\bbR^{n\times k}
\end{equation}
and is differentiable in the neighborhood. Specifically,
$f$ is well defined and differentiable on some neighborhood
\begin{equation}\label{eq:O-delta}
\STMnbr:=\{P\in\bbR^{n\times k}\,:\,\|P^{\T}P-I_k\|<\delta\}
\end{equation}
 of $\STM{k}{n}$,
where $0<\delta$ is a constant and $\|\cdot\|$ is some matrix norm.

\begin{table}[t]
\renewcommand{\arraystretch}{1.0}
\caption{Objective functions in the literature}\label{tbl:obj-funs}
\centerline{
\framebox{
\parbox{0.96\textwidth}{ 
\vspace{-0.4cm}
\begin{itemize}
  \item $\tr(P^{\T}AP)$, the symmetric eigenvalue problem (SEP)
        \cite{demm:1997,govl:2013,parl:1998,saad:1992},
        where $\tr(\cdot)$ is the matrix-trace function;
  \item $\frac {\tr(P^{\T}AP)}{\tr(P^{\T}BP)}$, the orthogonal linear discriminant analysis (OLDA)
        \cite{cazb:2018,fish:1936,ngbs:2010,zhln:2010,zhln:2013};
  \item $\frac {\tr(P^{\T}AP)}{\tr(P^{\T}BP)}+\tr(P^{\T}CP)$,  the sum of the trace ratios (SumTR)
        \cite{zhli:2014a,zhli:2014b};
  \item $\frac {\tr(P^{\T}D)}{\sqrt{\tr(P^{\T}BP)}}$, the orthogonal canonical correlation analysis (OCCA)
        \cite{cugh:2015,zhwb:2022};
  \item $\frac {\tr(P^{\T}AP+P^{\T}D)}{[\tr(P^{\T}BP)]^{\theta}}$ for $0\le\theta\le 1$, the $\theta$-trace ratio problem
        ($\Theta$TR)
        \cite{wazl:2023};
  \item $\tr(P^{\T}AP+P^{\T}D)$, the MAXBET subproblem (MBSub) \cite{liww:2015,tebe:1988,vdge:1984,wazl:2023,wazl:2022a,zhys:2020};
  \item $\sum_{i=1}^N\tr(P_i^{\T}A_iP_i+P_i^{\T}D_i)$,  the sum of coupled traces (SumCT)
        \cite{bacr:2004,bicc:2019,bomt:1998,rapc:2002,wazl:2022a}, where $P$ is column-partitioned as
        $[P_1,P_2,\ldots,P_N]$;
  \item $\phi(\bx)$ with $\bx=[\tr(P^{\T}A_1P),\ldots,\tr(P^{\T}A_NP)]^{\T}$, trace composition (TrCP), where $\phi(\bx)$ is a scalar function in $\bx\in\bbR^n$;
  \item $\sum_{i=1}^N\|P^{\T}A_iP\|_{\F}^2$, the uniform multidimensional scaling (UMDS) \cite{zhzl:2017};
  \item $\tr(P^{\T}AP)+\phi(\diag(PP^{\T}))$ \cite{edas:1999}, the density functional theory (DFT) of Hohenberg and Kohn \cite{hoko:1964}
        and Kohn and Sham \cite{kosh:1965}, where $\phi(\bx)$ is a scalar function in $\bx\in\bbR^n$, and
        $\diag(PP^{\T})$ extracts the diagonal entries of $PP^{\T}$ into a vector.
\end{itemize}
\vspace{-0.4cm}
}}}
\centerline{\small * $A$, $B$, and all $A_i$ are symmetric and may or may not be
positive semidefinite.}
\end{table}

Although in general objective function $f$ can be any differentiable function
that is well-defined on $\STMnbr$, in practical applications often $f$ is a composition of
matrix traces of linear or quadratic forms in $P$. A partial list of most commonly used ones in the literature is given
in Table~\ref{tbl:obj-funs}, where
$A$, $B$, $D$, all $A_i$ and $D_i$ are constant matrices, and $A$, $B$, and all $A_i$ are at least symmetric and may or may not be
positive semidefinite.
All but the last two in the table  are clearly composed of one or more matrix traces depending on $P$, and
the last two are no exceptions! To see that, we notice
$$
\|P^{\T}A_iP\|_{\F}^2=\tr((P^{\T}A_iP)^2), \quad
[\diag(PP^{\T})]_{(i)}=\tr(\be_i^{\T}PP^{\T}\be_i)=\tr(P^{\T}\be_i\be_i^{\T}P),
$$
where $[\diag(PP^{\T})]_{(i)}$ is the $i$th entries of vector $\diag(PP^{\T})$ and
 $\be_i$ is the $i$th column of the identity matrix.
TrCP is included in Table~\ref{tbl:obj-funs} to represent a broad class of objective functions some of which may have possibly appeared in the past literature, for example,
the monotone nonlinear eigenvector problem (mNEPv) \cite{balu:2024}: $\sum_{i=1}^N\psi_i(\bp^{\T}A_i\bp)$
        where $\bp\in\bbR^n$ and each $\psi_i(\cdot)$ is a single-variable convex function in $\bbR$.

Beyond Table~\ref{tbl:obj-funs}, there are matrix optimization problems
 that can be reduced to one alike for numerical purposes. For example,
the following least-squared minimization
 \begin{equation}\label{eq:OPLS}
\min_{P\in \STM{k}{n}}\|CP-B\|_{\F}^2
\end{equation}
can be reformulated into the MAXBET subproblem in the table with $A=-C^{\T}C$ and $D=C^{\T}B/2$.
It can be found in many real world applications
including
the orthogonal least squares regression (OLSR)
for feature extraction \cite{zhwn:2016,nizl:2017}, the  multidimensional similarity structure analysis (SSA) \cite[chapter 19]{boli:1987}, and  the unbalanced Procrustes problem \cite{chtr:2001,edas:1999,elpa:1999,godi:2004,huca:1962,zhys:2020,zhdu:2006}.

\subsection{Review of the NEPv and NPDo Approach}
Maximizing trace $\tr(P^{\T}AP)$, at the top of  Table~\ref{tbl:obj-funs}, has an explicit solution in terms
of the eigenvalues and eigenvectors of symmetric matrix $A$,
known as Fan's trace maximization principle \cite{fan:1949} \cite[p.248]{hojo:2013} (see also \cite{lilb:2013,liwz:2023,lili:2024} for later extensions). For that reason, it is often regarded indistinguishably as the symmetric eigenvalue problem (SEP)
that is ubiquitous throughout mathematics, science, engineering, and especially today's data sciences.
It has been well studied theoretically and numerically in NLA
\cite{demm:1997,govl:2013,li:2015,parl:1998,rutt:1994,saad:1992}
and often serves as the most distinguished illustrating example for optimization on the Stiefel and Grassmann manifolds
\cite{abms:2008,edas:1999}.
For the rest of
the objective functions, the so-called
NEPv approach and NPDo
approach have been investigated for numerically solving the associated optimization problems.

The basic idea of the NEPv approach \cite{zhys:2020,zhwb:2022,wazl:2023} is as follows:
\begin{enumerate}[(1)]
  \item establish  an NEPv
      \begin{equation}\label{eq:NEPv-form:intro}
      H(P)P=P\Omega, \,\, P\in\STM{k}{n}
      \end{equation}
      that either is or can be made equivalent to the first order optimality condition,
      also known as the KKT condition
      (see section~\ref{sec:KKT} for detail),
      where $H(P)\in\bbR^{n\times n}$ is a symmetric matrix-valued function dependent of $P$;
  \item solve NEPv \eqref{eq:NEPv-form:intro} by the self-consistent-field (SCF) iteration: given $P_0$, iteratively
      \begin{equation}\label{eq:SCF-form:NEPv:intro}
      \framebox{
      \parbox{12.0cm}{
      compute partial eigendecomposition $H(P_{i-1})\what P_i=\what P_i\Omega_i$
      associated with the $k$ largest (or smallest) eigenvalues of $H(P_{i-1})$
                  for $\what P_i\in\STM{k}{n}$,  and postprocess $\what P_i$ to $P_i$.
      }}
      \end{equation}
\end{enumerate}
While the idea of SCF seems rather natural, its convergence analysis is not and often has to be done on a case-by-case basis where
novelty lies \cite{cazb:2018,ball:2022,luli:2024,zhwb:2022,wazl:2023}. In particular, it is critical to know
what part of the spectrum of $H(P_{i-1})$ whose partial eigendecomposition in \eqref{eq:SCF-form:NEPv:intro} is about so as to move the objective function
$f$ up.
The SCF iteration \eqref{eq:SCF-form:NEPv:intro} differs from
the classical SCF  for solving the discretized
Kohn-Sham equations in its postprocessing from $\what P_i$ to $P_i$, which is not needed in the classical SCF for NEPv
that is {\em right-unitarily invariant\/} (see Definition~\ref{dfn:UIF} in the next section).
Indiscriminately, we use SCF to
refer to both the classical SCF and SCF \eqref{eq:SCF-form:NEPv:intro} when no confusion arises.

SCF, in connection with the NEPv approach, has been one of the default methods
for solving the discretized Kohn-Sham equations in the density function theory \cite{sacs:2010,yaml:2009}.
Since then, the same idea has been proven effective in several data science applications (see Table~\ref{tbl:obj-funs}):
OLDA \cite{zhln:2010,zhln:2013}, OCCA \cite{zhwb:2022}, MBSub \cite{wazl:2023,wazl:2022a}, and
$\Theta$TR \cite{wazl:2023}. Later, we will show that the approach will work on UMDS \cite{zhzl:2017} and TrCP, too.

Related, in \cite{wazl:2022a}, the NPDo approach is proposed
to numerically maximize SumCT.
A similar idea appeared before in \cite{bomt:1998} where each $P_i$ is a vector and $D_i=0$.
The basic idea of the NPDo approach is as follows:
\begin{enumerate}[(1)]
  \item establish the first order optimality condition, which takes the form
      \begin{equation}\label{eq:NPDo-form:intro}
      \scrH(P)=P\Lambda, \,\, P\in\STM{k}{n},
      \end{equation}
      where $\scrH(P)\in\bbR^{n\times k}$ is the Euclidean gradient of $f(P)$ and,
       provably, $\Lambda$ is positive semidefinite at optimality;
  \item solve NPDo \eqref{eq:NPDo-form:intro} by the self-consistent-field iteration: given $P_0$, iteratively
      \begin{equation}\label{eq:SCF-form:NPDo:intro}
      \framebox{
      \parbox{10.0cm}{
      compute polar
      decomposition\footnotemark
      $\scrH(P_{i-1})=\what P_i\Lambda_i$ of $\scrH(P_{i-1})$ for $\what P_i\in\STM{k}{n}$, and postprocess $\what P_i$ to $P_i$.
      }}
      \end{equation}
      \footnotetext {Throughout this paper, a polar decomposition of $B\in\bbR^{n\times k}$ $(k\le n)$ refers to $B=P\Omega$
         with $P\in\STM{k}{n}$ and positive semidefinite $\Omega\in\bbR^{k\times k}$.
         $\Omega=(B^{\T}B)^{1/2}$ is always unique, but $P\in\STM{k}{n}$
         is unique if and only if $\rank(B)=k$ \cite{li:1995}.
         The matrix $P$ in the decomposition is called an {\em orthonormal polar factor\/} of $B$.}
\end{enumerate}
A key prerequisite of the NPDo approach is that, at an optimality $P_*$, \eqref{eq:NPDo-form:intro} is a polar decomposition of $\scrH(P_*)$.
This is proved in \cite{wazl:2022a} for
SumCT under the condition that all $A_i$ are positive semidefinite, and later in this paper, we will prove it for more optimization problems, including those in Table~\ref{tbl:obj-funs}
that do not appear in ratio forms.
As $\scrH(P)$ to be decomposed also depends on orthonormal polar factor $P$, we call it a
{\em nonlinear polar decomposition with orthonormal polar factor dependency}, or NPDo in short.
Polar decomposition is often computed via SVD \cite{govl:2013}
which can be viewed as a special SEP \cite{demm:1997}. For that reason, NPDo may also be regarded as a special NEPv.

\subsection{Contributions}\label{ssec:contribute}
We observe that all objective functions in Table~\ref{tbl:obj-funs} are compositions of some scalar functions,
matrix traces such as $\tr(P^{\T}AP)$ and $\tr(P^{\T}D)$ in fact.
For example,
in $\Theta$TR \cite{wazl:2023},
$f(P)$ can be expressed as a composition of three functions $\tr(P^{\T}BP)$, $\tr(P^{\T}AP)$, and $\tr(P^{\T}D)$  by $\phi$:
\begin{equation}\label{eq:ThetaTR:cvx}
f(P)=\phi\circ T(P)
\quad\mbox{with}\,\, T(P)=\begin{bmatrix}
                             \tr(P^{\T}BP) \\
                             \tr(P^{\T}AP) \\
                             \tr(P^{\T}D)
                           \end{bmatrix}, \,\,
\phi(x_1,x_2,x_3)
    =\frac {x_2+x_3}{x_1^{\theta}},
\end{equation}
i.e., it is a composition of the 3-variable function $\phi$ with three matrix traces: $\tr(P^{\T}BP)$, $\tr(P^{\T}AP)$,
and $\tr(P^{\T}D)$. Each trace serves as a singleton unit of function in $P$ that does seem to be decomposable into finer units
for any benefit of study and numerical computations. For that reason, later in this paper, we shall call
such a  singleton unit of function in $P$ an {\em atomic\/} function. In its generality, an atomic function
is defined upon satisfying two basic conditions but may not necessarily be in a matrix trace form.

Unfortunately, $\phi$ in \eqref{eq:ThetaTR:cvx} is not convex in $\bx$, but $\phi^2$ for $0\le\theta\le 1/2$ is (more detail can be found in Remark~\ref{rk:app-NEPv-cvx} in section~\ref{sec:CVX-comp:NEPv}).
Except  for OLDA and SumTR,
all objective functions in Table~\ref{tbl:obj-funs}, either themselves or squared (for OCCA and $\Theta$TR with $0\le\theta\le 1/2$), are
convex compositions of atomic functions, assuming  $\phi$ for both TrCP and DFT are convex.

Our main contributions of this paper are summarized as follows:
\begin{enumerate}[(1)]
  \item creating two unifying frameworks of the NEPv and NPDo approaches, respectively, to numerically solve \eqref{eq:main-opt}
        by their corresponding SCF iterations,
        with guaranteed convergence to a KKT point that satisfies certain necessary conditions to be
        established for a maximizer;
  \item introducing the notion of {\em atomic functions\/} in $P$ with respect to both approaches, and showing that,
        \begin{equation}\label{eq:concrete-AF}
        [\tr\big((P^{\T}D)^m\big)]^s, \quad
        [\tr\big((P^{\T}AP)^m\big)]^s,\quad
        [\tr\big((\Gamma P^{\T}AP)^m\big)]^s,
        \end{equation}
        are concrete atomic functions, where $m\ge 1$ is an integer, $s\ge 1$ is a scalar, and $A$ is symmetric but may or may not be a positive semidefinite matrix depending on the circumstances, and $\Gamma$ is positive semidefinite;

  \item  showing the  NEPv and NPDo approaches
         work on each individual atomic function  for the approach and, more importantly, any convex composition $\phi\circ T$ of their respective atomic functions, where
         $\phi(\bx)$ for $\bx\in\mathfrak{D}\subseteq\bbR^N$ is convex, each entry of $T(P)\in\bbR^N$
         is an atomic function,
         and the partial derivative of $\phi$
         with respect to an entry may be required nonnegative, depending on the particular atomic function that occupies the entry.
\end{enumerate}
Although the two approaches look very much parallel to each other in presentation, there are differences in applicabilities
and numerical implementations, making them somewhat complementary to each other.
A brief comparison of the two approaches is given in section~\ref{sec:NPDo-vs-NEPv}.

\subsection{Organization and Notation}
After stating the KKT condition of maximization problem \eqref{eq:main-opt} in section~\ref{sec:KKT}, we divide
the rest of this paper into two parts. With maximization problem \eqref{eq:main-opt} in mind, in Part~I we
focus on the NPDo approach to solve the KKT condition  in three sections:
section~\ref{sec:NPDo-theory} creates a unifying NPDo framework, including {\bf the NPDo Ansatz}
to guarantee that the KKT condition is an NPDo at optimality of \eqref{eq:main-opt},
 the global convergence of the SCF iteration \eqref{eq:SCF-form:NPDo:intro}; section~\ref{sec:AF-NPD} develops a general theory that
governs atomic functions for NPDo and show that matrix-trace functions, $\tr\big((P^{\T}D)^m\big)$ and $\tr\big((P^{\T}AP)^m\big)$
and their powers of order higher than $1$, are atomic functions;
finally in section~\ref{sec:CVX-comp:NPDo}, we investigate the NPDo approach for convex compositions $\phi\circ T(P)$ of atomic functions and elaborate on a few  $T(P)\in\bbR^N$ of common matrix-trace functions that include some of
those appearing in Table~\ref{tbl:obj-funs}.
In Part~II, we focus on the NEPv approach for the same purpose. It also has three sections to address
the corresponding issues: a unifying framework built upon an ansatz -- {\bf the NEPv Ansatz},
the global convergence of the SCF iteration \eqref{eq:SCF-form:NEPv:intro}, atomic functions for NEPv, and their convex compositions $\phi\circ T(P)$  along with a few  $T(P)\in\bbR^N$ of common matrix-trace functions.
Both frameworks are very similar in appearance, but there are subtle differences in requirements
and ease to use, making each have advantages
over the other in circumstances.
A brief comparison to highlight the major differences between the two approaches is made in section~\ref{sec:NPDo-vs-NEPv}.
Concluding remarks are drawn in section~\ref{sec:conclu}. There are six appendices at the end to supplement necessary material.
Appendix~\ref{sec:angle-space} reviews
the canonical angles between subspaces of equal dimensions;
appendix~\ref{sec:prelim} cites a couple of well known inequalities for scalars and
establishes a few new ones for matrices to serve the main body of the paper;
appendixes~\ref{sec:Proof4NPDoCVG} and~\ref{sec:Proof4NEPvCVG} contain the proofs of main convergence theorems for NPDo
and NEPv, respectively;
appendix~\ref{sec:M-IN-Prod} briefly outlines the idea to
extend our developments to the case under the $M$-inner product, i.e., for the $M$-orthogonal constraint $P^{\T}MP=I_k$ instead of $P^{\T}P=I_k$;
and appendix~\ref{sec:f-inc-theta-TR:pf} refines \cite[Theorem 2.2]{wazl:2023} in the form of   {\bf the NEPv Ansatz}.


For notation, we follow the following convention:
\begin{itemize}
  \item $\bbR^{m\times n}$  is the set of $m\times n$ real matrices,  $\bbR^n=\bbR^{n\times 1}$, and $\bbR=\bbR^1$;
  \item $\STM{k}{n}$ in \eqref{eq:STM} denotes the Stiefel manifold and
        $\STMnbr$ in \eqref{eq:O-delta} is  some neighborhood of it; Also frequently, given  $D\in\bbR^{n\times k}$,
        $$
        \STM{k}{n}_{D+}:=\{X\in\STM{k}{n}\,:\, X^{\T}D\succeq 0\};
        $$
  \item $I_n\in\bbR^{n\times n}$ is the identity matrix or simply $I$ if its size is clear from the context, and $\be_j$ is the $j$th column of $I$ of an apt size;
  \item $B^{\T}$ stands for the transpose of a matrix/vector $B$;
  \item $\cR(B)$ is the column subspace of a matrix $B$, spanned by its columns, whose dimension is $\rank(B)$, the rank of $B$;
  \item For $B\in\bbR^{m\times n}$, unless otherwise explicitly stated, its SVD
        refers to the one $B=U\Sigma V^{\T}$, also known as the {\em thin\/} SVD of $B$,  with
        $$
        \Sigma=\diag(\sigma_1(B),\sigma_2(B),\ldots,\sigma_s(B))\in\bbR^{s\times s},
        \,\,
        U\in\STM{s}{m},\,\,
        V\in\STM{s}{n},
        $$
        where $s=\min\{m,n\}$,  the singular values $\sigma_j(B)$  are always arranged decreasingly as
        $$
        \sigma_1(B)\ge\sigma_2(B)\ge\cdots\ge\sigma_s(B)\ge 0,
        $$
        and $\sigma_{\min}(B)=\sigma_s(B)$;
        Accordingly, $\|B\|_2$, $\|B\|_{\F}$, and $\|B\|_{\tr}$ are the spectral, Frobenius, and
        trace         norms of $B$:
        $$
        \|B\|_2=\sigma_1(B),\,\,
        \|B\|_{\F}=\Big(\sum_{i=1}^s[\sigma_i(B)]^2\Big)^{1/2},\,\,
        \|B\|_{\tr}=\sum_{i=1}^s\sigma_i(B),
        $$
        respectively; The trace norm is  also known as the nuclear norm;
  \item For a symmetric matrix $A\in\bbR^{n\times n}$, $\eig(A)=\{\lambda_i(A)\}_{i=1}^n$ denotes the set of its
        eigenvalues (counted by multiplicities)
        arranged in the decreasing order:
        $$
        \lambda_1(A)\ge\lambda_2(A)\ge\cdots\ge\lambda_n(A),
        $$
        and $\tr_{\max,k}(A)=\sum_{i=1}^k\lambda_i(A)$ and $\tr_{\min,k}(A)=\sum_{i=1}^k\lambda_{n-i+1}(A)$,
        the sum of the $k$ largest  eigenvalues and that of the $k$ smallest  eigenvalues of $A$, respectively;
  \item A matrix $A\succ 0\, (\succeq 0)$ means that it is symmetric and positive definite (semi-definite), and
        accordingly
        $A\prec 0\, (\preceq 0)$ if $-A\succ 0\, (\succeq 0)$.
\end{itemize}





\section{KKT Condition}\label{sec:KKT}
Consider maximization problem \eqref{eq:main-opt} on the Stiefel manifold $\STM{k}{n}$ in its generality.
For $P=[p_{ij}]\in\STMnbr$ defined in \eqref{eq:O-delta}, denote by
\begin{equation}\label{eq:f-gradR(nk)}
\scrH(P):=\frac{\partial f(P)}{\partial P} \in\bbR^{n\times k}
\quad\mbox{with its $(i,j)$th entry}\,\,
\frac{\partial f(P)}{\partial p_{ij}},
\end{equation}
the partial derivative of $f(P)$ with respect to $P$ as a matrix variable in $\bbR^{n\times k}$, where
all entries of $P$ are treated as independent. It is also known as the {\em Euclidean gradient\/} in recent literature.
Throughout this paper, notation $\scrH(P)$ is reserved for the Euclidean gradient of objective function $f$ within the context.

As an optimization problem on the Stiefel manifold, the first order optimality condition
\eqref{eq:main-opt}, also known as the KKT condition, is given by setting the Riemannian gradient of $f$ with respect to the Stiefel manifold $\STM{k}{n}$ at $P$ to $0$.
It is well-known that the Riemannian gradient of a smooth function $f$  with respect to the Stiefel manifold at $P\in\STM{k}{n}$
can be calculated according to (see, e.g., \cite[(3.37)]{abms:2008})
\begin{equation}\label{eq:gradOnStM}
\grad f_{|{\STM{k}{n}}}(P)
  =\Pi_P\big(\scrH(P)\big)
  =\scrH(P)-P \cdot \sym\!\big(P^{\T}\scrH(P)\big),
\end{equation}
where the projection $\Pi_P(Z):=P-P\sym(P^{\T}Z)$ with $\sym(P^{\T}Z)=(P^{\T}Z+Z^{\T}P)/2$. 
Setting $\grad f_{|{\STM{k}{n}}}(P)=0$ yields
the first-order optimality condition:
\begin{equation}\label{eq:KKT}
\scrH(P)=P\Lambda
\quad \mbox{with} \quad
\Lambda^{\T}=\Lambda\in\bbR^{k\times k},\quad P\in\STM{k}{n},
\end{equation}
where
$\Lambda=\sym(P^{\T}\scrH(P))$.
The exact form of $\Lambda$, however, is not important, but its symmetry is, for example, it implies that $P^{\T}\scrH(P)$ is symmetric at any KKT point $P$.
The KKT condition \eqref{eq:KKT} can also be inferred from
treating $P\in\STM{k}{n}$ as the  orthogonality constraint $P^{\T}P=I_k$ and then
using the classical method of Lagrange multipliers for constrained optimization \cite{nowr:2006}.
Geometrically, the condition \eqref{eq:KKT} simply asks for that the Euclidean gradient belongs to the normal space to the Stiefel manifold
at $P$.

For simple functions, $f(P)=\tr(P^{\T}D)$ or $\tr(P^{\T}AP)$, the KKT condition~\eqref{eq:KKT} can be considered as solved.
In fact, for the two functions, \eqref{eq:KKT} becomes
$D=P\Lambda$ or $2AP=P\Lambda$, respectively, which, in consideration of \eqref{eq:main-opt},
tell us that a maximizer can be taken to be an orthonormal polar factor of $D$ \cite{hojo:2013,zhwb:2022},
or an orthonormal basis matrix of the eigenspace of $A$ associated with its $k$ largest eigenvalues \cite{lili:2024,lilb:2013,liwz:2023,stsu:1990}, respectively. In both cases, the maximizer as described
is considered a close form solution to the respective problem because of the numerical maturity by  existing NLA techniques and software
\cite{abbd:1999,bddrv:2000,demm:1997,govl:2013,li:2015,parl:1998}.

In general, equation \eqref{eq:KKT} is not an easy equation to solve in searching for a maximizer of \eqref{eq:main-opt} with guarantee. For example,
in the MAXBET subproblem, simply $f(P)=\tr(P^{\T}D)+\tr(P^{\T}AP)$, the sum of the two simple matrix-trace functions
and \eqref{eq:KKT} becomes $2AP+D=P\Lambda$ for which there is no existing
NLA technique that yields a solution to maximize $f(P)$ with guarantee. Having said that, we point out that
the eigenvalue-based method
\cite{zhys:2020}, which falls into the NEPv approach, has been demonstrated to be numerically efficient
\cite{zhys:2020} and often produces global maximizers.
The MAXBET subproblem is a special case of SumCT. As such, in \cite{wazl:2022a}, the NPDo approach has also been
successfully applied.

We now formally define the notion of a function being right-unitarily invariant, originally introduced in \cite{luli:2024}.
It is an important concept that we will frequently
refer to in the rest of this paper. However, our definition here differs from \cite[Definition 2.1]{luli:2024} slightly
in that we limit
the domain  to some neighborhood $\STMnbr$ of the Stiefel manifold $\STM{k}{n}$, rather than the entire
space $\bbR^{n\times k}$ used in \cite{luli:2024}. Carefully going through \cite{luli:2024}, one can see that
our definition here is actually sufficient for the development in \cite{luli:2024} as it is here.

\begin{definition}\label{dfn:UIF}
A function $F\,:\,\STMnbr\to\bbR^{p\times q}$ is said {\em right-unitarily invariant\/} if
$$ 
F(PQ)\equiv F(P)\quad\mbox{for $P\in\STMnbr$ and $Q\in\STM{k}{k}$}.
$$ 
\end{definition}

\begin{remark}\label{rk:right-UI}
When objective $f$ in \eqref{eq:main-opt} is right-unitarily invariant, $f$ is constant in the entire orbit
$\{P_0Q\,:\, Q\in\STM{k}{k}\}$ for any given $P_0\in\STM{k}{n}$. Hence
if $P_*\in\STM{k}{n}$ is a maximizer of \eqref{eq:main-opt} then any member of the orbit
$\{P_*Q\,:\, Q\in\STM{k}{k}\}$ is a maximizer, too, and $P_*$ is just a representative. In this sense, optimization
problem \eqref{eq:main-opt} is a problem on the Grassmann manifold $\scrG_k(\bbR^n)$, the collection of all $k$-dimensional subspaces in $\bbR^n$, equipped with some proper metric (see appendix~\ref{sec:angle-space}).
Numerically, while we attempt to compute some approximation to $P_*$, we actually compute an approximation to
some representative $\wtd P$ in the orbit $\{P_*Q\,:\, Q\in\STM{k}{k}\}$. As far as error/convergence analysis is concerned,
bounding some metric between subspaces $\cR(\wtd P)$ and $\cR(P_*)$ is an appropriate and right thing to do.
\end{remark}

\clearpage
\part{The NPDo Approach}

\section{The NPDo Framework}\label{sec:NPDo-theory}
In \cite{wazl:2022a}, an NPDo approach is proposed to
numerically maximize the sum of coupled traces (SumCT) in Table~\ref{tbl:obj-funs}. It is an SCF iterative procedure \eqref{eq:SCF-form:NPDo:intro}
that solves  the KKT condition \eqref{eq:KKT} for its solution with an eye on maximizing the sum.
Our general framework in this section is inspired by and bears similarity to the developments there, but in more abstract terms.

\subsection{The NPDo Ansatz}\label{ssec:NPDAssum}
The success of the NPDo approach in \cite{wazl:2022a} rests on  a monotonicity lemma
which motivates us to formulate the following ansatz to build our framework upon.
The key point of the assumption is the ability to generate an improved approximate maximizer $\wtd P$
from a given one $P$, where both the given $P$ and the improved $\wtd P$ may have to come out of possibly a
        strict subset $\bbP$ of $\STM{k}{n}$.
What $\bbP$ to use
depends on the underlying optimization problem \eqref{eq:main-opt} at hand,
as we will repeatedly demonstrate later by concrete examples.

\smallskip\noindent
{\bf The NPDo Ansatz.}
{\em
Let function $f$ be defined in some neighborhood $\STMnbr$ of $\STM{k}{n}$,
and denote by $\scrH(P)=\frac{\partial f(P)}{\partial P}$.
Given $P\in\bbP\subseteq\STM{k}{n}$ and $\what P\in\STM{k}{n}$,
if
\begin{equation}\label{eq:NPDo-assume}
\tr(\what P^{\T}\scrH(P))\ge\tr(P^{\T}\scrH(P))+\eta
\quad\mbox{for some $\eta\in\bbR$},
\end{equation}
then
there exists $Q\in\STM{k}{k}$ such that $\wtd P=\what PQ\in\bbP$ and
$f(\wtd P)\ge f(P)+\omega\eta$,
where $\omega$ is some positive constant, independent of $P$ and $\what P$.
}

\smallskip
For any given $P\in\STM{k}{n}$, by Lemma~\ref{lm:polar2max}, there is always
$\what P\in\STM{k}{n}$ such that \eqref{eq:NPDo-assume} holds with some $\eta>0$, unless for that given $P$,
\eqref{eq:KKT} holds with $\Lambda\succeq 0$. In fact, we can take
$\what P\in\STM{k}{n}$ to be an orthonormal polar factor of $\scrH(P)$, which also maximizes
$\tr(X^{\T}\scrH(P))$ over $X\in\STM{k}{n}$ to $\|\scrH(P)\|_{\tr}$  again by Lemma~\ref{lm:polar2max}, in which
case $\eta=\|\scrH(P)\|_{\tr}-\tr(P^{\T}\scrH(P))$.
Hence, for the purpose of solving \eqref{eq:main-opt},
we may relax the ansatz to $\eta\ge 0$ only.
As far as verifying this ansatz is concerned, it is the desirable aim, $f(\wtd P)\ge f(P)+\omega\eta$,
that  needs to be checked.
The necessity of also involving $\bbP$, a subset of $\STM{k}{n}$, can be best justified by Example~\ref{eg:motivation:NPDo} below.

\begin{example}\label{eg:motivation:NPDo}
Consider $f(P)=\tr(P^{\T}AP)+\tr((P^{\T}D)^2)$
where $0\preceq A\in\bbR^{n\times n}$ and
$D\in\bbR^{n\times k}$. It can be verified that $\scrH(P)=2AP+2DP^{\T}D$ (see also \eqref{eq:partD-lin-pow} in the next section). Suppose now that
\eqref{eq:NPDo-assume} holds for $P,\, \what P\in\STM{k}{n}$, or equivalently,
\begin{equation}\label{eq:motivation:NPDo-1}
2\tr(\what P^{\T}AP)+2\tr(\what P^{\T}DP^{\T}D)\ge 2\tr(P^{\T}AP)+2\tr((P^{\T}D)^2)+\eta.
\end{equation}
The right-hand side of this inequality seems relatable to $f(P)$, but $P$ and $\what P$ are coupled together in
its left-hand side. Somehow we have to separate them in order to establish the desired inequality
$f(\wtd P)\ge f(P)+\omega\eta$ as demanded by {\bf the NPDo Ansatz}. Indeed this is what we will do next.
Let $X=A^{1/2}\what P$ and $Y=A^{1/2}P$ where $A^{1/2}$ is the unique positive semidefinite square root of $A$. By Lemma~\ref{lm:vN-tr-ineq-ext2},
we get
\begin{equation}\label{eq:motivation:NPDo-2}
2\tr(\what P^{\T}AP)=2\tr(X^{\T}Y)\le\tr(X^{\T}X)+\tr(Y^{\T}Y)=\tr(\what P^{\T}A\what P)+\tr(P^{\T}AP),
\end{equation}
successfully separating $P$ and $\what P$ from their coupling by $\tr(\what P^{\T}AP)$.
Turning to $\tr(\what P^{\T}DP^{\T}D)$,
we assume that $P^{\T}D\succeq 0$, i.e., $P\in\bbP=\STM{k}{n}_{D+}$,
and let $Q\in\STM{k}{k}$ be an orthonormal polar factor
of $\what P^{\T}D$ and hence $Q^{\T}(\what P^{\T}D)\succeq 0$, implying $\wtd P=\what PQ\in\bbP$.
We get
\begin{align}
2\tr(\what P^{\T}DP^{\T}D)
  &\le 2\|\what P^{\T}DP^{\T}D\|_{\tr} \qquad (\mbox{by Lemma~\ref{lm:maxtrace}})\nonumber \\
  &\le\tr((Q^{\T}\what P^{\T}D)^2)+\tr((P^{\T}D)^2) \qquad (\mbox{by Lemma~\ref{lm:vN-tr-ineq-ext1}}) \nonumber \\
  &=\tr((\wtd P^{\T}D)^2)+\tr((P^{\T}D)^2).  \label{eq:motivation:NPDo-3}
\end{align}
Combine \eqref{eq:motivation:NPDo-1}, \eqref{eq:motivation:NPDo-2}, and \eqref{eq:motivation:NPDo-3} to get
$f(\wtd P)\ge f(P)+\eta$ upon noticing $\tr(\wtd P^{\T}A\wtd P)=\tr(\what P^{\T}A\what P)$.
We observe the critical conditions: $P^{\T}D\succeq 0$ and $Q^{\T}\what P^{\T}D\succeq 0$
that ensure \eqref{eq:motivation:NPDo-3}, which we use to
separate $P$ and $\what P$ from their coupling by $\tr(\what P^{\T}DP^{\T}D)$.
The first condition $P^{\T}D\succeq 0$ can be fulfilled by simply starting with $P\in\bbP$, while the second
condition $Q^{\T}\what P^{\T}D\succeq 0$ is made possible by the chosen $Q$ and, as a byproduct, $\wtd P\in\bbP$, too.
Besides this role of making $Q^{\T}\what P^{\T}D\succeq 0$, $Q$ also
increases the objective value as a result of the two inequality signs in the derivation of
\eqref{eq:motivation:NPDo-3}.
In our later use of
\eqref{eq:NPDo-assume}, we begin with some $P\in\bbP\subseteq\STM{k}{n}$ and then find some $\what P\in\STM{k}{n}$
such that $\eta>0$ in \eqref{eq:NPDo-assume}, and therefore having a flexibility of
judiciously choosing a proper $Q$ becomes a logical necessity.
\end{example}

\begin{remark}\label{rk:NPDo-assume}
A few comments on {\bf the NPDo Ansatz} are in order.
\begin{enumerate}[(i)]


  \item When $f(P)$ is right-unitarily invariant,
        it suffices to take $Q=I_k$ and $\wtd P=\what P$ because $f(\what P)=f(\wtd P)$
        regardless of $Q$.
        Introducing subset $\bbP$ of $\STM{k}{n}$ and judiciously choosing $Q$
        are for generality in order to deal with the case when $f(P)$
        is not right-unitarily invariant, e.g., the one in Example~\ref{eg:motivation:NPDo} and those from  Table~\ref{tbl:obj-funs} in section~\ref{sec:intro} that involve
        $D$ or $D_i$.
        Throughout this paper, we will assume that $\bbP$ is
        as inclusive as necessary to allow our proving arguments to go through. In particular, at the minimum, $\bbP$ should contain one or more maximizers of the associated optimization problem~\eqref{eq:main-opt}.

  \item For computational purposes, it is necessary to have an efficient way
        to construct $Q$ in the ansatz. That is often the case when
        it comes to common concrete objective functions $f$ that are in use today.
        In our later development, either
        a proper $\bbP$ can maximally increase the value of objective function $f$, e.g.,
        when $\tr\big((P^{\T}D)^m\big)$  for $m\ge 1$ is involved, or
        we have to have it for our theoretical proofs to go through. In fact for $\tr\big((P^{\T}D)^m\big)$,
        we may take $\bbP=\STM{k}{n}_{D+}$,
        and let $Q\in\STM{k}{k}$ be an orthonormal polar factor of
        $\what P^{\T}D$ to ensure
        $\wtd P^{\T}D=Q^{\T}(\what P^{\T}D)\succeq 0$.
        As a consequence, $(\wtd P^{\T}D)^m\succeq 0$ and $\|(\wtd P^{\T}D)^m\|_{\tr}=\tr((\wtd P^{\T}D)^m)$
        by Lemma~\ref{lm:maxtrace} and hence
        an orthonormal polar factor $Q$ of
        $\what P^{\T}D$ maximizes $\tr([(\what PZ)^{\T}D]^m)$ over $Z\in\STM{k}{k}$.
        Calculating this $Q$ via the SVD of $\what P^{\T}D\in\bbR^{k\times k}$ is efficient since $k$ is usually small
        (in the tens or no more than a couple of hundreds).
%
  \item It is tempting to stipulate  $f(\what P)\ge f(P)+\omega\eta$, but that is either false or just hard to prove, e.g.,
        for  the one in Example~\ref{eg:motivation:NPDo}. Often in our algorithms to solve~\eqref{eq:main-opt} iteratively, with $P$ being
        the current approximate maximizer, assuming {\bf the NPDo Ansatz}, we naturally
        compute $\what P$ that maximizes $\tr(X^{\T}\scrH(P))$ over $X\in\STM{k}{n}$. With that $\what P$, settling
        whether $f(\what P)\ge f(P)+\omega\eta$ or not can be a hard or even impossible task, for example,
        in Example~\ref{eg:motivation:NPDo} it is not clear if $f(\what P)\ge f(P)+\omega\eta$ at all.
   \item If $f(\cdot)$ extended to some convex set of $\bbR^{n\times k}$ that contains $\STM{k}{n}$ is convex, then $f(\cdot)$
         satisfies {\bf the NPDo Ansatz} with $\bbP=\STM{k}{n}$, $Q=I_k$, and $\omega=1$.
         This is because, by convexity of the extended $f(\cdot)$, we have
         $$
         f(\what P)\ge f(P)+\tr\big((\what P-P)^{\T}\scrH(P)\big)\ge f(P)+\eta,
         $$
         assuming \eqref{eq:NPDo-assume}.
\end{enumerate}
\end{remark}

\begin{table}[t]
\renewcommand{\arraystretch}{1.4}
\caption{\small {\bf The NPDo Ansatz} on objective functions in Table~\ref{tbl:obj-funs}}\label{tbl:obj-funs:NPD}
\centerline{\small
\begin{tabular}{|c|c|c|c|}
  \hline
    & $\scrH(P)$ & conditions &  by  \\ \hline\hline
SEP & $2AP$ & $A\succeq 0$ &  Thm.~\ref{thm:main-npd-cvx:T1a}  \\ \hline
MBSub & $2AP+D$ & $A\succeq 0$ & Thm.~\ref{thm:main-npd-cvx:T1a}, \cite{wazl:2022a}  \\ \hline
SumCT & $2[A_1P_1,\ldots,A_NP_N]+D$ & $A_i\succeq 0\,\forall i$ &  Thm.~\ref{thm:main-npd-cvx:T1a}, \cite{wazl:2022a}  \\ \hline
\multirow{2}{*}{TrCP} & \multirow{2}{*}{$2\sum_{i=1}^N\phi_i(\bx)A_iP$} & $A_i\succeq 0,\,\phi_i\ge 0\,\forall i$ &  \multirow{2}{*}{Thm.~\ref{thm:main-npd-cvx:T1a}}  \\
     &                                & convex $\phi$ &   \\ \hline
UMDS & $4\sum_{i=1}^NA_iPP^{\T}A_iP$ & $A_i\succeq 0\,\forall i$ &  Expl.~\ref{eg:T2}  \\ \hline
\multirow{2}{*}{DFT} & \multirow{2}{*}{$2AP+2\sum_{i=1}^n\phi_i(\bx)\be_i\be_i^{\T}P$} & $A\succeq 0,\,\phi_i\ge 0\,\forall i$ &  \multirow{2}{*}{Thm.~\ref{thm:main-npd-cvx:T1a}}  \\
     &                                & convex $\phi$ &   \\
  \hline
\multicolumn{4}{l}{\small * $\phi_i(\bx):=\partial\phi(\bx)/\partial x_i$  for $\bx=[x_i]$.} \\
\end{tabular}
}
\end{table}
As to the validity of  {\bf the NPDo Ansatz} on the objective functions in Table~\ref{tbl:obj-funs},
it holds for all, except for those
that involve quotients,  under reasonable conditions on the constant matrices and function $\phi$. Table~\ref{tbl:obj-funs:NPD}
details conditions under which
{\bf the NPDo Ansatz} holds, where the last column points to the places for justifications.
We point out that  we can take $\bbP=\STM{k}{n}$, $Q=I_k$, and $\omega=1$ for all  in Table~\ref{tbl:obj-funs:NPD}
but judicious choices of $\bbP$ and $Q$ can increase the values of objective functions more than $\omega\eta$ as stipulated by
{\bf the NPDo Ansatz} for MBSub and SumCT \cite[Theorem 5.2]{wazl:2022a}. Another benefit is its possibility
to steer away from a local maximum \cite[Remark~3.1(vi)]{wazl:2022a}.



The first immediate consequence of {\bf the NPDo Ansatz} is the following theorem that provides a characterization of
the maximizers of the associated optimization problem~\eqref{eq:main-opt}.

\begin{theorem}\label{thm:maxers-NPD}
Let $P_*\in\STM{k}{n}$ be a maximizer of \eqref{eq:main-opt}.
Suppose that {\bf the NPDo Ansatz} holds and $P_*\in\bbP$. Then
\eqref{eq:KKT} holds for $P=P_*$ and $\Lambda=\Lambda_*:=P_*^{\T}\scrH(P_*)\succeq 0$.
\end{theorem}


\begin{proof}
Any maximizer is a KKT point, and hence \eqref{eq:KKT} holds for $P=P_*$ and $\Lambda=\Lambda_*$.
Assume, to the contrary, that
$\Lambda_*=P_*^{\T}\scrH(P_*)\not\succeq 0$
(which means either $\Lambda_*$ is not symmetric or it is symmetric but indefinite or negative semidefinite).
Then by Lemma~\ref{lm:polar2max}, we have
$\tr(\what P^{\T}\scrH(P_*))=\|\scrH(P_*)\|_{\tr}\ge \tr(P_*^{\T}\scrH(P_*))+\eta$ for some $\eta>0$, where $\what P$ is an orthonormal polar factor of $\scrH(P_*)$.
By {\bf the NPDo Ansatz}, we can find $\wtd P=\what PQ\in\bbP$ such that
$f(\wtd P)\ge f(P_*)+\omega\eta> f(P_*)$, contradicting that $P_*$ is a maximizer.
\end{proof}

What this theorem says is that at a maximizer $P_*$, \eqref{eq:KKT} is a polar decomposition of $\scrH(P_*)$. Hence
solving \eqref{eq:main-opt} through its KKT condition is necessarily looking for
some $P_*$ so that \eqref{eq:KKT} is a polar decomposition. Since the matrix of which we are seeking
a  polar decomposition is a matrix-valued function that depends on its orthonormal polar factor, we
naturally call \eqref{eq:KKT} a {\em nonlinear polar decomposition with orthonormal polar factor dependency\/} (NPDo) of $\scrH(\cdot)$.

We note that $\scrH(P_*)$ has a unique polar decomposition if $\rank(\scrH(P_*))=k$ \cite{li:1995};
but it is not unique if $\rank(\scrH(P_*))<k$ \cite{high:2008,li:1993b,li:2014HLA}. However in the latter case,
it does not mean that any orthonormal polar factor of $\scrH(P_*)$, other than $P_*$, also satisfies \eqref{eq:KKT}, unless
$\scrH(\cdot)$ is right-unitarily invariant.

\subsection{SCF Iteration}\label{ssec:SCF4NPD}
The second immediate consequence of {\bf the NPDo Ansatz} is the global convergence of
an SCF iteration for solving optimization problem~\eqref{eq:main-opt} as outlined in \Cref{alg:SCF4NPDo}.
\begin{algorithm}[t]
\caption{NPDoSCF: NPDo \eqref{eq:KKT} solved by SCF} \label{alg:SCF4NPDo}
\begin{algorithmic}[1]
\REQUIRE Function $\scrH(P)\equiv\partial f(P)/\partial P$ satisfying {\bf the NPDo Ansatz}, $P^{(0)}\in\bbP$;
\ENSURE  an approximate maximizer of \eqref{eq:main-opt}. 
\FOR{$i=0,1,\ldots$ until convergence}
    \STATE compute $H_i=\scrH(P^{(i)})\in\bbR^{n\times k}$ and its thin SVD: $H_i=U_i\Sigma_iV_i^{\T}$;
    \STATE $\what P^{(i)}=U_iV_i^{\T}\in\STM{k}{n}$,
           an orthonormal polar factor of  $H_i$;
    \STATE calculate $Q_i\in\STM{k}{k}$ whose existence is stipulated by {\bf the NPDo Ansatz}
           and let $P^{(i+1)}=\what P^{(i)}Q_i\in\bbP$;
\ENDFOR
\RETURN the last $P^{(i)}$.
\end{algorithmic}
\end{algorithm}
This algorithm is similar to \cite[Algorithm 3.1]{wazl:2022a}, but the latter has more details that are dictated by
the particularity of objective function $f$ there.
We have a few general comments regarding the implementation of
\Cref{alg:SCF4NPDo} (NPDoSCF):
\begin{enumerate}[(1)]
  \item At Line 4 it refers to  {\bf the NPDo Ansatz} for the calculation of $Q_i$.
        Exactly how it is computed depends on the structure of $f$ at hand.
        If $f$ is right-unitarily invariant, we can simply take $Q_i=I_k$ as we commented in Remark~\ref{rk:NPDo-assume}(i).
        In Remark~\ref{rk:NPDo-assume}(ii),
        we commented on  the issue in the case when $f(P)$ involves and increases with $\tr\big((P^{\T}D)^m\big)$, e.g., the one in
        Example~\ref{eg:motivation:NPDo}, $Q_i$ can be taken to be
        an orthonormal polar factor of  $(\what P^{(i)})^{\T}D$. Later in section~\ref{sec:CVX-comp:NPDo} we will elaborate
        on how to choose $Q_i$ for a few convex compositions of  matrix-trace functions.
  \item A reasonable stopping criterion at Line 1 is
         \begin{equation}\label{eq:stop-1}
         \varepsilon_{\KKT}+\varepsilon_{\sym}:=\frac {\|\scrH(P)-P[P^{\T}\scrH(P)]\|_{\F}}{\xi}+\frac {\big\|[P^{\T}\scrH(P)]-[P^{\T}\scrH(P)]^{\T}\|_{\F}}{\xi}\le\epsilon,
         \end{equation}
         where $\epsilon$ is a given tolerance, and $\xi$ is some normalization quantity that should be
         designed according to the underlying $\scrH(P)$, but generically, $\xi=\|\scrH(P)\|_{\F}$, or any reasonable estimate of it,
         should work well.
         The significance of both $\varepsilon_{\KKT}$ and $\varepsilon_{\sym}$ is rather self-explanatory. In fact, we
         will call $\varepsilon_{\KKT}$ and $\varepsilon_{\sym}$ the {\em normalized residual for the KKT equation\/} \eqref{eq:KKT}
         and the {\em normalized residual for the symmetry\/} in $\Lambda=P^{\T}\scrH(P)$, respectively.
  \item Let us investigate the computational complexity per iterative step. Since how $H_i=\scrH(P^{(i)})$ and $Q_i$ are computed
        is generally problem-dependent, we will only examine the cost for all other operations.
      At Line 2, the thin SVD of $H_i\in\bbR^{n\times k}$ is often computed in two steps: compute a thin QR decomposition
      $H_i=WR$ and then the SVD of $R\in\bbR^{k\times k}$ followed by the product of $W$ with the left singular vector matrix
      of $R$. Hence the overall cost per SCF iterative step, stemming from the SVD of $R$ and three matrix products of
      an $n$-by-$k$ matrix with an $k$-by-$k$ matrix, is about $6nk^2+20k^3$ flops \cite[p.493]{govl:2013} which is linear in $n$ for small $k$.
\end{enumerate}
Convergence of \Cref{alg:SCF4NPDo} will be detailed later in \Cref{thm:cvg4SCF4NPDo:strong} in subsection~\ref{ssec:cvg4SCF4NPDo}.
%
It is shown that as the SCF iteration progresses, the value of the objective function
monotonically increases, any accumulation point of the generated approximation sequence
satisfies the necessary conditions in Theorem~\ref{thm:maxers-NPD} for a global maximizer, and
under certain conditions, the accumulation point can be proved to be the limit point of the entire approximation sequence.
In short, the NPDo approach is guaranteed to work.

\begin{remark}\label{rk:SCF4NPDo-var}
Currently as stated, $\what P^{(i)}$ is computed via the thin SVD of $\scrH(P^{(i)})$, and it is efficient if $k\ll n$.
Conceivably, $\what P^{(i)}$ can also be computed iteratively by other means (see, e.g., \cite{nabg:2010} and references therein), especially when $k\not\ll n$ and $n$ is not small, but how accurate does it have to be computed to?
Our convergence analysis in \Cref{thm:cvg4SCF4NPDo:strong} sheds light on the question where it lays out
the level of accuracy in $\what P^{(i)}$ to ensure desirable convergence properties.
\end{remark}

\subsection{Acceleration by LOCG and Convergence}\label{ssec:LOCG:NPD}
Although  \Cref{alg:SCF4NPDo}, an SCF iteration for solving
NPDo \eqref{eq:KKT}, is proved always convergent to KKT points under {\bf the NPDo Ansatz},
it may take many SCF iterations to converge to a solution with desired accuracy and
that can be costly for large scale problems, even though the complexity per SCF iterative step is linear in $n$. In fact, for $f(P)=\tr(P^{\T}AP)$ with $A\succeq 0$, \Cref{alg:SCF4NPDo}
is simply the subspace iteration which converges
linearly at the rate of $\lambda_{k+1}(A)/\lambda_k(A)$. This rate is $1$  if $\lambda_{k+1}(A)=\lambda_k(A)$, indicating possible divergence, but strictly less than $1$ otherwise. In the latter case, although the convergence is guaranteed,
it can be slow when $\lambda_{k+1}(A)<\lambda_k(A)$ only slightly such that $\lambda_{k+1}(A)/\lambda_k(A)\approx 1$ \cite{demm:1997,govl:2013}.
In \cite{wazl:2022a}, acceleration by a locally optimal conjugate gradient technique (LOCG) was demonstrated
to be rather helpful to speed things up for maximizing SumCT. The same idea can be used to speed up
\Cref{alg:SCF4NPDo}, too.
In this subsection, we will explain the idea, which draws inspiration from optimization \cite{poly:1987,taka:1965} and has been
increasingly used in NLA for linear systems and eigenvalue problems \cite{beli:2022,imlz:2016,knya:2001,li:2015,yali:2021}.

\subsubsection*{A variant of LOCG for Acceleration}
Without loss of generality, let $P^{(-1)}\in\STM{k}{n}$ be the approximate maximizer of \eqref{eq:main-opt}
from the very previous iterative step, and $P\in\STM{k}{n}$ the current approximate maximizer.
We are now looking for the next approximate maximizer
$P^{(1)}$, along the line of LOCG, according to
\begin{equation}\label{eq:LOCG}
P^{(1)}=\arg\max_{Y\in\STM{k}{n}}f(Y),\,\,\mbox{s.t.}\,\, \cR(Y)\subseteq\cR([P,\scrR(P),P^{(-1)}]),
\end{equation}
where
\begin{equation}\label{eq:R(P)}
\scrR(P):=\grad f_{|{\STM{k}{n}}}(P)=\scrH(P)-P\cdot\sym(P^{\T}\scrH(P)) 
\end{equation}
by \eqref{eq:gradOnStM}.
Initially for the first iteration, we don't have $P^{(-1)}$ and
it is understood that $P^{(-1)}$ is absent from \eqref{eq:LOCG}, i.e.,
simply $\cR(Y)\subseteq\cR([P,\scrR(P)])$.

We still have to numerically solve \eqref{eq:LOCG}. For that purpose, let $W\in\STM{\hat n}{n}$ be an orthonormal basis matrix of subspace
$\cR([P,\scrR(P),P^{(-1)}])$. Generically, $\hat n=3k$ but $\hat n<3k$ can happen.
It can be implemented by the Gram-Schmidt orthogonalization process, starting with orthogonalizing the columns of $\scrR(P)$ against $P$ since
$P\in\STM{k}{n}$ already. In MATLAB, to fully take advantage of its optimized functions, we simply set
$W=[\scrR(P),P^{(-1)}]$ (or $W=\scrR(P)$ for the first iteration) and then  do
$$
\framebox{
\begin{minipage}{4.9cm}
\tt  W=W-P*(P'*W); W=orth(W); \\
\tt W=W-P*(P'*W); W=orth(W);\\
\tt W=[P,W];
\end{minipage}
}
$$
where the first two lines  perform the classical Gram-Schmidt orthogonalization twice to almost ensure that
the resulting  columns of $W$ are fully orthogonal to the columns of $P$ at the end of the second line,
and {\tt orth} is a MATLAB function for
orthogonalization\footnote{Another option is to use MATLAB's thin {\tt qr}:
    {\tt [W,$\sim$]=qr(W,0)}.
    }.
It is important to note that the first $k$ columns of
the final $W$ are the same as those of $P$.

Now it follows from $\cR(Y)\subseteq\cR([P,\scrR(P),P^{(-1)}])=\cR(W)$ that in \eqref{eq:LOCG}
\begin{subequations}\label{eq:LOCGsub}
\begin{equation}\label{eq:LOCGsub:Y}
Y=WZ\quad\mbox{for $Z\in\STM{k}{\hat n}$}.
\end{equation}
Problem \eqref{eq:LOCG} becomes
\begin{equation}\label{eq:LOCGsub-1}
Z_{\opt}=\arg\max_{Z\in\STM{k}{m}} \wtd f(Z)\quad\mbox{with}\quad\wtd f(Z):=f(WZ),
\end{equation}
\end{subequations}
and $P^{(1)}=WZ_{\opt}$ for \eqref{eq:LOCG}.
It can verified that
\begin{subequations}\label{eq:KKT-reduced}
\begin{equation}\label{eq:KKT-reduced-1}
\wtd\scrH(Z):=\frac {\partial \wtd f(Z)}{\partial Z}
   =\left.W^{\T}\frac {\partial f(Y)}{\partial Y}\right|_{P=WZ}=W^{\T}\scrH(WZ),
\end{equation}
and the first order optimality condition for \eqref{eq:LOCGsub-1} is
\begin{equation}\label{eq:KKT-reduced-2}
\wtd\scrH(Z)=Z\wtd\Lambda
\quad \mbox{with} \quad
\wtd\Lambda^{\T}=\wtd\Lambda\in\bbR^{k\times k},\quad Z\in\STM{k}{\hat n}.
\end{equation}
\end{subequations}

\begin{lemma}\label{lm:NPDoAssump.-reduced}
Suppose that {\bf the NPDo Ansatz} holds for $f$, and let $\bbZ:=W^{\T}\bbP\subseteq\STM{k}{\hat n}$. If
$W\bbZ\subseteq\bbP$, then {\bf the NPDo Ansatz} holds for $\wtd f$ defined in \eqref{eq:LOCGsub-1}.
\end{lemma}

\begin{proof}
Let  $Z\in\bbZ$ and $\what Z\in\STM{k}{m}$ such that
\begin{equation}\label{eq:NPDo-assume:reduced}
\tr(\what Z^{\T}\wtd\scrH(Z))\ge\tr(Z^{\T}\wtd\scrH(Z))+\eta
\quad\mbox{for some $\eta\in\bbR$}.
\end{equation}
Set $P=WZ\in\bbP$ (because of $W\bbZ\subseteq\bbP$) and $\what P=W\what Z\in\STM{k}{n}$. Noticing that $\wtd\scrH(Z))=W^{\T}\scrH(WZ)$, we have
\eqref{eq:NPDo-assume} from \eqref{eq:NPDo-assume:reduced}. By {\bf the NPDo Ansatz} for $f$, there exists $Q\in\STM{k}{k}$ such that
$\wtd P=\what PQ=W(\what ZQ)=:W\wtd Z\in\bbP$ and
$f(\wtd P)\ge f(P)+\omega\eta$. Hence,
$$
\wtd f(\wtd Z)=\wtd f(\what ZQ)=f(\wtd P)\ge f(P)+\omega\eta
  =f(WZ)+\omega\eta
  =\wtd f(Z)+\omega\eta.
$$
Note also $\wtd Z=W^{\T}\wtd P\in\bbZ$. Hence {\bf the NPDo Ansatz} holds for $\wtd f$.
\end{proof}

As a consequence of this lemma and the results in subsections~\ref{ssec:NPDAssum} and \ref{ssec:SCF4NPD}, \Cref{alg:SCF4NPDo}
is applicable to compute $Z_{\opt}$ of \eqref{eq:LOCGsub-1} via NPDo \eqref{eq:KKT-reduced-2}.
We outline the resulting
method in Algorithm~\ref{alg:SCF4NPDo+LOCG}, which is an inner-outer iterative scheme for \eqref{eq:main-opt},
where at Line~4
any other method, if known, can also be inserted to replace \Cref{alg:SCF4NPDo} to
solve \eqref{eq:LOCGsub-1}.

\begin{algorithm}[t]
\caption{NPDoLOCG:  NPDo \eqref{eq:KKT} solved by LOCG
  } \label{alg:SCF4NPDo+LOCG}
\begin{algorithmic}[1]\REQUIRE Function $\scrH(P)\equiv\partial f(P)/\partial P$ satisfying {\bf the NPDo Ansatz}, $P^{(0)}\in\bbP$;
\ENSURE  an approximate maximizer of \eqref{eq:main-opt}.
\STATE $P^{(-1)}=[\,]$; \% null matrix
\FOR{$i=0,1,\ldots$ until convergence}
    \STATE compute $W_i\in\STM{\hat n}{n}$ such that $\cR(W_i)=\cR([P^{(i)},\scrR(P^{(i)}),P^{(i-1)}])$
           and $P^{(i)}$ occupies
           the first $k$ columns of $W_i$;
    \STATE solve \eqref{eq:LOCGsub-1} via NPDo \eqref{eq:KKT-reduced} for $Z_{\opt}$ by  \Cref{alg:SCF4NPDo}
           with initially $Z^{(0)}$ being the first $k$ columns of $I_{\hat n}$, or
           approximately such that
           $\wtd f(Z_{\opt})\ge\wtd f(Z^{(0)})+\omega\eta$ for some $\eta>0$;
    \STATE $P^{(i+1)}=W_iZ_{\opt}$;
\ENDFOR
\RETURN the last $P^{(i)}$.
\end{algorithmic}
\end{algorithm}

\begin{remark}\label{rk:SCF4NPDo+LOCG}
A few  comments regarding Algorithm~\ref{alg:SCF4NPDo+LOCG} are in order.
\begin{enumerate}[(i)]
  \item It is important to compute $W$ at Line~4 in such a way, as explained moments ago, that its first $k$ columns are exactly the same as those of $P^{(i)}$.
        As $P^{(i)}$ converges, conceivably $P^{(i+1)}$ changes little from $P^{(i)}$ and hence $Z_{\opt}$
        is increasingly close to the first $k$ columns of $I_{\hat n}$. This explains the choice of $Z^{(0)}$
        at Line 4.
%
  \item Another area of improvement is to solve \eqref{eq:LOCGsub-1} with an accuracy, comparable but fractionally better than the
        current $P^{(i)}$ as an approximate solution of \eqref{eq:main-opt}.  Specifically, if we use
        \eqref{eq:stop-1} at Line 2 to stop the for-loop: Lines 2--6 of  Algorithm~\ref{alg:SCF4NPDo+LOCG}, with tolerance $\epsilon$, then instead of using the same
        $\epsilon$ for \Cref{alg:SCF4NPDo} at its line 1,
        we can use a fraction, say $1/4$,
        of $\varepsilon_{\KKT}+\varepsilon_{\sym}$ evaluated at the current approximation $P=P^{(i)}$
        as stopping tolerance,
        when \Cref{alg:SCF4NPDo} is called upon at Line 4 of Algorithm~\ref{alg:SCF4NPDo+LOCG}.
\end{enumerate}
\end{remark}

Whether Algorithm~\ref{alg:SCF4NPDo+LOCG}
speeds up \Cref{alg:SCF4NPDo} depends on two factors at the runtime:
1) it takes significantly fewer the number of outer iterative steps than the number of SCF iterative steps
by \Cref{alg:SCF4NPDo} as it does without acceleration, and 2) the cost per SCF step on NPDo \eqref{eq:KKT-reduced} is significantly less than that on NPDo \eqref{eq:KKT}. Both factors are materialized for SumCT (see \cite[Example 4.1]{wazl:2022a}).

\subsection{Convergence Analysis}\label{ssec:cvg4SCF4NPDo}
\Cref{alg:SCF4NPDo} is stated under {\bf the NPDo Ansatz}, with which it can be shown that
\begin{equation}\label{eq:cvg4SCF4NPDo:obj-inc}
f(P^{(i+1)})\ge f(P^{(i)})
    +\omega\Big[\|\scrH(P^{(i)})\|_{\tr}-\tr\big([P^{(i)}]^{\T}\scrH(P^{(i)})\big)\Big],
\end{equation}
for every SCF iterative step, where $\omega$ is the constant from {\bf the NPDo Ansatz} and it is independent of
$i$. This inequality alone, if satisfied, turns out to have critical consequences in convergence of
the sequence $\{P^{(i)}\in\STM{k}{n}\}_{i=0}^{\infty}$, regardless whether {\bf the NPDo Ansatz} holds or not, and
is thus of interest in its own right.

In what follows, we first establish a general convergence theory
under an assumption like this inequality and then apply it to derive some
convergence results for \Cref{alg:SCF4NPDo,alg:SCF4NPDo+LOCG} and a variant of \Cref{alg:SCF4NPDo} in which
each $\what P^{(i)}$ is approximately calculated, not necessarily by the thin SVD as stated there.

\begin{theorem}\label{thm:cvg4SCF4NPDo:gl}
Let $\{P^{(i)}\in\STM{k}{n}\}_{i=0}^{\infty}$
be a sequence on $\STM{k}{n}$ such that
\begin{equation}\label{eq:cvg4SCF4NPDo:gl:obj-inc}
f(P^{(i+1)})\ge f(P^{(i)})
    +\kappa \underbrace{\Big[\|\scrH(P^{(i)})\|_{\tr}-\tr\big([P^{(i)}]^{\T}\scrH(P^{(i)})\big)\Big]}_{=:\eta_{*i}},
\end{equation}
where $\kappa>0$ is a constant, independent of $i$, and let $P_*$
be an accumulation point of the sequence $\{P^{(i)}\}_{i=0}^{\infty}$.
The following statements hold.
\begin{enumerate}[{\rm (i)}]
  \item $\cR(P_*)$ is an accumulation point of the sequence $\{\cR(P^{(i)})\}_{i=0}^{\infty}$;
  \item The sequence $\{f(P^{(i)})\}_{i=0}^{\infty}$ is monotonically increasing and convergent;
  \item The KKT condition \eqref{eq:KKT} is satisfied for $P=P_*$ and $\Lambda=\Lambda_*:=P_*^{\T}\scrH(P_*)\succeq 0$;
  \item We have two convergent
        series
        \begin{subequations}\label{eq:cvg4SCF4NPDo:series:gl}
        \begin{align}
        \sum_{i=1}^{\infty}\sigma_{\min}(\scrH(P^{(i)}))\,\big\|\sin\Theta\big(\cR(\scrH(P^{(i)})),\cR(P^{(i)})\big)\big\|_{\F}^2
                      &<\infty,    \label{eq:cvg4SCF4NPDo:series:gl-1} \\
        \sum_{i=1}^{\infty}\sigma_{\min}(\scrH(P^{(i)}))\,
                  \frac {\big\|\scrH(P^{(i)})-P^{(i)}\cdot\sym\big([P^{(i)}]^{\T}\scrH(P^{(i)})\big)\big\|_{\F}^2}
                        {\big\|\scrH(P^{(i)})\big\|_{\F}^2}
                      &<\infty,                 \label{eq:cvg4SCF4NPDo:series:gl-2} \\
        \sum_{i=1}^{\infty}\sigma_{\min}(\scrH(P^{(i)}))\,
                  \big\|P^{(i)}-P_*^{(i)}\big\|_{\F}^2
                      &<\infty,                 \label{eq:cvg4SCF4NPDo:series:gl-3}
        \end{align}
        \end{subequations}
        where $\Theta(\cdot,\cdot)$ is the diagonal matrix of the canonical angles between two subspaces (see appendix~\ref{sec:angle-space}), and $P_*^{(i)}$ is the orthonormal polar factor of $\scrH(P^{(i)})$.
\end{enumerate}
\end{theorem}

\begin{proof}
See appendix~\ref{sec:Proof4NPDoCVG}.
\end{proof}

\Cref{thm:cvg4SCF4NPDo:gl}(iv) has interesting and useful consequences.

\begin{corollary}\label{cor:cvg4SCF4NPDo:gl}
To the conditions of \Cref{thm:cvg4SCF4NPDo:gl} add this: $\sigma_{\min}(\scrH(P^{(i)}))$ is eventually bounded below away from $0$
uniformly\footnote {By which we mean that there exist a constant $c>0$ and an integer $K$ such that
  $\sigma_{\min}(\scrH(P^{(i)}))\ge c$ for all $i\ge K$.}.
Then
$$
\lim_{i\to\infty}\frac {\big\|\scrH(P^{(i)})-P^{(i)}\cdot\sym\big([P^{(i)}]^{\T}\scrH(P^{(i)})\big)\big\|_{\F}}
                       {\big\|\scrH(P^{(i)})\big\|_{\F}} =0.
$$
\end{corollary}

\Cref{cor:cvg4SCF4NPDo:gl} implies that increasingly
$\scrH(P^{(i)})\approx P^{(i)}\cdot\sym\big([P^{(i)}]^{\T}\scrH(P^{(i)})\big)$ towards
a polar decomposition of $\scrH(P^{(i)})$, which means that $P^{(i)}$ becomes a more and more accurate approximate solution
to NPDo \eqref{eq:KKT}, even in the absence of knowing whether the entire sequence $\{P^{(i)}\}_{i=0}^{\infty}$ converges or not.
The latter does require additional specifications as to how the sequence $\{P^{(i)}\}_{i=0}^{\infty}$ is generated
(see \Cref{thm:cvg4SCF4NPDo:strong}(c) below).
On the other hand, what if $\sigma_{\min}(\scrH(P^{(i)}))$ is not eventually bounded below away from $0$, e.g.,
$\sigma_{\min}(\scrH(P^{(i)}))$ behaves like $1/i$? We may still be able to draw useful conclusions, dependent on circumstances.
As an example, we state yet another corollary of \Cref{thm:cvg4SCF4NPDo:gl}(iv) with the help of \Cref{lm:series}.

\begin{corollary}\label{cor:cvg4SCF4NPDo:gl'}
In \Cref{thm:cvg4SCF4NPDo:gl}, if $\sum_{i=1}^{\infty}\sigma_{\min}(\scrH(P^{(i)}))$ diverges then there exists a convergent subsequence of
$$
\left\{\frac {\big\|\scrH(P^{(i)})-P^{(i)}\cdot\sym\big([P^{(i)}]^{\T}\scrH(P^{(i)})\big)\big\|_{\F}}
                       {\big\|\scrH(P^{(i)})\big\|_{\F}}\right\}_{i=1}^{\infty}
$$
that goes to $0$.
\end{corollary}


We emphasize again that \Cref{thm:cvg4SCF4NPDo:gl} does not assume  {\bf the NPDo Ansatz} and does not specify how the sequence $\{P^{(i)}\}_{i=0}^{\infty}$ is generated so long as \eqref{eq:cvg4SCF4NPDo:gl:obj-inc} is satisfied.
In the case of \Cref{alg:SCF4NPDo} where {\bf the NPDo Ansatz} is the prerequisite, we have
\eqref{eq:cvg4SCF4NPDo:obj-inc}, i.e., \eqref{eq:cvg4SCF4NPDo:gl:obj-inc} holds with $\kappa=\omega$.

\begin{remark}\label{rk:NPDo-low-rank}
There is a minor issue when it comes to any terms for the case $\rank\big(\scrH(P^{(i)})\big)<k$
in series in \eqref{eq:cvg4SCF4NPDo:series:gl}. Namely, if $\rank\big(\scrH(P^{(i)})\big)<k$, then
$\cR(\scrH(P^{(i)})$ and $\cR(P^{(i)})$ do not have an equal dimension and also
the orthonormal polar factor $P_*^{(i)}$ of $\scrH(P^{(i)})$ is not unique.
The non-uniqueness of $P_*^{(i)}$ poses no mathematical issue to \eqref{eq:cvg4SCF4NPDo:series:gl-3}
because $\sigma_{\min}(\scrH(P^{(i)}))=0$ implying the related term vanishes.
Rigourously, $\Theta(\cdot,\cdot)$ defined in appendix~\ref{sec:angle-space} for two subspaces of equal dimension
is no longer meaningful if $\rank\big(\scrH(P^{(i)})\big)<k$, but then again
we can drop off any term associated with $\rank\big(\scrH(P^{(i)})\big)<k$ from the series in \eqref{eq:cvg4SCF4NPDo:series:gl-1}.
\end{remark}

Next we state our convergence results for \Cref{alg:SCF4NPDo,alg:SCF4NPDo+LOCG}.
As stated in \Cref{alg:SCF4NPDo}, each
$\what P^{(i)}$ is computed as an orthonormal polar factor of $\scrH(P^{(i)})$ via its thin SVD, which nowadays
is considered a direct method and is efficient, too, particularly in the case $k\ll n$. It ensures
$[\what P^{(i)}]^{\T}\scrH(P^{(i)})\succeq 0$ and $\tr\big([\what P^{(i)}]^{\T}\scrH(P^{(i)})\big)=\|\scrH(P^{(i)})\|_{\tr}$
in the working precision, making the corresponding $\eta$ in {\bf the NPDo Ansatz} the largest possible, namely,
$\|\scrH(P^{(i)})\|_{\tr}-\tr\big([P^{(i)}]^{\T}\scrH(P^{(i)})\big)$.
In attempting to be
more general, we will consider a variation of \Cref{alg:SCF4NPDo} where each $\what P^{(i)}$ is approximately calculated, not necessarily by the thin SVD as stated there, but accurately enough so that
\begin{align}
\eta_i:=\,&\tr\big([\what P^{(i)}]^{\T}\scrH(P^{(i)})\big)-\tr\big([P^{(i)}]^{\T}\scrH(P^{(i)})\big)
             \label{eq:NPDo:eta(i):dfn} \\
      \ge\,&c\Big[\big\|\scrH(P^{(i)})\big\|_{\tr}-\tr\big([P^{(i)}]^{\T}\scrH(P^{(i)})\big)\Big]
             \label{eq:NPDo:eta(i):weak}
\end{align}
for some constant $c>0$, independent of $i$. Current $\what P^{(i)}$ as stated in \Cref{alg:SCF4NPDo} makes $c=1$, and thus any convergence result under \eqref{eq:NPDo:eta(i):weak} remains valid for \Cref{alg:SCF4NPDo}.

\begin{theorem}\label{thm:cvg4SCF4NPDo:strong}
Suppose that {\bf the NPDo Ansatz} holds, and let the sequence $\{P^{(i)}\}_{i=0}^{\infty}$ be generated
by a variation of \Cref{alg:SCF4NPDo}, in which $\what P^{(i)}$ at Line 3 is approximately computed, not necessarily by the thin SVD of $H_i$ but by any method, such that \eqref{eq:NPDo:eta(i):weak} is satisfied,
and let $P_*$ be an accumulation
point of the sequence $\{P^{(i)}\}_{i=0}^{\infty}$.
The following statements hold.
\begin{enumerate}[{\rm (a)}]
  \item All conclusions in \Cref{thm:cvg4SCF4NPDo:gl} are valid;
  \item Suppose that $\cR(P_*)$ is an isolated accumulation point of $\{\cR(P^{(i)})\}_{i=0}^{\infty}$. If
        \begin{equation}\label{eq:full-rank:assume}
        \rank(\scrH(P_*Q))=k\quad\mbox{for any $Q\in\STM{k}{k}$},
        \end{equation}
        then
        the entire sequence $\{\cR(P^{(i)})\}_{i=0}^{\infty}$ converges to $\cR(P_*)$;
  \item 
        Suppose that $P_*$ is an isolated accumulation point of $\{P^{(i)}\}_{i=0}^{\infty}$ and that
        \begin{equation}\label{eq:full-rank:assume'}
        \rank(\scrH(P_*))=k.
        \end{equation}
        Then the entire sequence $\{P^{(i)}\}_{i=0}^{\infty}$ converges to $P_*$ if one of the following
        two assumptions holds:
        \begin{enumerate}[{\rm (c1)}]
          \item $Q_i$ converges to $I_k$ as $i\to\infty$;
          \item $f(P_*)>f(P)$ for any $P\ne P_*$ and $\cR(P)=\cR(P_*)$, i.e.,
                $f(P)$ has a unique maximizer in the orbit $\{P_*Q\,:\,Q\in\STM{k}{k}\}$.
        \end{enumerate}
\end{enumerate}
\end{theorem}

\begin{proof}
See appendix~\ref{sec:Proof4NPDoCVG}.
\end{proof}

In \Cref{alg:SCF4NPDo}, $P^{(i+1)}=\what P^{(i)}Q_i$ and thus $\cR(P^{(i+1)})=\cR(\what P^{(i)})$.
If also $\what P^{(i)}$ is computed by the thin SVD then $\cR(\what P^{(i)})=\cR(\scrH(P^{(i)}))$ provided $\rank\big(\scrH(P^{(i)})\big)=k$.
If that is the case, $\cR(\scrH(P^{(i)}))$ in \eqref{eq:cvg4SCF4NPDo:series:gl-1} can be replaced with
$\cR(P^{(i+1)})$.
In the variation of Algorithm~\ref{alg:SCF4NPDo} as described in
Theorem~\ref{thm:cvg4SCF4NPDo:strong}, $\what P^{(i)}$ is computed to make
$\tr\big([\what P^{(i)}]^{\T}\scrH(P^{(i)})\big)$
large enough so that \eqref{eq:NPDo:eta(i):weak} is satisfied.
While it is nice to have Theorem~\ref{thm:cvg4SCF4NPDo:strong}, there is a limitation, i.e., the difficulty in verifying
\eqref{eq:NPDo:eta(i):weak} numerically as $\big\|\scrH(P^{(i)})\big\|_{\tr}$ is unknown without computing the thin SVD of
$\scrH(P^{(i)})$.
Nonetheless,
it still offers guidance and understanding when $\what P^{(i)}$ has to be computed iteratively.

As a corollary of \Cref{thm:cvg4SCF4NPDo:strong}(a), we find that
{\bf the NPDo Ansatz} is a sufficient condition for NPDo \eqref{eq:KKT} to have a solution because
there always exists an accumulation point $P_*$ of the sequence $\{P^{(i)}\}_{i=0}^{\infty}$
in $\STM{k}{n}$.

\begin{corollary}\label{cor:cvg4SCF4NPDo}
Under {\bf the NPDo Ansatz},
NPDo \eqref{eq:KKT} is solvable, i.e., there exists $P\in\STM{k}{n}$ such that $\Lambda=P^{\T}\scrH(P)\succeq 0$
and \eqref{eq:KKT} holds.
\end{corollary}

In the case when objective function $f$ is right-unitarily invariant, assumption (c2) of
\Cref{thm:cvg4SCF4NPDo:strong}(c) clearly does not hold.
In such a case, computing
$\cR(P_*)$ may be the ultimate goal because each maximizer $P_*$ is really
a representative from the orbit $\{P_*Q\,:\,Q\in\STM{k}{k}\}$.
Given $Q\in\bbR^{k\times k}$, let $g(P)=f(PQ)$. It can be verified that
$$
\frac {\partial g(P)}{\partial P}=\left.\frac {\partial f(\what P)}{\partial\what P}\right|_{\what P=PQ}\,Q^{\T}=\scrH(PQ)\,Q^{\T}.
$$
Thus if $f$ is right-unitarily invariant, then $g(P)\equiv f(P)$ and thus $\scrH(P)=\scrH(PQ)\,Q^{\T}$; if also
$Q\in\STM{k}{k}$, then we get
$$
\scrH(PQ)=\scrH(P)\,Q
$$
and as a result, condition \eqref{eq:full-rank:assume} is equivalently
to $\rank(\scrH(P_*))=k$. Also if $f$ is right-unitarily invariant, then there is no need to
choose $Q_i$ other than simply as $I_k$, which makes assumption (c1) in \Cref{thm:cvg4SCF4NPDo:strong}(c) automatically
satisfied, yielding

\begin{corollary}\label{cor:cvg4SCF4NPDo'}
Suppose that {\bf the NPDo Ansatz} holds and that objective $f$ is right-unitarily invariant, and let the sequence $\{P^{(i)}\}_{i=0}^{\infty}$ be generated by the variation of \Cref{alg:SCF4NPDo}, as described in
\Cref{thm:cvg4SCF4NPDo:strong}, and  with $Q_i=I_k$ for all iterative steps and $P_*$ be an accumulation
point of the sequence. If $P_*$ is an isolated accumulation point of $\{P^{(i)}\}_{i=0}^{\infty}$ and if
\eqref{eq:full-rank:assume'} holds, then the entire sequence $\{P^{(i)}\}_{i=0}^{\infty}$ converges to $P_*$.
\end{corollary}

Next we will perform a convergence analysis for Algorithm~\ref{alg:SCF4NPDo+LOCG}, assuming that
at least one SCF iteration is carried out in \Cref{alg:SCF4NPDo} when it is called
        to compute $Z_{\opt}$ at Line~4 of Algorithm~\ref{alg:SCF4NPDo+LOCG}.
This assumption is rather realistic.
In actual computation, as we explained
in Remark~\ref{rk:SCF4NPDo+LOCG}(ii), the computed $Z_{\opt}$ should be sufficiently more accurate
as an approximate solution for \eqref{eq:LOCGsub} than $P^{(i)}$ as an approximate solution for the original
problem \eqref{eq:main-opt} at that moment.

\begin{theorem}\label{thm:cvg4SCFvLOCG}
Suppose that {\bf the NPDo Ansatz} holds, and let sequence $\{P^{(i)}\}_{i=0}^{\infty}$ be generated by Algorithm~\ref{alg:SCF4NPDo+LOCG} in which, it is assumed that
$Z_{\opt}$ is computed by at least one SCF iteration within the variation of
\Cref{alg:SCF4NPDo}, as described in \Cref{thm:cvg4SCF4NPDo:strong},
when it is called to compute $Z_{\opt}$ at Line 4 of Algorithm~\ref{alg:SCF4NPDo+LOCG}.
Then all conclusions in \Cref{thm:cvg4SCF4NPDo:gl} are valid.
Furthermore, if $Z_{\opt}$ is computed to
        be an exact maximizer of \eqref{eq:LOCGsub} in each outer iterative step, then
$(P^{(i)})^{\T}\scrH(P^{(i)})\succeq 0$ for $i\ge 1$.
\end{theorem}

\begin{proof}
See appendix~\ref{sec:Proof4NPDoCVG}.
\end{proof}



\section{Atomic Functions for NPDo}\label{sec:AF-NPD}

Armed with the general theoretical framework for the NPDo approach in
section~\ref{sec:NPDo-theory},
in this section, we introduce the notion of atomic functions for NPDo, which serves as
a singleton unit of function on $\STMnbr$
for which {\bf the NPDo Ansatz} holds and thus the NPDo approach is guaranteed to work  for solving \eqref{eq:main-opt}, and more importantly,
the NPDo approach works on any convex composition of atomic functions.

In what follows, we will first formulate two conditions that define
atomic functions  and explain why the NPDo approach will work  on the atomic functions, and then we give concrete examples of
atomic functions that encompass nearly all practical ones that are in use today. We leave
the investigation of how the NPDo approach works on convex compositions of  atomic functions to section~\ref{sec:CVX-comp:NPDo},
along with a few convex compositions of our concrete atomic functions to guide the use of
the general result.

Combining the results in this section and the next section will yield
a large collection of objective functions, including those in Table~\ref{tbl:obj-funs:NPD}, for which
{\bf the NPDo Ansatz} holds and therefore the NPDo framework as laid out in section~\ref{sec:NPDo-theory}
works on them.


\subsection{Conditions on Atomic Functions}\label{ssec:AF-NPDo}
We are interested in functions $f$ defined on some neighborhood $\STMnbr$  of the Stiefel manifold $\STM{k}{n}$
that satisfy
\begin{subequations}\label{eq:cond4AF-NPDo}
\begin{align}
\tr\left(P^{\T}\frac{\partial f(P)}{\partial P}\right)&=\gamma\, f(P)\quad\mbox{for $P\in\bbP\subseteq\STM{k}{n}$},
              \label{eq:cond4AF-NPDo-a} \\
\intertext{and given $P\in\bbP$ and $\what P\in\STM{k}{n}$, there exists $Q\in\STM{k}{k}$ such
that $\wtd P=\what PQ\in\bbP$ and}
\tr\left(\what P^{\T}\frac{\partial f(P)}{\partial P}\right)
   &\le\alpha f(\wtd P)+\beta f(P), 
      \label{eq:cond4AF-NPDo-b}
\end{align}
\end{subequations}
where $\alpha>0,\,\beta\ge 0$,  and $\gamma=\alpha+\beta$ are constants that are dependent of $f$.
Some subset $\bbP$ of $\STM{k}{n}$ is also involved and, as we commented in Remark~\ref{rk:NPDo-assume}(i) for
{\bf the NPDo Ansatz}, $\bbP$ should be sufficiently inclusive to serve the purpose of
solving \eqref{eq:main-opt} with the function as objective, and for the case of $\bbP=\STM{k}{n}$,
$Q$ can be taken to be $I_k$. More comments on this are in Remark~\ref{rk:cond4AF-NPDo}(ii) below.

\begin{definition}\label{defn:AF:NPD}
A function $f$ defined on some neighborhood $\STMnbr$ of $\STM{k}{n}$ is an {\em atomic function\/}
for NPDo
if there are constants $\alpha>0,\,\beta\ge 0$, and $\gamma=\alpha+\beta$  such that both conditions in
\eqref{eq:cond4AF-NPDo} hold.
\end{definition}

The constants in the definition may vary with the atomic function in question. The descriptive word ``{\em atomic\/}'' is
used here to loosely suggest that such a function is somehow ``{\em unbreakable\/}'', such as the concrete
ones in the next subsection. Having said that, we also find that
{\em for two atomic functions $f_1$ and $f_2$, if they
share the same $\bbP$, the same constants $\alpha,\,\beta,\,\gamma$, and the same $Q$ for \eqref{eq:cond4AF-NPDo-b},
then any linear combination $f:=c_1f_1+c_2f_2$ with $c_1,\,c_2\ge 0$ but $c_1+c_2>0$
also satisfies \eqref{eq:cond4AF-NPDo}, and hence an atomic function as well.}
In fact, it can be verified that
\begin{align*}
\tr\left(P^{\T}\frac{\partial f(P)}{\partial P}\right)
  &=c_1\tr\left(P^{\T}\frac{\partial f_1(P)}{\partial P}\right)+c_2\tr\left(P^{\T}\frac{\partial f_2(P)}{\partial P}\right) \\
  &=c_1\gamma f_1(P)+c_2\gamma f_2(P) \\
  &=\gamma f(P), \\
\tr\left(\what P^{\T}\frac{\partial f(P)}{\partial P}\right)
  &=c_1\tr\left(\what P^{\T}\frac{\partial f_1(P)}{\partial P}\right)+c_2\tr\left(\what P^{\T}\frac{\partial f_2(P)}{\partial P}\right) \\
  &\le c_1[\alpha f_1(\wtd P)+\beta f_1(P)]+c_2[\alpha f_2(\wtd P)+\beta f(_2P)] \\
  &=\alpha f(\wtd P)+\beta f(P).
\end{align*}
Evidently, $f=c_1f_1+c_2f_2$ is ``{\em breakable\/}''. Nonetheless,
``{\em atomic\/}'' still seems to be suitably descriptive despite of what we just discussed.

Throughout this paper, we actually define two types of atomic functions. One type is what we just defined
in Definition~\ref{defn:AF:NPD}. It is for the NPDo approach.
The other type will come in Part~II later for the NEPv approach.

\begin{remark}\label{rk:cond4AF-NPDo}
There are two comments regarding Definition~\ref{defn:AF:NPD}.
\begin{enumerate}[(i)]
  \item Theoretically, each of the two conditions in \eqref{eq:cond4AF-NPDo} has interest of its own. For example, equation \eqref{eq:cond4AF-NPDo-a}
        is a partial differential equation (PDE) in its own right. In that regard, a natural question arises:
        does it have a close form solution, given $\gamma\in\bbR$ (that is not necessarily nonnegative)?
        In this paper, we group the two together because later we need both  to show that
        together they imply {\bf the NPDo Ansatz} for an atomic function and thus the NPDo approach works.
        Also importantly, we need $\gamma=\alpha+\beta$.
  \item How inclusive should $\bbP$ as a subset of $\STM{k}{n}$ be?
        Often certain
        necessary conditions for the maximizers of \eqref{eq:main-opt} with given atomic function as objective can be derived to limit the extent of searching. For example, as has been
        extensively exploited in \cite{zhwb:2022,wazl:2023,wazl:2022a,zhys:2020,luli:2024}, any maximizer $P_*$ must satisfy
        $P_*^{\T}D\succeq 0$ in the case where $f(P)$ contains $\tr(P^{\T}D)$ and increases as $\tr(P^{\T}D)$ does. In such a case,
        searching a maximizer can be naturally limited among those $P\in\STM{k}{n}$ such that $P^{\T}D\succeq 0$, i.e.,
        $P\in\bbP=\STM{k}{n}_{D+}$.
        As a result, it suffices to just require that the equality and inequality in \eqref{eq:cond4AF-NPDo} hold
        for all $P,\,\wtd P\in\bbP=\STM{k}{n}_{D+}$.
        In our later concrete examples in subsection~\ref{ssec:AF-concrete-NPD}, equation \eqref{eq:cond4AF-NPDo-a} even holds
        for all $P\in\bbR^{n\times k}$ for some atomic functions.
\end{enumerate}
\end{remark}

The next  theorem shows that if $f$ is an atomic function for NPDo, then so
is any of its positive powers of order higher than $1$, if well-defined, and moreover the $\alpha$-constant does not change but the $\beta$-constant will.

\begin{theorem}\label{thm:cond4set1-pow-induced}
Given function $f$ satisfying \eqref{eq:cond4AF-NPDo}, suppose that $f(P)\ge 0$ for $P\in\bbP$.
Let $g(P)=c[f(P)]^s$ where $c>0$ and $s>1$.
Then
\begin{subequations}\label{eq:cond4set1-pow-induced}
\begin{align}
\tr\left(P^{\T}\frac{\partial g(P)}{\partial P}\right)&=s\gamma\, g(P)\quad\mbox{for $P\in\bbP\subseteq\STM{k}{n}$},
          \label{eq:induced4pow-diff-eq}\\
\intertext{and given $P\in\bbP$ and $\what P\in\STM{k}{n}$, there exists $Q\in\STM{k}{k}$ such
that $\wtd P=\what PQ\in\bbP$ and}
\tr\left(\what P^{\T}\frac{\partial g(P)}{\partial P}\right)
  &\le \alpha g(\wtd P)+(s\gamma-\alpha) g(P). 
          \label{eq:induced-regularity4pow-diff-eq}
\end{align}
\end{subequations}
\end{theorem}

\begin{proof}
It can be seen that
$
\frac{\partial g(P)}{\partial P}=c\,s[f(P)]^{s-1}\frac{\partial f(P)}{\partial P},
$
and thus
$$
\tr\left(P^{\T}\frac{\partial g(P)}{\partial P}\right)=c\,s[f(P)]^{s-1}\tr\left(P^{\T}\frac{\partial f(P)}{\partial P}\right)
  =c\,s[f(P)]^{s-1}\gamma\, f(P)
  =s\gamma\, g(P),
$$
yielding \eqref{eq:induced4pow-diff-eq}. On the other hand, for $P\in\bbP$ and $\what P\in\STM{k}{n}$,
we have
\begin{align}
\tr\left(\what P^{\T}\frac{\partial g(P)}{\partial P}\right)
   &=cs\,[f(P)]^{s-1}\,\tr\left(\what P^{\T}\frac{\partial f(P)}{\partial P}\right)
              \notag  \\
   &\le cs\,[f(P)]^{s-1}\,\Big\{\alpha [f(\wtd P)]+\beta [f(P)]\Big\} \qquad (\mbox{by \eqref{eq:cond4AF-NPDo-b}})\notag\\
   &= cs\alpha\,[f(\wtd P)][f(P)]^{s-1}+\beta s c[f(P)]^s  \notag\\
   &\le cs\alpha\,\left\{\frac 1s [f(\wtd P)]^s+\frac {s-1}s [f(P)]^s\right\}+\beta s g(P) \label{eq:cond4set1-pow-induced:pf-1}\\
   &=\alpha g(\wtd P)+\alpha (s-1) g(P)+\beta s g(P) \notag \\
   &=\alpha g(\wtd P)+[\alpha (s-1)+\beta s] g(P), \notag
\end{align}
yielding \eqref{eq:induced-regularity4pow-diff-eq}, where we have used Lemma~\ref{lm:YoungIneq-ext} on
$[f(\wtd P)][f(P)]^{s-1}$ to get \eqref{eq:cond4set1-pow-induced:pf-1}.
\end{proof}

%


\begin{remark}\label{rk:induced4tr-diff-eq:2}
In Theorem~\ref{thm:cond4set1-pow-induced}, $g(P)=h(f(P))$ where $h(t)=ct^s$, i.e., $g=h\circ f$ is a composition function.
We claim that this is the only composition function that  satisfies the same type of PDE as $f$ does
in \eqref{eq:cond4AF-NPDo-a}.
Here is why. Given function $f$ satisfying \eqref{eq:cond4AF-NPDo-a}, let $g=h\circ f$. Suppose that $g$
also satisfies \eqref{eq:cond4AF-NPDo-a}, i.e.,
\begin{equation}\label{eq:tr-diff-eq:circ}
\tr\left(P^{\T}\frac{\partial g(P)}{\partial P}\right)=\tilde\gamma\, g(P),
\end{equation}
where $\tilde\gamma$ is a constant.
We claim that $h(t)=ct^s$. In fact, it follows from
$$
\frac{\partial g(P)}{\partial P}=h'(f(P))\frac{\partial f(P)}{\partial P}
$$
and \eqref{eq:cond4AF-NPDo-a} and \eqref{eq:tr-diff-eq:circ} that
$$
\tilde\gamma\, g(P)
  =\tr\left(P^{\T}\frac{\partial g(P)}{\partial P}\right)
  =h'(f(P))\tr\left(P^{\T}\frac{\partial f(P)}{\partial P}\right)
  =h'(f(P))\gamma f(P).
$$
Namely,
$$
\tilde\gamma h(f)=\gamma h'(f) f
\quad\Rightarrow\quad
\frac {h'(f)}{h(f)}=\frac {\tilde\gamma}{\gamma}\cdot\frac 1f
\quad\Rightarrow\quad
h(f)=c\,f^s,
$$
as expected, where $s=\tilde\gamma/\gamma$ and $c$ is some constant.
\end{remark}

\begin{theorem}\label{thm:NPDo-assume:AF}
{\bf The NPDo Ansatz} holds with $\omega=1/\alpha$ for atomic function $f$ that satisfies the conditions in \eqref{eq:cond4AF-NPDo}.
\end{theorem}

\begin{proof}
Given $P\in\bbP\subseteq\STM{k}{n}$ and $\what P\in\STM{k}{n}$, suppose  that \eqref{eq:NPDo-assume} holds,
i.e., $\tr(\what P^{\T}\scrH(P))\ge\tr(P^{\T}\scrH(P))+\eta$. We have by \eqref{eq:cond4AF-NPDo}
$$
\eta+\gamma f(P)=\eta+\tr(P^{\T}\scrH(P))
   \le\tr(\what P^{\T}\scrH(P))
   \le\alpha f(\wtd P)+\beta f(P)
$$
yielding $\eta/\alpha+ f(P)\le  f(\wtd P)$, as was to be shown.
\end{proof}

As a corollary of Theorem~\ref{thm:NPDo-assume:AF}, the NPDo approach as laid out in section~\ref{sec:NPDo-theory}
works on any atomic function for NPDo.

\subsection{Concrete Atomic Functions}\label{ssec:AF-concrete-NPD}
We will show  that
\begin{equation}\label{eq:concrete-AF:NPDo}
\framebox{
\parbox{11cm}{
for integer $m\ge 1$ and scalar $s\ge 1$,
$[\tr\big((P^{\T}D)^m\big)]^s$,
$[\tr\big((P^{\T}AP)^m\big)]^s$
and $[\tr\big((\Gamma P^{\T}AP)^m\big)]^s$
where $\Gamma\succeq 0$ and $A\succeq 0$
}}
\end{equation}
satisfy \eqref{eq:cond4AF-NPDo}
and hence are atomic functions for NPDo.
Therefore, by Theorem~\ref{thm:NPDo-assume:AF}, {\bf the NPDo Ansatz} holds for them.
We point out that the results we will prove in this subsection are actually for more general
$P$, $\what P$ and $\wtd P$ in $\bbR^{n\times k}$, which is more general than in $\STM{k}{n}$ as
specified in Definition~\ref{defn:AF:NPD}.

It is noted that $[\tr\big((P^{\T}AP)^m\big)]^s$ is a special case
of $[\tr\big((\Gamma P^{\T}AP)^m\big)]^s$ with $\Gamma=I_k$, but we still single it out because it
shows up more more frequently than the latter.

We start by considering $\tr\big((P^{\T}D)^m\big)$ and its power.

\begin{theorem}\label{thm:lin4tr-diff-eq}
Let $D\in\bbR^{n\times k}$, and let $m\ge 1$ be an integer.
\begin{enumerate}[{\rm (a)}]
  \item For $P\in\bbR^{n\times k}$, we have
        \begin{equation}\label{eq:lin-pow4tr-diff-eq}
        \tr\left(P^{\T}\frac{\partial \tr\big((P^{\T}D)^m\big)}{\partial P}\right)=m\, \tr\big((P^{\T}D)^m\big).
        \end{equation}
  \item Let $P,\,\what P\in\bbR^{n\times k}$.
        \begin{enumerate}[{\rm (i)}]
          \item For $m=1$, we have
        \begin{equation}\label{eq:regularity4lin4tr-diff-eq}
        \tr\left(\what P^{\T}\frac{\partial \tr(P^{\T}D)}{\partial P}\right)=\tr(\what P^{\T}D);
        \end{equation}
          \item For $m\ge 1$, if  $P^{\T}D\succeq 0$, then
        \begin{equation}\label{eq:regularity4lin-pow4tr-diff-eq}
        \tr\left(\what P^{\T}\frac{\partial \tr\big((P^{\T}D)^m\big)}{\partial P}\right)\le\tr((\wtd P^{\T}D)^m)+(m-1)\,\tr\big((P^{\T}D)^m\big),
        \end{equation}
        where $\wtd P=\what PQ$ for $Q\in\STM{k}{k}$ such that $\wtd P^{\T}D\succeq 0$.
        \end{enumerate}
\end{enumerate}
In particular, the conditions in \eqref{eq:cond4AF-NPDo} hold with $\bbP=\STM{k}{n}_{D+}$,
$\alpha=1$ and $\beta=m-1$,
and thus $\tr\big((P^{\T}D)^m\big)$ is an atomic function for NPDo.

Just for the case $m=1$, we can also take
$\bbP=\STM{k}{n}$ and $\wtd P=\what P$ in \eqref{eq:cond4AF-NPDo}, and then \eqref{eq:cond4AF-NPDo-b} becomes an equality.
Therefore, $\tr(P^{\T}D)$ is also an atomic function for NPDo with $\bbP=\STM{k}{n}$ and $Q=I_k$ in the definition.
\end{theorem}

\begin{proof}
Consider perturbing $P\in\bbR^{n\times k}$ to $P+E$ where $E\in\bbR^{n\times k}$ with
$\|E\|$  sufficiently tiny. We have
\begin{align}
[(P+E)^{\T}D]^m&=[P^{\T}D+E^{\T}D]^m \nonumber\\
  &=(P^{\T}D)^m+\sum_{i=0}^{m-1}(P^{\T}D)^iE^{\T}D(P^{\T}D)^{m-1-i}+O(\|E\|^2), \label{eq:lin-pow:pertb}\\
\tr([(P+E)^{\T}D]^m)
  &=\tr\big((P^{\T}D)^m\big)+m\,\tr(E^{\T}D(P^{\T}D)^{m-1})+O(\|E\|^2). \label{eq:tr-lin-pow:pertb}
\end{align}
Immediately, it follows from \eqref{eq:tr-lin-pow:pertb} that
\begin{equation} \label{eq:partD-lin-pow}
\frac{\partial \tr\big((P^{\T}D)^m\big)}{\partial P}=m\,D(P^{\T}D)^{m-1}.
\end{equation}
Equation \eqref{eq:lin-pow4tr-diff-eq} is a direct consequence
of \eqref{eq:partD-lin-pow}. This proves item (a).

Now, we prove item (b). 
Equation \eqref{eq:regularity4lin4tr-diff-eq}  which is for $m=1$ is easily verified.
In general, for $m>1$,
noticing the assumption $P^{\T}D\succeq 0$ and $\wtd P^{\T}D=Q^{\T}(\what P^{\T}D)\succeq 0$ for the case,
we have, by Lemma~\ref{lm:vN-tr-ineq-ext1},
\begin{align}
\tr\left(\what P^{\T}\frac{\partial \tr\big((P^{\T}D)^m\big)}{\partial P}\right)
  &=m\,\tr(\what P^{\T}D(P^{\T}D)^{m-1}) \nonumber\\
  &\le\tr((\wtd P^{\T}D)^m)+(m-1)\,\tr\big((P^{\T}D)^m\big),
  \label{eq:lin4tr-diff-eq:pf-1}
\end{align}
which is \eqref{eq:regularity4lin-pow4tr-diff-eq}.
\end{proof}

\begin{remark}\label{rk:lin4tr-diff-eq}
In obtaining \eqref{eq:lin4tr-diff-eq:pf-1} by Lemma~\ref{lm:vN-tr-ineq-ext1}, it is needed that
$P^{\T}D\succeq 0$ and $\wtd P^{\T}D=Q^{\T}(\what P^{\T}D)\succeq 0$. This explains the necessity of  having a strict subset $\bbP$ of
$\STM{k}{n}$ and aligning $\what P$ to $\wtd P\in\bbP$ by some $Q$ in \eqref{eq:cond4AF-NPDo} for defining atomic function for NPDo in general.
For the purpose of maximizing $\tr\big((P^{\T}D)^m\big)$, given $P\in\bbP=\STM{k}{n}_{D+}$ being the current approximation, to compute the next and hopefully improved
approximation, the NPDo approach will seek $\what P$ to maximize
$\tr\left(X^{\T}\frac{\partial \tr\big((P^{\T}D)^m\big)}{\partial P}\right)$, or equivalently,
$\tr(X^{\T}D(P^{\T}D)^{m-1})$, over $X\in\STM{k}{n}$, and hence $\what P$ is taken to be an orthonormal polar factor of $D(P^{\T}D)^{m-1}$.
For that $\what P$, likely $\what P^{\T}D\not\succeq 0$, and hence necessarily $\what P$ needs to be aligned to $\wtd P=\what PQ\in\bbP$
so that $\wtd P^{\T}D\succeq 0$.
\end{remark}


For any scalar $s>1$,  $f(P)=[\tr\big((P^{\T}D)^m\big)]^s$ is  well-defined for any $P\in\bbR^{n\times k}$
such that $\tr\big((P^{\T}D)^m\big)\ge 0$.
In particular, $[\tr\big((P^{\T}D)^m\big)]^s$ is well-defined
for
$$
P\in\bbR_{D+}^{n\times k}:=\{X\in\bbR^{n\times k}\,:\,X^{\T}D\succeq 0\}.
$$
With Theorem~\ref{thm:lin4tr-diff-eq}, a minor modification to the proof of Theorem~\ref{thm:cond4set1-pow-induced}
leads to

\begin{corollary}\label{cor:pow(lin)4tr-diff-eq}
Let $D\in\bbR^{n\times k}$, integer $m\ge 1$, scalar $s>1$, $g(P)=[\tr\big((P^{\T}D)^m\big)]^s$.
\begin{subequations}\label{eq:pow(lin)4tr-diff-eq}
\begin{enumerate}[{\rm (a)}]
  \item For $P\in\bbR^{n\times k}$ at which $g(P)$ is well defined, we have
        \begin{equation}\label{eq:lin-pow4trpow-diff-eq}
          \tr\left(P^{\T}\frac{\partial [\tr\big((P^{\T}D)^m\big)]^s}{\partial P}\right)
             =sm\, [\tr\big((P^{\T}D)^m\big)]^s;
        \end{equation}
  \item Let $P\in\bbR_{D+}^{n\times k}$, $\what P\in\bbR^{n\times k}$, and let $\wtd P=\what PQ$, where
        $Q\in\STM{k}{k}$ is an orthonormal polar factor of $\what P^{\T}D$. We have
        $\wtd P\in\bbR_{D+}^{n\times k}$ and
        \begin{equation}\label{eq:lin-pow-regularity4pow-diff-eq}
        \tr\left(\what P^{\T}\frac{\partial [\tr\big((P^{\T}D)^m\big)]^s}{\partial P}\right)
             \le [\tr((\wtd P^{\T}D)^m)]^s+(sm-1)[\tr\big((P^{\T}D)^m\big)]^s.
        \end{equation}
\end{enumerate}
\end{subequations}
In particular, the conditions in \eqref{eq:cond4AF-NPDo} hold with $\bbP=\STM{k}{n}_{D+}$, $\alpha=1$ and $\beta=sm-1$,
and thus $[\tr\big((P^{\T}D)^m\big)]^s$  for scalar $s>1$  is an atomic function for NPDo.
\end{corollary}


Next we consider $\tr\big((P^{\T}AP)^m\big)$ and its power.

\begin{theorem}\label{thm:quad4tr-diff-eq}
Let symmetric $A\in\bbR^{n\times n}$, and let $m\ge 1$ be an integer.
\begin{enumerate}[{\rm (a)}]
  \item For $P\in\bbR^{n\times k}$, we have
\begin{equation}\label{eq:quad-pow4tr-diff-eq}
\tr\left(P^{\T}\frac{\partial \tr\big((P^{\T}AP)^m\big)}{\partial P}\right)=2m\, \tr\big((P^{\T}AP)^m\big).
\end{equation}
  \item For $P,\,\what P\in\bbR^{n\times k}$, if $A\succeq 0$, then
        \begin{equation}\label{eq:regularity4quad-pow4tr-diff-eq}
        \tr\left(\what P^{\T}\frac{\partial \tr\big((P^{\T}AP)^m\big)}{\partial P}\right)\le\tr((\what P^{\T}A\what P)^m)+(2m-1)\,\tr\big((P^{\T}AP)^m\big).        \end{equation}
\end{enumerate}
In particular, the conditions in \eqref{eq:cond4AF-NPDo} hold with $\bbP=\STM{k}{n}$, $Q=I_k$ and
$\wtd P=\what P$ , $\alpha=1$ and $\beta=2m-1$,
and thus $\tr\big((P^{\T}AP)^m\big)$ is an atomic function for NPDo.
\end{theorem}

\begin{proof}
Consider perturbing $P\in\bbR^{n\times k}$ to $P+E$ where $E\in\bbR^{n\times k}$ with
$\|E\|$ sufficiently tiny. We have
\begin{align}
[(P+E)^{\T}A(P+E)]^m&=[P^{\T}AP+E^{\T}AP+P^{\T}AE+E^{\T}AE]^m \nonumber\\
  &=(P^{\T}AP)^m+\sum_{i=0}^{m-1}(P^{\T}AP)^i(E^{\T}AP+P^{\T}AE)(P^{\T}AP)^{m-1-i} \nonumber\\
  &\quad+O(\|E\|^2), \label{eq:quad-pow:pertb}\\
\tr([(P+E)^{\T}A(P+E)]^m)
  &=\tr\big((P^{\T}AP)^m\big)+m\,\tr(E^{\T}AP(P^{\T}AP)^{m-1}) \nonumber\\
  &\hphantom{\,=(P^{\T}AP)^m}
        +m\,\tr((P^{\T}AP)^{m-1}P^{\T}AE)+O(\|E\|^2) \nonumber\\
  &=\tr\big((P^{\T}AP)^m\big)+2m\,\tr(E^{\T}AP(P^{\T}AP)^{m-1})+O(\|E\|^2). \label{eq:tr-quad-pow:pertb}
\end{align}
Immediately, it follows from  \eqref{eq:tr-quad-pow:pertb} that
\begin{equation}\label{eq:partD-quad-pow}
\frac{\partial \tr\big((P^{\T}AP)^m\big)}{\partial P}=2m\,AP(P^{\T}AP)^{m-1}.
\end{equation}
Equation \eqref{eq:quad-pow4tr-diff-eq} is a direct consequence
of \eqref{eq:partD-quad-pow}. This proves item (a).

Next we prove item (b). Let $X=A^{1/2}\what P$ and $Y=A^{1/2}P$, where $A^{1/2}$ is the unique positive semidefinite square root of $A$.
We have
\begin{align*}
\tr\left(\what P^{\T}\frac{\partial \tr\big((P^{\T}AP)^m\big)}{\partial P}\right)
  &=2m\,\tr(\what P^{\T}AP(P^{\T}AP)^{m-1}) \\
  &=2m\,\tr(X^{\T}Y(Y^{\T}Y)^{m-1}) \\
  &\le\tr((X^{\T}X)^m)+(2m-1)\,\tr((Y^{\T}Y)^m) \qquad (\mbox{by Lemma~\ref{lm:vN-tr-ineq-ext2}})\\
  &=\tr((\what P^{\T}A\what P)^m)+(2m-1)\,\tr\big((P^{\T}AP)^m\big),
\end{align*}
which is \eqref{eq:regularity4quad-pow4tr-diff-eq}.
\end{proof}




With Theorem~\ref{thm:quad4tr-diff-eq}, a minor modification to the proof of Theorem~\ref{thm:cond4set1-pow-induced}
leads to

\begin{corollary}\label{cor:pow(quad)4tr-diff-eq}
Let symmetric $A\in\bbR^{n\times n}$, and let $m\ge 1$ be an integer and $s>1$ a scalar.
For $P,\,\wtd P\in\bbR^{n\times k}$,
if $A\succeq 0$, then
\begin{subequations}\label{eq:pow(quad)4tr-diff-eq}
\begin{align}
\tr\left(P^{\T}\frac{\partial [\tr\big((P^{\T}AP)^m\big)]^s}{\partial P}\right)
    &=2sm\, [\tr\big((P^{\T}AP)^m\big)]^s, \label{eq:quad-pow4trpow-diff-eq} \\
\tr\left(\wtd P^{\T}\frac{\partial [\tr\big((P^{\T}AP)^m\big)]^s}{\partial P}\right)
             &\le [\tr((\wtd P^{\T}A\wtd P)^m)]^s+(2sm-1)[\tr\big((P^{\T}AP)^m\big)]^s.
                      \label{eq:quad-pow-regularity4pow-diff-eq}
\end{align}
\end{subequations}
In particular, the conditions in \eqref{eq:cond4AF-NPDo} hold with $\bbP=\STM{k}{n}$, $Q=I_k$,
$\alpha=1$ and $\beta=2sm-1$,
and thus $[\tr\big((P^{\T}AP)^m\big)]^s$  for scalar $s>1$ is an atomic function for NPDo.
\end{corollary}


In the literature, there is another type of trace functions in the form $\tr(\Gamma P^{\T}AP)$,
where $\Gamma\in\bbR^{k\times k}$ and $A\in\bbR^{k\times k}$ are symmetric.
It is known as the {\em Brockett cost function} in the case when $\Gamma$
is diagonal and positive semi-definite, and optimizing it
has often been used as an illustrative example for optimization on the Stiefel manifold \cite[p.80]{abms:2008}, \cite{bicc:2019,broc:1991,liwz:2023}. It turns out that $\tr\big((\Gamma P^{\T}AP)^m\big)$ is also an atomic function for NPDo
if $\Gamma$ is positive semidefinite.

\begin{theorem}\label{thm:quad4tr-diff-eq-BCF}
Let symmetric $A\in\bbR^{n\times n}$, and let $m\ge 1$ be an integer.
\begin{enumerate}[{\rm (a)}]
  \item For $P\in\bbR^{n\times k}$, we have
\begin{equation}\label{eq:quad-pow4tr-diff-eq-BCF}
\tr\left(P^{\T}\frac{\partial \tr\big((\Gamma P^{\T}AP)^m\big)}{\partial P}\right)=2m\, \tr\big((\Gamma P^{\T}AP)^m\big).
\end{equation}
  \item For $P,\,\what P\in\bbR^{n\times k}$, if $A\succeq 0$ and $\Gamma\succeq 0$, then
        \begin{equation}\label{eq:regularity4quad-pow4tr-diff-eq-BCF}
        \tr\left(\what P^{\T}\frac{\partial \tr\big((\Gamma P^{\T}AP)^m\big)}{\partial P}\right)
             \le\tr((\Gamma\what P^{\T}A\what P)^m)+(2m-1)\,\tr\big((\Gamma P^{\T}AP)^m\big).
       \end{equation}
\end{enumerate}
In particular, the conditions in \eqref{eq:cond4AF-NPDo} hold with $\bbP=\STM{k}{n}$, $Q=I_k$ and
$\wtd P=\what P$ , $\alpha=1$ and $\beta=2m-1$,
and thus $\tr\big((\Gamma P^{\T}AP)^m\big)$ is an atomic function for NPDo.
\end{theorem}

\begin{proof}
Consider perturbing $P\in\bbR^{n\times k}$ to $P+E$ where $E\in\bbR^{n\times k}$ with
$\|E\|$ sufficiently tiny. We have
\begin{align}
&[\Gamma(P+E)^{\T}A(P+E)]^m=[\Gamma P^{\T}AP+\Gamma E^{\T}AP+\Gamma P^{\T}AE+\Gamma E^{\T}AE]^m \nonumber\\
  &\qquad=(\Gamma P^{\T}AP)^m+\sum_{i=0}^{m-1}(\Gamma P^{\T}AP)^i(\Gamma E^{\T}AP+\Gamma P^{\T}AE)(\Gamma P^{\T}AP)^{m-1-i} \nonumber\\
  &\qquad\quad+O(\|E\|^2), \label{eq:quad-pow-BCF:pertb}\\
&\tr([\Gamma (P+E)^{\T}A(P+E)]^m)
  =\tr\big((\Gamma P^{\T}AP)^m\big)+m\,\tr(E^{\T}AP(\Gamma P^{\T}AP)^{m-1}\Gamma) \nonumber\\
  &\hphantom{\,=(\Gamma P^{\T}AP)^m}
        +m\,\tr((\Gamma P^{\T}AP)^{m-1}\Gamma P^{\T}AE)+O(\|E\|^2) \nonumber\\
  &\qquad=\tr\big((\Gamma P^{\T}AP)^m\big)+2m\,\tr(E^{\T}AP(\Gamma P^{\T}AP)^{m-1}\Gamma)+O(\|E\|^2). \label{eq:tr-quad-pow-BCF:pertb}
\end{align}
Immediately, it follows from  \eqref{eq:tr-quad-pow-BCF:pertb} that
\begin{equation}\label{eq:partD-quad-pow-BCF}
\frac{\partial \tr\big((\Gamma P^{\T}AP)^m\big)}{\partial P}=2m\,AP(\Gamma P^{\T}AP)^{m-1}\Gamma .
\end{equation}
Equation \eqref{eq:quad-pow4tr-diff-eq-BCF} is a direct consequence
of \eqref{eq:partD-quad-pow-BCF}. This proves item (a).

Next we prove item (b). Let $X=A^{1/2}\what P\Gamma^{1/2}$ and $Y=A^{1/2}P\Gamma^{1/2}$, where $A^{1/2}$ and $\Gamma^{1/2}$
are the unique positive semidefinite square roots of $A$ and $\Gamma$, respectively.
We have
\begin{align*}
\tr\left(\what P^{\T}\frac{\partial \tr\big((\Gamma P^{\T}AP)^m\big)}{\partial P}\right)
  &=2m\,\tr(\what P^{\T}AP(\Gamma P^{\T}AP)^{m-1}\Gamma) \\
  &=2m\,\tr(X^{\T}Y(Y^{\T}Y)^{m-1}) \\
  &\le\tr((X^{\T}X)^m)+(2m-1)\,\tr((Y^{\T}Y)^m) \qquad (\mbox{by Lemma~\ref{lm:vN-tr-ineq-ext2}})\\
  &=\tr((\Gamma \what P^{\T}A\what P)^m)+(2m-1)\,\tr\big((\Gamma P^{\T}AP)^m\big),
\end{align*}
which is \eqref{eq:regularity4quad-pow4tr-diff-eq-BCF}.
\end{proof}

With Theorem~\ref{thm:quad4tr-diff-eq-BCF}, similar to \Cref{cor:pow(quad)4tr-diff-eq}, we have

\begin{corollary}\label{cor:pow(quad)4tr-diff-eq-BCF}
Let symmetric $A\in\bbR^{n\times n}$ and symmetric $\Gamma\in\bbR^{k\times k}$, and let $m\ge 1$ be an integer and scalar $s>1$.
For $P,\,\wtd P\in\bbR^{n\times k}$,
if $A\succeq 0$ and $\Gamma\succeq 0$, then
\begin{subequations}\label{eq:pow(quad)4tr-diff-eq-BCF}
\begin{align}
\tr\left(P^{\T}\frac{\partial [\tr\big((\Gamma P^{\T}AP)^m\big)]^s}{\partial P}\right)
    &=2sm\, [\tr\big((\Gamma P^{\T}AP)^m\big)]^s, \label{eq:quad-pow4trpow-diff-eq-BCF} \\
\tr\left(\wtd P^{\T}\frac{\partial [\tr\big((\Gamma P^{\T}AP)^m\big)]^s}{\partial P}\right)
             &\le [\tr((\Gamma \wtd P^{\T}A\wtd P)^m)]^s+(2sm-1)[\tr\big((\Gamma P^{\T}AP)^m\big)]^s.
                      \label{eq:quad-pow-regularity4pow-diff-eq-BCF}
\end{align}
\end{subequations}
In particular, the conditions in \eqref{eq:cond4AF-NPDo} hold with $\bbP=\STM{k}{n}$, $Q=I_k$,
$\alpha=1$ and $\beta=2sm-1$,
and thus $[\tr\big((\Gamma P^{\T}AP)^m\big)]^s$  for scalar $s>1$ is an atomic function for NPDo.
\end{corollary}

\begin{remark}\label{rk:convex-NPDo:AF}
As a consequence of the results in this subsection, the three functions in \eqref{eq:concrete-AF:NPDo}
satisfy {\bf the NPDo Ansatz} by \Cref{thm:NPDo-assume:AF}. It turns out that both $[\tr\big((P^{\T}AP)^m\big)]^s$
and $[\tr\big((\Gamma P^{\T}AP)^m\big)]^s$  are convex in $P\in\bbR^{n\times k}$ (see appendix~\ref{apx:convex-NPDo:AF}
for proofs) and thus
they satisfy {\bf the NPDo Ansatz} by convexity alone; see \Cref{rk:NPDo-assume}.
On the other hand, $\tr(P^{\T}D)$ is linear in $P$ and thus convex but  $\tr\big((P^{\T}D)^m\big)$ is not convex in general for $m\ge 2$.
\end{remark}

\section{Convex Composition}\label{sec:CVX-comp:NPDo}
The concrete atomic functions for NPDo in \eqref{eq:concrete-AF:NPDo} provides a limited collection of objective functions
for which the NPDo approach provably works. In this section, we will vastly expand the collection to include
any convex composition of atomic functions, provided that some of the partial derivatives
of the composing convex function are  nonnegative.

Specifically, we are interested in a special case of optimization problem \eqref{eq:main-opt} on the Stiefel manifold $\STM{k}{n}$ where
$f$ is a convex composition of atomic functions for NPDo, namely,
\begin{equation}\label{eq:Master-OptSTM}
\max_{P\in\STM{k}{n}}\,f(P)\quad\mbox{with}\quad f(P):=(\phi\circ T)(P)\equiv\phi(T(P)),
\end{equation}
where $T\,:\, P\in\STM{k}{n} \to T(P)\in\mathfrak{D}\subseteq\bbR^N$ whose components are atomic functions
dependent of just a few or all columns of $P$,
and $\phi\,:\,\mathfrak{D}\to\bbR$ is
convex and differentiable.
Denote the partial derivatives of $\phi$
with respect to $\bx=[x_1,x_2,\ldots,x_N]^{\T}\in\mathfrak{D}\subseteq\bbR^N$ by
\begin{equation}\label{eq:phi-i}
\phi_i(\bx)=\frac {\partial \phi(\bx)}{\partial x_i}\quad\mbox{for $1\le i\le N$}.
\end{equation}

Our goal is to solve \eqref{eq:Master-OptSTM} by \Cref{alg:SCF4NPDo} and its accelerating  variation in Algorithm~\ref{alg:SCF4NPDo+LOCG} with convergence guarantee. To that end, we
will have to place consistency conditions upon all components of $T(P)$.
Let
\begin{equation}\label{eq:T-general}
T(P)=\begin{bmatrix}
                                        f_1(P_1) \\
                                        f_2(P_2) \\
                                        \vdots \\
                                        f_N(P_N) \\
                                      \end{bmatrix},
\end{equation}
where each $P_i$ is a submatrix of $P$, consisting of a few or all columns of $P$.
We point out that
it is possible that some of  $P_i$ may share common column(s) of $P$, different from
the situation in SumCT in Table~\ref{tbl:obj-funs}. Alternatively, we can write $P_i=PJ_i$ where $J_i\in\bbR^{k\times k_i}$ is submatrices
      of $I_k$, taking the columns of $I_k$ with the same column indices as $P_i$
      to $P$. Each $J_i$ acts as a column selector.

The KKT condition \eqref{eq:KKT} for \eqref{eq:Master-OptSTM} becomes
\begin{subequations}\label{eq:KKT-comp}
\begin{gather}
\scrH(P):=\frac {\partial f(P)}{\partial P}
  =\sum_{i=1}^N\phi_i(T(P))\frac {\partial f_i(P_i)}{\partial P_i}J_i^{\T}
  =P\Lambda, \label{eq:KKT-comp-1}\\
\mbox{with} \,\, \Lambda^{\T}=\Lambda\in\bbR^{k\times k},\,\, P\in\STM{k}{n}. \label{eq:KKT-comp-2}
\end{gather}
\end{subequations}
The consistency conditions on atomic functions $f_i(P_i)$ for $1\le i\le N$ are
\begin{subequations}\label{eq:cond4AF-NPDo:cvx}
\begin{align}
\tr\left(P_i^{\T}\frac{\partial f_i(P_i)}{\partial P_i}\right)
     &=\gamma_i\, f_i(P_i)\quad\mbox{for $P\in\bbP\subseteq\STM{k}{n}$},
              \label{eq:cond4AF-NPDo:cvx-a} \\
\intertext{and given $P\in\bbP$ and $\what P\in\STM{k}{n}$, there exists $Q\in\STM{k}{k}$ such
that $\wtd P=\what PQ\in\bbP$ and} 
\tr\left(\what P_i^{\T}\frac{\partial f_i(P_i)}{\partial P_i}\right)
   &\le\alpha f_i(\wtd P_i)+\beta_i f_i(P_i), 
      \label{eq:cond4AF-NPDo:cvx-b}
\end{align}
\end{subequations}
where $\alpha>0,\,\beta_i\ge 0$,  and $\gamma_i=\alpha+\beta_i$ are  constants, $\what P_i$ and $\wtd P_i$ are the
submatrices of $\what P$ and $\wtd P$, respectively, with the same column indices as $P_i$ to $P$.
It is important to keep in mind that some of the inequalities in \eqref{eq:cond4AF-NPDo:cvx-b} may actually be equalities,
e.g.,  for $f_i(P_i)=\tr(P_i^{\T}D_i)$ it is an equality by Theorem~\ref{thm:lin4tr-diff-eq}.

On the surface, it looks like that each $f_i$ is simply an atomic function for NPDo, but there are three
built-in consistency requirements in \eqref{eq:cond4AF-NPDo:cvx} among
all components $f_i(P_i)$: 1) the same $\bbP$ for all;  2) the same $\alpha$ for all, and
3) the same $Q$ to give $\wtd P=\what PQ$ for all.

\begin{theorem}\label{thm:main-npd-cvx}
Consider $f=\phi\circ T$, where $T(\cdot)$ takes the form in \eqref{eq:T-general} satisfying
\eqref{eq:cond4AF-NPDo:cvx} and $\phi$ is convex and differentiable with partial derivatives denoted by
$\phi_i$ as in \eqref{eq:phi-i}.
If $\phi_i(\bx)\ge 0$ for those $i$ for which \eqref{eq:cond4AF-NPDo:cvx-b} does not become an equality, then {\bf the NPDo Ansatz} holds  with $\omega=1/\alpha$.
\end{theorem}

\begin{proof}
Given $P\in\bbP\subseteq\STM{k}{n}$ and $\what P\in\STM{k}{n}$, suppose that \eqref{eq:NPDo-assume} holds, i.e.,
$\tr(\what P^{\T}\scrH(P))\ge\tr(P^{\T}\scrH(P))+\eta$.
Let $\wtd P=\what PQ$ where $Q\in\STM{k}{k}$ is the one dictated by the consistency conditions in \eqref{eq:cond4AF-NPDo:cvx}.
Write
$$
\bx=T(P)\equiv [x_1,x_2,\ldots,x_N]^{\T}, \quad
\wtd\bx=T(\wtd P)\equiv [\wtd x_1,\wtd x_2,\ldots,\wtd x_N]^{\T},
$$
i.e., $x_i=f_i(P_i)$ and $\wtd x_i=f_i(\wtd P_i)$.
Noticing $\scrH(P)$ in \eqref{eq:KKT-comp}, we have
\begin{align*}
\tr(P^{\T}\scrH(P))
  &=\sum_{i=1}^N\phi_i(\bx)\tr\left(P^{\T}\frac {\partial f_i(P_i)}{\partial P_i}J_i^{\T}\right) \\
  &=\sum_{i=1}^N\phi_i(\bx)\tr\left(P_i^{\T}\frac {\partial f_i(P_i)}{\partial P_i}\right)
      \\
  &=\sum_{i=1}^N\gamma_i\phi_i(\bx)\,x_i, \qquad(\mbox{by \eqref{eq:cond4AF-NPDo:cvx-a}})\\
\tr(\what P^{\T}\scrH(P))
  &=\sum_{i=1}^N\phi_i(\bx)\tr\left(\what P_i^{\T}\frac {\partial f_i(P_i)}{\partial P_i}\right) \\
  &\le\sum_{i=1}^N\phi_i(\bx)\,(\alpha\wtd x_i+\beta_i x_i),
\end{align*}
where the last inequality is due to $\phi_i\ge 0$ when the corresponding \eqref{eq:cond4AF-NPDo:cvx-b}
does not become an equality.
Plug them into \eqref{eq:NPDo-assume}  and simplify the resulting inequality with
the help of $\gamma_i=\alpha+\beta_i$ to get
$$
\eta/\alpha+\nabla\phi(\bx)^{\T}\bx=\eta/\alpha+\sum_{i=1}^N\phi_i(\bx)\,x_i\le\sum_{i=1}^N\phi_i(\bx)\,\wtd x_i
   =\nabla\phi(\bx)^{\T}\wtd\bx.
$$
Finally apply Lemma~\ref{lm:convex-mono} to yield $f(\wtd P)\ge f(P)+\eta/\alpha$.
\end{proof}

With Theorem~\ref{thm:main-npd-cvx} come the general results established in section~\ref{sec:NPDo-theory}.
In particular, \Cref{alg:SCF4NPDo} (NPDoSCF) and its accelerating  variation in Algorithm~\ref{alg:SCF4NPDo+LOCG}
can be applied to find a maximizer of \eqref{eq:Master-OptSTM},
except that the calculation of $Q_i$ at Line 4 of \Cref{alg:SCF4NPDo} remains to be specified. This missing detail is in general
dependent of the particularity of the mapping $T$ and the convex function $\phi$, to which we shall return after we showcase
a few concrete
mappings of $T$,
where
$A_i\in\bbR^{n\times n}$ for $1\le i\le\ell$ are at least symmetric
and $D_i\in\bbR^{n\times k_i}$ with $1\le k_i\le k$ for $1\le i\le t$.

\begin{example}\label{eg:T1}
Consider
the {\bf first} concrete mapping of $T$:
\begin{equation}\label{eq:T-tr}
T_1\,:\, P\in\STM{k}{n} \to T_1(P):=\begin{bmatrix}
                                        \tr(P_1^{\T}A_1P_1) \\
                                        \vdots \\
                                        \tr(P_{\ell}^{\T}A_{\ell}P_{\ell}) \\
                                        \tr(P_{\ell+1}^{\T}D_1) \\
                                        \vdots \\
                                        \tr(P_{\ell+t}^{\T}D_t) \\
                                      \end{bmatrix} \in\bbR^{\ell+t}.
\end{equation}
Either $\ell=0$ or $t=0$ (i.e., TrCP in Table~\ref{tbl:obj-funs}) is allowed. If all $A_i\succeq 0$, then
the consistency conditions in \eqref{eq:cond4AF-NPDo:cvx} are satisfied with $\bbP=\STM{k}{n}$, $Q=I_k$,
$\alpha=1$, $\beta_i=1$ for $1\le i\le\ell$ and $\beta_{\ell+j}=0$ for $1\le j\le t$, by
Theorems~\ref{thm:lin4tr-diff-eq} and \ref{thm:quad4tr-diff-eq}.
In particular, now \eqref{eq:cond4AF-NPDo:cvx-b} for $\ell+1\le i\le\ell+t$ are equalities.
Thus Theorem~\ref{thm:main-npd-cvx} applies, assuming $\phi_j(\bx)\ge 0$ for $1\le j\le\ell$.
Two existing special cases of $T_1$
are
\begin{enumerate}[(1)]
  \item $\ell=t$, $P_i=P_{\ell+i}$ for $1\le i\le\ell$, $P=[P_1,P_2,\ldots P_{\ell}]$, and $\phi(\bx)=\sum_{i=1}^{\ell+t}x_i$, which
gives SumCT investigated by \cite{wazl:2022a} (see also Table~\ref{tbl:obj-funs}), and
  \item $t=0$, $k=1$, $P_i=\bp$ (a unit vector) for $1\le i\le\ell$, which gives the main problem of \cite{balu:2024}
  (in the paper, $\phi(\bx)=\sum_{i=1}^{\ell}\psi_i(x_i)$
for $\bx=[x_1,x_2,\ldots,x_{\ell}]^{\T}$ with each $\psi_i$ being a convex function of a single-variable).
\end{enumerate}

Despite that we can take $\bbP=\STM{k}{n}$ and $Q=I_k$ here, with favorable compositions of $P_i$
as submatrices of $P$, we can find a better $Q$, other than $I_k$,
so that the objective value increases more than {\bf the NPDo Ansatz} suggests. Here are two of them:
\begin{equation}\label{eq:Part-Consistent}
\framebox{
\parbox{12.0cm}{
\vspace{-0.4cm}
\begin{enumerate}[(a)]
  \item $J_{\ell+i}^{\T}J_{\ell+j}=0$ for  $i\ne j$, which means that $P_{\ell+i}$ and $P_{\ell+j}$ share no common column of $P$;
  \item For $1\le i\le\ell,\, 1\le j\le t$, either $J_{\ell+j}^{\T}J_i=0$ or no row of $J_{\ell+j}^{\T}J_i$  is $0$, which means
        either $P_i$ and $P_{\ell+j}$ share no common column of $P$ or  $P_{\ell+j}$  is a submatrix of $P_i$.
\end{enumerate}
\vspace{-0.4cm}
}}
\end{equation}
\begin{equation}\label{eq:Part-Consistent'}
\framebox{
\parbox{12.0cm}{
\vspace{-0.4cm}
\begin{enumerate}[(a)]
  \item either $J_{\ell+i}^{\T}J_{\ell+j}=0$ or $J_{\ell+i}^{\T}J_{\ell+j}=I$ for any $i\ne j$,
        which means that $P_{\ell+i}$ and $P_{\ell+j}$ either
        share no common column of $P$, or $P_{\ell+i}=P_{\ell+j}$, i.e., the same submatrix of $P$;
  \item the same as item (b) in \eqref{eq:Part-Consistent}.
\end{enumerate}
\vspace{-0.4cm}
}}
\end{equation}
For \eqref{eq:Part-Consistent}, we determine $Q$ implicitly by
$\wtd P_{\ell+j}=\what P_{\ell+j}S_j$ where $S_j$ is an orthonormal polar factor of
$\phi_{\ell+j}(T_1(P)) \what P_{\ell+j}^{\T}D_j$ for $1\le j\le t$. In the case of \eqref{eq:Part-Consistent'},
$J_{\ell+j}$ for $1\le j\le t$ can be divided into no more than $t$ groups, and within each group all $J_{\ell+j}$ are the same
and two $J_{\ell+j}$ from different groups share no common column of $P$ at all.
For ease of presentation, let us say the set of indices $\{1,2,\ldots,t\}$ is divided into
$\tau$ exclusive subsets $\bbI_q$ for $1\le q\le\tau$ such that
$$
\framebox{
\parbox{10.0cm}{
$
\cup_{j=1}^{\tau}\bbI_j=\{1,2,\ldots,t\},\,\,
\bbI_i\cap\bbI_j=\emptyset\,\,\mbox{for $i\ne j$},
$
and $J_{\ell+i}=J_{\ell+j}$ if $i,\,j$ belong to the same $\bbI_q$ but $J_{\ell+i}^{\T}J_{\ell+j}=0$ otherwise.
}}
$$
%
Now determine $Q$ implicitly by taking just one index $j$ from each $\bbI_q$ for $1\le q\le\tau$ and letting
$\wtd P_{\ell+j}=\what P_{\ell+j}S_q$ where $S_q$ is an orthonormal polar factor of
\begin{equation}\label{eq:Part-Consistent'-1}
\what P_{\ell+j}^{\T}\Big[\sum_{i\in\bbI_q}\phi_{\ell+i}(T_1(P))\, D_i\Big].
\end{equation}
%
\end{example}

\begin{example}\label{eg:T2}
The {\bf second} concrete mapping of $T$ is
\begin{equation}\label{eq:T-F}
T_2\,:\, P\in\STM{k}{n} \to T_2(P):=\begin{bmatrix}
                                        \|P_1^{\T}A_1P_1\|_{\F}^2 \\
                                        \vdots \\
                                        \|P_{\ell}^{\T}A_{\ell}P_{\ell}\|_{\F}^2 \\
                                        \|P_{\ell+1}^{\T}D_1\|_{\F}^2 \\
                                        \vdots \\
                                        \|P_{\ell+t}^{\T}D_t\|_{\F}^2 \\
                                      \end{bmatrix} \in\bbR^{\ell+t}.
\end{equation}
Either $\ell=0$ or $t=0$ is allowed. To use Theorem~\ref{thm:quad4tr-diff-eq}, we notice that
$$
\|P_i^{\T}A_iP_i\|_{\F}^2=\tr((P_i^{\T}A_iP_i)^2),\quad
\|P_{\ell+j}^{\T}D_j\|_{\F}^2=\tr(P_{\ell+j}^{\T}D_jD_j^{\T}P_{\ell+j}).
$$
Hence if all $A_i\succeq 0$, then
the consistency conditions in \eqref{eq:cond4AF-NPDo:cvx} are satisfied with $\bbP=\STM{k}{n}$, $Q=I_k$,
$\alpha=1$, $\beta_i=3$ for $1\le i\le\ell$, and $\beta_{\ell+j}=1$ for $1\le j\le t$.
Theorem~\ref{thm:main-npd-cvx} applies, assuming $\phi_j(\bx)\ge 0$ for $1\le j\le\ell+t$.


A special case of $T_2$ is $t=0$ and $P_i=P$ for $1\le i\le\ell$, for which \eqref{eq:Master-OptSTM} with $T=T_2$
gives the key optimization problem in the uniform multidimensional scaling (UMDS)  \cite{zhzl:2017}
(see also Table~\ref{tbl:obj-funs}).
\end{example}

\begin{example}\label{eg:T3}
More generally, the {\bf third} concrete mapping of $T$ is
\begin{equation}\label{eq:T-pow-tr}
T_3\,:\, P\in\STM{k}{n} \to T_3(P):=\begin{bmatrix}
                                        \tr((P_1^{\T}A_1P_1)^{m_1}) \\
                                        \vdots \\
                                        \tr((P_{\ell}^{\T}A_{\ell}P_{\ell})^{m_{\ell}}) \\
                                        \tr((P_{\ell+1}^{\T}D_1)^{m_{\ell+1}}) \\
                                        \vdots \\
                                        \tr((P_{\ell+t}^{\T}D_t)^{m_{\ell+t}}) \\
                                      \end{bmatrix} \in\bbR^{\ell+t},
\end{equation}
where integer $m_i\ge 1$ for all $1\le i\le\ell+t$. It reduces to Example~\ref{eg:T1} if  $m_i=1$ for all $1\le i\le\ell+t$.
Either $\ell=0$ or $t=0$ is allowed. Suppose all $A_i\succeq 0$.
Suppose all $P_i$ together have the properties in \eqref{eq:Part-Consistent}. Then
the consistency conditions in \eqref{eq:cond4AF-NPDo:cvx} are satisfied with
\begin{equation}\label{eq:D-sep}
\bbP=\{P\in\STM{k}{n}\,:\, P_{\ell+j}^{\T}D_j\succeq 0\,\,\mbox{for $1\le j\le t$}\},
\end{equation}
$\alpha=1$, $\beta_i=2m_i-1$ for $1\le i\le\ell$ and $\beta_{\ell+j}=m_{\ell+j}-1$ for $1\le j\le t$, by
Theorems~\ref{thm:lin4tr-diff-eq} and~\ref{thm:quad4tr-diff-eq}.
Theorem~\ref{thm:main-npd-cvx} applies, assuming $\phi_j(\bx)\ge 0$ for $1\le j\le\ell$
and for each $j\in\{\ell+1,\ldots,\ell+t\}$ with $m_j\ge 2$. $Q$ is implicitly determined by
$\wtd P_{\ell+j}=\what P_{\ell+j}S_j$ where $S_j$ is an orthonormal polar factor of
$\what P_{\ell+j}^{\T}D_j$ for $1\le j\le t$.
\end{example}

We pointed out in Example~\ref{eg:T1} that judiciously choosing $Q$ to go from $\what P$ to $\wtd P$ in
\eqref{eq:cond4AF-NPDo:cvx} can increase the objective function value more than Theorem~\ref{thm:main-npd-cvx} suggests.
To further strengthen this point,
%
we consider a special case of $T_1$:
$P_i=P$ for $1\le i\le \ell+t$. For ease of future reference,  denote the special $T_1$ by
\begin{equation}\label{eq:T-tr-a}
T_{1a}\,:\, P\in\STM{k}{n} \to T_{1a}(P):=\begin{bmatrix}
                                        \tr(P^{\T}A_1P) \\
                                        \vdots \\
                                        \tr(P^{\T}A_{\ell}P) \\
                                        \tr(P^{\T}D_1) \\
                                        \vdots \\
                                        \tr(P^{\T}D_t) \\
                                      \end{bmatrix} \in\bbR^{\ell+t}.
\end{equation}

\begin{theorem}\label{thm:main-npd-cvx:T1a}
Consider $f=\phi\circ T_{1a}$ where $\phi$ is convex and differentiable with
$\phi_j(\bx)\ge 0$ for $1\le j\le\ell$,  and suppose $A_i\succeq 0$ for $1\le i\le\ell$.
Given $\what P\in\STM{k}{n},\,P\in\STM{k}{n}$, let $\wtd P=\what PQ$ where
$Q$ is an orthonormal polar factor of $\what P^{\T}\scrD(P)$ with
\begin{equation}\label{eq:scrD}
\scrD(P)=\sum_{j=1}^{t}\phi_{\ell+j}(T_{1a}(P))\, D_j.
\end{equation}
If
$
\tr(\what P^{\T}\scrH(P))\ge\tr(P^{\T}\scrH(P))+\eta,
$
then
$
f(\wtd P)\ge f(P)+\eta+\delta,
$
where $\delta=\|\what P^{\T}\scrD(P)\|_{\tr}-\tr(\what P^{\T}\scrD(P))$.
In particular, {\bf the NPDo Ansatz} holds  with $\omega=1$.
\end{theorem}

\begin{proof}
Along the lines in the proof of Theorem~\ref{thm:main-npd-cvx}, here we will have
\begin{align*}
\tr(P^{\T}\scrH(P))
  &=2\sum_{i=1}^{\ell}\phi_i(\bx)\,x_i+\sum_{j=1}^{t}\phi_{\ell+j}(\bx)\,x_{\ell+j}, \\
\|\what P^{\T}\scrD(P)\|_{\tr}
  &=\tr(\wtd P^{\T}\scrD(P)) \qquad(\mbox{since $\wtd P^{\T}\scrD(P)=Q^{\T}[\what P^{\T}\scrD(P)]\succeq 0$}) \\
  &=\sum_{j=1}^{t}\phi_{\ell+j}(\bx)\,\wtd x_{\ell+j}, \\
\tr(\what P^{\T}\scrH(P))
  &=2\sum_{i=1}^{\ell}\phi_i(\bx)\,\tr(\what P^{\T}A_iP)+\sum_{j=1}^{t}\phi_{\ell+j}(\bx)\,\tr(\what P^{\T}D_j) \\
  &\le\sum_{i=1}^{\ell}\phi_i(\bx)\,\Big[\tr(\what P^{\T}A_i\what P)+\tr(P^{\T}A_iP)\Big]
       +\tr(\what P^{\T}\scrD(P)) \\
  &=\sum_{i=1}^{\ell}\phi_i(\bx)\,\Big[\wtd x_i+x_i\Big]+\|\what P^{\T}\scrD(P)\|_{\tr}-\delta \\
  &=\sum_{i=1}^{\ell}\phi_i(\bx)\,\Big[\wtd x_i+x_i\Big]+\sum_{j=1}^{t}\phi_{\ell+j}(\bx)\,\wtd x_{\ell+j}-\delta.
\end{align*}
Plug them into $\eta+\tr(P^{\T}\scrH(P)\le\tr(\what P^{\T}\scrH(P)$ and simplify the resulting inequality to get
$
\eta+\delta+\nabla\phi(\bx)^{\T}\bx\le\nabla\phi(\bx)^{\T}\wtd\bx,
$
and then apply Lemma~\ref{lm:convex-mono} to conclude the proof.
\end{proof}

\begin{table}[t]
\renewcommand{\arraystretch}{1.4}
\caption{Condition on $\phi_j$ and choice of $Q_i$ at Line 4 of \Cref{alg:SCF4NPDo}}\label{tbl:T1T2T3T1a}
\centerline{\small
\begin{tabular}{|c|c|}
  \hline
$T_1$ and $T_{1a}$
       & $\phi_j\ge 0$ for $1\le j\le \ell$, $Q_i=I_k$. \\ \hline
$T_1$ with \eqref{eq:Part-Consistent} and $t\ge 1$
       & \parbox{9cm}{\strut $\phi_j\ge 0$ for $1\le j\le \ell$,
         $Q_i$ is implicitly determined by
         $\wtd P_{\ell+j}^{(i+1)}=\what P_{\ell+j}^{(i)}S_j$
         where         $S_j$ is an orthonormal polar factor of $\phi_{\ell+j}(T_1(P))\big[\what P_{\ell+j}^{(i)}\big]^{\T}D_j$ for
         $1\le j\le t$. \strut} \\ \hline
$T_1$ with \eqref{eq:Part-Consistent'} and $t\ge 1$
       & \parbox{9cm}{\strut $\phi_j\ge 0$ for $1\le j\le \ell$,
         $Q_i$ is implicitly determined by: for just one element $j$ from each $\bbI_q$ ($1\le q\le\tau$),
         $\wtd P_{\ell+j}^{(i+1)}=\what P_{\ell+j}^{(i)}S_q$
         where  $S_q$ is an orthonormal polar factor of
         $\big[\what P_{\ell+j}^{(i)}\big]^{\T}\big[\sum_{p\in\bbI_q}\phi_{\ell+p}(T_1(P))D_p\big]$ for  $1\le q\le \tau$.\strut} \\ \hline
$T_{1a}$ with $t\ge 1$
       & \parbox{9cm}{\strut $\phi_j\ge 0$ for $1\le j\le \ell$, $Q_i$ is an orthonormal polar factor of
        $\big[\what P^{(i)}\big]^{\T}\scrD(P^{(i)})$.\strut} \\ \hline
$T_2$
       & $\phi_j\ge 0$ for $1\le j\le \ell+t$, $Q_i=I_k$. \\ \hline
$T_3$ with $t=0$
       & $\phi_j\ge 0$ for $1\le j\le \ell$, $Q_i=I_k$. \\ \hline
$T_3$ with \eqref{eq:Part-Consistent} and $t\ge 1$
       & \parbox{9cm}{\strut $\phi_j\ge 0$ for $1\le j\le\ell$
         and for each $j\in\{\ell+1,\ldots,\ell+t\}$ with $m_j\ge 2$, $Q_i$ is implicitly determined by
         $\wtd P_{\ell+j}^{(i+1)}=\what P_{\ell+j}^{(i)}S_j$
         where         $S_j$ is an orthonormal polar factor of $\big[\what P_{\ell+j}^{(i)}\big]^{\T}D_j$ for
         $1\le j\le t$.\strut} \\ \hline
\multicolumn{2}{l}{\small * $\phi_j(\bx):=\partial\phi(\bx)/\partial x_j$  for $\bx=[x_j]$.}
\end{tabular}
}
\end{table}

Theorem~\ref{thm:main-npd-cvx:T1a} improves Theorem~\ref{thm:main-npd-cvx} when it comes to $T=T_{1a}$:
the objective value
increases additional $\delta$ more. We notice, by Lemma~\ref{lm:polar2max},
that $\delta\ge 0$ always and it is strict
unless $\what P^{\T}\scrD(P)\succeq 0$.
Also note $\wtd P$ satisfies
$\wtd P^{\T}\scrD(P)\succeq 0$.
As a result of Theorem~\ref{thm:main-npd-cvx:T1a}, along the same line of the proof of Theorem~\ref{thm:maxers-NPD},
we have another necessary condition on a maximizer $P_*$ of \eqref{eq:Master-OptSTM} with $T=T_{1a}$
in Corollary~\ref{cor:T1a}, beyond the ones in
Theorem~\ref{thm:maxers-NPD}.

\begin{corollary}\label{cor:T1a}
Suppose $A_i\succeq 0$ for $1\le i\le\ell$ and that $\phi$ is convex and differentiable with
$\phi_j(\bx)\ge 0$ for $1\le j\le\ell$. If $P_*$ is a  maximizer of \eqref{eq:Master-OptSTM} with $T=T_{1a}$, then we
have not only \eqref{eq:KKT} for $P=P_*$ with $\Lambda=\Lambda_*:=P_*^{\T}\scrH(P_*)\succeq 0$
but also $P_*^{\T}\scrD(P_*)\succeq 0$.
\end{corollary}

In Table~\ref{tbl:T1T2T3T1a}, we list conditions on partial derivatives $\phi_j$
and the best choices of $Q_i$ at Line~4 of \Cref{alg:SCF4NPDo}
that, when it is not $I_k$, can increase the objective value even more per SCF iterative step than {\bf the NPDo Ansatz} suggests.
Having said that, we notice that $P^{(i)}$ as $i$ varies may belong to different subsets $\bbP$ of
$\STM{k}{n}$. For example, with $T=T_{1a}$, if $Q_i$ is calculated according to Theorem~\ref{thm:main-npd-cvx:T1a}, i.e., $Q_i$ is an orthonormal polar factor of $\big[\what P^{(i)}\big]^{\T}\scrD(P^{(i)})$,
then $P^{(i+1)}\in\bbP_i:=\{X\in\STM{k}{n}\,:\, X^{\T}\scrD(P^{(i)})\succeq 0\}$ that varies from one iterative step to the next. Eventually,
$\bbP_i$ approaches $\bbP_*:=\{X\in\STM{k}{n}\,:\, X^{\T}\scrD(P_*)\succeq 0\}$ by
Corollary~\ref{cor:T1a}.
Similar comments can be said about $T_1$ with~\eqref{eq:Part-Consistent} or~\eqref{eq:Part-Consistent'} that we discussed
towards the end of Example~\ref{eg:T1}.
Numerically, such variations in $\bbP$ does not pose any problem for \Cref{alg:SCF4NPDo}
to compute an approximate maximizer for the maximization problem \eqref{eq:Master-OptSTM}.

\begin{remark}\label{rk:app-NPDo-cvx}
We conclude this section by commenting on the applicability of the results of this section to the objective functions
in Table~\ref{tbl:obj-funs} via convex compositions of atomic functions for NPDo. Essentially our results are applicable to
all of those that are not in the quotient form, i.e., SEP, MBSub, SumCT, TrCP, UMDS, and DFT, assuming that matrices $A$ and $A_i$
are positive semidefinite. SEP is simply about the atomic function $\tr(P^{\T}AP)$.
For OLDA and SumTR, the corresponding composing functions $\phi$ are
$x_2/x_1$ and $x_2/x_1+x_3$, respectively, but both are non-convex.
The composing function for OCCA is $\phi_0(\bx)=x_2/\sqrt{x_1}$ where $\bx=[x_1,x_2]^{\T}$, which is not convex but whose square
$\phi(\bx):=[\phi_0(\bx)]^2=x_2^2/x_1$ is convex for $x_2\ge 0$ and $x_1>0$. Unfortunately,
the atomic function associated with $x_1$ is $\tr(P^{\T}BP)$, for which
$\phi_1(\bx):=\partial\phi(\bx)/\partial x_1=-(x_2/x_1)^2\le 0$, violating the conditions of Theorems~\ref{thm:main-npd-cvx}
and~\ref{thm:main-npd-cvx:T1a}. A similar argument applies to $\Theta$TR.
So the NPDo approach does not work for OLDA, OCCA, and $\Theta$TR. But, fortunately, the NEPv approach next will.
\end{remark}

\clearpage
\part{The NEPv Approach}

\section{The NEPv Framework}\label{sec:NEPv-theory}
There are cases for which
\begin{equation}\label{eq:natNEPv}
\scrH(P):=\frac{\partial f(P)}{\partial P}\equiv H(P)P
\end{equation}
for $P\in\STM{k}{n}$ (or even $\bbR^{n\times k}$), where $H(P)\in\bbR^{n\times n}$ is a symmetric matrix-valued function
dependent of $P$, e.g., for
$f(P)=\frac {\tr(P^{\T}AP)}{\tr(P^{\T}BP)}+\tr(P^{\T}CP)$ from the sum of trace ratios (SumTR) \cite{zhli:2014a,zhli:2014b},
which includes SEP and OLDA as special cases,
\begin{equation}\label{eq:H(p):sumTR}
\scrH(P)=2\left[\frac 1{\tr(P^{\T}BP)}\left(A-\frac {\tr(P^{\T}AP)}{\tr(P^{\T}BP)}\, B\right)+C\right]P\equiv H(P)P
\end{equation}
for $P\in\bbR^{n\times k}$, where $H(P)$ is easily identified.
In fact, Lu and Li~\cite[Lemma 2.1]{luli:2024} show that
\eqref{eq:natNEPv} always hold for some symmetric and right-unitarily invariant $H(P)$ if $f$ is right-unitarily invariant.
As a result of \eqref{eq:natNEPv}, the KKT condition \eqref{eq:KKT} is an NEPv:
\begin{equation}\label{eq:NEPv-form}
H(P)P=P\Omega, \,\, P\in\STM{k}{n}.
\end{equation}
Necessarily $\Omega=P^{\T}H(P)P\in\bbR^{k\times k}$ is symmetric.

But not all $\scrH(P)={\partial f(P)}/{\partial P}$ take the form $H(P)P$, and in the latter,
we can still construct some $H(P)$ to turn the KKT condition \eqref{eq:KKT} equivalently into an NEPv in the form of
\eqref{eq:NEPv-form} under some mild condition.
For example, for
$f(P)=\frac {\tr(P^{\T}AP+P^{\T}D)}{[\tr(P^{\T}BP)]^{\theta}}$
of the $\theta$-trace ratio problem ($\Theta$TR)
which includes OCCA and the MAXBET subproblem as special cases,
the authors of \cite{wazl:2023} used
\begin{equation}\label{eq:H(P):theta-TR}
H(P)=\frac 2{[\tr(P^{\T}BP)]^{\theta}}
     \left(A+\frac {DP^{\T}+PD^{\T}}2-\theta\,\frac {\tr(P^{\T}AP+P^{\T}D)}{\tr(P^{\T}BP)}\,B\right).
\end{equation}
In general, we can always take
\begin{equation}\label{eq:H(P):always}
H(P):=[\scrH(P)]P^{\T}+P[\scrH(P)]^{\T}=\left[\frac{\partial f(P)}{\partial P}\right] P^{\T}+P\left[\frac{\partial f(P)}{\partial P}\right]^{\T}.
\end{equation}
In the case when $\scrH(P)\equiv H_0(P)P$, this $H(P)$ becomes $2\,H_0(P)$ for $P\in\STM{k}{n}$.

Why $H(P)$ in \eqref{eq:H(P):theta-TR} and \eqref{eq:H(P):always}  work for $\Theta$TR and
in general, respectively, can be best explained by the next theorem.

\begin{theorem}\label{thm:H(P)-eligibility}
Let $H(P)\in\bbR^{n\times n}$ be a symmetric matrix-valued function satisfying
\begin{equation}\label{eq:cond:KKT=NEPv}
H(P)P-\frac{\partial f(P)}{\partial P}=P\,\scrM(P)\quad\mbox{for $P\in\STM{k}{n}$},
\end{equation}
where $\scrM(P)\in\bbR^{k\times k}$ is some matrix-valued function.
$P\in\STM{k}{n}$ is a solution to the KKT condition \eqref{eq:KKT} if and only if
it is a solution to NEPv \eqref{eq:NEPv-form}
and $\scrM(P)$ is  symmetric.
\end{theorem}

\begin{proof}
If $P$ is a solution to the KKT condition \eqref{eq:KKT}, i.e.,
$\scrH(P)=P\Lambda$, $P\in\STM{k}{n}$, and $\Lambda=\Lambda^{\T}$.
Then, by \eqref{eq:cond:KKT=NEPv},
$$
H(P)P=P\Lambda+P\scrM(P)=P(\Lambda+\scrM(P))=:P\Omega,
$$
where $\Omega=\Lambda+\scrM(P)$ is symmetric because alternatively $\Omega=P^{\T}H(P)P$ which is symmetric, and hence $\scrM(P)=\Omega-\Lambda$ is also symmetric.
On the other hand, if $P$ is a solution to NEPv \eqref{eq:NEPv-form} such that $\scrM(P)$ is  symmetric, then again
by  \eqref{eq:cond:KKT=NEPv}
$$
\scrH(P)=P\Omega-P\scrM(P)=P\big(\Omega-\scrM(P)\big)=:P\Lambda,
$$
where $\Lambda=\Omega-\scrM(P)$ is symmetric because $\scrM(P)$ is assumed symmetric.
\end{proof}

According to Theorem~\ref{thm:H(P)-eligibility}, to solve the KKT condition \eqref{eq:KKT} via solving
NEPv \eqref{eq:NEPv-form} with an $H(P)$ that satisfies \eqref{eq:cond:KKT=NEPv}, we need to limit the solutions to those  of the NEPv such that
$\scrM(P)$ is symmetric.
Return to the concrete $H(P)$ given by \eqref{eq:H(P):theta-TR} and \eqref{eq:H(P):always}. It can be
verified that, for $H(P)$ in \eqref{eq:H(P):theta-TR} for $\Theta$TR,
$$
H(P)P-\scrH(P)=P\left(\frac 1{[\tr(P^{\T}BP)]^{\theta}}\, {D^{\T}P}\right),
$$
and hence any solution to the resulting NEPv \eqref{eq:NEPv-form} such that $D^{\T}P$ is symmetric is a KKT
point of $\Theta$TR and vice versa.
Similarly, for $H(P)$ in \eqref{eq:H(P):always} in general,
$$
H(P)P-\scrH(P)=P([\scrH(P)]^{\T}P)
$$
and hence any solution to the resulting NEPv \eqref{eq:NEPv-form} such that $[\scrH(P)]^{\T}P$ is symmetric is a KKT
point and vice versa.

We note that \eqref{eq:cond:KKT=NEPv} is a guiding equation for $H(P)$, and satisfying \eqref{eq:cond:KKT=NEPv}
yields a candidate $H(P)$ and the resulting NEPv \eqref{eq:NEPv-form}.
In general, given $\scrH(P)$, there are infinitely many $H(P)$
that satisfy \eqref{eq:cond:KKT=NEPv}.


Our goal in this part is still the same as
in Part~I, namely establishing conditions under which SCF \eqref{eq:SCF-form:NEPv:intro} on NEPv \eqref{eq:NEPv-form} is provably convergent,
except that the conditions will be imposed on $H(P)$, instead of $\scrH(P)$ earlier.
The developments in this section follow the lines of \cite{wazl:2023,zhwb:2022}, but in more abstract terms.

\subsection{The NEPv Ansatz}\label{ssec:NEPvAssum}
The successes of the NEPv approach used in \cite{wazl:2023,zhwb:2022} for solving
OCCA and $\Theta$TR relies on certain monotonicity lemmas which inspire us to make the following
ansatz to build our framework upon. It also requires a sufficiently inclusive subset $\bbP$ of
$\STM{k}{n}$ as in the NPDo framework in Part~I.

\smallskip\noindent
{\bf The NEPv Ansatz.}
{\em
For function $f$ defined in some neighborhood $\STMnbr$ of the Stiefel manifold $\STM{k}{n}$,
there is a
symmetric matrix-valued function $H(P)\in\bbR^{n\times n}$ such that for
$\what P\in\STM{k}{n},\,P\in\bbP\subseteq\STM{k}{n}$,
if
\begin{equation}\label{eq:NEPv-assume}
\tr(\what P^{\T}H(P)\what P)\ge\tr(P^{\T}H(P)P)+\eta
\quad\mbox{for some $\eta\in\bbR$},
\end{equation}
then
there exists $Q\in\STM{k}{k}$ such that $\wtd P=\what PQ\in\bbP$ and
$f(\wtd P)\ge f(P)+\omega\eta$,
where $\omega$ is some positive constant, independent of $P$ and $\what P$.
}

\smallskip
For any $P\in\STM{k}{n}$, there is always
        $\what P\in\STM{k}{n}$ such that \eqref{eq:NEPv-assume} holds with some $\eta>0$, unless for that given $P$,
         $H(P)P=P\Omega$ holds and the eigenvalues of $\Omega$ consist of the $k$ largest eigenvalues of $H(P)$.
        In fact, we can take
        $\what P\in\STM{k}{n}$ to be an orthonormal basis matrix of the eigenspace of $H(P)$
        associated with its $k$ largest eigenvalues, which also maximizes
        $\tr(X^{\T}H(P)X)$ over $X\in\STM{k}{n}$ \cite{fan:1949}.
Hence, for the purpose of solving \eqref{eq:main-opt},
we may relax the ansatz to $\eta\ge 0$ only.
        In general, it is the desirable aim, $f(\wtd P)\ge f(P)+\omega\eta$, in {\bf the NEPv Ansatz}
        that needs to be verified before the general theory of this section can be applied.
Below we will use the same objective function in Example~\ref{eg:motivation:NPDo} to  rationalize this ansatz.

\begin{example}\label{eg:motivation:NEPv}
Consider $f(P)=\tr(P^{\T}AP)+\tr((P^{\T}D)^2)$
where $A\in\bbR^{n\times n}$ is symmetric and
$D\in\bbR^{n\times k}$. Note now no longer $A$ is required to be positive semidefinite as it had to be in Example~\ref{eg:motivation:NPDo}.
Since $\scrH(P)=2AP+2DP^{\T}D$, no longer there exists a symmetric $H(P)$ such that \eqref{eq:natNEPv} holds.
Our discussion above leads us to use
$H(P)=2A+2(DP^{\T}DP^{\T}+PD^{\T}PD^{\T})$ for which
$$
H(P)P-\scrH(P)=P\big[2(D^{\T}P)^2\big])\quad\mbox{for}\quad P\in\STM{k}{n},
$$
satisfying \eqref{eq:cond:KKT=NEPv}.
To achieve equivalency between the KKT condition~\eqref{eq:KKT} and  NEPv~\eqref{eq:NEPv-form}, according to
Theorem~\ref{thm:H(P)-eligibility} we should limit the scope to those $P\in\STM{k}{n}$ such that $(D^{\T}P)^2$ is symmetric.
Actually we will further limit the scope to $P\in\bbP=\STM{k}{n}_{D+}$.
Suppose now that
\eqref{eq:NEPv-assume} holds for $P\in\STM{k}{n}_{D+}$ and $\what P\in\STM{k}{n}$, or equivalently,
\begin{equation}\label{eq:motivation:NEPv-1}
2\tr(\what P^{\T}A\what P)+4\tr(\what P^{\T}DP^{\T}DP^{\T}\what P)
    \ge 2\tr(P^{\T}AP)+4\tr((P^{\T}D)^2)+\eta.
\end{equation}
Next let $Q\in\STM{k}{k}$ be an orthonormal polar factor
of $\what P^{\T}D$ and let $\wtd P=\what PQ\in\bbP$. We find that
$\tr(\wtd P^{\T}A\wtd P)=\tr(\what P^{\T}A\what P)$, but it remains to break the second term in the left-hand side of
\eqref{eq:motivation:NEPv-1} apart so that $P$ and $\what P$ are detached. For that purpose, we note
\begin{align}
2\tr(\what P^{\T}DP^{\T}DP^{\T}\what P)
   &\le 2\|\what P^{\T}DP^{\T}DP^{\T}\what P\|_{\tr} \qquad (\mbox{by Lemma~\ref{lm:maxtrace}}) \nonumber \\
   &\le 2\|\what P^{\T}DP^{\T}D\|_{\tr}\|P^{\T}\what P\|_2 \nonumber \\
   &\le 2\|\what P^{\T}DP^{\T}D\|_{\tr} \qquad(\mbox{since $\|P^{\T}\what P\|_2\le 1$})\nonumber \\
  &\le\tr((Q^{\T}\what P^{\T}D)^2)+\tr((P^{\T}D)^2) \qquad(\mbox{by Lemma~\ref{lm:vN-tr-ineq-ext1}})\nonumber \\
  &=\tr((\wtd P^{\T}D)^2)+\tr((P^{\T}D)^2). \label{eq:motivation:NEPv-2}
\end{align}
Combine \eqref{eq:motivation:NEPv-1} and \eqref{eq:motivation:NEPv-2} to get
$f(\wtd P)\ge f(P)+\eta/2$ upon noticing $\tr(\wtd P^{\T}A\wtd P)=\tr(\what P^{\T}A\what P)$.
As in Example~\ref{eg:motivation:NPDo}, we observe the critical conditions:
$P^{\T}D\succeq 0$ and $\wtd P^{\T}D\succeq 0$ (i.e., $P,\,\wtd P\in\bbP$), that we used to derive \eqref{eq:motivation:NEPv-2},
where $\wtd P^{\T}D=Q^{\T}\what P^{\T}D\succeq 0$ is again made possible by the chosen $Q$.
\end{example}

\begin{remark}\label{rk:NEPv-assume}
A few comments on {\bf the NEPv Ansatz} are in order.
\begin{enumerate}[(i)]
  \item {\bf The NEPv Ansatz} critically involves a symmetric matrix-valued function $H(P)\in\bbR^{n\times n}$ that has to be constructed. In
        Theorem~\ref{thm:H(P)-eligibility}, we provide a guiding equation~\eqref{eq:cond:KKT=NEPv}
        for the purpose
        so that the KKT condition
        \eqref{eq:KKT} and NEPv \eqref{eq:NEPv-form} are equivalent as far as solving the associated optimization problem
        \eqref{eq:main-opt} is concerned. It is fulfilled naturally when $\scrH(P)\equiv H(P)P$ for $P\in\STM{k}{n}$ exactly (e.g., for OLDA, SumTR, TrCP, UMDS, and DFT), but
        at other times, we will have to construct $H(P)$ individually based on the particularity of $\scrH$
        as in \cite{zhwb:2022} for OCCA, \cite{wazl:2023} for $\Theta$TR, \cite{zhys:2020}
        for MAXBET and Example~\ref{eg:motivation:NEPv}, or we simply use the generic  \eqref{eq:H(P):always}.
        The ansatz does not demand any explicit association of $H(P)$ with $\scrH(P)$, but conceivably
        they should be highly related, such as the relation imposed by~\eqref{eq:cond:KKT=NEPv}.

  \item One may argue that the ansatz 
        might be made unnecessarily complicated. After all
        $\tr(\what P^{\T}H(P)\what P)=\tr(\wtd P^{\T}H(P)\wtd P)$. Should  we get rid of $\what P$ in the ansatz altogether?
        One possibility is to require that
        \begin{equation}\label{eq:NEPv-assume-strong}
        \framebox{
        \parbox{8.5cm}{
        $\tr(\wtd P^{\T}H(P)\wtd P)\ge\tr(P^{\T}H(P)P)+\eta$ for $P,\,\wtd P\in\bbP$
        implies $f(\wtd P)\ge f(P)+\omega\eta$.
        }}
        \end{equation}
        This is a stronger version, however,
        assuming
        that for any $\what P\in\STM{k}{n}$, there exists $Q\in\STM{k}{k}$ such that $\wtd P=\what PQ\in\bbP$.
        Here is why.
        Suppose that \eqref{eq:NEPv-assume-strong} holds. Given $\what P\in\STM{k}{n},\,P\in\bbP\subseteq\STM{k}{n}$,
        let $\wtd P=\what PQ\in\bbP$ for some $Q\in\STM{k}{k}$. If \eqref{eq:NEPv-assume} holds, then
        $$
        \tr(\wtd P^{\T}H(P)\wtd P)=\tr(\what P^{\T}H(P)\what P)\ge\tr(P^{\T}H(P)P)+\eta,
        $$
        which, under \eqref{eq:NEPv-assume-strong}, yields $f(\wtd P)\ge f(P)+\omega\eta$, proving the desired inequality
        of {\bf the NEPv Ansatz}.

  \item When $f(P)$ is right-unitarily invariant, $H(P)$ always exists such that \eqref{eq:natNEPv} holds and can be taken to be right-unitarily invariant, too \cite[Lemma 2.1]{luli:2024}.
        In such a case, the ansatz 
        can be simplified to: $Q=I_k$ and $\wtd P=\what P$ always because $f(\what P)=f(\wtd P)$
        regardless of $Q$. Also often $\bbP=\STM{k}{n}$.


  \item Introducing a subset $\bbP$ of $\STM{k}{n}$ and judiciously choosing  $Q$
        are for generality to deal with the case when $f(P)$
        is not right-unitarily invariant, e.g., the one in Example~\ref{eg:motivation:NEPv} and those in Table~\ref{tbl:obj-funs} in section~\ref{sec:intro} that involve
        $D$ or $D_i$. Suitable $Q$ can also increase the objective value  more than $\omega\eta$.
        For example, for $\Theta$TR with $H(P)$ given
        by \eqref{eq:H(P):theta-TR}, {\bf the NEPv Ansatz} holds with taking
        $Q$ to be an orthonormal polar factor of $\what P^{\T}D$. In fact, along the line of the proof of
        \cite[Lemma 2.1]{wazl:2023},
        assuming $B\succeq 0$, $\rank(B)>n-k$, and $\tr(P^{\T}AP+P^{\T}D)\ge 0$ in the case of $0<\theta<1$ but otherwise
        no need to impose nonnegativity  on $\tr(P^{\T}AP+P^{\T}D)$ for $\theta\in\{0,1\}$,
        we can improve the conclusion of \cite[Theorem 2.2]{wazl:2023}  to (in the current notation)
        \begin{align}
        f(\wtd P) \ge f(P)&+\frac 12\left(\frac {\tr_{\min,k}(B)}{\tr_{\max,k}(B)}\right)^{\theta}\,\Big[\tr(\what P^{\T}H(P)\what P)-\tr(P^{\T}H(P)P)\Big] \nonumber\\
               &+[\tr_{\max,k}(B)]^{-\theta}\,\Big[\|\what P^{\T}D\|_{\tr}-\tr(\what P^{\T}DP^{\T}\what P)\Big],
                   \label{eq:f-inc-theta-TR:refined}
        \end{align}
        where $0\le\theta\le 1$.
        The last term is contributed by
        the selection of $Q$ as described.
        A proof of \eqref{eq:f-inc-theta-TR:refined} is given in Appendix~\ref{sec:f-inc-theta-TR:pf}.

        The comments in Remark~\ref{rk:NPDo-assume}(ii) on $\bbP$ apply here, too.

  \item It is tempting to stipulate  $f(\what P)\ge f(P)+\omega\eta$, but that is either false or just hard to prove
        for the one in Example~\ref{eg:motivation:NEPv} and some of those in Table~\ref{tbl:obj-funs} that involve $D$. Often in our algorithms to solve~\eqref{eq:main-opt} iteratively, with $P$ being
        the current approximate maximizer, assuming {\bf the NEPv Ansatz}, we naturally
        compute $\what P$ that maximizes $\tr(X^{\T}H(P)X)$ over $X\in\STM{k}{n}$. With that $\what P$, settling
        whether $f(\what P)\ge f(P)+\omega\eta$ or not can be a hard task, as
        in Example~\ref{eg:motivation:NEPv}
        where the objective function involves
        $\tr((P^{\T}D)^2)$.
\end{enumerate}
\end{remark}

\begin{table}[t]
\renewcommand{\arraystretch}{1.4}
\caption{\small {\bf The NEPv Ansatz} on objective functions in Table~\ref{tbl:obj-funs}}\label{tbl:obj-funs:NEPv}
\centerline{\small
\begin{tabular}{|c|c|c|c|c|}
  \hline
    & $H(P)$ & conditions & $\bbP$ &  by  \\ \hline\hline
SEP & $A$ & none & $\STM{k}{n}$ & \cite{fan:1949}, \cite[p.248]{hojo:2013}  \\ \hline
OLDA & \eqref{eq:H(P):theta-TR} with $\theta=1$, $D=0$  & $B\succeq 0$, $\tr_{\min,k}(B)>0$ & $\STM{k}{n}$ & \eqref{eq:f-inc-theta-TR:refined}   \\ \hline
OCCA & \eqref{eq:H(P):theta-TR} with $\theta=1/2$, $A=0$ & $B\succeq 0$, $\tr_{\min,k}(B)>0$ & $\STM{k}{n}_{D+}$ & \eqref{eq:f-inc-theta-TR:refined}  \\ \hline
\multirow{3}{*}{$\Theta$TR} & \multirow{2}{*}{\eqref{eq:H(P):theta-TR} for case $0<\theta<1$}  & $B\succeq 0$, $\tr_{\min,k}(B)>0$
           & \multirow{2}{*}{$\STM{k}{n}_{D+}$} & \multirow{2}{*}{\eqref{eq:f-inc-theta-TR:refined}}  \\
           &                    & $\tr(P^{\T}AP+P^{\T}D)\ge 0$   &                 &            \\ \cline{2-5}
 & \eqref{eq:H(P):theta-TR} for case $\theta=1$  & $B\succeq 0$, $\tr_{\min,k}(B)>0$
           & $\STM{k}{n}_{D+}$ & \eqref{eq:f-inc-theta-TR:refined}  \\ \hline
MBSub & \eqref{eq:H(P):theta-TR} with $\theta=0$ & none & $\STM{k}{n}_{D+}$ & \eqref{eq:f-inc-theta-TR:refined}  \\ \hline
SumCT & \eqref{eq:H(P):always} & $A_i\succeq 0\,\forall i$ & \cite{wazl:2022a} & Thm.~\ref{thm:main-NEPv-cvx:Hi(Pi)}  \\ \hline
TrCP & $2\sum_{i=1}^N\phi_i(\bx)A_i$ & convex $\phi$ & $\STM{k}{n}$ & Expl.~\ref{eg:T1a-NEPv}  \\ \hline
UMDS & $4\sum_{i=1}^NA_iPP^{\T}A_i$ & $A_i\succeq 0\,\forall i$ & $\STM{k}{n}$ & Expl.~\ref{eg:T2a}  \\ \hline
DFT & $2A+2\sum_{i=1}^n\phi_i(\bx)\be_i\be_i^{\T}$ & convex $\phi$  & $\STM{k}{n}$ & Expl.~\ref{eg:T1a-NEPv} \\
  \hline
\multicolumn{5}{l}{\small * $\phi_i(\bx):=\partial\phi(\bx)/\partial x_i$  for $\bx=[x_i]$.
                   $\Theta$TR for $\theta=0$ becomes MBSub.} \\
\end{tabular}
}
\end{table}
As to the validity of  {\bf the NEPv Ansatz} on the objective functions in Table~\ref{tbl:obj-funs},
it holds for all, except SumTR,  under mild conditions on the constant matrices and function $\phi$.
Table~\ref{tbl:obj-funs:NEPv}
provides the details on $H(P)$ and conditions under which
{\bf the NEPv Ansatz} holds, where the last column refers to places for justifications.
We leave $\bbP$ unspecified for SumCT but refer it to \cite{wazl:2022a} because it is more complicated to fit the space in the table.
In fact, it is required that
each $P_i$ falls in $\{X\in\STM{k_i}{n}\,:\, X^{\T}D_i\succeq 0\}$. Later in subsection~\ref{ssec:PineP} we will argue that for SumCT it would be more efficient
to go for the NPDo approach in Part~I.
        For  SumTR with $H(P)$ given by
        \eqref{eq:H(p):sumTR},  $\bbP=\STM{k}{n}$ and $Q=I_k$ and {\bf the NEPv Ansatz} does not hold.
        This can be drawn from the counterexample, \cite[Example 4.1]{zhli:2014b}, for which SCF diverges,
        but later we will show, under {\bf the NEPv Ansatz}, SCF is guaranteed to converge!

Comparing Table~\ref{tbl:obj-funs:NPD} for  NPDo  with Table~\ref{tbl:obj-funs:NEPv} for NEPv here, we find that,
among those in Table~\ref{tbl:obj-funs},
{\bf the NEPv Ansatz}  provably holds for three more of them, which are OLDA, OCCA, and $\Theta$TR (all involving ratios), than {\bf the NPDo Ansatz} does. This observation that {\bf the NEPv Ansatz} is satisfied more often than
{\bf the NPDo Ansatz} among those in Table~\ref{tbl:obj-funs} is not an accident.
In fact {\bf the NPDo Ansatz} is stronger than {\bf the NEPv Ansatz} with
the generic $H(P)$ in \eqref{eq:H(P):always},
as shown by the next theorem.

\begin{theorem}\label{thm:NPDo>=NEPv}
Let function $f$ be defined on some neighborhood $\STMnbr$ of $\STM{k}{n}$ and let $H(P)$ be as in \eqref{eq:H(P):always}.
Then {\bf the NPDo Ansatz} implies {\bf the NEPv Ansatz}.
\end{theorem}

\begin{proof}
Suppose that {\bf the NPDo Ansatz} holds.
Given $\what P\in\STM{k}{n},\,P\in\bbP\subseteq\STM{k}{n}$ such that \eqref{eq:NEPv-assume} holds,
let $W\in\STM{k}{k}$ be an orthonormal polar factor of $\what P^{\T}\scrH(P)$, and set $\check P=\what PW_1$.
Then $\check P^{\T}\scrH(P)=W_1^{\T}[\what P^{\T}\scrH(P)]\succeq 0$, and thus by Lemma~\ref{lm:maxtrace}
\begin{equation}\label{eq:NPDo>=NEPv:pf-1}
\tr(\check P^{\T}\scrH(P))=\|\check P^{\T}\scrH(P)\|_{\tr}=\|\what P^{\T}\scrH(P)\|_{\tr}\ge\tr(\what P^{\T}\scrH(P)).
\end{equation}
Recalling \eqref{eq:H(P):always}, we have
\begin{align*}
\tr(\what P^{\T}H(P)\what P)&=2\tr(\what P^{\T}\scrH(P)P^{\T}\what P) \\
   &\le 2\|\what P^{\T}\scrH(P)P^{\T}\what P\|_{\tr} \qquad (\mbox{by Lemma~\ref{lm:maxtrace}}) \\
   &\le 2\|\what P^{\T}\scrH(P)\|_{\tr}\|P^{\T}\what P\|_2 \\
   &\le 2\|\what P^{\T}\scrH(P)\|_{\tr} \qquad(\mbox{since $\|P^{\T}\what P\|_2\le 1$}) \\
   &\le 2\tr(\check P^{\T}\scrH(P)). \qquad(\mbox{by \eqref{eq:NPDo>=NEPv:pf-1}})
\end{align*}
Now noticing that $\tr(P^{\T}H(P)P)=2\tr(P^{\T}\scrH(P))$, we get from inequality \eqref{eq:NEPv-assume} that
$$
\tr(\check P^{\T}\scrH(P))\ge \tr(P^{\T}\scrH(P))+\eta/2.
$$
By {\bf the NPDo Ansatz}, there exists $W_2\in\STM{k}{k}$ such that $\wtd P=\check PW_2=\what P(W_1W_2)\in\bbP$ and
$f(\wtd P)\ge f(P)+(\omega/2)\eta$, verifying {\bf the NEPv Ansatz}.
\end{proof}

The first immediate consequence of {\bf the NEPv Ansatz} is the following theorem that provides a characterization of
the maximizers of the associated optimization problem~\eqref{eq:main-opt}.

\begin{theorem}\label{thm:maxers-NEPv}
Let $P_*\in\STM{k}{n}$ be a maximizer of \eqref{eq:main-opt}.
Suppose that {\bf the NEPv Ansatz} holds and $P_*\in\bbP$.
Then NEPv
\eqref{eq:NEPv-form} holds for $P=P_*$ and the eigenvalues of $\Omega=\Omega_*:=P_*^{\T}H(P_*)P_*$
consist of the first $k$ largest eigenvalues of $H(P_*)$.
\end{theorem}

\begin{proof}
Consider
\begin{equation}\label{eq:maxers-NEPv:pf-1}
\max_{P\in\STM{k}{n}}\tr(P^{\T}H(P_*) P).
\end{equation}
We claim $P_*$ is a maximizer of \eqref{eq:maxers-NEPv:pf-1}; otherwise there would be some
$\what P\in\STM{k}{n}$ such that
$$
\tr(\what P^{\T}H(P_*)\what P)\ge\tr(P_*^{\T}H(P_*)P_*)+\eta
$$
for some $\eta>0$. Invoking  {\bf the NEPv Ansatz}, we can find $\wtd P=\what PQ\in\bbP$ such that
$f(\wtd P)\ge f(P_*)+\omega\eta> f(P_*)$, contradicting that $P_*$ is a maximizer.
Thus $P_*$ is a maximizer of \eqref{eq:maxers-NEPv:pf-1} whose KKT condition is
$H(P_*)P=P\Omega$ which $P_*$ will have to satisfy, i.e.,
$H(P_*)P_*=P_*\Omega_*$, where $\Omega_*=P_*^{\T}H(P_*)P_*$ whose eigenvalues consists of
the $k$ largest ones of $H(P_*)$.
\end{proof}

\subsection{SCF Iteration}\label{ssec:SCF4NEPv}
The second immediate consequence of {\bf the NEPv Ansatz} is the global convergence of an SCF iteration
for solving optimization problem~\eqref{eq:main-opt} as outlined in Algorithm~\ref{alg:SCF4NEPv}.
\begin{algorithm}[t]
\caption{NEPvSCF: NEPv \eqref{eq:NEPv-form} solved by SCF} \label{alg:SCF4NEPv}
\begin{algorithmic}[1]
\REQUIRE Symmetric matrix-valued function $H(P)$  satisfying {\bf the NEPv Ansatz}, $P^{(0)}\in\bbP$;
\ENSURE  an approximate maximizer of \eqref{eq:main-opt}. 
\FOR{$i=0,1,\ldots$ until convergence}
    \STATE compute $H_i=H(P^{(i)})\in\bbR^{n\times n}$;
    \STATE solve SEP $H_i\what P^{(i)}=\what P^{(i)}\Omega_i$ for $\what P^{(i)}\in\STM{k}{n}$,
           an orthonormal basis matrix of the eigenspace associated with the first $k$ largest eigenvalues of $H_i$;
    \STATE calculate $Q_i\in\STM{k}{k}$ and let $P^{(i+1)}=\what P^{(i)}Q_i\in\bbP$, according to {\bf the NEPv Ansatz};
\ENDFOR
\RETURN the last $P^{(i)}$.
\end{algorithmic}
\end{algorithm}
This algorithm is similar to \cite[Algorithm 2]{zhwb:2022}, \cite[Algorithm 4.1]{wazl:2023},
but the latter two have more details that are dictated by
the particularity of $f$ there. A few commens are in order:
\begin{enumerate}[(1)]
  \item A reasonable stopping criterion at Line 1 is
         \begin{equation}\label{eq:stop-NEPv}
         \varepsilon_{\NEPv}:=\frac {\|H(P)P-P[P^{\T}H(P)P]\|_{\F}}{\xi}\le\epsilon,
         \end{equation}
         where $\epsilon$ is a given tolerance, and $\xi$ is some normalization quantity that should be
         designed according to the underlying $H(P)$, but generically, $\xi=\|H(P)\|_{\F}$, or any reasonable estimate of it,
         should work well.

  \item At Line 4 it refers to  {\bf the NEPv Ansatz} for the calculation of $Q_i$.
Exactly how it is computed depends on the structure of $f$ at hand. We commented on the similar issue for \Cref{alg:SCF4NPDo} earlier. In the case of Example~\ref{eg:motivation:NEPv}, $Q_i$ is taken to be
an orthonormal polar factor of $(\what P^{(i)})^{\T}D$ to make $P^{(i+1)}\in\STM{k}{n}_{D+}$.

  \item In Algorithm~\ref{alg:SCF4NEPv}, we explicitly state that it is for $H(P)$ that satisfies {\bf the NEPv Ansatz}, without which
we cannot guarantee convergence as stated in the theorems in the rest of this section, but numerically the body of
the algorithm can still be implemented. There is a level-shifting technique that can help achieve local convergence
\cite{ball:2022,luli:2024}.

\end{enumerate}
Convergence of \Cref{alg:SCF4NEPv} will be detailed later in \Cref{thm:cvg4SCF4NEPv:strong} in subsection~\ref{ssec:cvg4SCF4NEPv}.
%
It is shown that as the SCF iteration progresses, the value of the objective function
monotonically increases, any accumulation point of the generated approximation sequence
satisfies the necessary conditions in Theorem~\ref{thm:maxers-NEPv} for a global maximizer, and
under certain conditions, the accumulation point can be proved to be the limit point of the entire approximation sequence.
In short, the NEPv approach is guaranteed to work.

\begin{remark}\label{rk:SCF4MEPv-var}
Currently as stated, $\what P^{(i)}$ is computed to an orthonormal basis matrix of the eigenspace associated with the first $k$ largest eigenvalues of $H(P^{(i)})$, and that can be costly if it is done by a full eigendecomposition.
      Conceivably, $\what P^{(i)}$ can also be computed iteratively by other means (see, e.g., \cite{goye:2002,knya:2001,li:2015,quye:2010} and references therein), especially when $k\not\ll n$ and $n$ is not small, but how accurate does it have to be computed to?
      Our convergence analysis in \Cref{thm:cvg4SCF4NEPv:strong} sheds light on the question  where it lays out
the level of accuracy in $\what P^{(i)}$ to ensure desirable convergence properties.
\end{remark}

%

\subsection{Acceleration by LOCG and Convergence}\label{ssec:LOCG:NEPv}
At Line 3 of Algorithm~\ref{alg:SCF4NEPv}, an $n\times n$ SEP is involved and that can be expensive for large/huge $n$.
As in subsection~\ref{ssec:LOCG:NPD}, the same idea for acceleration can be applied here to
speed things up. It has in fact been partially demonstrated in \cite[section 5]{wazl:2022a} on
MBSub.

\subsubsection*{A variant of LOCG for Acceleration}
We adopt the same setup at the beginning of subsection~\ref{ssec:LOCG:NPD} up to \eqref{eq:KKT-reduced}.
Differently, here we will need a  symmetric matrix-valued function $\wtd H(Z)$ for the dimensionally reduced
maximization problem \eqref{eq:LOCGsub} so that {\bf the NEPv Ansatz} can be passed on from $f$ with $H$ to
$\wtd f$ with $\wtd H$.

As we commented in Remark~\ref{rk:NEPv-assume}(i),  {\bf the NEPv Ansatz} does not impose any explicit relation between $\scrH(P)$ and $H(P)P$, e.g., through condition~\eqref{eq:cond:KKT=NEPv},
but in order to figure out a symmetric matrix-valued function $\wtd H(Z)$ from $H(P)$,  let us assume \eqref{eq:cond:KKT=NEPv}
for the moment. Once we figure out what $\wtd H(Z)$ should be for the case, we will then show the newly derived $\wtd H(Z)$ works for $\wtd f$ even without
\eqref{eq:cond:KKT=NEPv} as a prerequisite, i.e., {\bf the NEPv Ansatz} gets passed on from  $f$ with $H$ to
$\wtd f$ with $\wtd H$.
It follows from \eqref{eq:cond:KKT=NEPv} that
$$ 
W^{\T}\left(H(P)P-\frac{\partial f(P)}{\partial P}\right)=W^{\T}P\,\scrM(P)\quad\mbox{for $P=WZ$, $Z\in\STM{k}{\hat n}$}.
$$ 
or, equivalently,
\begin{equation}\label{eq:cond:KKT=NEPv:reduced}
\big(W^{\T}H(WZ)W\big)Z-\frac{\partial \wtd f(Z)}{\partial Z}=Z\,\wtd\scrM(Z)\quad\mbox{for $Z\in\STM{k}{\hat n}$},
\end{equation}
where $\wtd\scrM(Z):=\scrM(WZ)$. Immediately, \eqref{eq:cond:KKT=NEPv:reduced} sheds light on
what $\wtd H(Z)$ should be and, accordingly, the associated NEPv
for the reduced maximization problem \eqref{eq:LOCGsub}, namely,
\begin{equation}\label{eq:NEPv-form:reduced}
 \wtd H(Z)Z:=\big(W^{\T}H(WZ)W\big)Z =Z\wtd\Omega, \,\, Z\in\STM{k}{\hat n}.
\end{equation}
By Theorem~\ref{thm:H(P)-eligibility}, we conclude that, in the case of \eqref{eq:cond:KKT=NEPv},
$Z\in\STM{k}{m}$ is a solution to the KKT condition \eqref{eq:KKT-reduced} for
the reduced problem \eqref{eq:LOCGsub-1} if and only if
it is a solution to NEPv \eqref{eq:NEPv-form:reduced}
and $\wtd\scrM(Z)$ is  symmetric.

Now that we have found the form of $\wtd H$ for $\wtd f$, given $H$ for $f$, we can safely relinquish \eqref{eq:cond:KKT=NEPv}
and \eqref{eq:cond:KKT=NEPv:reduced} going forward. In the next lemma, we show that
$\wtd f$ with $\wtd H$ of the reduced problem inherits {\bf the NEPv Ansatz} for $f$ with $H$ of the original problem.

\begin{lemma}\label{lm:NEPvAssump.-reduced}
Suppose that {\bf the NEPv Ansatz} holds for $f(P)$ with $H(P)$, and let $\bbZ:=W^{\T}\bbP\subseteq\STM{k}{\hat n}$. If
$W\bbZ\subseteq\bbP$, then {\bf the NEPv Ansatz} holds for $\wtd f(Z)$ defined in \eqref{eq:LOCGsub-1}
with $\wtd H(Z)=W^{\T}H(WZ)W$.
\end{lemma}

\begin{proof}
Let $\what Z\in\STM{k}{m}$ and $Z\in\bbZ:=W^{\T}\bbP\subseteq\STM{k}{\hat n}$ such that
\begin{equation}\label{eq:NEPv-assume:reduced}
\tr(\what Z^{\T}\wtd H(Z)\what Z)\ge\tr(Z^{\T}\wtd H(Z)Z)+\eta.
\end{equation}
Set $P=WZ\in\bbP$ because of $W\bbZ\subseteq\bbP$ and $\what P=W\what Z\in\STM{k}{n}$. We get
\eqref{eq:NEPv-assume} from \eqref{eq:NEPv-assume:reduced}. By {\bf the NEPv Ansatz} for $f$ with $H$,
there exists $Q\in\STM{k}{k}$ such that
$\wtd P=\what PQ=W(\what ZQ)=:W\wtd Z\in\bbP$ and
$f(\wtd P)\ge f(P)+\omega\eta$, i.e.,
$$
\wtd f(\wtd Z)=\wtd f(\what ZQ)=f(\wtd P)\ge f(P)+\omega\eta
  =f(WZ)+\omega\eta
  =\wtd f(Z)+\omega\eta.
$$
Note also $\wtd Z=W^{\T}\wtd P\in W^{\T}\bbP=\bbZ$. Hence {\bf the NEPv Ansatz} holds for $\wtd f$ with $\wtd H$.
\end{proof}

As a consequence of this lemma, and the results in subsections~\ref{ssec:NEPvAssum} and \ref{ssec:SCF4NEPv}, Algorithm~\ref{alg:SCF4NEPv}
is applicable to compute $Z_{\opt}$ of \eqref{eq:LOCGsub-1} via NEPv \eqref{eq:NEPv-form:reduced}.
We outline the resulting
method in Algorithm~\ref{alg:SCF4NEPv+LOCG}, which is an inner-outer iterative scheme for \eqref{eq:main-opt},
where at Line~4
any other method, if known, can be inserted to replace Algorithm~\ref{alg:SCF4NEPv} to
solve \eqref{eq:LOCGsub-1}. The same comments we made in Remark~\ref{rk:SCF4NPDo+LOCG} and after are applicable to
Algorithm~\ref{alg:SCF4NEPv+LOCG}, too.

\begin{algorithm}[t]
\caption{NEPvLOCG: NEPv \eqref{eq:NEPv-form} solved by LOCG} \label{alg:SCF4NEPv+LOCG}
\begin{algorithmic}[1]\REQUIRE Symmetric matrix-valued function $H(P)$ satisfying {\bf the NEPv Ansatz}, $P^{(0)}\in\bbP$;
\ENSURE  an approximate maximizer of \eqref{eq:main-opt}.
\STATE $P^{(-1)}=[\,]$; \% null matrix
\FOR{$i=0,1,\ldots$ until convergence}
    \STATE compute $W_i\in\STM{\hat n}{n}$ such that $\cR(W_i)=\cR([P^{(i)},\scrR(P^{(i)}),P^{(i-1)}])$
           and $P^{(i)}$ occupies
           the first $k$ columns of $W_i$;
    \STATE solve \eqref{eq:LOCGsub-1} via NEPv \eqref{eq:NEPv-form:reduced} for $Z_{\opt}$ by  Algorithm~\ref{alg:SCF4NEPv}
           with initially $Z^{(0)}$ being the first $k$ columns of $I_{\hat n}$;
    \STATE $P^{(i+1)}=W_iZ_{\opt}$;
\ENDFOR
\RETURN the last $P^{(i)}$.
\end{algorithmic}
\end{algorithm}

\subsection{Convergence Analysis}\label{ssec:cvg4SCF4NEPv}
\Cref{alg:SCF4NEPv} is stated under {\bf the NEPv Ansatz}, with which it can be shown that
\begin{equation}\label{eq:cvg4SCF4NEPv:obj-inc}
f(P^{(i+1)})\ge f(P^{(i)})
    +\omega\Big[\tr_{\max,k}\big(H(P^{(i)})\big)-\tr\big([P^{(i)}]^{\T}H(P^{(i)})P^{(i)}\big)\Big],
\end{equation}
for every SCF iterative step, where $\omega$ is the constant from {\bf the NEPv Ansatz} and it is independent of
$i$. This inequality alone, if satisfied, turns out to have critical consequences in convergence of
the sequence $\{P^{(i)}\in\STM{k}{n}\}_{i=0}^{\infty}$, regardless whether {\bf the NEPv Ansatz} holds or not, but
the existence of some symmetric matrix-valued function $H(P)\in\bbR^{n\times n}$ is still required, just that
{\bf the NEPv Ansatz} holds or not is inconsequential.

In what follows, we first establish a general convergence theory
under an assumption like inequality \eqref{eq:cvg4SCF4NEPv:obj-inc} and then apply it to derive some
convergence results for \Cref{alg:SCF4NEPv,alg:SCF4NEPv+LOCG} and a variant of \Cref{alg:SCF4NEPv} in which
each $\what P^{(i)}$ is approximately calculated, not necessarily an exact orthonormal basis matrix of the eigenspace
associated with the first $k$ largest eigenvalues of $H(P^{(i)})$.

\begin{theorem}\label{thm:cvg4SCF4NEPv:gl}
Let $\{P^{(i)}\in\STM{k}{n}\}_{i=0}^{\infty}$
be a sequence on $\STM{k}{n}$ such that
\begin{equation}\label{eq:cvg4SCF4NEPv:gl:obj-inc}
f(P^{(i+1)})\ge f(P^{(i)})
    +\kappa \underbrace{\Big[\tr_{\max,k}\big(H(P^{(i)})\big)-\tr\big([P^{(i)}]^{\T}H(P^{(i)})P^{(i)}\big)\Big]}_{=:\eta_{*i}},
\end{equation}
where $\kappa>0$ is a constant, independent of $i$, and let $P_*$
be an accumulation point of the sequence $\{P^{(i)}\}_{i=0}^{\infty}$.
The following statements hold.
\begin{enumerate}[{\rm (i)}]
  \item $\cR(P_*)$ is an accumulation point of the sequence $\{\cR(P^{(i)})\}_{i=0}^{\infty}$;
  \item The sequence $\{f(P^{(i)})\}_{i=0}^{\infty}$ is monotonically increasing and convergent;
  \item The equation \eqref{eq:NEPv-form} holds for $P=P_*$ and the eigenvalues of $\Omega_*=P_*^{\T}H(P_*)P_*$
        consist of the first $k$ largest eigenvalues of $H(P_*)$.
        Furthermore in the case when $H(P)$
        satisfies \eqref{eq:cond:KKT=NEPv}, if $\scrM(P_*)$ is symmetric, then $P_*$ is a KKT point.
  \item We have two convergent series
        \begin{subequations}\label{eq:cvg4SCF4NEPv:series:gl}
        \begin{align}
         \sum_{i=1}^{\infty}\delta_i\,\big\|\sin\Theta\big(\cP_{\max,k}^{(i)},\cR(P^{(i)})\big)\big\|_{\F}^2
                      &<\infty,,      \label{eq:cvg4SCF4NEPv:series:gl-1} \\
         \sum_{i=1}^{\infty}\delta_i\,
                       \frac {\big\|H(P^{(i)})P^{(i)}-P^{(i)}\Lambda_i\big\|_{\F}^2}
                             {\big\|H(P^{(i)})\big\|_{\F}^2}
                      &<\infty,                    \label{eq:cvg4SCF4NEPv:series:gl-2}
        \end{align}
        \end{subequations}
        where
        $\cP_{\max,k}^{(i)}$ is the invariant subspace of $H(P^{(i)})$ associated with its $k$ largest eigenvalues, and
        $$
        \delta_i=\lambda_k(H(P^{(i)}))-\lambda_{k+1}(H(P^{(i)})), \quad
        \Lambda_i=[P^{(i)}]^{\T}H(P^{(i)})P^{(i)}.
        $$
\end{enumerate}
\end{theorem}

\begin{proof}
See appendix~\ref{sec:Proof4NEPvCVG}.
\end{proof}

\Cref{thm:cvg4SCF4NEPv:gl}(iv) has interesting and useful consequences.

\begin{corollary}\label{cor:cvg4SCF4NEPv:gl}
To the conditions of \Cref{thm:cvg4SCF4NEPv:gl} add this: $\delta_i=\lambda_k(H(P^{(i)}))-\lambda_{k+1}(H(P^{(i)}))$
is eventually bounded below away from $0$
uniformly\footnote {By which we mean that there exist a constant $c>0$ and an integer $K$ such that
  $\sigma_{\min}(\scrH(P^{(i)}))\ge c$ for all $i\ge K$.}.
Then
$$
\lim_{i\to\infty}\frac {\big\|H(P^{(i)})P^{(i)}-P^{(i)}\Lambda_i\big\|_{\F}}
                       {\big\|H(P^{(i)})\big\|_{\F}} =0.
$$
\end{corollary}

\Cref{cor:cvg4SCF4NEPv:gl} implies that increasingly $H(P^{(i)})P^{(i)}\approx P^{(i)}\Lambda_i=P^{(i)}\big([P^{(i)}]^{\T}H(P^{(i)})P^{(i)}\big)$,
which means that $P^{(i)}$ becomes a more and more accurate approximate solution
to NEPv \eqref{eq:NEPv-form}, even in the absence of knowing whether the entire sequence $\{P^{(i)}\}_{i=0}^{\infty}$ converges or not. The latter does require additional additional specifications as to how the sequence
$\{P^{(i)}\}_{i=0}^{\infty}$ is generated.
(see \Cref{thm:cvg4SCF4NEPv:strong}(c) below).
On the other hand, what if $\delta_i=\lambda_k(H(P^{(i)}))-\lambda_{k+1}(H(P^{(i)}))$ is not eventually bounded below away from $0$, e.g.,
$\delta_i=\lambda_k(H(P^{(i)}))-\lambda_{k+1}(H(P^{(i)}))$ behaves like $1/i$?
We may still be able to draw useful conclusions, dependent on circumstances.
As an example, we state yet another corollary of \Cref{thm:cvg4SCF4NEPv:gl}(iv) with the help of \Cref{lm:series}.

\begin{corollary}\label{cor:cvg4SCF4NEPv:gl'}
In \Cref{thm:cvg4SCF4NEPv:gl}, if $\sum_{i=1}^{\infty}\delta_i$ diverges then there exists a convergent subsequence of
$$
\left\{\frac {\big\|H(P^{(i)})P^{(i)}-P^{(i)}\Lambda_i\big\|_{\F}}
                       {\big\|H(P^{(i)})\big\|_{\F}}\right\}_{i=1}^{\infty}
$$
that goes to $0$.
\end{corollary}


We emphasize again that \Cref{thm:cvg4SCF4NEPv:gl} does not assume  {\bf the NEPv Ansatz} and does not specify how the sequence $\{P^{(i)}\}_{i=0}^{\infty}$ is generated so long it has \eqref{eq:cvg4SCF4NEPv:gl:obj-inc} satisfied.
In the case of \Cref{alg:SCF4NEPv} where {\bf the NEPv Ansatz} is the prerequisite, we have
\eqref{eq:cvg4SCF4NEPv:obj-inc}, or, equivalently, \eqref{eq:cvg4SCF4NEPv:gl:obj-inc} holds with $\kappa=\omega$.

Next we state our convergence results for \Cref{alg:SCF4NEPv,alg:SCF4NEPv+LOCG}.
As stated in \Cref{alg:SCF4NEPv}, each
$\what P^{(i)}$ is computed as an orthonormal basis matrix of the eigenspace
associated with the first $k$ largest eigenvalues of $H(P^{(i)})$.
The cost of a full eigendecomposition of $H(P^{(i)})$ is $4n^3/3$ flops \cite[p.463]{govl:2013} which is too expensive
for large or even modest $n$,  since we have to do it at every SCF iterative step.
Fortunately, we do not need the full eigendecomposition but the top $k$ eigenvalues and their associate eigenvectors.
Since $k$ is usually small such as a few tens or smaller, a better option is some iterative methods geared for
extreme eigenpairs \cite{goye:2002,knya:2001,li:2015,quye:2010}.
Furthermore, as far as always moving the objective value up is concerned,
it suffices to calculate $\what P^{(i)}$ just well enough such that
$\tr([\what P^{(i)}]^{\T}H_i\what P^{(i)})\ge \tr([P^{(i)}]^{\T}H_i P^{(i)})$ (usually strictly). This observation can become very useful when
the $k$th and $(k+1)$st eigenvalues of $H_i$ are very close, in which case convergence to the $k$th eigenvector by
an iterative method is often very slow.
In attempting to be
more general, we will consider a variation of \Cref{alg:SCF4NEPv} where each $\what P^{(i)}$ is approximately calculated
accurately enough so that
%
\begin{align}
\eta_i:=\,&\tr\big([\what P^{(i)}]^{\T}H(P^{(i)})\what P^{(i)}\big)-\tr\big([P^{(i)}]^{\T}H(P^{(i)}) P^{(i)}\big)
             \label{eq:NEPv:eta(i):dfn} \\
      \ge\,&c\Big[\tr_{\max,k}\big(H(P^{(i)})\big)-\tr\big([P^{(i)}]^{\T}H(P^{(i)})P^{(i)}\big)\Big]
             \label{eq:NEPv:eta(i):weak}
\end{align}
for some constant $c>0$, independent of $i$.
Current $\what P^{(i)}$ as stated in Algorithm~\ref{alg:SCF4NEPv} makes $c=1$,  and thus any convergence result under \eqref{eq:NEPv:eta(i):weak} remains valid for \Cref{alg:SCF4NEPv}.

\begin{theorem}\label{thm:cvg4SCF4NEPv:strong}
Suppose that {\bf the NEPv Ansatz} holds, and let the sequence $\{P^{(i)}\}_{i=0}^{\infty}$ be generated
by a variation of Algorithm~\ref{alg:SCF4NEPv}, in which $\what P^{(i)}$ at Line 3 is approximately computed such that \eqref{eq:NEPv:eta(i):weak} is satisfied.
The following statements hold.
\begin{enumerate}[{\rm (a)}]
  \item All conclusions in \Cref{thm:cvg4SCF4NPDo:gl} are valid;
  \item Suppose that $\cR(P_*)$ is an isolated accumulation point of $\{\cR(P^{(i)})\}_{i=0}^{\infty}$. If
        \begin{equation}\label{eq:eig-gap:assume}
        \lambda_k(H(P_*Q))-\lambda_{k+1}(H(P_*Q))>0\quad\mbox{for any $Q\in\STM{k}{k}$},
        \end{equation}
        then
        the entire sequence $\{\cR(P^{(i)})\}_{i=0}^{\infty}$ converges to $\cR(P_*)$.
  \item Suppose that $P_*$ is an isolated accumulation point of $\{P^{(i)}\}_{i=0}^{\infty}$. If
        \begin{equation}\label{eq:eig-gap:assume'}
        \lambda_k(H(P_*))-\lambda_{k+1}(H(P_*))>0
        \end{equation}
        and if
        $f(P_*)>f(P)$ for any $P\ne P_*$ and $\cR(P)=\cR(P_*)$,
        i.e.,
        $f(P)$ has unique maximizer in the orbit $\{P_*Q\,:\,Q\in\STM{k}{k}\}$,
        then
        the entire sequence $\{P^{(i)}\}_{i=0}^{\infty}$ converges to $P_*$.
  \item Suppose that $f(\cdot)$ and $H(\cdot)$ are right-unitarily invariant. Define, for the purpose of alignment,
        $V_i^{(i)}\in\STM{k}{k}$ to be the orthonormal polar factor of $\big[P^{(i)}\big]^{\T}P_*$ for $i\ge 0$.
        If $P_*$ is an isolated accumulation point of $\{P^{(i)}V_i\}_{i=0}^{\infty}$
        and if \eqref{eq:eig-gap:assume'} holds, then
        the entire sequence $\{P^{(i)}V_i\}_{i=0}^{\infty}$ converges to $P_*$.
\end{enumerate}
\end{theorem}

\begin{proof}
See appendix~\ref{sec:Proof4NEPvCVG}.
\end{proof}

In Theorem~\ref{thm:cvg4SCF4NEPv:strong}(d), the right-unitarily invariance assumption on $H(\cdot)$
may be considered implied, thanks to Remark~\ref{rk:NEPv-assume}(iii).
Also that $P_*$ is an accumulation point of $\{P^{(i)}V_i\}_{i=0}^{\infty}$ is implied by the fact that $P_*$ be an accumulation
point of $\{P^{(i)}\}_{i=0}^{\infty}$. This is because if $\{P^{(i)}\}_{i\in\bbI}$ is a subsequence that converges
to $P_*$ then $\big[P^{(i)}\big]^{\T}P_*\to I_k$ as $\bbI\ni i\to\infty$, implying their orthonormal polar factor
$V_i\to I_k$ as $\bbI\ni i\to\infty$ \cite{li:1995} and thus $P^{(i)}V_i\to I_k$ as $\bbI\ni i\to\infty$.

In \Cref{alg:SCF4NEPv}, $P^{(i+1)}=\what P^{(i)}Q_i$ and thus $\cR(P^{(i+1)})=\cR(\what P^{(i)})$.
If also $\what P^{(i)}$ is computed to an orthonormal basis matrix of the eigenspace associated with the first $k$ largest eigenvalues of $H(P^{(i)})$  then $\cR(\what P^{(i)})=\cP_{\max,k}^{(i)}$.
If that is the case, $\cP_{\max,k}^{(i)}$ in \eqref{eq:cvg4SCF4NEPv:series:gl-1} can be replaced with
$\cR(P^{(i+1)})$.
In the variation of Algorithm~\ref{alg:SCF4NEPv} as described in
Theorem~\ref{thm:cvg4SCF4NEPv:strong}, $\what P^{(i)}$ is computed to make
$\tr\big([\what P^{(i)}]^{\T}H_i\what P^{(i)}\big)$
large enough so that \eqref{eq:NEPv:eta(i):weak} is satisfied.
While it is nice to have Theorem~\ref{thm:cvg4SCF4NEPv:strong}, there is a limitation, i.e., the difficulty in verifying
\eqref{eq:NEPv:eta(i):weak} numerically as $\tr_{\max,k}\big(H(P^{(i)})\big)$ is unknown without computing the top  eigenspace of dimension $k$
of $H(P^{(i)})$, or, equivalently, an orthonormal basis matrix of the eigenspace. Nonetheless,
it still offers guidance and understanding when $\what P^{(i)}$ has to be computed iteratively.

As a corollary of Theorem~\ref{thm:cvg4SCF4NEPv:strong}(a), we establish a sufficient condition for NEPv \eqref{eq:NEPv-form} to have a solution because
there always exists an accumulation point $P_*$ of the sequence $\{P^{(i)}\}_{i=0}^{\infty}$
in $\STM{k}{n}$.

\begin{corollary}\label{cor:cvg4SCF4NEPv}
Under {\bf the NEPv Ansatz},
NEPv \eqref{eq:NEPv-form} is solvable, i.e., there exists $P\in\STM{k}{n}$ such that
  \eqref{eq:NEPv-form} holds and
the eigenvalues of $\Omega=P^{\T}H(P)P$ are the $k$ largest ones of $H(P)$.
\end{corollary}



Next we can perform a convergence analysis of Algorithm~\ref{alg:SCF4NEPv+LOCG} similarly to what we did in subsection~\ref{ssec:LOCG:NPD}.
We state our convergence results in the next theorem.

\begin{theorem}\label{thm:cvg4SCFvLOCG:NEPv}
Suppose that {\bf the NEPv Ansatz} holds, and let sequence $\{P^{(i)}\}_{i=0}^{\infty}$ be generated by Algorithm~\ref{alg:SCF4NEPv+LOCG}
in which, it is assumed that $Z_{\opt}$ is computed by at least one SCF iteration, within
the variation of Algorithm~\ref{alg:SCF4NEPv} as described in
Theorem~\ref{thm:cvg4SCF4NEPv:strong}, when it is called to compute $Z_{\opt}$
at Line 4 of Algorithm~\ref{alg:SCF4NEPv+LOCG}. Suppose that  there is a constant $c'>0$, independent of $i$,
such that eventually
\begin{multline}\label{eq:cvg4SCFvLOCG:NEPv:assume}
\tr_{\max,k}\big(W_i^{\T}H(P^{(i)})W_i\big)-\tr\big([P^{(i)}]^{\T}H(P^{(i)})P^{(i)}\big) \\
   \ge c'\Big[\tr_{\max,k}\big(H(P^{(i)})\big)-\tr\big([P^{(i)}]^{\T}H(P^{(i)})P^{(i)}\big)\Big].
\end{multline}
Then all conclusions in \Cref{thm:cvg4SCF4NEPv:gl} are valid.
\end{theorem}

\begin{proof}
See appendix~\ref{sec:Proof4NEPvCVG}.
\end{proof}

\section{Atomic Functions for NEPv}\label{sec:AF-NEPv}
Armed with the general theoretical framework for the NEPv approach in
section~\ref{sec:NEPv-theory},
in this section, we introduce the notion of atomic functions for NEPv, which serves as
a singleton unit of function on $\STMnbr$
for which the NEPv approach is guaranteed to work for solving \eqref{eq:main-opt}, and more importantly,
the NEPv approach works on any convex composition of atomic functions,
provided that some of the partial derivatives
of the composing convex function are  nonnegative.


In what follows, we first formulate two conditions that define
atomic function and prove why the NEPv approach will work on the atomic functions, and then we give concrete examples of
atomic functions that encompass nearly all practical ones that are in use today,  and we leave
investigating how the NEPv approach will work on convex compositions of these atomic functions to section~\ref{sec:CVX-comp:NEPv}.

Combining the results in this section and the next section will yield
a large collection of objective functions, including those in Table~\ref{tbl:obj-funs:NEPv}, for which
{\bf the NEPv Ansatz} holds.

\subsection{Conditions on Atomic Functions}\label{ssec:AF-NEPv}
Suppose that, for  function $f$ defined on some neighborhood $\STMnbr$ of the Stiefel manifold $\STM{k}{n}$,
we have already constructed an associated symmetric matrix-valued function $H(P)\in\bbR^{n\times n}$.
%
We are interested in those
that satisfy
\begin{subequations}\label{eq:cond4AF-NEPv}
\begin{align}
\tr(P^{\T}H(P)P)&=\munderbar\gamma\, f(P)\quad\mbox{for $P\in\bbP\subseteq\STM{k}{n}$},
              \label{eq:cond4AF-NEPv-a} \\
\intertext{and given $P\in\bbP$ and $\what P\in\STM{k}{n}$, there exists $Q\in\STM{k}{k}$ such
that $\wtd P=\what PQ\in\bbP$ and}
\tr(\what P^{\T}H(P)\what P)&\le\munderbar\alpha \,f(\wtd P)+\munderbar\beta \,f(P),
      \label{eq:cond4AF-NEPv-b}
\end{align}
\end{subequations}
where $\munderbar\alpha>0,\,\munderbar\beta\ge 0$,  and $\munderbar\gamma=\munderbar\alpha+\munderbar\beta$ are constants.
Here  a subset $\bbP\subseteq\STM{k}{n}$ is also involved.

\begin{definition}\label{defn:AF:NEPv}
A function $f$ defined on some neighborhood $\STMnbr$ of $\STM{k}{n}$ is an {\em atomic function\/}
for NEPv
if there are
a symmetric matrix-valued function $H(P)\in\bbR^{n\times n}$ for $P\in\STM{k}{n}$ and
constants $\munderbar\alpha>0,\,\munderbar\beta\ge 0$, and $\munderbar\gamma=\munderbar\alpha+\munderbar\beta$  such that both conditions in
\eqref{eq:cond4AF-NEPv} hold.
\end{definition}

An atomic function according to Definition~\ref{defn:AF:NEPv} is of the second type
in this paper and is for  the NEPv approach, in contrast to the first type that is defined in subsection~\ref{ssec:AF-NPDo} of
Part~I for the NPDo approach. As in subsection~\ref{ssec:AF-NPDo}, here it also can be verified that
for two atomic functions $f_1$ and $f_2$ with $H_1$ and $H_2$, respectively,
that share the same $\bbP$, the same constants $\munderbar\alpha,\,\munderbar\beta,\,\munderbar\gamma$, and the same $Q$ for \eqref{eq:cond4AF-NEPv-b},
any linear combination $f:=c_1f_1+c_2f_2$ with $c_1H_1+c_2H_2$ for $c_1,\,c_2>0$
satisfies \eqref{eq:cond4AF-NEPv}, and hence is an atomic function for NEPv as well.

\begin{remark}\label{rk:regularity4tr-diff-eq:H(P)}
An alternative to \eqref{eq:cond4AF-NEPv-b} is
\begin{equation}\tag{\ref{eq:cond4AF-NEPv-b}$'$}
\tr(\wtd P^{\T}H(P)\wtd P)\le\munderbar\alpha \,f(\wtd P)+\munderbar\beta \,f(P)
           \quad\mbox{for $P,\,\wtd P\in\bbP\subseteq\STM{k}{n}$},
\end{equation}
without referring to an intermediate $\what P$.
We claim that (\ref{eq:cond4AF-NEPv-b}$'$) is stronger than \eqref{eq:cond4AF-NEPv-b},
assuming for any $\what P\in\STM{k}{n}$, there exists $Q\in\STM{k}{k}$ such
that $\wtd P=\what PQ\in\bbP$. Here is why. Suppose (\ref{eq:cond4AF-NEPv-b}$'$) holds.
Given any $\what P\in\STM{k}{n}$, let $\wtd P=\what PQ\in\bbP$ for some $Q\in\STM{k}{k}$. Then
by (\ref{eq:cond4AF-NEPv-b}$'$) we have for any $P\in\bbP$
$$
\tr(\what P^{\T}H(P)\what P)=\tr(\wtd P^{\T}H(P)\wtd P)\le\munderbar\alpha \,f(\wtd P)+\munderbar\beta \,f(P),
$$
yielding \eqref{eq:cond4AF-NEPv-b}.
In view of this observation, in our later developments, we may verify
(\ref{eq:cond4AF-NEPv-b}$'$) directly if it can be verified.
\end{remark}


In relating $\scrH(P)$ to $H(P)$ via, e.g., \eqref{eq:cond:KKT=NEPv} when \eqref{eq:natNEPv} does not hold,
$H(P)$ in general is not unique  \cite{luli:2024}. As a result,
satisfying \eqref{eq:cond4AF-NEPv} may depend on both $f$ and
the choice of $H(P)$. In other words, it is possible that the conditions in \eqref{eq:cond4AF-NEPv} are satisfied
for one choice of $H(P)$ but may not for another.

\begin{remark}\label{rk:tr-diff-eq:H(P)}
Previously, \eqref{eq:cond4AF-NPDo-a} appears explicitly as a partial differential equation (PDE), but \eqref{eq:cond4AF-NEPv-a} here
does not.
Nonetheless, it is likely a PDE in disguise, especially when $H(P)P$ is related to
$\scrH(P):={\partial f(P)}/{\partial P}$
through condition \eqref{eq:cond:KKT=NEPv}.
\end{remark}

The next theorem basically says that the conditions in \eqref{eq:cond4AF-NPDo} that define the atomic function for  NPDo  are stronger
than the ones in \eqref{eq:cond4AF-NEPv} for  NEPv  with the generic $H(P)$ given by \eqref{eq:H(P):always}.

\begin{theorem}\label{thm:scrH2H:tr-diff-eq}
Let $H(P)$ be as in \eqref{eq:H(P):always}.
\begin{enumerate}[{\rm (a)}]
  \item Equation \eqref{eq:cond4AF-NPDo-a} implies \eqref{eq:cond4AF-NEPv-a} with $\munderbar\gamma=2\gamma$.
  \item Inequality \eqref{eq:cond4AF-NPDo-b}
        implies \eqref{eq:cond4AF-NEPv-b} with
        $\munderbar\alpha=2\alpha$, $\munderbar\beta=2\beta$, and $\munderbar\gamma=\munderbar\alpha+\munderbar\beta=2\gamma$.
\end{enumerate}
\end{theorem}

\begin{proof}
Assuming \eqref{eq:cond4AF-NPDo-a}, for $H(P)$ as given in  \eqref{eq:H(P):always}, we have
for $P\in\bbP\subseteq\STM{k}{n}$
$$
\tr(P^{\T}H(P)P)=2\tr\left(P^{\T}\scrH(P)\right)=2\gamma\, f(P),
$$
as was to be shown. Assume \eqref{eq:cond4AF-NPDo-b}. Let
$\what P\in\STM{k}{n},\,P\in\bbP$ and
let $W_1\in\STM{k}{k}$ be an orthonormal polar factor of
$\what P^{\T}\scrH(P)$ and $\check P=\what PW_1$.
Then we again have \eqref{eq:NPDo>=NEPv:pf-1} and
\begin{align}
\tr(\what P^{\T}H(P)\what P)&=2\tr\left(\what P^{\T}\scrH(P)P^{\T}\what P\right) \nonumber \\
   &\le 2\left\|\what P^{\T}\scrH(P)P^{\T}\what P\right\|_{\tr} \qquad (\mbox{by Lemma~\ref{lm:maxtrace}}) \nonumber\\
   &\le 2\left\|\what P^{\T}\scrH(P)\right\|_{\tr} \|P^{\T}\what P\|_2  \nonumber\\
   &\le 2\left\|\what P^{\T}\scrH(P)\right\|_{\tr} \qquad (\mbox{since $\|P^{\T}\what P\|_2\le 1$})  \nonumber\\
   &= 2\tr\left(\check P^{\T}\scrH(P)\right)
        \qquad (\mbox{by \eqref{eq:NPDo>=NEPv:pf-1}}). \label{eq:scrH2H:tr-diff-eq:pf-1}
\end{align}
Now use \eqref{eq:cond4AF-NPDo-b}
to conclude that there is  $W_2\in\STM{k}{k}$ such that $\wtd P=\check PW_2=\what P(W_1W_2)\in\bbP$
and
\begin{equation}\label{eq:scrH2H:tr-diff-eq:pf-2}
2\tr\left(\check P^{\T}\scrH(P)\right)\le 2\alpha \,f(\wtd P)+2\beta f(P).
\end{equation}
Combine \eqref{eq:scrH2H:tr-diff-eq:pf-1} and \eqref{eq:scrH2H:tr-diff-eq:pf-2} to get
\eqref{eq:cond4AF-NEPv-b} with
        $\munderbar\alpha=2\alpha$, $\munderbar\beta=2\beta$, and $\munderbar\gamma=\munderbar\alpha+\munderbar\beta=2\gamma$.
\end{proof}


\begin{theorem}\label{thm:cond4set2-pow-induced}
Given function $f$ defined on $\STMnbr$ and its associated symmetric $H(P)$ that satisfy \eqref{eq:cond4AF-NEPv},
suppose $f(P)\ge 0$ for $P\in\bbP$. Let $g(P)=c[f(P)]^s$ where $c>0,\,s>1$,
and let its associated symmetric matrix-valued function be $H_g(P)=cs\,[f(P)]^{s-1}H(P)$. Then
\begin{subequations}\label{eq:cond4set1-pow-induced:H(P)}
\begin{align}
\tr(P^{\T}H_g(P)P)&=s\munderbar\gamma\, g(P)\quad\mbox{for $P\in\bbP\subseteq\STM{k}{n}$},
          \label{eq:induced4pow-diff-eq:H(P)}\\
\intertext{and given $P\in\bbP$ and $\what P\in\STM{k}{n}$, there exists $Q\in\STM{k}{k}$ such
that $\wtd P=\what PQ\in\bbP$ and}
\tr(\what P^{\T} H_g(P)\what P)&\le \munderbar\alpha g(\wtd P)+(s\munderbar\gamma-\munderbar\alpha) g(P),
          \label{eq:induced-regularity4pow-diff-eq:H(P)}
\end{align}
\end{subequations}
where $\munderbar\alpha$, $\munderbar\beta$ and $\munderbar\gamma=\munderbar\alpha+\munderbar\beta$ are as in  \eqref{eq:cond4AF-NEPv-b}.
\end{theorem}

\begin{proof}
We have
\begin{align*}
\tr(P^{\T}H_g(P)P)
   &=cs\,[f(P)]^{s-1}\,\tr(P^{\T}H(P)P) \\
   &=cs\,[f(P)]^{s-1}\,\munderbar\gamma\, f(P) \qquad(\mbox{by \eqref{eq:cond4AF-NEPv-a}})\\
   &=s\munderbar\gamma\,g(P), \\
\tr(\what P^{\T}H_g(P)\what P)
   &=cs\,[f(P)]^{s-1}\,\tr(\what P^{\T}H(P)\what P)  \\
   &\le cs\,[f(P)]^{s-1}\,\Big[\munderbar\alpha \,f(\wtd P)+\munderbar\beta \,[f(P)\Big]
                     \qquad(\mbox{by \eqref{eq:cond4AF-NEPv-b}})\\
   &= cs\munderbar\alpha\,f(\wtd P)\,[f(P)]^{s-1}+\munderbar\beta s c\,[f(P)]^s \\
   &\le cs\munderbar\alpha\,\left\{\frac 1s [f(\wtd P)]^s+\frac {s-1}s [f(P)]^s\right\}+\munderbar\beta s g(P)
                     \qquad(\mbox{by Lemma~\ref{lm:YoungIneq-ext}})\\
   &=\munderbar\alpha g(\wtd P)+\munderbar\alpha (s-1) g(P)+\munderbar\beta s g(P) \\
   &=\munderbar\alpha g(\wtd P)+[\munderbar\alpha (s-1)+\munderbar\beta s] g(P),
\end{align*}
as expected.
\end{proof}

Finally, we show that {\bf the NEPv Ansatz} holds for atomic functions for NEPv. As a corollary, the NEPv approach as laid out in section~\ref{sec:NEPv-theory}
works on any atomic function for NEPv.

\begin{theorem}\label{thm:NEPv-assume:AF}
{\bf The NEPv Ansatz} holds with $\omega=1/\munderbar\alpha$ for atomic function $f$  with a symmetric matrix-valued function $H(P)\in\bbR^{n\times n}$
 satisfying the conditions in \eqref{eq:cond4AF-NEPv}.
\end{theorem}

\begin{proof}
Given $P\in\bbP\subseteq\STM{k}{n}$ and $\what P\in\STM{k}{n}$, suppose  that \eqref{eq:NEPv-assume} holds, i.e.,
$\tr(\what P^{\T}H(P)\what P)\ge\tr(P^{\T}H(P)P)+\eta$.
We have by \eqref{eq:cond4AF-NEPv}
$$
\eta+\munderbar\gamma f(P)=\eta+\tr(P^{\T}H(P)P)
   \le\tr(\what P^{\T}H(P)\what P)
   \le\munderbar\alpha f(\wtd P)+\munderbar\beta f(P)
$$
yielding $\eta/\munderbar\alpha+ f(P)\le  f(\wtd P)$, as was to be shown.
\end{proof}

\subsection{Concrete Atomic Functions}\label{ssec:AF-concrete-NEPv}
We will show  that
\begin{equation}\label{eq:concrete-AF:NEPv}
\framebox{
\parbox{10cm}{
for integer $m\ge 1$ and scalar $s\ge 1$,
$[\tr\big((P^{\T}D)^m\big)]^s$,
$[\tr\big((P^{\T}AP)^m\big)]^s$ (assuming $A\succeq 0$ if $m\ge 2$ or $s>1$),
and $[\tr\big((\Gamma P^{\T}AP)^m\big)]^s$ (assuming $A\succeq 0$, $\Gamma\succeq 0$)
}}
\end{equation}
with proper symmetric $H(P)$ to be given in the theorems and corollaries below,
satisfy \eqref{eq:cond4AF-NEPv}
and hence are atomic functions for NEPv.
Therefore, by Theorem~\ref{thm:NEPv-assume:AF}, {\bf the NEPv Ansatz} holds for them.

These atomic functions are the same in form as the ones in subsection~\ref{ssec:AF-concrete-NPD}.
According to \Cref{thm:scrH2H:tr-diff-eq}, they
are atomic function for NEPv with the generic construction \eqref{eq:H(P):always} of symmetric matrix-valued $H(\cdot)$.
However, we can do better with $[\tr\big((P^{\T}AP)^m\big)]^s$ because its partial derivative $\partial f(P)/\partial P$
takes the form $H(P)P$ exactly and thus deserves special attention. The other two functions in \eqref{eq:concrete-AF:NEPv}
do not enjoy such a property and therefore we may have to use the generic construction \eqref{eq:H(P):always} for their associated $H(\cdot)$. As an illustration, we will state the detail for $[\tr\big((P^{\T}D)^m\big)]^s$ in
\Cref{thm:lin4tr-diff-eq:H(P)} and \Cref{cor:pow(lin)4tr-diff-eq:H(P)}
and for $[\tr\big((P^{\T}AP)^m\big)]^s$ in \Cref{thm:quad4tr-diff-eq:H(P)} and \Cref{cor:pow(quad)4tr-diff-eq:H(P)}
but will just briefly comment on $[\tr\big((\Gamma P^{\T}AP)^m\big)]^s$ because its partial derivative
does not take the form $H(P)P$ exactly, as it is for $[\tr\big((P^{\T}D)^m\big)]^s$, and thus
its treatment is very much the same as what we will do for the latter.
Table~\ref{tbl:NPDo-vs-NEPv:AF} summarizes the detail for $[\tr\big((P^{\T}D)^m\big)]^s$
and $[\tr\big((P^{\T}AP)^m\big)]^s$.


When inequality \eqref{eq:cond4AF-NPDo-b} or \eqref{eq:cond4AF-NEPv-b}
become an equality, there is an important implication when it comes to
verify the corresponding ansatz for
 the composition
of atomic functions by a convex function $\phi$, namely, for equality \eqref{eq:cond4AF-NPDo-b} or \eqref{eq:cond4AF-NEPv-b},
the corresponding partial derivative $\phi_j(\bx):=\partial\phi(\bx)/\partial x_j$
can be of any sign but in general is required to be nonnegative otherwise. We have seen this
in Theorem~\ref{thm:main-npd-cvx} and will see it again in Theorem~\ref{thm:main-NEPv-cvx} later.

\begin{table}[t]
\renewcommand{\arraystretch}{1.1}
\caption{\small Concrete atomic functions for NPDo and NEPv}\label{tbl:NPDo-vs-NEPv:AF}
\centerline{\small
\begin{tabular}{|c||c|c|c|c|}
  \hline
  & \multicolumn{2}{c|}{$[\tr\big((P^{\T}D)^m\big)]^s$} & \multicolumn{2}{c|}{$[\tr\big((P^{\T}AP)^m\big)]^s$} \\ \hline\hline
\multirow{4}{*}{NPDo}    & $m=s=1$ & $m\ge 2$ or $s>1$ & \multicolumn{2}{c|}{$A\succeq 0$,} \\ \cline{2-3}
    & \eqref{eq:cond4AF-NPDo-b} an equality, & \eqref{eq:cond4AF-NPDo-b} an inequality,
    & \multicolumn{2}{c|}{\eqref{eq:cond4AF-NPDo-b} an inequality,}  \\
    & $\bbP=\STM{k}{n}_{D+}$, & $\bbP=\STM{k}{n}_{D+}$.
    & \multicolumn{2}{c|}{$\bbP=\STM{k}{n}$.} \\
    & or $\bbP=\STM{k}{n}$. &
    & \multicolumn{2}{c|}{} \\ \hline
\multirow{4}{*}{NEPv}  & \multicolumn{2}{c|}{} & $m=s=1$ & $m\ge 2$ or $s>1$  \\ \cline{4-5}
    & \multicolumn{2}{c|}{\eqref{eq:cond4AF-NEPv-b} an inequality,} &  \eqref{eq:cond4AF-NEPv-b} an equality,
    & \eqref{eq:cond4AF-NEPv-b} an inequality, \\
    & \multicolumn{2}{c|}{$\bbP=\STM{k}{n}_{D+}$,} & $\bbP=\STM{k}{n}$,
    & $A\succeq 0$, $\bbP=\STM{k}{n}$, \\ \cline{4-5}
    & \multicolumn{2}{c|}{symmetric $H_g(P)$ given by \eqref{eq:lin-pow4tr-Hg(P)}} &
      \multicolumn{2}{c|}{symmetric $H_g(P)$ given by \eqref{eq:quad-pow4tr-Hg(P)}} \\ \hline
\multicolumn{5}{l}{\small * Integer $m\ge 1$, scalar $s\ge 1$ and $A$ is symmetric. }
\end{tabular}
}
\end{table}

\begin{theorem}\label{thm:lin4tr-diff-eq:H(P)}
Let $D\in\bbR^{n\times k}$, integer $m\ge 1$  and $f(P)=\tr\big((P^{\T}D)^m\big)$ for which
we use
\begin{equation}\label{eq:H(P):lin-pow4tr-diff-eq}
H(P)=m\,\Big[D(P^{\T}D)^{m-1}P^{\T}+P(D^{\T}P)^{m-1}D^{\T}\Big],
\end{equation}
and thus $H(P)P-\scrH(P)\equiv P\big[m(D^{\T}P)^m\big]$ for $P\in\STM{k}{n}$.
Then we have
\begin{subequations}\label{eq:cond4set2-lin4tr-diff-eq:H(P)}
\begin{align}
\tr(P^{\T}H(P)P)
   &= 2m\, \tr\big((P^{\T}D)^m\big)\quad\mbox{for $P\in\STM{k}{n}$}; \label{eq:lin-pow4tr-diff-eq:H(P)} \\
\tr(\wtd P^{\T}H(P)\wtd P)
   &\le  2\tr((\wtd P^{\T}D)^m)+2(m-1)\,\tr\big((P^{\T}D)^m\big)
           \quad\mbox{for $P,\,\wtd P\in\bbP$},
           \label{eq:regularity4lin-pow4tr-diff-eq:H(P)}
\end{align}
\end{subequations}
where $\bbP=\STM{k}{n}_{D+}$. 
They, as argued in Remark~\ref{rk:regularity4tr-diff-eq:H(P)}, imply that \eqref{eq:cond4AF-NEPv} holds with $\munderbar\alpha=2$ and $\munderbar\beta=2(m-1)$,
and thus $f(P)=\tr\big((P^{\T}D)^m\big)$ is an atomic function for NEPv.
Furthermore, any solution $P$ to NEPv \eqref{eq:NEPv-form} with
$H(P)$ in \eqref{eq:H(P):lin-pow4tr-diff-eq} such that $(P^{\T}D)^m$ is symmetric is a solution to the KKT condition \eqref{eq:KKT} and vice versa.
\end{theorem}

\begin{proof}
$H(P)$ in the theorem is in fact the generic one in \eqref{eq:H(P):always} for the case
and in the notation of Theorem~\ref{thm:H(P)-eligibility},
$\scrM(P)=m(P^{\T}D)^m$. Hence any solution $P$ to NEPv \eqref{eq:NEPv-form} such that $(P^{\T}D)^m$ is symmetric is a solution to
the KKT condition \eqref{eq:KKT} and vice versa.

Equation \eqref{eq:lin-pow4tr-diff-eq:H(P)} can be straightforwardly verified.
Now we prove \eqref{eq:regularity4lin-pow4tr-diff-eq:H(P)}.
Inequality \eqref{eq:regularity4lin-pow4tr-diff-eq:H(P)} for $m=1$
in fact holds for all $P\in\STM{k}{n}$. To see this, for $m=1$ and $P\in\STM{k}{n}$,
since $\wtd P^{\T}D\succeq 0$,
$$
\tr(\wtd P^{\T}H(P)\wtd P)=2\tr(\wtd P^{\T}DP^{\T}\wtd P)
   \le2\|\wtd P^{\T}DP^{\T}\wtd P\|_{\tr}
   \le 2\|\wtd P^{\T}D\|_{\tr}=2\tr(\wtd P^{\T}D)
$$
by Lemma~\ref{lm:maxtrace}.
In general for $m>1$, suppose both $P^{\T}D\succeq 0$ and $\wtd P^{\T}D\succeq 0$.
Then
\begin{align*}
\tr(\wtd P^{\T}H(P)\wtd P)&=2m\,\tr(\wtd P^{\T}D(P^{\T}D)^{m-1}P^{\T}\wtd P) \\
   &\le 2m\left\|\wtd P^{\T}D(P^{\T}D)^{m-1}P^{\T}\wtd P\right\|_{\tr} \qquad (\mbox{by Lemma~\ref{lm:maxtrace}}) \\
   &\le 2m\left\|\wtd P^{\T}D(P^{\T}D)^{m-1}\right\|_{\tr} \|P^{\T}\wtd P\|_2  \\
   &\le 2m\left\|\wtd P^{\T}D(P^{\T}D)^{m-1}\right\|_{\tr} \qquad (\mbox{since $\|P^{\T}\wtd P\|_2\le 1$})  \\
   &\le 2\tr((\wtd P^{\T}D)^m)+2(m-1)\,\tr\big((P^{\T}D)^m\big),
\end{align*}
where the last inequality is due to Lemma~\ref{lm:vN-tr-ineq-ext}.

Finally \eqref{eq:cond4set2-lin4tr-diff-eq:H(P)} implies  \eqref{eq:cond4AF-NEPv}, as
argued in Remark~\ref {rk:regularity4tr-diff-eq:H(P)}.
\end{proof}


For any scalar $s>1$,  $[\tr\big((P^{\T}D)^m\big)]^s$ is  well-defined for any $P\in\bbR^{n\times k}$
such that $\tr\big((P^{\T}D)^m\big)\ge 0$.
In particular, $[\tr\big((P^{\T}D)^m\big)]^s$ is well-defined
for $P\in\STM{k}{n}_{D+}$.
Combining Theorems~\ref{thm:cond4set2-pow-induced} and \ref{thm:lin4tr-diff-eq:H(P)}, we obtain the following corollary.

\begin{corollary}\label{cor:pow(lin)4tr-diff-eq:H(P)}
Let $D\in\bbR^{n\times k}$, integer $m\ge 1$, scalar $s>1$, $g(P)=[\tr\big((P^{\T}D)^m\big)]^s$, and $\bbP=\STM{k}{n}_{D+}$.
Let $H(P)$
be as in \eqref{eq:H(P):lin-pow4tr-diff-eq}
and
\begin{align}
H_g(P)&=s\,[\tr\big((P^{\T}D)^m\big)]^{s-1}H(P) \nonumber\\
      &=s\,[\tr\big((P^{\T}D)^m\big)]^{s-1}m\,\Big[D(P^{\T}D)^{m-1}P^{\T}+P(D^{\T}P)^{m-1}D^{\T}\Big]
          \label{eq:lin-pow4tr-Hg(P)}
\end{align}
for which $H_g(P)P-\partial g(P)/\partial P\equiv P\big[sm\,[\tr\big((P^{\T}D)^m\big)]^{s-1}\,(D^{\T}P)^m\big]$.
Then
\begin{subequations}\label{eq:pow(lin)4tr-diff-eq:H(P)}
\begin{align}
\tr(P^{\T}H_g(P)P)&=2sm\, [\tr\big((P^{\T}D)^m\big)]^s \quad\mbox{for $P\in\bbP$},
                    \label{eq:lin-pow4trpow-diff-eq:H(P)}              \\
\tr(\wtd P^{\T}H_g(P)\wtd P)
             &\le 2[\tr((\wtd P^{\T}D)^m)]^s+2(sm-1)[\tr\big((P^{\T}D)^m\big)]^s
                \quad\mbox{for $P,\,\wtd P\in\bbP$}. \label{eq:lin-pow-regularity4pow-diff-eq:H(P)}
\end{align}
They, as argued in Remark~\ref{rk:regularity4tr-diff-eq:H(P)}, imply that \eqref{eq:cond4AF-NEPv} holds with $\munderbar\alpha=2$ and $\munderbar\beta=2(sm-1)$,
and thus $g(P)=[\tr\big((P^{\T}D)^m\big)]^s$  for scalar $s>1$ is an atomic function for NEPv.
\end{subequations}
\end{corollary}

Next we consider $\tr\big((P^{\T}AP)^m\big)$ and its power.

\begin{theorem}\label{thm:quad4tr-diff-eq:H(P)}
Let symmetric $A\in\bbR^{n\times n}$, integer $m\ge 1$, and
$f(P)=\tr\big((P^{\T}AP)^m\big)$ for which we use
\begin{equation}\label{eq:H(P):quad-pow4tr-diff-eq}
H(P):=2m\,A(PP^{\T}A)^{m-1}
\end{equation}
and thus $\scrH(P)\equiv H(P)P$ for $P\in\bbR^{n\times k}$.
\begin{subequations}\label{eq:quad4tr-diff-eq:H(P)}
\begin{enumerate}[{\rm (a)}]
  \item For $P\in\bbR^{n\times k}$, we have
        \begin{equation}\label{eq:quad-pow4tr-diff-eq:H(P)}
        \tr(P^{\T}H(P)P)=2m\, f(P)\equiv 2m\,\tr\big((P^{\T}AP)^m\big).
        \end{equation}
  \item Let $P,\,\wtd P\in\bbR^{n\times k}$.
        \begin{enumerate}[{\rm (i)}]
          \item For $m=1$, we always have
                \begin{equation}\label{eq:regularity4quad4tr-diff-eq:H(P)}
                \tr(\wtd P^{\T}H(P)\wtd P)= 2\tr(\wtd P^{\T}A\wtd P);
                \end{equation}
          \item For $m>1$, if $A\succeq 0$, then
                \begin{equation}\label{eq:regularity4quad-pow4tr-diff-eq:H(P)}
                \tr(\wtd P^{\T}H(P)\wtd P)\le 2\tr((\wtd P^{\T}A\wtd P)^m)+2(m-1)\,\tr\big((P^{\T}AP)^m\big).
                \end{equation}
         \end{enumerate}
\end{enumerate}
\end{subequations}
They, as argued in Remark~\ref{rk:regularity4tr-diff-eq:H(P)}, imply that \eqref{eq:cond4AF-NEPv} holds with $\munderbar\alpha=2$ and $\munderbar\beta=2(m-1)$, and $\bbP=\STM{k}{n}$,
and thus $f(P)=\tr\big((P^{\T}AP)^m\big)$ is an atomic function for NEPv.
\end{theorem}

\begin{proof}
With $H(P)$ as in \eqref{eq:H(P):quad-pow4tr-diff-eq}, equation \eqref{eq:quad-pow4tr-diff-eq:H(P)} is straightforwardly
verified.

For $m=1$, $H(P)=2A$ and hence immediately we have \eqref{eq:regularity4quad4tr-diff-eq:H(P)}.

Consider $m>1$ and suppose $A\succeq 0$.
Let $X=A^{1/2}\wtd P$ and $Y=A^{1/2}P$, where $A^{1/2}$ is the positive semidefinite square root of $A$.
We have
\begin{align*}
\wtd P^{\T}H(P)\wtd P&=2m\,\wtd P^{\T}AP(P^{\T}AP)^{m-2}P^{\T}A\wtd P\\
   &=2m\,X^{\T}Y(Y^{\T}Y)^{m-2}Y^{\T}X \\
   &=2m\,X^{\T}(YY^{\T})^{m-1}X, \\
\tr(\wtd P^{\T}H(P)\wtd P)
  &=2m\,\tr(X^{\T}(YY^{\T})^{m-1}X) \\
  &=2m\,\tr(XX^{\T}(YY^{\T})^{m-1}) \\
  &\le 2\tr((XX^{\T})^m)+2(m-1)\,\tr((YY^{\T})^m) \qquad (\mbox{by Lemma~\ref{lm:vN-tr-ineq-ext}}) \\
  &= 2\tr((X^{\T}X)^m)+2(m-1)\,\tr((Y^{\T}Y)^m)  \\
  &=2\tr((\wtd P^{\T}A\wtd P)^m)+2(m-1)\,\tr\big((P^{\T}AP)^m\big),
\end{align*}
which is \eqref{eq:regularity4quad-pow4tr-diff-eq:H(P)}.
\end{proof}


We emphasize that  \eqref{eq:quad-pow4tr-diff-eq:H(P)}, \eqref{eq:regularity4quad4tr-diff-eq:H(P)},
and \eqref{eq:regularity4quad-pow4tr-diff-eq:H(P)} actually holds for any $P,\,\wtd P\in\bbR^{n\times k}$,
broader than
what the conditions in \eqref{eq:cond4AF-NEPv} entail.
With Theorem~\ref{thm:quad4tr-diff-eq:H(P)} and using a similar proof to that of
Theorem~\ref{thm:cond4set2-pow-induced}, we get the following corollary that is valid for
all $P,\,\wtd P\in\bbR^{n\times k}$,  broader than
simply combining Theorems~\ref{thm:cond4set2-pow-induced} with~\ref{thm:quad4tr-diff-eq:H(P)}.

\begin{corollary}\label{cor:pow(quad)4tr-diff-eq:H(P)}
Let symmetric $A\in\bbR^{n\times n}$ be positive semidefinite, integer $m\ge 1$, scalar $s>1$, $g(P)=[\tr\big((P^{\T}AP)^m\big)]^s$,
and let
\begin{align}
H_g(P)&=s\,[\tr\big((P^{\T}AP)^m\big)]^{s-1}H(P) \nonumber \\
      &=s\,[\tr\big((P^{\T}AP)^m\big)]^{s-1}\cdot 2m\,A(PP^{\T}A)^{m-1}  \label{eq:quad-pow4tr-Hg(P)}
\end{align}
for which $\partial g(P)/\partial P\equiv H_g(P)P$ for $P\in\bbR^{n\times k}$,
where $H(P)$ is as in \eqref{eq:H(P):quad-pow4tr-diff-eq}.
For $P,\,\wtd P\in\bbR^{n\times k}$,
we have
\begin{subequations}\label{eq:pow(quad)4tr-diff-eq:H(P)}
\begin{align}
\tr(P^{\T}H_g(P)P)
    &=2sm\, [\tr\big((P^{\T}AP)^m\big)]^s, \label{eq:quad-pow4trpow-diff-eq:H(P)} \\
\tr(\wtd P^{\T}H_g(P)\wtd P)
             &\le 2[\tr((\wtd P^{\T}A\wtd P)^m)]^s+2(sm-1)[\tr\big((P^{\T}AP)^m\big)]^s.
                      \label{eq:quad-pow-regularity4pow-diff-eq:H(P)}
\end{align}
\end{subequations}
They, as argued in Remark~\ref{rk:regularity4tr-diff-eq:H(P)}, imply that \eqref{eq:cond4AF-NEPv} holds with $\munderbar\alpha=2$ and $\munderbar\beta=2(sm-1)$, and $\bbP=\STM{k}{n}$,
and thus $g(P)=[\tr\big((P^{\T}AP)^m\big)]^s$  for scalar $s>1$ is an atomic function for NEPv.
\end{corollary}

We remark again in passing that we will omit the detail on $[\tr\big((\Gamma P^{\T}AP)^m\big)]^s$ being an atomic function for NEPv,
as it can be similarly done to \Cref{thm:lin4tr-diff-eq:H(P)} and \Cref{cor:pow(lin)4tr-diff-eq:H(P)} where
the generic construction \eqref{eq:H(P):always} of symmetric matrix-valued $H(\cdot)$ is used for
$[\tr\big((P^{\T}D)^m\big)]^s$, i.e., to go with $g(P)=[\tr\big((\Gamma P^{\T}AP)^m\big)]^s$ for applying the NEPv framework we will use
$$
H_g(P)=s[\tr\big((\Gamma P^{\T}AP)^m\big)]^{s-1}\cdot 2m\Big[AP(\Gamma P^{\T}AP)^{m-1}\Gamma P^{\T}+P(\Gamma P^{\T}AP)^{m-1}\Gamma P^{\T}A\Big].
$$

%
%

\section{Convex Composition}\label{sec:CVX-comp:NEPv}
We are interested in solving the same optimization problem on the Stiefel manifold $\STM{k}{n}$ as in \eqref{eq:Master-OptSTM}
by Algorithm~\ref{alg:SCF4NEPv} and its accelerating variation in Algorithm~\ref{alg:SCF4NEPv+LOCG}
with convergence guarantee.
In that regard, we stick to the initial setup at the beginning of section~\ref{sec:CVX-comp:NPDo} up to the paragraph containing
\eqref{eq:T-general}. We then go along a different path -- the path of the NEPv approach. To that end, we will
have to specify what $H(P)$, a symmetric matrix-valued function, to use for a given objective
function $f=\phi\circ T$ in \eqref{eq:Master-OptSTM},
assuming that a symmetric matrix-valued function
has already been constructed for each component of $T(P)$.

In its generality, each component $f_i(P_i)$ of $T(P)$ in  \eqref{eq:T-general} may involve a few but not necessarily
all columns of $P$.
It turns out that, for the case when not all $P_i=P$, we do not have a feasible way to construct
a symmetric matrix-valued function $H(P)$  for $f(P)=\phi\circ T(P)$ out of those for the components of $T(P)$.
That leaves us the only option of using the generic $H(P)$ in \eqref{eq:H(P):always}:
$$
H(P):=[\scrH(P)]P^{\T}+P[\scrH(P)]^{\T}
  \equiv\left[\frac{\partial f(P)}{\partial P}\right] P^{\T}+P\left[\frac{\partial f(P)}{\partial P}\right]^{\T},
$$
completely ignoring the symmetric matrix-valued functions for the components of $T(P)$ already known. Furthermore, with this generic
$H(P)$, in order to fulfill {\bf the NEPv Ansatz}, we will have to assume the components of
$T(P)$ are atomic functions for NPDo satisfying \eqref{eq:cond4AF-NPDo:cvx},
which means that {\bf the NPDo Ansatz} will hold for $f$. Consequently, our previous NPDo approach in Part~I will work
on such function $f$ in the first place, making it unnecessary resort to the NEPv approach for solving
the optimization problem \eqref{eq:Master-OptSTM}. More detail will be explained in subsection~\ref{ssec:PineP}.
Besides explaining the extra complexity that not all $P_i=P$ may bring, in subsection~\ref{ssec:PineP} we will also demonstrate that
the NEPv approach can still be made to work with the generic $H(P)$, just that the approach may not be as effective as the NPDo approach in Part~I.

\subsection{All $P_i$ are the entire $P$}\label{ssec:PieqP}
Our  focus is on
the case when all $P_i=P$, i.e., each component $f_i(P_i)=f_i(P)$.
Specifically, we will consider a special case of $T(P)$ in \eqref{eq:T-general}:
\begin{equation}\label{eq:T0-general:NEPv}
T_0(P)=\begin{bmatrix}
                                        f_1(P) \\
                                        f_2(P) \\
                                        \vdots \\
                                        f_N(P) \\
                                      \end{bmatrix},
\end{equation}
where $f_i$ for $1\le i\le N$ are atomic functions for NEPv, whose associated
symmetric matrix-valued functions are $H_i(P)\in\bbR^{n\times n}$ for $1\le i\le N$, respectively.

Our first task is to create a proper symmetric matrix-valued function $H(P)\in\bbR^{n\times n}$
to go with $f=\phi\circ T_0$ from $H_i(P)$ for $1\le i\le N$. To that end,
we will follow what we did in subsection~\ref{ssec:LOCG:NEPv} to first figure out what
$H(P)$ should be for the circumstance
when \eqref{eq:cond:KKT=NEPv} holds for each $f_i$, namely,
\begin{equation}\label{eq:cond:Hi(P)}
H_i(P)P-\frac {\partial f_i(P)}{\partial P}=P\,\scrM_i(P)
\quad\mbox{for $1\le i\le N$},
\end{equation}
and then show that the newly created $H(P)$ can serve the purpose for us as far as inheriting
{\bf the NEPv Ansatz} from $f_i$ with $H_i$ for $1\le i\le N$ is concerned,
without the need to assume \eqref{eq:cond:Hi(P)} anymore.
Recall notation $\phi_i(\bx)$ in \eqref{eq:phi-i} for the $i$th partial derivative of $\phi(\bx)$.
For $f=\phi\circ T_0$, we have
$$
\scrH(P):=\frac {\partial f(P)}{\partial P}
  =\sum_{i=1}^N\phi_i(T_0(P))\,\frac {\partial f_i(P)}{\partial P},
$$
and hence naturally, we may choose
\begin{equation}\label{eq:H(P)-comp-form}
H(P)=\sum_{i=1}^N\phi_i(T_0(P))\,H_i(P),
\end{equation}
for which, with \eqref{eq:cond:Hi(P)}, we find
\begin{align}
H(P)P-\frac {\partial f(P)}{\partial P}
 &=\sum_{i=1}^N\phi_i(T_0(P))\,\left(H_i(P)P-\frac {\partial f_i(P)}{\partial P}\right) \nonumber \\
 &=P\sum_{i=1}^N\phi_i(T_0(P))\,\scrM_i(P) \nonumber \\
 &=:P\scrM(P). \label{eq:H(P)-comp-form:eligibility}
\end{align}
Therefore the symmetric $H(P)\in\bbR^{n\times n}$ in \eqref{eq:H(P)-comp-form} fits the one suggested by
Theorem~\ref{thm:H(P)-eligibility}. In particular, any solution $P_*$ to NEPv \eqref{eq:NEPv-form}
with $H(P)$ given by \eqref{eq:H(P)-comp-form} satisfies the KKT condition~\eqref{eq:KKT} if
$\scrM(P_*)$ defined in \eqref{eq:H(P)-comp-form:eligibility} is symmetric and vice versa, as guaranteed by Theorem~\ref{thm:H(P)-eligibility}.

Next we will show that $f=\phi\circ T_0$ with $H(P)$ in \eqref{eq:H(P)-comp-form} inherits
{\bf the NEPv Ansatz} from $f_i$ with $H_i$ for $1\le i\le N$
without assuming \eqref{eq:cond:Hi(P)}.
To that end, we  place some consistency conditions upon all components of $T_0(P)$ in \eqref{eq:T0-general:NEPv} as follows:
for $1\le i\le N$
%
%
%
%
%
\begin{subequations}\label{eq:cond4AF-NEPv:cvx}
\begin{align}
\tr(P^{\T}H_i(P)P)&=\munderbar\gamma_i\, f_i(P)\quad\mbox{for $P\in\bbP\subseteq\STM{k}{n}$},
              \label{eq:cond4AF-NEPv:cvx-a} \\
\intertext{and given $\what P\in\STM{k}{n}$ and $P\in\bbP$, there exists $Q\in\STM{k}{k}$ such
that $\wtd P=\what PQ\in\bbP$ and} 
\tr(\what P^{\T}H_i(P)\what P)
   &\le\munderbar\alpha f_i(\wtd P)+\munderbar\beta_i f_i(P),
      \label{eq:cond4AF-NEPv:cvx-b}
\end{align}
\end{subequations}
where $\munderbar\alpha>0,\,\munderbar\beta_i\ge 0$,  and
$\munderbar\gamma_i=\munderbar\alpha+\munderbar\beta_i$ are  constants.
It is important to keep in mind that some of the inequalities in \eqref{eq:cond4AF-NEPv:cvx-b} may  actually be equalities,
e.g.,  for $f_i(P)=\tr(P^{\T}A_iP)$ it is an equality by Theorem~\ref{thm:lin4tr-diff-eq}.

On the surface, it looks like that each $f_i$ is simply an atomic function for NEPv,  but there are three
built-in consistency requirements in \eqref{eq:cond4AF-NEPv:cvx} among
all $f_i$: 1) the same $\bbP$ for all;  2) the same $\munderbar\alpha$ for all, and
3) the same $Q$ to give $\wtd P=\what PQ$ for all.

\begin{theorem}\label{thm:main-NEPv-cvx}
Consider $f=\phi\circ T_0$, where $T_0(\cdot)$ takes the form in \eqref{eq:T0-general:NEPv} and $\phi$ is convex and differentiable with partial derivatives denoted by
$\phi_i$ as in \eqref{eq:phi-i}. Let $H(P)$ be given by \eqref{eq:H(P)-comp-form} with $H_i(P)$ for $1\le i\le N$
satisfying \eqref{eq:cond4AF-NEPv:cvx}.
If $\phi_i(\bx)\ge 0$ for those $i$ for which \eqref{eq:cond4AF-NEPv:cvx-b} does not become an equality, then
{\bf the NEPv Ansatz}   with $\omega=1/\munderbar\alpha$ holds for $f=\phi\circ T_0$ with $H$.
\end{theorem}

\begin{proof}
Given $P\in\bbP\subseteq\STM{k}{n}$ and $\what P\in\STM{k}{n}$, suppose that \eqref{eq:NEPv-assume} holds, i.e.,
$\tr(\what P^{\T}H(P)\what P)\ge\tr(P^{\T}H(P)P)+\eta$.
Let $\wtd P=\what PQ$ where $Q\in\STM{k}{k}$ is the one dictated by the consistency conditions in \eqref{eq:cond4AF-NEPv:cvx}.
Write
$$
\bx=T_0(P)\equiv [x_1,x_2,\ldots,x_N]^{\T}, \quad
\wtd\bx=T_0(\wtd P)\equiv [\wtd x_1,\wtd x_2,\ldots,\wtd x_N]^{\T},
$$
i.e., $x_i=f_i(P)$ and $\wtd x_i=f_i(\wtd P)$.
Noticing $H(P)$ in \eqref{eq:H(P)-comp-form}, we have by \eqref{eq:cond4AF-NEPv:cvx}
\begin{align*}
\tr(P^{\T}H(P)P)
  &=\sum_{i=1}^N\phi_i(\bx)\tr(P^{\T}H_i(P)P) \\
  &=\sum_{i=1}^N\munderbar\gamma_i\phi_i(\bx)\,x_i, \qquad(\mbox{by \eqref{eq:cond4AF-NEPv:cvx-a}}) \\
\tr(\what P^{\T}H(P)\what P)
  &=\sum_{i=1}^N\phi_i(\bx)\tr(\what P^{\T}H_i(P)\what P)  \\
  &\le\sum_{i=1}^N\phi_i(\bx)\,(\munderbar\alpha\wtd x_i+\munderbar\beta_i x_i),
\end{align*}
where the last inequality is due to $\phi_i\ge 0$ when the corresponding \eqref{eq:cond4AF-NEPv:cvx-b}
does not become an equality.
Plug them into $\eta+\tr(P^{\T}H(P)P)\le\tr(\what P^{\T}H(P)\what P)$ and simplify the resulting inequality with
the help of $\munderbar\gamma_i=\munderbar\alpha+\munderbar\beta_i$ to get
$$
\eta/\munderbar\alpha+\nabla\phi(\bx)^{\T}\bx=\eta/\munderbar\alpha+\sum_{i=1}^N\phi_i(\bx)\,x_i\le\sum_{i=1}^N\phi_i(\bx)\,\wtd x_i
   =\nabla\phi(\bx)^{\T}\wtd\bx.
$$
Finally apply Lemma~\ref{lm:convex-mono} to yield $f(\wtd P)\ge f(P)+\eta/\munderbar\alpha$.
\end{proof}

With Theorem~\ref{thm:main-NEPv-cvx} come the general results established in section~\ref{sec:NEPv-theory}.
In particular, Algorithm~\ref{alg:SCF4NEPv} (NEPvSCF) and its accelerating variation in Algorithm~\ref{alg:SCF4NEPv+LOCG} can be applied to find a maximizer of \eqref{eq:Master-OptSTM},
except that the calculation of $Q_i$ at Line 4 of Algorithm~\ref{alg:SCF4NEPv} remains to be specified. This missing detail is in general
dependent of the particularity of the mapping $T_0$ and the convex function $\phi$.
What we will do in Examples~\ref{eg:T1a-NEPv} and \ref{eg:T3a} below provides some ideas on this matter.

In the rest of this section,
$A_i\in\bbR^{n\times n}$ for $1\le i\le\ell$ are at least symmetric
and $D_i\in\bbR^{n\times k}$ for $1\le i\le t$.

\begin{example}\label{eg:T1a-NEPv}
Consider $T_{1a}$  in \eqref{eq:T-tr-a}, as a special case of $T_0$, and optimization problem \eqref{eq:Master-OptSTM} with
$f=\phi\circ T_{1a}$. For this example, we will use
\begin{equation}\label{eq:H(P)-comp-T1a}
H(P)=\sum_{i=1}^{\ell}\phi_i(T_{1a}(P))\underbrace{2A_i}_{=:H_i(P)}
     +\sum_{j=1}^{t} \phi_{\ell+j}(T_{1a}(P))\, \underbrace{\big(D_jP^{\T}+PD_j^{\T}\big)}_{=:H_{\ell+j}(P)},
\end{equation}
for which $H(P)P-\scrH(P)\equiv P\big[\scrD(P)^{\T}P\big]$ for $P\in\STM{k}{n}$, where, as in \eqref{eq:scrD},
$$
\scrD(P)=\sum_{j=1}^{t}\phi_{\ell+j}(T_{1a}(P))\, D_j.
$$
Any solution $P$ to NEPv~\eqref{eq:NEPv-form} with
$H(P)$ in \eqref{eq:H(P)-comp-T1a} such that $[\scrD(P)]^{\T}P$ is symmetric is a solution to the KKT condition \eqref{eq:KKT}
and vice versa.
It can be seen that \eqref{eq:cond4AF-NEPv:cvx-b} is an equality for $1\le i\le\ell$ and hence
it does not need to require $\phi_i\ge 0$ for $1\le i\le\ell$. In addition to this, instead of treating each
$H_{\ell+j}(P)$ separately, we can treat $H_{\ell+j}(P)$ for $1\le j\le t$ collectively all at once
through $\scrD(P)$, as
we did in \eqref{thm:main-npd-cvx:T1a}, making all $\phi_{\ell+j}\ge 0$ unnecessary as well. A much more improved
version of Theorem~\ref{thm:main-NEPv-cvx}
is stated as Theorem~\ref{thm:main-NEPv-cvx:T1a} below,
according to which, the best $Q_i$ at Line~4 of Algorithm~\ref{alg:SCF4NEPv}
when applied to $\phi\circ T_{1a}$
is an orthonormal polar factor of $[\what P^{(i)}]^{\T}\scrD(P^{(i)})$.
A special case of $T_{1a}$ is: $t=0$, $k=1$, $P_i=\bp$ (a unit vector) for $1\le i\le\ell$, which gives the main problem of \cite{balu:2024}
  (in the paper, $\phi(\bx)=\sum_{i=1}^{\ell}\psi_i(x_i)$
for $\bx=[x_1,x_2,\ldots,x_{\ell}]^{\T}$ with each $\psi_i$ being a convex function of a single-variable).
\end{example}

\begin{theorem}\label{thm:main-NEPv-cvx:T1a}
Consider $f=\phi\circ T_{1a}$, and let $\scrD(P)$ be as in \eqref{eq:scrD}
and $H(P)$ as in \eqref{eq:H(P)-comp-T1a}.
Given $\what P\in\STM{k}{n},\,P\in\STM{k}{n}$, let $\wtd P=\what PQ$ where
$Q$ is an orthonormal polar factor of $\what P^{\T}\scrD(P)$.
If
$$
\tr(\what P^{\T}H(P)\what P)\ge\tr(P^{\T}H(P)P)+\eta,
$$
then
$
f(\wtd P)\ge f(P)+\frac 12\eta+\delta,
$
where $\delta=\|\what P^{\T}\scrD(P)\|_{\tr}-\tr(\what P^{\T}\scrD(P)P^{\T}\what P)$.
In particular, {\bf the NEPv Ansatz} holds  with $\omega=1/2$.
\end{theorem}

\begin{proof}
Along the lines of the proof of Theorem~\ref{thm:main-NEPv-cvx}, here we will have
\begin{align}
\tr(P^{\T}H(P)P)
  &=2\sum_{i=1}^{\ell}\phi_i(\bx)\,x_i+2\sum_{i=1}^{t}\phi_{\ell+i}(\bx)\,x_{\ell+i}, \nonumber\\
\|\what P^{\T}\scrD(P)\|_{\tr}
  &=\tr(\wtd P^{\T}\scrD(P)) \qquad(\mbox{since $\wtd P^{\T}\scrD(P)=Q^{\T}[\what P^{\T}\scrD(P)]\succeq 0$}) \nonumber\\
  &=\sum_{i=1}^{t}\phi_{\ell+i}(\bx)\,\wtd x_{\ell+i}, \label{eq:main-NEPv-cvx:T1a:pf-1}\\
\tr(\what P^{\T}H(P)\what P)
  &=2\sum_{i=1}^{\ell}\phi_i(\bx)\,\tr(\what P^{\T}A_i\what P)+2\sum_{i=1}^{t}\phi_{\ell+i}(\bx)\,\tr(\what P^{\T}D_iP^{\T}\what P) \nonumber\\
  &=2\sum_{i=1}^{\ell}\phi_i(\bx)\,\tr(\wtd P^{\T}A_i\wtd P)
       +2\tr(\what P^{\T}\scrD(P)P^{\T}\what P) \nonumber\\
  &=2\sum_{i=1}^{\ell}\phi_i(\bx)\,\wtd x_i+2\|\what P^{\T}\scrD(P)\|_{\tr}-2\delta \nonumber\\
  &=2\sum_{i=1}^{\ell}\phi_i(\bx)\,\wtd x_i+2\sum_{i=1}^{t}\phi_{\ell+i}(\bx)\,\wtd x_{\ell+i}-2\delta, \nonumber
\end{align}
where the last equality is due to \eqref{eq:main-NEPv-cvx:T1a:pf-1}.
Plug them into $\eta+\tr(P^{\T}H(P)P)\le\tr(\what P^{\T}H(P)\what P)$ and simplify the resulting inequality to get
$
\frac 12\eta+\delta+\nabla\phi(\bx)^{\T}\bx\le\nabla\phi(\bx)^{\T}\wtd\bx,
$
and then apply Lemma~\ref{lm:convex-mono} to conclude the proof.
\end{proof}

Theorem~\ref{thm:main-NEPv-cvx:T1a} improves Theorem~\ref{thm:main-NEPv-cvx} in that
the objective value
        increases additional $\delta$ more.  We notice, by Lemma~\ref{lm:maxtrace}, that
        $$
        \tr(\what P^{\T}\scrD(P)P^{\T}\what P)
          \le\|\what P^{\T}\scrD(P)P^{\T}\what P\|_{\tr}\le\|\what P^{\T}\scrD(P)\|_{\tr}\|P^{\T}\what P\|_2
          \le\|\what P^{\T}\scrD(P)\|_{\tr}
        $$
        and hence
        $\delta\ge 0$ and it is strict when any one of the inequalities above is strict.
%
%
Theorem~\ref{thm:main-NEPv-cvx:T1a} compares favorably against Theorem~\ref{thm:main-npd-cvx:T1a}. Both are about
mapping $T_{1a}$, but Theorem~\ref{thm:main-NEPv-cvx:T1a} puts no condition on symmetric matrices $A_i$ and
no condition on the partial derivatives, whereas Theorem~\ref{thm:main-npd-cvx:T1a} requires all $A_i\succeq 0$ and
$\phi_i\ge 0$ for $1\le i\le \ell$.

Also note that, in Theorem~\ref{thm:main-NEPv-cvx:T1a}, $\wtd P$ satisfies
$\wtd P^{\T}\scrD(P)\succeq 0$. Along the same line of the proof of Theorem~\ref{thm:maxers-NEPv},
we establish another necessary condition in Corollary~\ref{cor:T1a-NEPv} for any maximizer $P_*$ of \eqref{eq:Master-OptSTM}
with $T=T_{1a}$, besides the ones in
Theorem~\ref{thm:maxers-NEPv}.

\begin{corollary}\label{cor:T1a-NEPv}
Consider \eqref{eq:Master-OptSTM} with $T=T_{1a}$ and let $H(P)$ be as in \eqref{eq:H(P)-comp-T1a}.
If $P_*$ is a  maximizer of \eqref{eq:Master-OptSTM}, then we
have not only NEPv \eqref{eq:NEPv-form} satisfied by $P=P_*$ and $\Omega=\Omega_*:=P_*^{\T}H(P_*)P_*$ whose eigenvalues
consist of the $k$ largest ones of $H(P_*)$, 
but also $P_*^{\T}\scrD(P_*)\succeq 0$.
\end{corollary}

\begin{example}\label{eg:T2a}
Consider $T_{2a}$, a special case of $T_2$ in \eqref{eq:T-F},
\begin{equation}\label{eq:T-F-a}
T_{2a}\,:\, P\in\STM{k}{n} \to T_{2a}(P):=\begin{bmatrix}
                                        \|P^{\T}A_1P\|_{\F}^2 \\
                                        \vdots \\
                                        \|P^{\T}A_{\ell}P\|_{\F}^2 \\
                                        \|P^{\T}D_1\|_{\F}^2 \\
                                        \vdots \\
                                        \|P^{\T}D_t\|_{\F}^2 \\
                                      \end{bmatrix} \in\bbR^{\ell+t}.
\end{equation}
Either $\ell=0$ or $t=0$ is allowed. Notice that
$$
\|P^{\T}A_iP\|_{\F}^2=\tr((P^{\T}A_iP)^2), \quad
\|P^{\T}D_j\|_{\F}^2=\tr(P^{\T}D_jD_j^{\T}P).
$$
For optimization problem \eqref{eq:Master-OptSTM} with
$f=\phi\circ T_{2a}$, we will use
\begin{equation}\label{eq:H(P)-comp-T2a}
H(P)=\sum_{i=1}^{\ell}\phi_i(T_{2a}(P))\,\underbrace{4A_iPP^{\T}A_i}_{=:H_i(P)}
     +\sum_{j=1}^{t} \phi_{\ell+j}(T_{2a}(P))\, \underbrace{D_jD_j^{\T}}_{=:H_{\ell+j}(P)},
\end{equation}
for which $\scrH(P)\equiv H(P)P$ for $P\in\bbR^{n\times k}$.
If all $A_i\succeq 0$, then, by Theorem~\ref{thm:quad4tr-diff-eq:H(P)},
the consistency conditions in \eqref{eq:cond4AF-NEPv:cvx} are satisfied with $\bbP=\STM{k}{n}$, $Q=I_k$,
$\munderbar\alpha=2$, $\munderbar\beta_i=2$ for $1\le i\le\ell$ and $\munderbar\beta_{\ell+j}=0$ for $1\le j\le t$.
Note, for the example, \eqref{eq:cond4AF-NEPv:cvx-b} for $\ell+1\le i\le \ell+t$ are equalities and hence
Theorem~\ref{thm:main-NEPv-cvx} requires
$\phi_i\ge 0$ for $1\le i\le\ell$ only.
Optimization problem \eqref{eq:Master-OptSTM} with $T=T_{2a}$ for $t=0$
and
$\phi(\bx)=\sum_{i=1}^{\ell}x_i$
gives the key optimization problem in the uniform multidimensional scaling (UMDS) method \cite{zhzl:2017}.

Comparing the conclusion here and that of Example~\ref{eg:T2} (for all $P_i=P$), we don't require $\phi_{\ell+j}\ge 0$ for $1\le j\le t$ here, everything else being equal. In particular, both require $A_i\succeq 0$ for all $i$. Next, we explain how to make the NEPv approach work on $f=\phi\circ T_{2a}$ even if some $A_i\not\succeq 0$, yielding yet another example for which the NEPv
approach works but the NPDo approach may not.
Let $\delta_i\in\bbR$ such that $\what A_i=A_i-\delta_iI\succeq 0$ for $1\le i\le\ell$.
This is always possible by letting $\delta_i$ be some lower bound of the eigenvalues of $A_i$, and numerically, $\delta_i$ can be estimated cheaply \cite{zhli:2011}. Notice that
$$
\tr((P^{\T}A_iP)^2)=\tr((P^{\T}\what A_iP)^2)+2\delta_i\tr(P^{\T}\what A_iP)+k\delta_i^2.
$$
Define
\begin{equation}\label{eq:T-F-a:hat}
\what T_{2a}\,:\, P\in\STM{k}{n} \to \what T_{2a}(P):=\begin{bmatrix}
                                        \tr((P^{\T}\what A_1P)^2) \\
                                        \vdots \\
                                        \tr((P^{\T}\what A_{\ell}P)^2) \\
                                        \tr(P^{\T}\what A_1P) \\
                                        \vdots \\
                                        \tr(P^{\T}\what A_{\ell}P) \\
                                        \|P^{\T}D_1\|_{\F}^2 \\
                                        \vdots \\
                                        \|P^{\T}D_t\|_{\F}^2 \\
                                      \end{bmatrix} \in\bbR^{2\ell+t},
\end{equation}
and an affine transformation: $\scrA\,:\,\hat\bx\in\bbR^{2\ell+t}\,\to\,\bx\in\bbR^{\ell+t}$ by
$$
\bx=\scrA(\hat\bx)=\begin{bmatrix}
                     \hat\bx_{(1:\ell)}+2\Delta\,\hat\bx_{(\ell+1:2\ell)})+k\bd \\
                     \hat\bx_{(2\ell+1:2\ell+t)})
                   \end{bmatrix},
$$
where $\hat\bx_{(i:j)}$ is the sub-vector of $\hat\bx$ from its $i$th entry to $j$th entry, $\Delta=\diag(\delta_1,\ldots,\delta_{\ell})\in\bbR^{\ell\times\ell}$,
and $\bd^{\T}=[\delta_1^2,\ldots,\delta_{\ell}^2]$. Finally, $f(P)=\hat\phi\circ\what T_{2a}(P)$ where
$\hat\phi(\hat\bx)=\phi(\bx)\equiv\phi(\scrA(\hat\bx))$ is convex in $\hat\bx$ \cite[p.79]{bova:2004}. It can be verified that
$$
\hat\phi_i(\hat\bx):=\frac {\partial \hat\phi(\hat\bx)}{\partial \hat x_i}
  =
  \begin{cases}
    \phi_i(\bx), &\quad\mbox{for $1\le i\le\ell$}, \\
    2\delta_{i-\ell}\phi_{i-\ell}(\bx), &\quad\mbox{for $\ell+1\le i\le 2\ell$}, \\
    \phi_{i-\ell}(\bx), &\quad\mbox{for $2\ell+1\le i\le 2\ell+t$}.
  \end{cases}
$$
Theorem~\ref{thm:main-NEPv-cvx} is now applicable to $f(P)=\hat\phi\circ\what T_{2a}(P)$,  but requiring only
$\phi_i\ge 0$ for $1\le i\le\ell$  and without requiring any of $A_i\succeq 0$.
\end{example}

\begin{example}\label{eg:T3a}
Consider $T_{3a}$, a special case of $T_3$ in \eqref{eq:T-pow-tr},
\begin{equation}\label{eq:T-pow-tr-a}
T_{3a}\,:\, P\in\STM{k}{n} \to T_{3a}(P):=\begin{bmatrix}\tr((P^{\T}A_1P)^{m_1}) \\
                                        \vdots \\
                                        \tr((P^{\T}A_{\ell}P)^{m_{\ell}}) \\
                                        \tr((P^{\T}D)^{m_{\ell+1}})
                                      \end{bmatrix} \in\bbR^{\ell+1},
\end{equation}
where integer $m_i\ge 1$ for all $1\le i\le\ell+1$  and $D\in\bbR^{n\times k}$. It reduces to an even more special case of Example~\ref{eg:T1a-NEPv}
if all $m_i=1$ for $1\le i\le\ell+1$.
Either $\ell=0$ or without the last component $\tr((P^{\T}D)^{m_{\ell+1}})$  is allowed.
For optimization problem \eqref{eq:Master-OptSTM} with
$f=\phi\circ T_{3a}$, we will use
\begin{multline}\label{eq:H(P)-comp-T3a}
H(P)=\sum_{i=1}^{\ell}\phi_i(T_{3a}(P))\,\underbrace{2m_i\,A_i(PP^{\T}A_i)^{m_i-1}}_{=:H_i(P)} \\
     +\phi_{\ell+1}(T_{3a}(P))\, \underbrace{m_{\ell+1}\,\big[D(P^{\T}D)^{m_{\ell+1}-1}P^{\T}+P(D^{\T}P)^{m_{\ell+1}-1}D^{\T}\big]}_{=:H_{\ell+1}(P)},
\end{multline}
for which $H(P)P-\scrH(P)\equiv P\big[\phi_{\ell+1}(T_{3a}(P))\,m_{\ell+1}\,(D^{\T}P)^{m_{\ell+1}}\big]$ for $P\in\STM{k}{n}$.
Any solution $P$ to NEPv~\eqref{eq:NEPv-form} with
$H(P)$ in \eqref{eq:H(P)-comp-T3a} such that $(D^{\T}P)^{m_{\ell+1}}$ is symmetric is a solution to the KKT condition \eqref{eq:KKT}.
Suppose all $A_i\succeq 0$ and let $\bbP=\STM{k}{n}_{D+}$.
Then
the consistency conditions in \eqref{eq:cond4AF-NEPv:cvx} are satisfied with
$\munderbar\alpha=2$, $\munderbar\beta_i=2(m_i-1)$ for $1\le i\le\ell+1$, by
Theorems~\ref{thm:lin4tr-diff-eq:H(P)} and \ref{thm:quad4tr-diff-eq:H(P)}.
In applying Theorem~\ref{thm:main-NEPv-cvx} with $T_0=T_{3a}$ we need $\phi_i\ge 0$ for $1\le i\le\ell+1$
because now we are not sure if any of the inequalities in \eqref{eq:cond4AF-NEPv:cvx-b} is an equality.
Lastly, the best $Q_i$ at Line 4 of Algorithm~\ref{alg:SCF4NEPv} when applied to
$\phi\circ T_{3a}$
is an orthonormal polar factor of $[\what P^{(i)}]^{\T}D$ and vice versa.
\end{example}

\subsection{Not all $P_i$ are the entire $P$}\label{ssec:PineP}
We now consider $T(P)$ in \eqref{eq:T-general} in its generality, i.e., some of the $P_i$
do not contain all columns of $P$.
To proceed, first we attempt to construct a symmetric $H(P)\in\bbR^{n\times n}$
from $H_i(P_i)$ for each component $f_i(P_i)$ of $T(P)$.
Assume, for the moment, that
\begin{equation}\label{eq:cond:Hi(Pi)}
H_i(P_i)P_i-\frac {\partial f_i(P_i)}{\partial P_i}=P_i\scrM_i(P_i),
\end{equation}
where $\scrM_i(P_i)\in\bbR^{k_i\times k_i}$.
Recalling the expression for $\scrH(P)$ in \eqref{eq:KKT-comp-1} as a linear combination of
$\frac {\partial f_i(P_i)}{\partial P_i}J_i^{\T}$,
we have, by using \eqref{eq:cond:Hi(Pi)} and $P_i=PJ_i$,
\begin{align*}
\frac {\partial f_i(P_i)}{\partial P_i}J_i^{\T}
  &=\Big[H_i(P_i)P_i-P_i\scrM_i(P_i)\Big]J_i^{\T} \\
  &=H_i(P_i)P_iJ_i^{\T}-PJ_i\scrM_i(P_i)J_i^{\T} \\
  &=H_i(P_i)P_iJ_i^{\T}P^{\T}P-PJ_i\scrM_i(P_i)J_i^{\T} \\
  &=\Big[H_i(P_i)P_iJ_i^{\T}P^{\T}+PJ_iP_i^{\T}H_i(P_i)\Big]P-P\Big[J_iP_i^{\T}H_i(P_i)P+J_i\scrM_i(P_i)J_i^{\T}\Big],
\end{align*}
and as a result
\begin{align*}
\scrH(P):=\frac {\partial f(P)}{\partial P}
  &=\sum_{i=1}^N\phi_i(T(P))\frac {\partial f_i(P_i)}{\partial P_i}J_i^{\T} \\
  &=\underbrace{\left(\sum_{i=1}^N\phi_i(T(P))\Big[H_i(P_i)P_iJ_i^{\T}P^{\T}+PJ_iP_i^{\T}H_i(P_i)\Big]\right)}_{=:H(P)}P  \\
  &\quad-P\underbrace{\left(\sum_{i=1}^N\phi_i(T(P))\Big[J_iP_i^{\T}H_i(P_i)P+J_i\scrM_i(P_i)J_i^{\T}\Big]\right)}_{=:\scrM(P)}\,,
\end{align*}
taking the form
\begin{equation}\label{eq:H(P)-from-Hi(Pi):cond}
H(P)P-\frac {\partial f(P)}{\partial P}=P\scrM(P),
\end{equation}
where the symmetric $H(P)\in\bbR^{n\times n}$ is given by
\begin{equation}\label{eq:H(P)-comp-form:Hi(Pi)}
H(P)=\sum_{i=1}^N\phi_i(T(P))\Big[H_i(P_i)P_iP_i^{\T}+P_iP_i^{\T}H_i(P_i)\Big].
\end{equation}
This construction reminds us of the technique that was first used in \cite{zhwb:2022} for OCCA to turn the KKT condition
into an NEPv, where a single term $D$ is converted to $DP^{\T}+PD^{\T}$. The same technique was later used in
\cite{zhys:2020,wazl:2023,wazl:2022a} and in this paper too in \eqref{eq:H(P):always}.


Although $H(P)$  in \eqref{eq:H(P)-comp-form:Hi(Pi)} seems to be a good candidate to build a  framework with
as laid out in section~\ref{sec:NEPv-theory}, there is
an apparent obstacle that is hard, if at all possible, to cross over. Recall the key foundation
in the consistency conditions in \eqref{eq:cond4AF-NEPv:cvx}. For the current case, we would need similar ones, i.e.,
conditions such as
\begin{subequations}\label{eq:cond4AF-NEPv:cvx:Hi(Pi)}
\begin{align}
\tr(P_i^{\T}H_i(P_i)P_i)&=\munderbar\gamma_i\, f_i(P_i)\quad\mbox{for $P\in\bbP\subseteq\STM{k}{n}$},
              \label{eq:tr-diff-eq:cvx:Hi(Pi)} \\
\intertext{and given $\what P\in\STM{k}{n}$ and $P\in\bbP$, there exists $Q\in\STM{k}{k}$ such
that $\wtd P=\what PQ\in\bbP$ and}
\tr(\what P_i^{\T}H_i(P_i)\what P_i)
   &\le\munderbar\alpha f_i(\wtd P_i)+\munderbar\beta_i f_i(P_i),
      \label{eq:regularity4tr-diff-eq:cvx:Hi(Pi)}
\end{align}
\end{subequations}
where $\munderbar\alpha>0,\,\munderbar\beta_i\ge 0$, and
$\munderbar\gamma_i=\munderbar\alpha+\munderbar\beta_i$ are  constants.
In return, we would then relate $\tr(P^{\T}H(P)P)$ to
$\sum_i\gamma_ix_i\phi_i(\bx)$
by equality and bound
$\tr(\what P^{\T}H(P)\what P)$ from above in terms of
$\sum_i\gamma_ix_i\phi_i(\bx)$
and
$\sum_i\gamma_i\wtd x_i\phi_i(\bx)$
where $\bx\equiv [x_i]=T(P)$ and $\wtd\bx\equiv [\wtd x_i]=T(\wtd P)$. The former is rather straightforward because
$$
\tr(P^{\T}H(P)P)
  =2\sum_{i=1}^N\phi_i(T(P))\tr(P_i^{\T}H_i(P_i)P_i)
  =2\sum_{i=1}^N\gamma_ix_i\phi_i(\bx),
$$
but the latter seems to be insurmountable because
$$
\tr(\what P^{\T}H(P)\what P)
  =2\sum_{i=1}^N\phi_i(\bx)\tr(\what P^{\T}H_i(P_i)P_iP_i^{\T}\what P)
$$
and it is not clear how to bound $\tr(\what P^{\T}H_i(P_i)P_iP_i^{\T}\what P)$ from above,
given \eqref{eq:regularity4tr-diff-eq:cvx:Hi(Pi)}.
So the route via $H(P)$ in \eqref{eq:H(P)-comp-form:Hi(Pi)} is blocked. This ends our first attempt of constructing $H(P)$ from
$H_i(P_i)$ for $1\le i\le N$.

We will seek an alternative route, as our second attempt. Assume the following setting: among the $N$ components
$f_i(P_i)$,
the first $N_1$ of them involve $P_i$ that are not the entire $P$ but
the last $N_2$ do, i.e.,
\begin{equation}\label{eq:mixture-P}
P_i=P\quad\mbox{for $N_1+1\le i\le N$},
\end{equation}
where $N_1+N_2=N$.
We  resort the generic symmetric matrix-valued function in \eqref{eq:H(P):always}
for the first $N_1$ components $f_i(P_i)$ for $1\le i\le N_1$ but the individual $H_i(P)$
for the last $N_2$ components $f_i(P)$ for $N_1+1\le i\le N$ to get:
\begin{equation}\label{eq:H(P)-comp-form:Hi(Pi):alt}
H(P)=\sum_{i=1}^{N_1}\phi_i(T(P))\left(\frac {\partial f_i(P_i)}{\partial P_i}P_i^{\T}
             +P_i\left[\frac {\partial f_i(P_i)}{\partial P_i}\right]^{\T}\right)
             +\sum_{i=N_1+1}^N\phi_i(T(P))\,H_i(P).
\end{equation}
In doing so, we completely ignore $H_i(P_i)$ for $1\le i\le N_1$ that we presumably already know and should take
 advantage of but cannot.
To proceed, we will also need
the same consistency conditions as in \eqref{eq:cond4AF-NPDo:cvx} for NPDo for $f_i(P_i)$ for
$1\le i\le N_1$ but  the ones in \eqref{eq:cond4AF-NEPv:cvx:Hi(Pi)} for $f_i(P)$ for
$N_1+1\le i\le N$. This makes
the first $N_1$ components of $T(P)$ atomic functions for both NPDo and for NEPv.

In this setting $N_1=0$ or $N_2=0$ are allowed. The case $N_1=0$ is the one we have already dealt with
in subsection~\ref{ssec:PieqP}, and for the case $N_2=0$, the NPDo approach can
be applied in the first place.
It remains to conquer the case both $N_1,\,N_2\ge 1$. Our next theorem will help us to do that.

\begin{theorem}\label{thm:main-NEPv-cvx:Hi(Pi)}
Consider $f=\phi\circ T$, where $T(\cdot)$ takes the form in \eqref{eq:T-general}
with \eqref{eq:mixture-P} and $\phi$
is convex and differentiable with partial derivatives denoted by
$\phi_i$ as in \eqref{eq:phi-i}. Let $H(P)$ be given by \eqref{eq:H(P)-comp-form:Hi(Pi):alt}.
Suppose that
\begin{enumerate}[{\rm (i)}]
  \item the conditions in \eqref{eq:cond4AF-NPDo:cvx} hold for $1\le i\le N_1$,
  \item the conditions in \eqref{eq:cond4AF-NEPv:cvx:Hi(Pi)} hold with $Q=I_k$ for $N_1+1\le i\le N$,
  \item $\phi_i(\bx)\ge 0$ for those $1\le i\le N_1$
        for which \eqref{eq:cond4AF-NPDo:cvx-b} does not become an equality
        and for those $N_1+1\le i\le N$
        for which \eqref{eq:regularity4tr-diff-eq:cvx:Hi(Pi)} does not become an equality,
  \item $\munderbar\alpha=2\alpha$ for $\alpha$ in \eqref{eq:cond4AF-NPDo:cvx-b}
        and $\munderbar\alpha$ in \eqref{eq:regularity4tr-diff-eq:cvx:Hi(Pi)}.
\end{enumerate}
Then {\bf the NEPv Ansatz} holds  with $\omega=1/(2\alpha)$ and with the $Q$-matrix
specified in the proof.
\end{theorem}

\begin{proof}
Given $\what P\in\STM{k}{n}$ and $P\in\bbP\subseteq\STM{k}{n}$,
suppose that \eqref{eq:NEPv-assume} holds, i.e.,
$\tr(\what P^{\T}H(P)\what P)\ge\tr(P^{\T}H(P)P)+\eta$.
Write
$$
\bx=T(P)\equiv [x_1,x_2,\ldots,x_N]^{\T}, \quad
\wtd\bx=T(\wtd P)\equiv [\wtd x_1,\wtd x_2,\ldots,\wtd x_N]^{\T},
$$
where $x_i=f_i(P_i)$, and $\wtd x_i=f_i(\wtd P_i)$ with $\wtd P$ to be defined later in \eqref{eq:main-NEPv-cvx:Hi(Pi):pf-1}.
Write
$\scrH_i(P_i)=\frac {\partial f_i(P_i)}{\partial P_i}$.
For $H(P)$  in \eqref{eq:H(P)-comp-form:Hi(Pi):alt}, we have by \eqref{eq:cond4AF-NPDo:cvx}
and \eqref{eq:cond4AF-NEPv:cvx:Hi(Pi)}
\begin{align*}
\tr(P^{\T}H(P)P)
  &=2\sum_{i=1}^{N_1}\phi_i(\bx)\tr(P_i^{\T}\scrH_i(P_i))+\sum_{i=N_1+1}^N\phi_i(\bx)\tr(P^{\T}H_i(P)P) \\
  &=2\sum_{i=1}^{N_1}\gamma_i\phi_i(\bx)\,x_i
    +\sum_{i=N_1+1}^N\munderbar\gamma_i\phi_i(\bx)\, x_i,
          \qquad\mbox{(by \eqref{eq:cond4AF-NPDo:cvx-a} and \eqref{eq:tr-diff-eq:cvx:Hi(Pi)})}.
\end{align*}
Let $W_1\in\STM{k}{k}$ be an orthonormal polar factor of $\what P^{\T}\munderbar\scrH(P)$ and $Z=\what PW_1$,
where
$$
\munderbar\scrH(P):=\sum_{i=1}^{N_1}\phi_i(T(P))\frac {\partial f_i(P_i)}{\partial P_i}J_i^{\T}.
$$
We have
\begin{align*}
\tr(\what P^{\T}H(P)\what P)
  &=2\sum_{i=1}^{N_1}\phi_i(\bx)\tr(\what P^{\T}\scrH_i(P_i)P_i^{\T}\what P)
    +\sum_{i=N_1+1}^N\phi_i(\bx)\,\tr(\what P^{\T}H_i(P)\what P)  \\
  &=2\sum_{i=1}^{N_1}\phi_i(\bx)\tr(\what P^{\T}\scrH_i(P_i)J_i^{\T}P^{\T}\what P)
    +\sum_{i=N_1+1}^N\phi_i(\bx)\,\tr(\what P^{\T}H_i(P)\what P)  \\
  &=2\tr(\what P^{\T}\munderbar\scrH(P)P^{\T}\what P)
    +\sum_{i=N_1+1}^N\phi_i(\bx)\,\tr(\what P^{\T}H_i(P)\what P) \\
  &\le 2\|\what P^{\T}\munderbar\scrH(P)\|_{\tr}\|P^{\T}\what P\|_2
    +\sum_{i=N_1+1}^N\phi_i(\bx)\,\tr(\what P^{\T}H_i(P)\what P)\\
  &\le 2\|\what P^{\T}\munderbar\scrH(P)\|_{\tr}
    +\sum_{i=N_1+1}^N\phi_i(\bx)\,\tr(\what P^{\T}H_i(P)\what P)\\
  &=2\tr(Z^{\T}\munderbar\scrH(P))
    +\sum_{i=N_1+1}^N\phi_i(\bx)\,\tr(Z^{\T}H_i(P)Z) \\
  &=2\sum_{i=1}^{N_1}\phi_i(\bx)\tr(Z_i^{\T}\scrH_i(P_i))
    +\sum_{i=N_1+1}^N\phi_i(\bx)\,\tr(Z^{\T}H_i(P)Z),
\end{align*}
where $Z_i=ZJ_i$ for $1\le i\le N$ are submatrices of $Z$.
Now use the second consistency condition in \eqref{eq:cond4AF-NPDo:cvx-b} to get $W_2\in\STM{k}{k}$ and set
\begin{equation}\label{eq:main-NEPv-cvx:Hi(Pi):pf-1}
\wtd P=ZW_2=\what P(W_1W_2)=:\what PQ.
\end{equation}
Then
\begin{align*}
\tr(\what P^{\T}H(P)\what P)
  &\le2\sum_{i=1}^{N_1}\phi_i(\bx)\tr(Z_i^{\T}\scrH_i(P_i))
    +\sum_{i=N_1+1}^N\phi_i(\bx)\,\tr(\wtd P^{\T}H_i(P)\wtd P) \\
  &\le2\sum_{i=1}^{N_1}\phi_i(\bx)\,(\alpha\wtd x_i+\beta_i x_i)
    +\sum_{i=N_1+1}^N\phi_i(\bx)\,(\munderbar\alpha\wtd x_i+\munderbar\beta_i x_i)\\
  &\le2\sum_{i=1}^{N_1}\phi_i(\bx)\,(\alpha\wtd x_i+\beta_i x_i)
    +\sum_{i=N_1+1}^N\phi_i(\bx)\,(2\alpha\wtd x_i+\munderbar\beta_i x_i).
\end{align*}
Plug them into $\eta+\tr(P^{\T}H(P)P)\le\tr(\what P^{\T}H(P)\what P)$ and simplify the resulting inequality with
the help of $\gamma_i=\alpha+\beta_i$ for $1\le i\le N_1$ and $\munderbar\gamma_i=2\alpha+\munderbar\beta_i$
for $N_1+1\le i\le N$ to get
$$
\eta/(2\alpha)+\nabla\phi(\bx)^{\T}\bx=\eta/(2\alpha)+\sum_{i=1}^N\phi_i(\bx)\,x_i\le\sum_{i=1}^N\phi_i(\bx)\,\wtd x_i
   =\nabla\phi(\bx)^{\T}\wtd\bx.
$$
Finally apply Lemma~\ref{lm:convex-mono} to yield $f(\wtd P)\ge f(P)+\eta/(2\alpha)$.
\end{proof}

We now explain why  the resulting Algorithm~\ref{alg:SCF4NEPv} is not going to be competitive to
\Cref{alg:SCF4NPDo} per SCF iterative step in the case $N_2=0$. First
both algorithms work due to the same set of consistency conditions in \eqref{eq:cond4AF-NPDo:cvx}. Next
carefully examining the proof of Theorem~\ref{thm:main-NEPv-cvx:Hi(Pi)}, we find that
$Q_i$ at Line 4 of Algorithm~\ref{alg:SCF4NEPv} is the product of two $k\times k$ orthogonal matrices,
$W_1$ from an orthonormal polar factor of $[\what P^{(i)}]^{\T}\scrH(P^{(i)})$ and $W_2$ from the second consistency condition
in \eqref{eq:cond4AF-NPDo:cvx-b}. \Cref{alg:SCF4NPDo} also needs $Q_i$ at its line 4 but just one
orthogonal matrix dictated by the second consistency condition
in \eqref{eq:cond4AF-NPDo:cvx},
making its SCF step cheaper, not to mention
there is an additional partial eigendecomposition to compute
at Line 3 of Algorithm~\ref{alg:SCF4NEPv}. Having said that, the significance of Theorem~\ref{thm:main-NEPv-cvx:Hi(Pi)}
lies in the mixture case: both $N_1\ge 1$ and $N_2\ge 1$, for which the theorem ensures that
Algorithm~\ref{alg:SCF4NEPv} can still be applied with guaranteed convergence. The next example falls into such a category.

\begin{example}\label{eg:T4}
Consider $T_4$ that shares some similarity with $T_2$ in \eqref{eq:T-F} and $T_{2a}$ in \eqref{eq:T-F-a}:
\begin{equation}\label{eq:T-F-b}
T_4\,:\, P\in\STM{k}{n} \to T_4(P):=\begin{bmatrix}
                                        \|P_1^{\T}A_1P_1\|_{\F}^2 \\
                                        \vdots \\
                                        \|P_{\ell}^{\T}A_{\ell}P_{\ell}\|_{\F}^2 \\
                                        \tr(P^{\T}B_1P) \\
                                        \vdots \\
                                        \tr(P^{\T}B_tP)
                                      \end{bmatrix} \in\bbR^{\ell+t},
\end{equation}
where $A_i\succeq 0$ for $1\le i\le \ell$ and $B_i\in\bbR^{n\times n}$ for $1\le i\le t$ are symmetric.
Each $\|P_i^{\T}A_iP_i\|_{\F}^2=\tr\big(\big[P_i^{\T}A_iP_i\big]^2\big)$ is an atomic function of NPDo by Theorem~\ref{thm:quad4tr-diff-eq},
and each $\tr(P^{\T}B_iP)$ is an atomic function of NEPv (by  Theorem~\ref{thm:quad4tr-diff-eq:H(P)})
but may not be an atomic function of NPDo unless $B_i\succeq 0$ also.
We can be do away with $B_i\not\succeq 0$ by shifting $B_i$ to $B_i-\delta I\succeq 0$ as we did in
the second part of Example~\ref{eg:T2a}, but the shift works only if the corresponding
partial derivative $\phi_{\ell+i}\ge 0$.
If we do not have that, then the NPDo approach is not guaranteed to work even with the shifting technique.
With Theorem~\ref{thm:main-NEPv-cvx:Hi(Pi)}, however,
we can make the NEPv approach work with
\begin{align*}
H(P)&=\sum_{i=1}^{\ell}\phi_i(T(P))\left(\frac {\partial \tr\big(\big[P_i^{\T}A_iP_i\big]^2\big)}{\partial P_i}P_i^{\T}
             +P_i\left[\frac {\partial \tr\big(\big[P_i^{\T}A_iP_i\big]^2\big)}{\partial P_i}\right]^{\T}\right) \\
    &\quad         +\sum_{i=1}^t\phi_{\ell+i}(T(P))\,2B_i \\
    &=\sum_{i=1}^{\ell}\phi_i(T(P))\,4\Big(A_iP_iP_i^{\T}AP_iP_i^{\T}+P_iP_i^{\T}A_iP_iP_i^{\T}A_i\Big)
     +\sum_{i=1}^t\phi_{\ell+i}(T(P))\,2B_i,
\end{align*}
assuming $\phi_i\ge 0$ for $1\le i\le\ell$ but no requirement  to impose on $\phi_{\ell+i}$ for $1\le i\le t$ is necessary,
because all $\tr\big(\big[P_i^{\T}A_iP_i\big]^2\big)$ share the same
$\alpha=1$ and $\beta=3$ in \eqref{eq:cond4AF-NPDo:cvx-b}
while all $\tr(P^{\T}B_iP)$ share the same $\munderbar\alpha=2$ and $\munderbar\beta=0$ in \eqref{eq:regularity4tr-diff-eq:cvx:Hi(Pi)}
which is also an equality for the case.
\end{example}

\begin{remark}\label{rk:app-NEPv-cvx}
We conclude section~\ref{sec:CVX-comp:NEPv} by commenting on the applicability of the results of this section
to the objective functions
in Table~\ref{tbl:obj-funs} via convex compositions of atomic functions for NEPv.
Essentially our results are applicable to
all but OLDA and SumTR, for which the corresponding composing functions $\phi$ are
$x_2/x_1$ and $x_2/x_1+x_3$, respectively. Both are non-convex, and yet {\bf the NEPv Ansatz} still holds for
OLDA but does not for SumTR.
With the generic $H(P)$ as in \eqref{eq:H(P):always}, SumCT
can be handled too through the convex composition of atomic functions but the resulting NEPv approach may not be competitive to the NPDo approach,
as we have argued in subsection~\ref{ssec:PineP}.
The composing function for OCCA is $x_2/\sqrt{x_1}$, which is not convex but whose square $x_2^2/x_1$ is convex for $x_2\ge 0$ and $x_1>0$.
%
%
For $\Theta$TR, the theory in subsection~\ref{ssec:PieqP} can only handle $0\le\theta\le 1/2$
and also on the objective function squared, however:
\begin{equation}\label{eq:ThetaTR:cvx2}
[f(P)]^2=\phi\circ T(P)
\quad\mbox{with}\,\, T(P)=\begin{bmatrix}
                             \tr(P^{\T}BP) \\
                             \tr(P^{\T}AP) \\
                             \tr(P^{\T}D)
                           \end{bmatrix}, \,\,
\phi(\bx)
    =\frac {(x_2+x_3)^2}{x_1^{2\theta}},
\end{equation}
where $\bx\equiv [x_1,x_2,x_3]^{\T}$.
This $T$ has the form of $T_{1a}$  of Example~\ref{eg:T1a-NEPv} and it can be verified that the associated symmetric matrix-valued
function $H(P)$ by \eqref{eq:H(P)-comp-T1a} differs from the one in \eqref{eq:H(P):theta-TR} \cite{wazl:2023} by a scalar factor only.
We claim that $\phi$ is convex for $x_1>0$ and $x_2+x_3\ge 0$, provided $0\le\theta\le 1/2$,
and, since also $\phi_3(\bx):=\partial\phi(\bx)/\partial x_3\ge 0$ for $x_1>0$ and $x_2+x_3\ge 0$, Theorem~\ref{thm:main-NEPv-cvx:T1a} applies.
We note that
$\phi_0(x,y)= {y^2}/{x^{2\theta}}$
for
$x>0$ and $y\ge 0$
is convex if $0\le\theta\le 1/2$ but is not convex if $\theta>1/2$. In fact, the Hessian matrix of $\phi_0$ is
$$
\begin{bmatrix}
2\theta(2\theta+1) {y^2}/{x^{2\theta+2}} & -4\theta y/{x^{2\theta+1}} \\
-4\theta y/{x^{2\theta+1}} &  2/{x^{2\theta}}
\end{bmatrix}
=\frac 2{x^{2\theta}}
 \begin{bmatrix}
    y/x &  \\
    & 1
 \end{bmatrix}
 \begin{bmatrix}
   \theta(2\theta+1) & -2\theta \\
   -2\theta & 1
 \end{bmatrix}
 \begin{bmatrix}
    y/x &  \\
    & 1
 \end{bmatrix}
$$
which is positive semidefinite for $x>0$ and $y\ge 0$ if and only if $0\le\theta\le 1/2$.
This implies
$\phi(\bx)\equiv\phi_0(x_1,x_2+x_3)$ is convex for $x_1>0$ and $x_2+x_3\ge 0$, provided  $0\le\theta\le 1/2$.
However, the results in
\cite{wazl:2023} says that the NEPv approach works on $f$ for $\Theta$TR for $0\le\theta\le 1$,
much bigger range for $\theta$
than what Theorem~\ref{thm:main-NEPv-cvx:T1a} for $f^2$ implies.
Theorem~\ref{thm:main-NEPv-cvx:T1a} also yields an inequality on how much the  objective function squared,
$f^2$, increases, in contrast to the previous \eqref{eq:f-inc-theta-TR:refined} which is for the original objective function, $f$, for $0\le\theta\le 1$ and OCCA.
In any case, the best $Q_i$ at Line 4 of Algorithm~\ref{alg:SCF4NEPv} when applied to $\Theta$TR
is an orthonormal polar factor of $[\what P^{(i)}]^{\T}D$ and $\bbP=\STM{k}{n}_{D+}$.
\end{remark}

\section{A Brief Comparison of the NPDo and NEPv Approaches}\label{sec:NPDo-vs-NEPv}
The developments of the frameworks for both NPDo and NEPv follow the same pattern: an ansatz that implies
the global convergence of their respective SCF iterations, the definition of atomic functions for an approach, and the fulfillment of
the ansatz by the atomic functions for the approach and their convex compositions. Subtly, between the two approaches, there are differences in applicabilities
and numerical implementations, making them somewhat complementary to each other.
%
\begin{table}[t]
\renewcommand{\arraystretch}{1.4}
\caption{\small NPDo {\em vs.} NEPv on three convex compositions of matrix traces}\label{tbl:NPDo-vs-NEPv}
\centerline{\small
\begin{tabular}{|c|c|c|c|c|}
  \hline
\multirow{2}{*}{$f$}    & \multicolumn{2}{c|}{NPDo} & \multicolumn{2}{c|}{NEPv} \\ \cline{2-5}
    & conditions &  by & conditions & by \\ \hline\hline
\multirow{2}{*}{$\phi\circ T_{1a}$} & $A_i\succeq 0,\,1\le i\le\ell$, & \multirow{2}{*}{Thm.~\ref{thm:main-npd-cvx:T1a}} &
        \multirow{ 2}{*}{none} & \multirow{ 2}{*}{Thm.~\ref{thm:main-NEPv-cvx:T1a}}  \\
      & $\phi_j\ge 0,\,1\le j\le\ell$ &  &
         &   \\ \hline
\multirow{2}{*}{$\phi\circ T_{2a}$} & $A_i\succeq 0,\,1\le i\le\ell$, & \multirow{2}{*}{Expl.~\ref{eg:T2}} &
        \multirow{2}{*}{$\phi_j\ge 0,\,1\le j\le\ell$} & \multirow{2}{*}{Expl.~\ref{eg:T2a}}  \\
                   & $\phi_j\ge 0,\,1\le j\le\ell+t$ &  &
         &   \\ \hline
\multirow{2}{*}{$\phi\circ T_{3a}$} & $A_i\succeq 0,\,1\le i\le\ell$, & \multirow{2}{*}{Expl.~\ref{eg:T3}} &
        $A_i\succeq 0,\,1\le i\le\ell$, & \multirow{2}{*}{Expl.~\ref{eg:T3a}}  \\
                   & $\phi_j\ge 0,\,1\le j\le\ell+1$ &  &
        $\phi_j\ge 0,\,1\le j\le\ell+1$ &   \\
        \hline
\multirow{2}{*}{$\phi\circ T_4$} & $A_i,\,B_j\succeq 0,\,\forall i, j$ & \multirow{2}{*}{Thm.~\ref{thm:main-npd-cvx}} &
        $A_i\succeq 0,\,1\le i\le\ell$, & \multirow{2}{*}{Expl.~\ref{eg:T4}}  \\
                   & $\phi_j\ge 0,\,1\le j\le\ell+t$ &  &
        $\phi_j\ge 0,\,1\le j\le\ell$ &   \\
        \hline
\multicolumn{5}{l}{\small * $\phi_j(\bx):=\partial\phi(\bx)/\partial x_j$  for $\bx=[x_j]$.} \\
\end{tabular}
}
\end{table}
We outline some notable differences below.


\smallskip
{\bf The NEPv approach requires weaker conditions.}
We have
demonstrated that {\bf the NDPo Ansatz} demands more on an objective function than {\bf the NEPv Ansatz}, and hence
the NEPv approach provably works on a wider collection of maximization problems \eqref{eq:main-opt} on the Stiefel manifold
than the NPDo approach.
\begin{enumerate}[(i)]
  \item The NPDo approach requires that the KKT
        condition $\scrH(P)\equiv{\partial f(P)}/{\partial P}=P\Lambda$
        is a polar decomposition at optimality, in order for the approach to be even considered,
        whereas the NEPv approach does not impose that
        the KKT condition must be an NPDo at optimality;
  \item {\bf The NPDo Ansatz} implies {\bf the NEPv Ansatz} with the generic symmetric matrix-valued function $H(P)$ in \eqref{eq:H(P):always} (see Theorem~\ref{thm:NPDo>=NEPv});
  \item Atomic functions for NPDo are also atomic functions for NEPv  with the generic symmetric matrix-valued function $H(P)$ in \eqref{eq:H(P):always} (see Theorem~\ref{thm:scrH2H:tr-diff-eq});
  \item Among those in Table~\ref{tbl:obj-funs}, the NEPv approach is guaranteed to work for three more
        than  the NPDo approach does (Table~\ref{tbl:obj-funs:NPD} {\em vs.} Table~\ref{tbl:obj-funs:NEPv});
  \item When it comes to the concrete atomic functions $[\tr\big((P^{\T}AP)^m\big)]^s$
        where integer $m\ge 1$ and scalar $s\ge 1$,
        it is required that $A\succeq 0$ always for the NPDo approach, whereas for the NEPv approach $A$ being symmetric
        suffices for $m\in\{1,2\}$ and $s=1$ (see Examples~\ref{eg:T1a-NEPv} and~\ref{eg:T2a});
  \item As a further demonstration, in Table~\ref{tbl:NPDo-vs-NEPv}, we summarize
        what are required by both approaches on four convex compositions of matrix-trace functions, and
        it clearly indicates that NEPv requires weaker conditions than  NPDo does for
        $\phi\circ T_{1a}$, $\phi\circ T_{2a}$, and $\phi\circ T_4$.
\end{enumerate}

\smallskip
{\bf The NPDo approach is easier to use.}
The NPDo approach, if it provably works, is easier to implement and more flexible
to use.
\begin{enumerate}[(a)]
  \item The NPDo approach relies on
        SVDs of tall and skinny matrices \cite{demm:1997,govl:2013} during its SCF
        iterations, whereas the NEPv approach needs solutions to potentially large scale eigenvalue problems \cite{bddrv:2000,lesy:1998,li:2015,parl:1998,saad:1992}, not to mention that the NEPv approach requires constructing a symmetric matrix-valued function $H(P)$.
  \item In general for the NPDo approach,
        the atomic functions in a convex composition are allowed to be of submatrices of $P$
        consisting of a few, not necessarily all, columns of $P$, such as $P_i$ in
        $\tr(P_i^{\T}A_iP_i)$ of SumCT in Table~\ref{tbl:obj-funs}, but, more generally,
        different $P_i$ can share common columns of $P$, whereas no $P_i$ in SumCT share common columns of $P$;
        For the NEPv approach, allowing such flexibility in the involved atomic functions in the convex composition
        forces us to use  the generic symmetric matrix-valued function $H(P)$ in \eqref{eq:H(P):always}
        and ask for the same conditions as the NPDo approach requires, and hence we may as well go for
        the NPDo approach in the first place,
        as we have argued in subsection~\ref{ssec:PineP}.
\end{enumerate}

\section{Concluding Remarks}\label{sec:conclu}
The first order optimality condition, also known as the KKT condition, for optimizing a function $f$ over the Stiefel manifold takes the form
\begin{equation}\label{eq:KKT'}
\scrH(P):=\frac{\partial f(P)}{\partial P}=P\Lambda
\quad \mbox{with} \quad
\Lambda^{\T}=\Lambda\in\bbR^{k\times k},\quad P^{\T}P=I_k.
\end{equation}
This is an $n\times k$ matrix equation in $P$ on the Stiefel manifold, upon noticing $\Lambda=[P^{\T}\scrH(P)+\scrH(P)^{\T}P]/2$.
Any maximizer is a solution. Except for very special objective functions such as $\tr(P^{\T}AP)$ or $\tr(P^{\T}D)$, solving this equation rightly
and efficiently for a maximizer is a challenging task. For example, it often has infinitely many solutions and maximizers
hide among them. Hence we  need to not only solve the nonlinear equation of the KKT condition but also  find the right ones.
Inspired by recent works \cite{wazl:2023,wazl:2022a,zhys:2020,zhli:2014a,zhli:2014b,zhwb:2022}, in this paper,
we establish two unifying frameworks, one for the NEPv approach and the other for the NPDo approach, for optimization on the Stiefel manifold.
Our frameworks are built upon two fundamental ansatzes,
{\bf the NPDo Ansatz} and {\bf the NEPv Ansatz}. When a respective ansatz is
satisfied, the corresponding approach, the NPDo or NEPv approach, is guaranteed to work
in the sense of  global convergence from any given initial point.
To expand the applicability of the approaches, we propose the theories of atomic functions for each approach
and show that any convex composition $\phi\circ T$ of atomic functions for any of the two approaches satisfies the corresponding ansatz under some mild conditions.
It is demonstrated that the
commonly used
matrix-trace functions
$$
\framebox{
\parbox{10cm}{
$[\tr\big((P^{\T}D)^m\big)]^s$,
$[\tr\big((P^{\T}AP)^m\big)]^s$,
and $[\tr\big((\Gamma P^{\T}AP)^m\big)]^s$ for integer $m\ge 1$ and scalar $s\ge 1$
}}
$$
are concrete atomic functions for both approaches ($A$ may be required to be positive semidefinite depending on circumstances
and $\Gamma$ is always positive semidefinite),
and that nearly all optimization problems on the Stiefel manifold recently investigated in the literature
for various machine learning applications are about some compositions of these concrete matrix-trace functions. Although not all of them are convex compositions, some may still satisfy {\bf the NEPv Ansatz}.
These concrete atomic functions in combination with convex compositions lead to a large collection of objective functions on which
one of or both  approaches work.

%
%
%
%


However when an ansatz fails to hold for a given objective function, the conclusions from
our global convergence analysis will likely fail. In the case for
the NEPv approach, a remedy via the level-shifting SCF exists to ensure locally linear convergence
when $f(P)$ is right-unitarily
invariant \cite{ball:2022} or when $f(P)$ contains and increases with $\tr(P^{\T}D)$ \cite{luli:2024},
where sharp estimations of linear convergence rate are obtained.
But there are more works to do, especially for
the NPDo approach for which a remedy remains to be found. In \cite{teli:2025}, it is investigated how
$\tr(P^{\T}D)$ may help determine a particular $P$ among all orthonormal basis matrices of a subspace.

Not all objective functions (or their simple transformations like $f^2$ for $\Theta$TR as explained in Remark~\ref{rk:app-NEPv-cvx}) that satisfy
the ansatz(es) take the form of some convex compositions of atomic functions, even for some of those in Table~\ref{tbl:obj-funs}.
For example, in the case of OLDA, $f(P)=\phi\circ T(P)$ where $\phi(\bx)=x_1/x_2$ and
$T(P)=\begin{bmatrix}
        \tr(P^{\T}AP)\\
        \tr(P^{\T}BP)
      \end{bmatrix}$.
Although this $f(P)$ is not a convex composition of $\tr(P^{\T}AP)$ and $\tr(P^{\T}BP)$, it still
satisfies {\bf the NEPv Ansatz}. In fact, OLDA is a special case of $\Theta$TR: $\theta=1$ and $D=0$, and
inequality \eqref{eq:f-inc-theta-TR:refined} applies and yields, for OLDA,
$$
f(\wtd P) \ge f(P)+\frac 12\cdot\frac {\tr_{\min,k}(B)}{\tr_{\max,k}(B)}\Big[\tr(\what P^{\T}H(P)\what P)-\tr(P^{\T}H(P)P)\Big],
$$
assuming $B\succeq 0$ and $\rank(B)>n-k$,
where $\wtd P=\what P$ and $H(P)$ is given by \eqref{eq:H(P):theta-TR}
upon setting $\theta=1$ and $D=0$.

Numerical demonstrations on the performances of the NEPv approach and the NPDo approach on various machine learning applications
have been well documented   by the author and his collaborators in
their recent works \cite{wazl:2023,wazl:2022a,zhys:2020,zhwb:2022,wawl:2025,wazl:2026}.
We expect more to come as the unifying frameworks in this paper have significantly expanded the domains on which the approaches provably work.

Throughout the article, we limit ourselves to the Stiefel manifold in the standard inner product
$\langle \bx,\by\rangle=\by^{\T}\bx$ for $\bx,\,\by\in\bbR^n$ and to the field of real numbers, because that is where
most optimization on the Stiefel manifold from machine learning dominantly falls into. But our developments can be extended to
the field of complex numbers and other inner-products, separately and combined.
For the case of the field of complex numbers, minor modifications suffice: replacing
        transpose $(\cdot)^{\T}$ with conjugate transpose $(\cdot)^{\HH}$ and  $\tr\big((P^{\T}D)^m\big)$ with $\Re(\tr((P^{\HH}D)^m))$
        (where $\Re(\cdot)$ takes the real part of a complex number).
However, extending to the case of the $M$-inner product requires additional algebraic manipulations and analysis. An outline on how to proceed is explained
        in appendix~\ref{sec:M-IN-Prod}.

\clearpage
\appendix

\section{Canonical Angles}\label{sec:angle-space}
We introduce a metric on
Grassmann manifold
$\scrG_k(\bbR^n)$, the collection of all $k$-dimensional subspaces in $\bbR^n$.
Let
$\cX=\cR(X)$ and $\cY=\cR(Y)$ be two points in $\scrG_k(\bbR^n)$, where $X,\,Y\in\STM{k}{n}$. The canonical angles
$\theta_1(\cX,\cY)\ge\cdots\ge\theta_k(\cX,\cY)$ between $\cX$ and $\cY$ are defined by \cite{stsu:1990}
$$
	0\le\theta_i(\cX,\cY):=\arccos \sigma_i(X^{\T}Y)\le\frac {\pi}2 \quad\mbox{for $1\le i\le k$},
$$
and accordingly,
the diagonal matrix of the canonical angles between $\cX$ and $\cY$ is
	$$
	\Theta(\cX,\cY)=\diag(\theta_1(\cX,\cY),\dots,\theta_k(\cX,\cY))\in\bbR^{k\times k}.
	$$
It is known that
	\begin{align}
	\dist_2(\cX,\cY)&:=\|\sin\Theta(\cX,\cY)\|_2=\sin\theta_1(\cX,\cY), \label{eq:sinTheta} \\
    \dist_{\F}(\cX,\cY)&:=\|\sin\Theta(\cX,\cY)\|_{\F}=\Big[\sum_{i=1}^k\sin^2\theta_i(\cX,\cY)\Big]^{1/2} \label{eq:sinTheta:F}
	\end{align}
are two unitarily invariant metrics on the
Grassmann manifold
$\scrG_k(\bbR^n)$  \cite[p.99]{sun:2001}.

The orthonormal basis matrix $X$ of $\cX$ is not unique, and neither is $Y$ of $\cY$. For that reason, it is not expected
that $\|X-Y\|_{\F}$ be in the order of $\Theta(\cX,\cY)$. In particular, more than likely $\|X-Y\|_{\F}>0$ even in the case
$\cX=\cY$ unless the basis matrices $X$ and $Y$ are judicially chosen. We now explain how to align $X$ with $Y$
according to $\Theta(\cX,\cY)$, namely replace $X$ with $\wtd X:=XQ_*$ where
\begin{equation}\label{eq:basis-align-Q}
Q_*=\arg\min_{Q\in\STM{k}{n}}\|XQ-Y\|_{\F}^2.
\end{equation}
Notice that for any $Q\in\STM{k}{n}$
\begin{equation}\label{eq:XQ-Y}
\|XQ-Y\|_{\F}^2=\tr([XQ-Y]^{\T}[XQ-Y])=2k-2\Re(Q^{\T}X^{\T}Y),
\end{equation}
where $\Re(\cdot)$ extract the real part of a complex number.
The last quantity in \eqref{eq:XQ-Y} is minimized by the orthonormal polar factor of $X^{\T}Y$, called it $Q_*$, at which
\begin{equation}\label{eq:basis-align}
\|\wtd X-Y\|_{\F}^2=2\sum_{i=1}^k(1-\cos\theta_i)=4\sum_{i=1}^k\sin^2(\theta_i/2)=4\Big\|\sin\frac 12\Theta(\cX,\cY)\Big\|_{\F}^2,
\end{equation}
yielding $\|\wtd X-Y\|_{\F}=2\|\sin(\Theta(\cX,\cY)/2)\|_{\F}$, where we write $\theta_i\equiv \theta_i(\cX,\cY)$ for short.
Hence the new orthonormal basis matrix $\wtd X$ of $\cX$ is aligned with $Y$ of $\cY$ well enough so that
$X=Y$ if $\Theta(\cX,\cY)=0$. As a consequence of \eqref{eq:basis-align-Q} -- \eqref{eq:basis-align}, we have always
\begin{equation}\label{eq:basis-not-aligned}
\|X-Y\|_{\F}\ge 2\Big\|\sin\frac 12\Theta(\cX,\cY)\Big\|_{\F}.
\end{equation}

In passing, we point out that it is possible to extend the notion of canonical angles for two subspaces of non-equal dimension
\cite{lizh:2015}. It is beyond what we will need in this paper.

\section{Preliminary Lemmas}\label{sec:prelim}
In this section, we collect a few results, some known and some likely new, that we need in our proofs.
We will point out an earliest reference, if known, to each one. Some likely appear before but we are not aware of
any reference to.
For completeness, we will provide proofs for those we cannot find references.

\begin{lemma}[Young's Inequality]\label{lm:YoungIneq}
Given $a,\,b\ge 0$, and $p,\,q\ge 1$ such that $\frac 1p+\frac 1q=1$, we have
$$
a^{1/p}b^{1/q}\le\frac 1p\,a+\frac 1q\,b.
$$
In particular for $p=q=2$, it becomes $2\sqrt{ab}\le a+b$.
\end{lemma}

\begin{lemma}\label{lm:YoungIneq-ext}
For $a,\,b\ge 0$, and $\mu,\,\nu\ge 0$ such that $\mu+\nu\ge 1$,  we have
$$
a^{\mu}b^{\nu}\le\frac {\mu}{\mu+\nu}\, a^{\mu+\nu}+\frac {\nu}{\mu+\nu}\, b^{\mu+\nu}.
$$
\end{lemma}

\begin{proof}
Let $\tau:=\mu+\nu\ge 1$. Then $\mu/\tau+\nu/\tau=1$. Using Young's Inequality, we get
$$
a^{\mu}b^{\nu}=\left(a^{\mu/\tau}b^{\nu/\tau}\right)^{\tau}
   \le\left(\frac {\mu}{\tau}\, a+\frac {\nu}{\tau}\, b\right)^{\tau}
   \le\frac {\mu}{\tau}\, a^{\tau}+\frac {\nu}{\tau}\, b^{\tau},
$$
as was to be shown, where the last inequality follows from that fact that $x^{\tau}$ with $\tau\ge 1$ is convex on $[0,\infty)$.
\end{proof}

\begin{lemma}\label{lm:convex-mono}
Let $\phi\,:\,\mathfrak{D}\subseteq\bbR^N\to\bbR$ be convex and differentiable and let $\bx,\,\wtd\bx\in\mathfrak{D}$,
where $\mathfrak{D}$ is a convex domain in $\bbR^N$.
If
$$
\nabla\phi(\bx)^{\T}\wtd\bx\ge\nabla\phi(\bx)^{\T}\bx+\eta,
$$
then $\phi(\wtd \bx)\ge\phi(\bx)+\eta$.
\end{lemma}

\begin{proof}
Likely, the result of this lemma is  known, and it has a short proof.
Since $\phi$ is convex, we have
\begin{align*}
\phi(\wtd\bx)&\ge\phi(\bx)+[\nabla\phi(\bx)]^{\T}(\wtd \bx-\bx) \\
   &=[\nabla\phi(\bx)]^{\T}\wtd \bx+\phi(\bx)-[\nabla\phi(\bx)]^{\T}\bx \\
   &\ge [\nabla\phi(\bx)]^{\T}\bx+\eta+\phi(\bx)-[\nabla\phi(\bx)]^{\T}\bx \\
   &=\phi(\bx)+\eta,
\end{align*}
as was to be shown.
\end{proof}

\begin{lemma}[von Neumann's trace inequality \cite{neum:1937}, {\cite[p.183]{hojo:1991}}]\label{lm:vN-tr-ineq}
For $B,\,C\in\bbR^{n\times k}$, we have
$$
|\tr(B^{\T}C)|\le\sum_{i=1}^k\sigma_i(B)\,\sigma_i(C).
$$
\end{lemma}

In the next four lemmas, any arbitrary nonnegative power of a positive semidefinite matrix $B$ is understood as
$B^{\mu}=U\Lambda^{\mu} U^{\T}$ for any $\mu\ge 0$, where $B=U\Lambda U^{\T}$ is the eigendecomposition of $B$,
and $\Lambda^{\mu}$ is obtained by taking the $\mu$th power of every diagonal entry of $\Lambda$.

\begin{lemma}\label{lm:vN-tr-ineq-ext}
For $B,\,C\in\bbR^{k\times k}$ that are positive semidefinite, and $\mu,\,\nu\ge 0$ such that $\mu+\nu\ge 1$,  we have
$$
\tr(B^{\mu}C^{\nu})\le\|B^{\mu}C^{\nu}\|_{\tr}\le\frac {\mu}{\mu+\nu}\, \tr(B^{\mu+\nu})+\frac {\nu}{\mu+\nu}\, \tr(C^{\mu+\nu}).
$$
\end{lemma}

\begin{proof}
That $\tr(B^{\mu}C^{\nu})\le\|B^{\mu}C^{\nu}\|_{\tr}$ is a corollary of Weyl's majorant theorem \cite[p.42]{bhat:1996}.
Let $Q\in\STM{k}{k}$ such that $Q^{\T}B^{\mu}C^{\nu}\succeq 0$, which yields
$\|B^{\mu}C^{\nu}\|_{\tr}=\tr(Q^{\T}B^{\mu}C^{\nu})$. Note that $B\succeq 0$ and thus
$(Q^{\T}B^{\mu})^{\T}(Q^{\T}B^{\mu})=B^{2\mu}$, implying the singular values of $Q^{\T}B^{\mu}$ are given by $\{[\sigma_i(B)]^{\mu}\}_{i=1}^k$.
Since $C\succeq 0$, the singular values of $C^{\nu}$ are $\{[\sigma_i(C)]^{\nu}\}_{i=1}^k$. Hence by Lemma~\ref{lm:vN-tr-ineq} and then by Lemma~\ref{lm:YoungIneq-ext},
we get
\begin{align}
\|B^{\mu}C^{\nu}\|_{\tr}&=\tr([Q^{\T}B^{\mu}]C^{\nu}) \label{eq:vN-tr-ineq-ext-pf-0}\\
  &\le \sum_{i=1}^k[\sigma_i(B)]^{\mu}\,[\sigma_i(C)]^{\nu} \nonumber\\
  &\le \sum_{i=1}^k \left(\frac {\mu}{\mu+\nu}\, [\sigma_i(B)]^{\mu+\nu}+\frac {\nu}{\mu+\nu}\, [\sigma_i(C)]^{\mu+\nu}\right)
   \label{eq:vN-tr-ineq-ext-pf-1}\\
  &=\frac {\mu}{\mu+\nu}\, \tr(B^{\mu+\nu})+\frac {\nu}{\mu+\nu}\, \tr(C^{\mu+\nu}), \nonumber
\end{align}
as expected.
\end{proof}

\begin{lemma}\label{lm:vN-tr-ineq-ext1}
For $B,\,C\in\bbR^{k\times k}$ where $C$ is positive semidefinite, and $\nu\ge 0$,  we have
$$
\tr(BC^{\nu})\le\|BC^{\nu}\|_{\tr}\le\frac {1}{1+\nu}\, \tr((Q^{\T}B)^{1+\nu})+\frac {\nu}{1+\nu}\, \tr(C^{1+\nu}),
$$
where $Q\in\STM{k}{k}$ such that $Q^{\T}B\succeq 0$.
\end{lemma}

\begin{proof}
Again $\tr(BC^{\nu})\le\|BC^{\nu}\|_{\tr}$ is a corollary of Weyl's majorant theorem. Despite that $B$ may not be positive semidefinite (possibly not even symmetric), we still have  \eqref{eq:vN-tr-ineq-ext-pf-0} with $\mu=1$ so long as $Q^{\T}BC^{\nu}\succeq 0$. Also note
$(Q^{\T}B)^{\T}(Q^{\T}B)=B^{\T}B$ and hence
the singular values of $Q^{\T}B$ are given by $\{\sigma_i(B)\}_{i=1}^k$. So
we still get \eqref{eq:vN-tr-ineq-ext-pf-1} with $\mu=1$.
The proof is completed upon noticing the eigenvalues of $Q^{\T}B$ are the same as the singular values of $B$.
\end{proof}

\begin{lemma}\label{lm:vN-tr-ineq-ext2}
For $X,\,Y\in\bbR^{n\times k}$ and $\mu,\,\nu\ge 0$,  we have
\begin{multline*}
\tr((X^{\T}X)^{\mu}X^{\T}Y(Y^{\T}Y)^{\nu}) \\
   \le\frac {1+2\mu}{2(\mu+\nu+1)}\,\tr((X^{\T}X)^{\mu+\nu+1})+\frac {1+2\nu}{2(\mu+\nu+1)}\,\tr((Y^{\T}Y)^{\mu+\nu+1}).
\end{multline*}
\end{lemma}

\begin{proof}
The singular values of $(X^{\T}X)^{\mu}X^{\T}$ and $Y(Y^{\T}Y)^{\nu}$ are $\{[\sigma_i(X)]^{1+2\mu}\}_{i=1}^k$
and \linebreak $\{[\sigma_i(Y)]^{1+2\nu}\}_{i=1}^k$, respectively. Hence by Lemma~\ref{lm:vN-tr-ineq} and then by Lemma~\ref{lm:YoungIneq-ext},
we get
\begin{align*}
&\tr((X^{\T}X)^{\mu}X^{\T}Y(Y^{\T}Y)^{\nu}) \\
  &\qquad\le \sum_{i=1}^k[\sigma_i(X)]^{1+2\mu}[\sigma_i(Y)]^{1+2\nu} \\
  &\qquad\le \sum_{i=1}^k\left(\frac {1+2\mu}{2(\mu+\nu+1)}[\sigma_i(X)]^{2(\mu+\nu+1)}
                +\frac {1+2\nu}{2(\mu+\nu+1)}[\sigma_i(Y)]^{2(\mu+\nu+1)}\right)        \\
  &\qquad=\frac {1+2\mu}{2(\mu+\nu+1)}\,\tr((X^{\T}X)^{\mu+\nu+1})+\frac {1+2\nu}{2(\mu+\nu+1)}\,\tr((Y^{\T}Y)^{\mu+\nu+1}),
\end{align*}
as was to be shown.
\end{proof}

\begin{lemma}\label{lm:vN-tr-ineq-ext3}
For $X,\,Y\in\bbR^{n\times k}$ and $\mu,\,\nu\ge 0$,  we have
\begin{multline*}
\|(X^{\T}X)^{\mu}X^{\T}Y(Y^{\T}Y)^{\nu}\|_{\F}^2 \\
   \le\frac {1+2\mu}{2(\mu+\nu+1)}\,\|(X^{\T}X)^{\mu+\nu+1}\|_{\F}^2+\frac {1+2\nu}{2(\mu+\nu+1)}\,\|(Y^{\T}Y)^{\mu+\nu+1}\|_{\F}^2.
\end{multline*}
\end{lemma}

\begin{proof}
Notice that
\begin{align*}
\|(X^{\T}X)^{\mu}X^{\T}Y(Y^{\T}Y)^{\nu}\|_{\F}^2
   &=\tr((X^{\T}X)^{\mu}X^{\T}Y(Y^{\T}Y)^{2\nu}Y^{\T}X(X^{\T}X)^{\mu}) \\
   &=\tr(X(X^{\T}X)^{2\mu}X^{\T}Y(Y^{\T}Y)^{2\nu}Y^{\T}).
\end{align*}
Hence, similarly to the proof of Lemma~\ref{lm:vN-tr-ineq-ext2}, we get
\begin{align*}
&\|(X^{\T}X)^{\mu}X^{\T}Y(Y^{\T}Y)^{\nu}\|_{\F}^2 \\
  &\qquad\le \sum_{i=1}^k[\sigma_i(X)]^{2+4\mu}[\sigma_i(Y)]^{2+4\nu} \\
  &\qquad\le \sum_{i=1}^k\left(\frac {2+4\mu}{4(\mu+\nu+1)}[\sigma_i(X)]^{4(\mu+\nu+1)}
                +\frac {2+4\nu}{4(\mu+\nu+1)}[\sigma_i(Y)]^{4(\mu+\nu+1)}\right)        \\
  &\qquad=\frac {1+2\mu}{2(\mu+\nu+1)}\,\tr((X^{\T}X)^{2(\mu+\nu+1)})+\frac {1+2\nu}{2(\mu+\nu+1)}\,\tr((Y^{\T}Y)^{2(\mu+\nu+1)}) \\
  &\qquad=\frac {1+2\mu}{2(\mu+\nu+1)}\,\|(X^{\T}X)^{\mu+\nu+1}\|_{\F}^2+\frac {1+2\nu}{2(\mu+\nu+1)}\,\|(Y^{\T}Y)^{\mu+\nu+1}\|_{\F}^2,
\end{align*}
as was to be shown.
\end{proof}

Lemma~\ref{lm:vN-tr-ineq-ext3} is actually not needed in this article. We present it here because of its similarity to Lemma~\ref{lm:vN-tr-ineq-ext2}.
A special of it for $\mu=\nu=0$ is hidden in the proofs in \cite{zhzl:2017}.

\begin{lemma}[{\cite[Lemma 3]{zhwb:2022}}]\label{lm:maxtrace}
For $H\in\bbR^{k\times k}$, we have
	$
	|\tr(H)|\le\|H\|_{\tr}. 
	$
If
	$
	|\tr(H)|=\|H\|_{\tr}, 
	$
then
$H$ is symmetric and is either positive semi-definite when $\tr(H)\ge 0$, or negative semi-definite when $\tr(H)\le 0$.
\end{lemma}

\begin{remark}[\cite{wazl:2022a}]\label{rk:tr-max}
As a corollary of Lemma~\ref{lm:maxtrace}, for any $H\in\bbR^{k\times k}$, if $H\not\succeq 0$
(which means either $H$ is not symmetric or $H$ is symmetric but indefinite or even negative semidefinite),
then $\tr(H)<\|H\|_{\tr}$.
Now let $H=U\Sigma V^{\T}$ be the SVD of $H$ \cite{govl:2013} where $\Sigma\in\bbR^{k\times k}$
and set $Q=UV^{\T}$, an orthonormal polar factor of $H$. Then
$Q^{\T}H=V\Sigma V^{\T}\succeq 0$ and $\tr(Q^{\T}H)=\|H\|_{\tr}>\tr(H)$.
\end{remark}

\begin{lemma}\label{lm:maxtrace-proj}
Let $H\in\bbR^{n\times n}$ be symmetric and $P_*\in\STM{k}{n}$ whose column space $\cR(P_*)$ is the invariant subspace
of $H$ associated with its $k$ largest eigenvalues. Suppose that $\lambda_k(H)-\lambda_{k+1}(H)>0$.
Given $P\in\STM{k}{n}$, let
$$
\eta=\tr(P_*^{\T}HP_*)-\tr(P^{\T}HP), \quad \epsilon=\sqrt{\frac {\eta}{\lambda_k(H)-\lambda_{k+1}(H)}}.
$$
Then\footnote {The first inequality in \eqref{eq:eig2max} actually holds so long as
      $\cR(P_*)$ is a $k$-dimensional invariant subspace of $H$. It is the second inequality
      that needs the condition of $\cR(P_*)$ being associated with the $k$ largest eigenvalues of $H$.}
\begin{equation}\label{eq:eig2max}
\frac {\|HP-P(P^{\T}HP)\|_{\F}}{\lambda_1(H)-\lambda_n(H)}\le\|\sin\Theta(\cR(P),\cR(P_*))\|_{\F}\le {\epsilon}.
\end{equation}
\end{lemma}

\begin{proof}
We know that $\eta\ge 0$ by Fan's trace minimization principle \cite{fan:1949}.
The second inequality in \eqref{eq:eig2max} is \cite[Theorem~1]{kove:1991} and can also be derived from some of the estimates in
\cite[Chapter~3]{wein:1974} and by a minor modification to the proof of \cite[Theorem~2.2]{li:2004c}. It remains to show
the first inequality in \eqref{eq:eig2max}.
Expand $P$ to $[P,P_{\bot}]\in\STM{n}{n}$. We find
$$
[P,P_{\bot}]^{\T}\big(HP-P(P^{\T}HP)\big)=\begin{bmatrix}
                                    0 \\
                                    P_{\bot}^{\T}HP
                                  \end{bmatrix},
$$
implying
\begin{align}
\|HP-P(P^{\T}HP)\|_{\F}&=\left\|[P,P_{\bot}]^{\T}\big(HP-P(P^{\T}HP)\big)\right\|_{\F} \nonumber \\
   &=\|P_{\bot}^{\T}HP\|_{\F} \nonumber\\
   &=\|P_{\bot}^{\T}(H-\xi I)P\|_{\F}, \label{eq:eig2max:pf-1}
\end{align}
for any $\xi\in\bbR$ because $P_{\bot}^{\T}P=0$. But in what follows, we will take $\xi=[\lambda_1(H)+\lambda_n(H)]/2$.
Next we expand $P_*$ to $[P_*,P_{*\bot}]\in\STM{n}{n}$. Since the column space of $P_*$ is the invariant subspace
of $H$ associated with the $k$ largest eigenvalues of $H$, we have
$$
H[P_*,P_{*\bot}]=[P_*,P_{*\bot}]\begin{bmatrix}
                                  P_*^{\T}HP_* &  \\
                                   & P_{*\bot}^{\T}HP_{*\bot}
                                \end{bmatrix},
$$
and
\begin{align}
P_{\bot}^{\T}(H-\xi I)P&=P_{\bot}^{\T}[P_*,P_{*\bot}]\begin{bmatrix}
                                  P_*^{\T}(H-\xi I)P_* &  \\
                                   & P_{*\bot}^{\T}(H-\xi I)P_{*\bot}
                                \end{bmatrix}\begin{bmatrix}
                                               P_*^{\T}\\ P_{*\bot}^{\T}
                                             \end{bmatrix}P \nonumber\\
   &=P_{\bot}^{\T}P_*P_*^{\T}(H-\xi I)P_*P_*^{\T}P
     +P_{\bot}^{\T}P_{*\bot}P_{*\bot}^{\T}(H-\xi I)P_{*\bot}P_{*\bot}^{\T}P.  \label{eq:eig2max:pf-2}
\end{align}
Noticing that
\begin{gather*}
\|P_{\bot}^{\T}P_*\|_{\F}=\|P_{*\bot}^{\T}P\|_{\F}=\|\sin\Theta(\cR(P),\cR(P_*))\|_{\F}, \\
\|P_*^{\T}P\|_2\le 1, \qquad\|P_{\bot}^{\T}P_{*\bot}\|_2\le 1,
\end{gather*}
and, for $\xi=[\lambda_1(H)+\lambda_n(H)]/2$,
\begin{align*}
\|P_*^{\T}(H-\xi I)P_*\|_2&\le\|H-\xi I\|_2=\frac 12 [\lambda_1(H)-\lambda_n(H)], \\
\|P_{*\bot}^{\T}(H-\xi I)P_{*\bot}\|_2&\le\|H-\xi I\|_2=\frac 12 [\lambda_1(H)-\lambda_n(H)],
\end{align*}
we get from \eqref{eq:eig2max:pf-2},
\begin{align*}
\|P_{\bot}^{\T}(H-\xi I)P\|_{\F}&\le 2\|H-\xi I\|_2\|\sin\Theta(\cR(P),\cR(P_*))\|_{\F} \\
    &=[\lambda_1(H)-\lambda_n(H)]\|\sin\Theta(\cR(P),\cR(P_*))\|_{\F},
\end{align*}
which together with \eqref{eq:eig2max:pf-1} yield the first inequality in \eqref{eq:eig2max}, as expected.
\end{proof}

%

The next lemma are likely well-known. For example, they are implied in the discussion in \cite{wazl:2022a} before
\cite[Lemma~ 3.2]{wazl:2022a} there.

\begin{lemma}\label{lm:polar2max}
Let $B\in\bbR^{n\times k}$.
\begin{enumerate}[{\rm (a)}]
  \item $\tr(P^{\T}B)\le\|B\|_{\tr}$ for any $P\in\STM{k}{n}$;
  \item $\tr(P^{\T}B)=\|B\|_{\tr}$ where $P\in\STM{k}{n}$ if and only if $B=P\Lambda$ with $\Lambda\succeq 0$;
  \item We have
        $$
        \max_{P\in\STM{k}{n}}\tr(P^{\T}B)=\|B\|_{\tr}
        $$
        and the optimal value $\|B\|_{\tr}$ is achieved by any orthonormal polar factor $P_*$ of $B$.
\end{enumerate}
\end{lemma}

Lemma~\ref{lm:polar2max} says that $\tr(P^{\T}B)$  is bounded  above by $\|B\|_{\tr}$ always and the upper bound
$\|B\|_{\tr}$ is achieved
by any orthonormal polar factor $P_*$ of $B$ and also any maximizer of $\tr(P^{\T}B)$ over $P\in\STM{k}{n}$ is an orthonormal polar factor of $B$.
For numerical computation, an orthonormal polar factor of $B$ can be constructed from the thin SVD $B=U\Sigma V^{\T}$ as
$P_*=UV^{\T}$ of $B$. Conceivably,
the closer $\tr(P^{\T}B)$ is to the upper bound, the closer $P$ approaches to an orthonormal polar factor of $B$.
The results of the next lemma quantify the last statement, where the last two inequalities in \eqref{eq:polar2max-4} are due to
Dr.~Zhongmin Teng\footnote {
            Private communication.
            Previously in the earlier versions of this paper the third inequality in \eqref{eq:polar2max-4} 
            is proved under the condition
            $\cR(P)=\cR(P_*)$, in which case $\sin\Theta(\cR(P),\cR(P_*))=0$.
            The proof of the second inequality in \eqref{eq:polar2max} is actually through proving
            $\|\sin\frac 12\Theta(\cR(P),\cR(P_*))\|_{\F}\le \frac 12{\epsilon}$, which is stronger.
            Since $\cR(P_*)$ is the same as $\cR(B)$ here, it can be replaced with $\cR(B)$.
            The second inequality in \eqref{eq:polar2max-4} is also due to Dr.~Zhongming Teng.
            }
and significantly
improve their earlier versions of this paper at arXiv (see also \cite{li:2026}).

\begin{lemma}\label{lm:polar2max'}
Let $B\in\bbR^{n\times k}$ and suppose $\rank(B)=k$. Let $P_*$ be the unique
orthonormal polar factor of $B$.  Given $P\in\STM{k}{n}$, let
$$
\eta=\|B\|_{\tr}-\tr(P^{\T}B), \quad
\epsilon=\sqrt{\frac {2\eta}{\sigma_{\min}(B)}}.
$$
We have
\begin{gather}
\frac {\|B-P(P^{\T}B)\|_{\F}}{\|B\|_2}\le\|\sin\Theta(\cR(P),\cR(P_*))\|_{\F}\le {\epsilon}\,,
        \label{eq:polar2max} \\
\frac {\|B-P(P^{\T}B)\|_{\F}}{\|B\|_2}\le\frac {\|B-P\cdot\sym(P^{\T}B)\|_{\F}}{\|B\|_2}
        \le\|P-P_*\|_{\F}\le {\epsilon}\,.
        \label{eq:polar2max-4}
\end{gather}
\end{lemma}

\begin{proof}
Let $\theta_i$ for $1\le i\le k$ be the canonical angles between
subspaces $\cR(P)$ and $\cR(P_*)$ where $\pi/2\ge\theta_1\ge\cdots\ge\theta_k\ge 0$.
Then the singular values of $P^{\T}P_*\in\bbR^{k\times k}$ are $\cos\theta_i$ for $1\le i\le k$. Let $B=P_*\Lambda$ be the polar decomposition of $B$.
We have
$P^{\T}B=P^{\T}P_*\Lambda$ and hence
\begin{equation}\label{eq:polar2max:pf-1}
\|B\|_{\tr}-\eta=\tr(P^{\T}B)=\tr([P^{\T}P_*]\Lambda)\le\sum_{i=1}^k\sigma_i(B)\cos\theta_{k-i+1}
\end{equation}
by Lemma~\ref{lm:vN-tr-ineq}.
Noticing $\|B\|_{\tr}=\sum_{i=1}^k\sigma_i(B)$, we get from \eqref{eq:polar2max:pf-1}
\begin{align}
\eta\ge\sum_{i=1}^k\sigma_i(B)\big[1-\cos\theta_{k-i+1}\big]&=\sum_{i=1}^k\sigma_i(B)\big[2\sin^2(\theta_{k-i+1}/2)\big]
       \label{eq:polar2max:pf-2}\\
  &\ge\sum_{i=1}^k\sigma_i(B)\cdot\frac 12\sin^2\theta_{k-i+1} \label{eq:polar2max:pf-3}\\
  &\ge\frac 12\sigma_{\min}(B)\|\sin\Theta(\cR(P),\cR(P_*))\|_{\F}^2, \nonumber
\end{align}
yielding the second inequality in \eqref{eq:polar2max}, where we have used
$$
\sin\theta\le 2\sin\frac {\theta}2=\frac {\sin\theta}{\cos(\theta/2)}\le\sqrt 2\sin\theta
\quad\mbox{for}\quad 0\le\theta\le\frac {\pi}2.
$$
Now we expand $P$ to $[P,P_{\bot}]\in\STM{n}{n}$. We find
\begin{equation}\label{eq:polar2max:pf-9}
[P,P_{\bot}]^{\T}\big(B-P(P^{\T}B)\big)=\begin{bmatrix}
                                    0 \\
                                    P_{\bot}^{\T}B
                                  \end{bmatrix},
\end{equation}
implying
\begin{equation}\label{eq:polar2max:pf-1'}
\|B-P(P^{\T}B)\|_{\F}=\left\|[P,P_{\bot}]^{\T}\big(B-P(P^{\T}B)\big)\right\|_{\F}=\|P_{\bot}^{\T}B\|_{\F}.
\end{equation}
It follows from the CS decomposition \cite{stsu:1990} that the singular values of $P_{\bot}^{\T}P_*\in\bbR^{k\times (n-k)}$ comes from
$\sin\theta_i$ for $1\le i\le k$, with possibly some additional zeroes.
We have
$P_{\bot}^{\T}B=P_{\bot}^{\T}P_*\Lambda$ and hence
$$
\|P_{\bot}^{\T}B\|_{\F}=\|P_{\bot}^{\T}P_*\Lambda\|_{\F}
    \le\|P_{\bot}^{\T}P_*\|_{\F}\|B\|_2=\|B\|_2\|\sin\Theta(\cR(P),\cR(P_*))\|_{\F},
$$
which, together \eqref{eq:polar2max:pf-1'}, yield the first inequality in \eqref{eq:polar2max}.

The first inequality in \eqref{eq:polar2max-4} holds
because\footnote {In fact, it holds for any unitarily invariant norm.}
$$
P^{\T}B=\arg\min_{M\in\bbR^{k\times k}}\|B-PM\|_{\F}.
$$
The second and third inequalities in \eqref{eq:polar2max-4} which, along with their proofs below, are due to
Dr.~Zhongming Teng as we commented earlier.
Similarly to \eqref{eq:polar2max:pf-9}, we have
\begin{equation}\label{eq:polar2max:pf-9'}
[P,P_{\bot}]^{\T}\big(B-P\cdot\sym(P^{\T}B)\big)=\begin{bmatrix}
                                    \frac 12[P^{\T}B-B^{\T}P] \\
                                    P_{\bot}^{\T}B
                                  \end{bmatrix},
\end{equation}
implying
\begin{equation}\label{eq:polar2max:pf-10}
\|B-P\cdot\sym(P^{\T}B)\|_{\F}^2=\frac 14\|P^{\T}B-B^{\T}P\|_{\F}^2+\|P_{\bot}^{\T}B\|_{\F}^2.
\end{equation}
We now bound the two terms in the right-hand side of \eqref{eq:polar2max:pf-10}. Recalling $B=P_*\Lambda$, we have
\begin{align*}
P^{\T}B-B^{\T}P&=P^{\T}P_*\Lambda-\Lambda P_*^{\T}P=(P^{\T}P_*-I_k)\Lambda-\Lambda (P_*^{\T}P-I_k), \\
\|P^{\T}B-B^{\T}P\|_{\F}&\le 2\|\Lambda\|_2\|P^{\T}P_*-I_k\|_{\F},\\
\|P_{\bot}^{\T}B\|_{\F}&=\|P_{\bot}^{\T}P_*\Lambda\|_{\F}\le\|\Lambda\|_2\|P_{\bot}^{\T}P_*\|_{\F}.
\end{align*}
Noting $\|\Lambda\|_2=\|B\|_2$, we get
\begin{align*}
\|B-P\cdot\sym(P^{\T}B)\|_{\F}
    &\le\|B\|_2\big(\|P^{\T}P_*-I_k\|_{\F}^2+\|P_{\bot}^{\T}P_*\|_{\F}^2\big)^{1/2} \\
    &=\|B\|_2\left\|\begin{bmatrix}
                     I_k-P^{\T}P_* \\
                     -P_{\bot}^{\T}P_*
                   \end{bmatrix}\right\|_{\F} \\
    &=\|B\|_2\left\|\begin{bmatrix}
                     P^{\T} \\
                     P_{\bot}^{\T}
                   \end{bmatrix}(P-P_*)\right\|_{\F} \\
    &=\|B\|_2\|P-P_*\|_{\F},
\end{align*}
yielding the second inequality in \eqref{eq:polar2max-4}.
For the third inequality there,
%
Again let $B=P_*\Lambda$ be the polar decomposition of $B$, and let $\wtd\Lambda=\Lambda-\sigma_{\min}(B) I_k$.
We know that $\wtd\Lambda\succeq 0$ because $\sigma_{\min}(B)=\lambda_{\min}(\Lambda)$. Hence
$$
\tr\big(P^{\T}P_*\wtd\Lambda\,\big)\le\|P^{\T}P_*\wtd\Lambda\,\|_{\tr}\le\|P^{\T}P_*\|_2\|\wtd\Lambda\,\|_{\tr}
   =\tr\big(\wtd\Lambda\,\big),
$$
yielding $\tr\big((I_k-P^{\T}P_*)\wtd\Lambda\,\big)=\tr\big(\wtd\Lambda\,\big)-\tr\big(P^{\T}P_*\wtd\Lambda\,\big)\ge 0$, and therefore
$$
0\le\tr\big((I_k-P^{\T}P_*)\wtd\Lambda\,\big)=\tr\big((I_k-P^{\T}P_*)\Lambda\big)-\sigma_{\min}(B)\tr(I_k-P^{\T}P_*)
$$
which gives
$
\tr(I_k-P^{\T}P_*)\le\dfrac {\tr\big((I_k-P^{\T}P_*)\Lambda\big)}{\sigma_{\min}(B)}
    =\dfrac {\eta}{\sigma_{\min}(B)}.
$
Finally,
\begin{align*}
\|P-P_*\|_{\F}^2&=\tr\big((P-P_*)^{\T}(P-P_*)\big) \\
       &=2k-2\tr(P^{\T}P_*) \\
       &=2\tr(I_k-P^{\T}P_*)
       \le \frac {2\eta}{\sigma_{\min}(B)}=\epsilon^2,
\end{align*}
as was to be shown.
\end{proof}

\Cref{lm:maxtrace-proj,lm:polar2max,lm:polar2max'}
have since been extracted out and extended to the complex number field in \cite{li:2026}.

\begin{lemma}\label{lm:series}
Let $\{a_i\}_{i=1}^{\infty}$ and $\{b_i\}_{i=1}^{\infty}$ be two nonnegative sequences, and suppose that
  $\sum_{i=1}^{\infty}a_ib_i$ is convergent but $\sum_{i=1}^{\infty}a_i$ is divergent. Then there is a convergent subsequence
  $\{b_{i_j}\}_{j=1}^{\infty}$
  of $\{b_i\}_{i=1}^{\infty}$ that goes to $0$.
\end{lemma}

\begin{proof}
We will prove this by constructing the convergent subsequence
  as follows. For, $\epsilon_1=1/1$, there is an integer $i_1\ge 1$ such that $b_{i_1}\le\epsilon_1$ otherwise
  $b_i\ge \epsilon_1$ for all $i$ and hence the partial sum of $\sum_{i=1}^{\infty}a_ib_i$:
  $$
  \sum_{i=1}^ka_ib_i\ge\epsilon_1\sum_{i=1}^ka_i
  $$
  diverges as $k\to\infty$, contradicting to that $\sum_{i=1}^{\infty}a_i$ is convergent. Next for $\epsilon_1=1/2$, there is an integer $i_2>i_1$ such that $b_{i_2}\le\epsilon_2$ otherwise
  $b_i\ge \epsilon_2$ for all $i\ge i_1+1$ and hence the partial sum of $\sum_{i=i_1+1}^{\infty}a_ib_i$:
  $$
  \sum_{i=i_1+1}^ka_ib_i\ge\epsilon_2\sum_{i=i_1+1}^ka_i
  $$
  diverges as $k\to\infty$, contradicting to that $\sum_{i=i_1+1}^{\infty}a_i$ is convergent. In general, let $\epsilon_j=1/j$, there is an integer $i_j>i_{j-1}$ such that $b_{i_j}\le\epsilon_j$ otherwise
  $b_i\ge \epsilon_j$ for all $i\ge i_{j-1}+1$ and hence the partial sum of $\sum_{i=i_{j-1}+1}^{\infty}a_ib_i$:
  $$
  \sum_{i=i_{j-1}+1}^ka_ib_i\ge\epsilon_j\sum_{i=i_{j-1}+1}^ka_i
  $$
  diverges as $k\to\infty$, contradicting to that $\sum_{i=i_{j-1}+1}^{\infty}a_i$ is convergent.
\end{proof}

\section{Proofs of \Cref{thm:cvg4SCF4NPDo:gl,thm:cvg4SCF4NPDo:strong,thm:cvg4SCFvLOCG}}
\label{sec:Proof4NPDoCVG}

\begin{proof}[Proof of \Cref{{thm:cvg4SCF4NPDo:gl}}]
Let $\{P^{(i)}\}_{i\in\bbI}$ be a subsequence of $\{P^{(i)}\}_{i=0}^{\infty}$ that converges to $P_*$, i.e.,
\begin{equation}\label{eq:thm:cvg4SCF4NPDo-subs-1:pf-1}
\lim_{\bbI\ni i\to \infty}\|P^{(i)}-P_*\|_{\F}= 0,
\end{equation}
where $\bbI$ is an infinite subset of $\{1,2,\ldots\}$.

It can be seen that \cite[pp.125-127]{sun:2001}
\begin{equation}\label{eq:cvg4SCF4NPDo':pf-0}
0\le\dist_2(\cR(P^{(i)}),\cR(P_*))\le \|P^{(i)}-P_*\|_2\to 0\,\,\mbox{as $\bbI\ni i\to\infty$},
\end{equation}
i.e., $\cR(P_*)$ is an accumulation point of the sequence $\{\cR(P^{(i)})\}_{i=0}^{\infty}$.
This proves item~(i).

That $\{f(P^{(i)})\}_{i=0}^{\infty}$ is monotonically increasing is a straightforward consequence of
\eqref{eq:cvg4SCF4NPDo:gl:obj-inc}. Since $f(P)$ is assumed differentiable in $P$ and $\STM{k}{n}$ is compact,
the sequence $\{f(P^{(i)})\}_{i=0}^{\infty}$ is uniformly bounded and hence convergent. This proves item~(ii).

We now prove item~(iii).
We need to show $\scrH(P_*)=P_*\Lambda_*$ and $\Lambda_*=P_*^{\T}\scrH(P_*)\succeq 0$, or, equivalently,
$\cR(\scrH(P_*))\subseteq\cR(P_*)$ and $P_*^{\T}\scrH(P_*)\succeq 0$.
Assume, to the contrary, that either $\cR(\scrH(P_*))\not\subseteq\cR(P_*)$ or
$P_*^{\T}\scrH(P_*)\not\succeq 0$, or both. Then, by Lemma~\ref{lm:polar2max},
$$
\delta:=\|\scrH(P_*)\|_{\tr}-\tr(P_*^{\T}\scrH(P_*))>0.
$$
Since $\|\scrH(P)\|_{\tr}$, $\tr(P^{\T}\scrH(P))$, and $f(P)$ are continuous in $P\in\bbR^{n\times k}$, it follows from
\eqref{eq:thm:cvg4SCF4NPDo-subs-1:pf-1} that there is an $i_0\in\bbI$ such that
\begin{subequations}\label{eq:thm:cvg4SCF4NPDo-subs-1:pf-2}
\begin{gather}
\Big|\|\scrH(P_*)\|_{\tr}-\|\scrH(P^{(i_0)})\|_{\tr}\Big| <\delta/3, \label{eq:thm:cvg4SCF4NPDo-subs-1:pf-2a}\\
\Big|\tr((P^{(i_0)})^{\T}\scrH(P^{(i_0)}))-\tr(P_*^{\T}\scrH(P_*))\Big|<\delta/3, \label{eq:thm:cvg4SCF4NPDo-subs-1:pf-2b}\\
f(P_*)-\kappa\delta/6<f(P^{(i_0)})\le f(P_*).  \label{eq:thm:cvg4SCF4NPDo-subs-1:pf-2c}
\end{gather}
\end{subequations}
By \eqref{eq:cvg4SCF4NPDo:gl:obj-inc} and using \eqref{eq:thm:cvg4SCF4NPDo-subs-1:pf-2}, we have
\begin{align*}
f(P^{(i_0+1)})&\ge f(P^{(i_0)})+\kappa\Big[\big\|\scrH(P^{(i_0)})\big\|_{\tr}-\tr((P^{(i_0)})^{\T}\scrH(P^{(i_0)}))\Big] \\
  &>f(P_*)-\frac {\kappa\delta}6+\kappa\Big[\|\scrH(P_*)\|_{\tr}-\frac {\delta}3-\tr(P_*^{\T}\scrH(P_*))-\frac {\delta}3\Big] \\
  &=f(P_*)+\frac {\kappa\delta}6>f(P_*),
\end{align*}
contradicting $f(P^{(i)})\le\lim_{j\to\infty}f(P^{(j)})= f(P_*)$ for all $i$.

To prove item (iv), we notice that
$\kappa\sum_{i=0}^m\eta_{*i}\le f(P^{(m+1)})-f(P^{(0)})$ for any $m\ge 1$.
By the uniform boundedness of $\{f(P^{(i)})\}_{i=0}^{\infty}$ and that $\kappa>0$ is a constant, we conclude that
\begin{align}
\sum_{i=1}^{\infty}\eta_{*i}
   &=\sum_{i=1}^{\infty}\Big[\big\|\scrH(P^{(i)})\big\|_{\tr}-\tr\big([P^{(i)}]^{\T}\scrH(P^{(i)})\big)\Big]
       \nonumber \\
   &\le\frac 1{\kappa}\lim_{m\to\infty}\big[f(P^{(m+1)})-f(P^{(0)})\big] \nonumber \\
   &<\infty.        \label{eq:thm:cvg4SCF4NPDo-subs-1:pf-3}
\end{align}
It follows from \Cref{lm:polar2max'} that
\begin{equation}\label{eq:thm:cvg4SCF4NPDo-subs-1:pf-3'}
\sigma_{\min}(\scrH(P^{(i)}))\big\|\sin\Theta\big(\cR(\scrH(P^{(i)})),\cR(P^{(i)})\big)\big\|_{\F}^2
   \le 2\eta_{*i},
\end{equation}
which, combined with \eqref{eq:thm:cvg4SCF4NPDo-subs-1:pf-3},
yield \eqref{eq:cvg4SCF4NPDo:series:gl-1}.
Again by \Cref{lm:polar2max'}, we get
\begin{subequations}\label{eq:thm:cvg4SCF4NPDo-subs-1:pf-4}
\begin{gather}
\sigma_{\min}(\scrH(P^{(i)}))\,
    \frac {\big\|\scrH(P^{(i)})- P^{(i)}\cdot\sym\big([ P^{(i)}]^{\T}\scrH(P^{(i)})\big)\big\|_{\F}^2}
          {\big\|\scrH(P^{(i)})\big\|_2^2}
   \le 2\eta_{*i}, \label{eq:thm:cvg4SCF4NPDo-subs-1:pf-4a}\\
\sigma_{\min}(\scrH(P^{(i)}))\,\big\|P^{(i)}-P_*^{(i)}\big\|_{\F}^2 \le 2\eta_{*i}, \label{eq:thm:cvg4SCF4NPDo-subs-1:pf-4b}
\end{gather}
\end{subequations}
Inequalities in \eqref{eq:thm:cvg4SCF4NPDo-subs-1:pf-4}, combined with \eqref{eq:thm:cvg4SCF4NPDo-subs-1:pf-3}
and $\big\|\scrH(P^{(i)})\big\|_2\le\big\|\scrH(P^{(i)})\big\|_{\F}$,
yield \eqref{eq:cvg4SCF4NPDo:series:gl-2} and \eqref{eq:cvg4SCF4NPDo:series:gl-3}.
\end{proof}

The following lemma is an equivalent restatement of \cite[Lemma 4.10]{moso:1983}
(see also \cite[Proposition 7]{kaqi:1999}) in the context of a metric space.

\begin{lemma}[{\cite[Lemma 4.10]{moso:1983}}]\label{lm:isolatedconvg}
Let $\scrG$ be a metric space with metric $\dist(\cdot,\cdot)$, and let
$\{\by_i\}_{i=0}^{\infty}$ be a sequence in $\scrG$. If
$\by_*\in \scrG$ is an isolated accumulation point 	
of the sequence such that, for every subsequence $\{\by_i\}_{i\in\bbI}$
converging to $\by_*$, there is an infinite subset $\widehat{\bbI}\subseteq \bbI$ satisfying
$\dist(\by_i,\by_{i+1})\to 0$ as $\what\bbI\ni i\to\infty$,
then the entire sequence $\{\by_i\}_{i=0}^{\infty}$ converges to $\by_*$.
\end{lemma}

In applying this lemma, on the Grassmann manifold
$\scrG_k(\bbR^n)$, we will use the unitarily invariant metric $\dist_2(\cdot,\cdot)$ in \eqref{eq:sinTheta}  in appendix~\ref{sec:angle-space}, and, on matrix space
$\bbR^{n\times k}$, we will use $\|X-Y\|_2$ as the metric.

\begin{proof}[Proof of \Cref{thm:cvg4SCF4NPDo:strong}]
Recall that we are considering the variation of \Cref{alg:SCF4NPDo} as described in the theorem.
We claim item (a) by verifying \eqref{eq:cvg4SCF4NPDo:gl:obj-inc}. In fact,
by {\bf the NPDo Ansatz}, it follows from \eqref{eq:NPDo:eta(i):weak} that
$$
f(P^{(i+1)})\ge f(P^{(i)})+\omega\eta_i\ge \omega\,c\,\eta_{*i}
$$
verifying \eqref{eq:cvg4SCF4NPDo:gl:obj-inc} with $\kappa=\omega\,c$.

We now prove item (b). Let $\{\cR(P^{(i)})\}_{i\in\bbI_1}$ be a subsequence that converges to $\cR(P_*)$.
Since $\{P^{(i)}\}_{i\in\bbI_1}$ is a bounded sequence, it has a convergent subsequence
$\{P^{(i)}\}_{i\in\bbI_2}$ that converges to $\check P_*$, where $\bbI_2\subseteq\bbI_1$. Clearly $\cR(\check P_*)=\cR(P_*)$,
implying
\begin{subequations}\label{eq:cvg4SCF4NPDo':pf-1}
\begin{alignat}{2}
&\rank(\scrH(\check P_*))=k\quad &&\mbox{by \eqref{eq:full-rank:assume}, and}, \label{eq:cvg4SCF4NPDo':pf-1a}\\
&\scrH(\check P_*)=\check P_*\big[\check P_*^{\T}\scrH(\check P_*)\big]\quad &&\mbox{by \Cref{thm:cvg4SCF4NPDo:gl}(iii)}.
             \label{eq:cvg4SCF4NPDo':pf-1b}
\end{alignat}
\end{subequations}
Next, $\{P^{(i+1)}\}_{i\in\bbI_2}$ has a convergent subsequence $\{P^{(i+1)}\}_{i\in\bbI_3}$, say
converging to $\wtd P_*$, where $\bbI_3\subseteq\bbI_2$. As a result of \Cref{thm:cvg4SCF4NPDo:gl}(ii), we have
$$
\lim_{\bbI_3\ni i\to\infty} f(P^{(i)})=\lim_{\bbI_3\ni i\to\infty} f(P^{(i+1)})=\lim_{i\to\infty} f(P^{(i)}),
$$
and, thus, $f(\check P_*)=f(\wtd P_*)$, a fact that we will need later for proving item (c) with assumption (c2).
It follows from
\eqref{eq:NPDo:eta(i):weak} that
\begin{equation}\label{eq:thm:cvg4SCF4NPDo:strong:pf-4}
\big\|\scrH(P^{(i)})\big\|_{\tr}-\tr\big([\what P^{(i)}]^{\T}\scrH(P^{(i)})\big)
   \le (1-c)\Big[\big\|\scrH(P^{(i)})\big\|_{\tr}-\tr\big([P^{(i)}]^{\T}\scrH(P^{(i)})\big)\Big].
\end{equation}
Now invoke \Cref{lm:polar2max'} to get
$$
\|\sin\Theta(\cR(P^{(i+1)}),\cR(\scrH(P^{(i)})))\|_{\F}
   \le\sqrt{2(1-c)\frac {\big\|\scrH(P^{(i)})\big\|_{\tr}-\tr\big([P^{(i)}]^{\T}\scrH(P^{(i)})\big)}{\sigma_{\min}(\scrH(P^{(i)}))}}
$$
since $\cR(P^{(i+1)})=\cR(\what P^{(i)})$ because $P^{(i+1)}=\what P^{(i)}Q_i$. Letting $\bbI_3\ni i\to\infty$ yields
\linebreak
$\|\sin\Theta(\cR(\wtd P_*),\cR(\scrH(\check P_*))\|_{\F}=0$, implying
$\cR(\wtd P_*)=\cR(\scrH(\check P_*))=\cR(\check P_*)$.
Therefore, as $\bbI_3\ni i\to\infty$, by \eqref{eq:cvg4SCF4NPDo':pf-0} we have
$$
\dist_2(\cR(P^{(i)}),\cR(P^{(i+1)})\le\dist_2(\cR(P^{(i)},\cR(\check P_*))+\dist_2(\cR(\wtd P_*),\cR(P^{(i+1)})\to 0.
$$
Now use Lemma~\ref{lm:isolatedconvg} to conclude that the entire sequence
$\{\cR(P^{(i)})\}_{i=0}^{\infty}$ converges to $\cR(P_*)$. This completes the proof of item~(b).

Consider item (c).
%
%
Let $\{P^{(i)}\}_{i\in\bbI_1}$ be a subsequence of $\{P^{(i)}\}_{i=0}^{\infty}$ that converges to $P_*$.
Since $\{\what P^{(i)}\}_{i\in\bbI_1}$ and $\{P^{(i+1)}\}_{i\in\bbI_1}$ are bounded sequences,
they have convergent subsequences
$\{\what P^{(i)}\}_{i\in\bbI_2}$ and $\{P^{(i+1)}\}_{i\in\bbI_2}$ that converge to $\what P_*$ and $\wtd P_*$, respectively,
where $\bbI_2\subseteq\bbI_1$.
Recalling \eqref{eq:thm:cvg4SCF4NPDo:strong:pf-4} and by
\Cref{lm:polar2max'}, we have
$$
\|\what P^{(i)}-P_*^{(i)}\|_{\F}
\le\sqrt{2(1-c)\frac {\big\|\scrH(P^{(i)})\big\|_{\tr}-\tr\big([P^{(i)}]^{\T}\scrH(P^{(i)})\big)}{\sigma_{\min}(\scrH(P^{(i)}))}},
$$
where $P_*^{(i)}$ is the orthonormal polar factor of $\scrH(P^{(i)})$.
As $\bbI_1\supseteq\bbI_2\ni i\to\infty$, $\what P^{(i)}$ goes to $\what P_*$ and
$P_*^{(i)}$ goes to the orthonormal polar factor of $\scrH(P_*)$, which is unique
under \eqref{eq:full-rank:assume'} and hence
$\what P_*=P_*$. Under assumption (c1), $P^{(i+1)}=\what P^{(i)}Q_i$ converges to $\what P_*$ as $\bbI_1\supseteq\bbI_2\ni i\to\infty$.
Hence as $\bbI_2\ni i\to\infty$
$$
\|P^{(i)}-P^{(i+1)}\|_{\F}\le\|P^{(i)}-P_*\|_{\F}+\|P_*-P^{(i+1)}\|_{\F}\to 0.
$$
Now use Lemma~\ref{lm:isolatedconvg} to conclude that the entire sequence
$\{P^{(i)}\}_{i=0}^{\infty}$ converges to $P_*$.
We also claim $\wtd P_*=P_*$ under assumption (c2) because
always
 $\cR(\wtd P_*)=\cR(\what P_*)=\cR(P_*)$.
Finally, we invoke Lemma~\ref{lm:isolatedconvg} to conclude that the entire sequence
$\{P^{(i)}\}_{i=0}^{\infty}$ converges to $P_*$ also under assumption (c2).
\end{proof}

\begin{proof}[Proof of \Cref{thm:cvg4SCFvLOCG}]
We start by verifying \eqref{eq:cvg4SCF4NPDo:gl:obj-inc}.
Since $Z^{(0)}$ is the first $k$ columns of $I_{\hat n}$, we have
\begin{equation}\label{eq:LOCG-init}
W_iZ^{(0)}=P^{(i)},\quad
\wtd f(Z^{(0)})=f(P^{(i)}).
\end{equation}
We note that the first SCF iteration in \Cref{alg:SCF4NPDo} for computing $Z_{\opt}$ is:
$Z^{(1)}=\what Z^{(0)}\wtd Q_1$ with $\what Z^{(0)}$ satisfying
$$
\tr\big([\what Z^{(0)}]^{\T}\wtd\scrH(Z^{(0)})\big)-\tr\big([Z^{(0)}]^{\T}\wtd\scrH(Z^{(0)})\big)
   \ge c\Big[\big\|\wtd\scrH(Z^{(0)})\big\|_{\tr}-\tr\big([Z^{(0)}]^{\T}\wtd\scrH(Z^{(0)})\big)\Big],
$$
where $\wtd\scrH(\cdot)$ is given as in \eqref{eq:KKT-reduced-1} and $\wtd Q_1\in\STM{k}{k}$ is
determined according to the inherited {\bf NPDo Ansatz} as guaranteed by \Cref{lm:NPDoAssump.-reduced}.
Regardless,
\begin{equation}\label{eq:cvg4SCFvLOCG:pf-1}
\wtd f(Z^{(1)})
\ge\wtd f(Z^{(0)}) +c\,\omega\Big[\underbrace{\big\|\wtd\scrH(Z^{(0)})\big\|_{\tr}-\tr\big([Z^{(0)}]^{\T}\wtd\scrH(Z^{(0)})\big)}_{=:\eta_{*i}}\Big].
\end{equation}
Keeping in mind \eqref{eq:LOCGsub} and \eqref{eq:KKT-reduced}, we get
\begin{gather*}
[Z^{(0)}]^{\T}\wtd\scrH(Z^{(0)})=[W_iZ^{(0)}]^{\T}\scrH(P^{(i)})=[P^{(i)}]^{\T}\scrH(P^{(i)}).
\end{gather*}
By construction $\cR(P^{(i)})\subseteq\cR(W_i)$ and $\cR(\scrR(P^{(i)}))\subseteq\cR(W_i)$, we have
\begin{align*}
W_iW_i^{\T}\scrH(P^{(i)})
   &=W_iW_i^{\T}\big[\cR(\scrR(P^{(i)}))+P^{(i)}\cdot\sym([P^{(i)}]^{\T}\scrH(P^{(i)}))\big] \\
   &=\cR(\scrR(P^{(i)}))+P^{(i)}\cdot\sym([P^{(i)}]^{\T}\scrH(P^{(i)})) \\
   &=\scrH(P^{(i)}).
\end{align*}
Therefore $\big\|\wtd\scrH(Z^{(0)})\big\|_{\tr}=\big\|W_i^{\T}\scrH(P^{(i)})\big\|_{\tr}
   =\big\|W_iW_i^{\T}\scrH(P^{(i)})\big\|_{\tr}=\big\|\scrH(P^{(i)})\big\|_{\tr}$
and thus
\begin{equation}\label{eq:cvg4SCFvLOCG:pf-2}
\eta_{*i}=\big\|\scrH(P^{(i)})\big\|_{\tr}-\tr\big([P^{(i)}]^{\T}\scrH(P^{(i)})\big)\ge 0.
\end{equation}
Notice $\wtd f(Z^{(1)})=f(W_iZ^{(1)})$. Because at least one SCF iteration is carried out, we have
\begin{equation}\label{eq:cvg4SCFvLOCG:pf-3}
f(P^{(i+1)})=\wtd f(Z_{\opt})\ge \wtd f(Z^{(1)})\ge f(P^{(i)})+c\,\omega\,\eta_{*i},
\end{equation}
upon putting all together, where $Z_{\opt}$ is the computed approximation of the optimizer to \eqref{eq:LOCGsub-1}.
This  verifies \eqref{eq:cvg4SCF4NPDo:gl:obj-inc} with $\kappa=c\,\omega$.
Therefore
all conclusions in \Cref{thm:cvg4SCF4NPDo:gl} are valid.

%

Finally, if $Z_{\opt}$ is assumed to be computed an exact maximizer of \eqref{eq:LOCGsub} in each outer iterative step, then
$$
W_i^{\T}\wtd\scrH(Z_{\opt})=Z_{\opt}\cdot \wtd\Lambda_{\opt}
\quad\Rightarrow\quad
W_iW_i^{\T}\scrH(W_iZ_{\opt})=W_iZ_{\opt}\cdot \wtd\Lambda_{\opt},
$$
i.e., $W_iW_i^{\T}\scrH(P^{(i+1)})=P^{(i+1)}\cdot \wtd\Lambda_{\opt}$,
is a polar decomposition, by \Cref{thm:cvg4SCF4NEPv:strong} applied to the reduced problem \eqref{eq:LOCGsub-1}, and hence
$$
0\preceq Z_{\opt}^{\T}\wtd \scrH(Z_{\opt})=Z_{\opt}^{\T}W_i^{\T}\scrH(W_iZ_{\opt})=(P^{(i+1)})^{\T}\scrH(P^{(i+1)}),
$$
as was to be shown.
\end{proof}

\section{Proofs of \Cref{thm:cvg4SCF4NEPv:gl,thm:cvg4SCF4NEPv:strong,thm:cvg4SCFvLOCG:NEPv} }
\label{sec:Proof4NEPvCVG}

\begin{proof}[Proof of \Cref{thm:cvg4SCF4NEPv:gl}]
Let $\{P^{(i)}\}_{i\in\bbI}$ be a subsequence of $\{P^{(i)}\}_{i=0}^{\infty}$ that converges to $P_*$, i.e.,
\begin{equation}\label{eq:thm:cvg4SCF4NEPv-subs-1:pf-1}
\lim_{\bbI\ni i\to \infty}\|P^{(i)}-P_*\|_{\F}= 0,
\end{equation}
where $\bbI$ is an infinite subset of $\{1,2,\ldots\}$.

It can be seen that \cite[pp.125-127]{sun:2001}
\begin{equation}\label{eq:cvg4SCF4NEPv':pf-0}
0\le\dist_2(\cR(P^{(i)}),\cR(P_*))\le \|P^{(i)}-P_*\|_2\to 0\,\,\mbox{as $\bbI\ni i\to\infty$},
\end{equation}
i.e., $\cR(P_*)$ is an accumulation point of the sequence $\{\cR(P^{(i)})\}_{i=0}^{\infty}$.
This proves item~(i).

That $\{f(P^{(i)})\}_{i=0}^{\infty}$ is monotonically increasing is a straightforward consequence of
\eqref{eq:cvg4SCF4NEPv:gl:obj-inc}. Since $f(P)$ is assumed differentiable in $P$ and $\STM{k}{n}$ is compact,
the sequence $\{f(P^{(i)})\}_{i=0}^{\infty}$ is uniformly bounded and hence convergent. This proves item~(ii).

We now prove item~(iii).
We need to show $H(P_*)P_*=P_*\Omega_*$ and the eigenvalues of $\Omega_*$ are the $k$ largest ones of $H(P_*)$,
or, equivalently,
\begin{equation}\label{eq:thm:cvg4SCF4NEPv-1:pf-1a}
\tr(\Omega_*)=\tr(P_*^{\T}H(P_*)P_*)=\max_{P\in\STM{k}{n}}\tr(P^{\T}H(P_*)P)=\tr_{\max,k}(H(P_*)).
\end{equation}
Assume, to the contrary, that \eqref{eq:thm:cvg4SCF4NEPv-1:pf-1a} does not hold, i.e.,
$$
\delta:=\tr_{\max,k}(H(P_*))-\tr(P_*^{\T}H(P_*)P_*)>0.
$$
Since $\tr_{\max,k}(H(P))$, $\tr(P^{\T}H(P)P)$, and $f(P)$ are continuous in $P\in\STM{k}{n}$, there is an $i_0\in\bbI$ such that
\begin{subequations}\label{eq:thm:cvg4SCF4NEPv-1:pf-2}
\begin{gather}
\Big|\tr_{\max,k}(H(P_*))-\tr_{\max,k}(H(P^{(i_0)}))\Big| <\delta/3, \label{eq:thm:cvg4SCF4NEPv-1:pf-2a}\\
\Big|\tr((P^{(i_0)})^{\T}H(P^{(i_0)})P^{(i_0)})-\tr(P_*^{\T}H(P_*)P_*)\Big|<\delta/3, \label{eq:thm:cvg4SCF4NEPv-1:pf-2b}\\
f(P_*)-\kappa\delta/6<f(P^{(i_0)})\le f(P_*).  \label{eq:thm:cvg4SCF4NEPv-1:pf-2c}
\end{gather}
\end{subequations}
By  \eqref{eq:cvg4SCF4NEPv:gl:obj-inc},
we have
\begin{align*}
f(P^{(i_0+1)})&\ge f(P^{(i_0)})+\kappa\Big[\tr_{\max,k}(H(P^{(i_0)}))-\tr((P^{(i_0)})^{\T}H(P^{(i_0)})P^{(i_0)})\Big] \\
  &>f(P_*)-\frac {\kappa\delta}6+\kappa\Big[\tr_{\max,k}(H(P_*))-\frac {\delta}3-\tr(P_*^{\T}H(P_*)P_*)-\frac {\delta}3\Big] \\
  &=f(P_*)+\frac {\kappa\delta}6>f(P_*),
\end{align*}
contradicting $f(P^{(i)})\le\lim_{j\to\infty}f(P^{(j)})= f(P_*)$ for all $i$.
That $P_*$ is a KKT point if $H(P)$
        satisfies \eqref{eq:cond:KKT=NEPv} and $\scrM(P_*)$ is symmetric is a consequence of Theorem~\ref{thm:H(P)-eligibility}.

To prove item (iv), we notice that
$\kappa\sum_{i=0}^m\eta_{*i}\le f(P^{(m+1)})-f(P^{(0)})$ for any $m\ge 1$.
By the uniform boundedness of $\{f(P^{(i)})\}_{i=0}^{\infty}$ and that $\kappa>0$ is a constant, we conclude that
\begin{align}
\sum_{i=1}^{\infty}\eta_{*i}
  &=\sum_{i=1}^{\infty}\Big[\tr_{\max,k}(H(P^{(i)}))-\tr([P^{(i)}]^{\T}H(P^{(i)})P^{(i)})\Big] \nonumber \\
   &\le\frac 1{\kappa}\lim_{m\to\infty}\big[f(P^{(m+1)})-f(P^{(0)})\big] \nonumber \\
  &<\infty. \label{eq:thm:cvg4SCF4NEPv-subs-1:pf-3}
\end{align}
It follows from Lemma~\ref{lm:maxtrace-proj} that
\begin{equation}\label{eq:thm:cvg4SCF4NEPv-subs-1:pf-3'}
\delta_i\,\big\|\sin\Theta\big(\cP_{\max,k}^{(i)},\cR(P^{(i)})\big)\big\|_{\F}^2\le \eta_{*i},
\end{equation}
which, combined with \eqref{eq:thm:cvg4SCF4NEPv-subs-1:pf-3}, yield \eqref{eq:cvg4SCF4NEPv:series:gl-1}.
Again by Lemma~\ref{lm:maxtrace-proj}, we get
\begin{equation}\label{eq:thm:cvg4SCF4NEPv-subs-1:pf-4}
\delta_i\,\frac {\big\|H(P^{(i)})P^{(i)}-P^{(i)}\Lambda_i\big\|_{\F}^2}
                {\big[\lambda_1(H(P^{(i)}))-\lambda_n(H(P^{(i)}))\big]^2}
     \le \eta_{*i}.
\end{equation}
Inequality \eqref{eq:thm:cvg4SCF4NEPv-subs-1:pf-4}, combined with \eqref{eq:thm:cvg4SCF4NEPv-subs-1:pf-3}
and\footnote {We ignore the case $\big[\lambda_1(H(P^{(i)}))-\lambda_n(H(P^{(i)}))\big]=0$ because that corresponds
              to $H(P^{(i)})=\xi I_n$ for some $\xi\in\bbR$.}
\begin{equation}\label{eq:thm:cvg4SCF4NEPv-subs-1:pf-5}
0<\big[\lambda_1(H(P^{(i)}))-\lambda_n(H(P^{(i)}))\big]\le 2\|H(P^{(i)})\|_2\le2\|H(P^{(i)})\|_{\F},
\end{equation}
yield \eqref{eq:cvg4SCF4NEPv:series:gl-2}.
\end{proof}

\begin{proof}[Proof of Theorem~\ref{thm:cvg4SCF4NEPv:strong}]
Recall that we are considering the variation of \Cref{alg:SCF4NEPv} as described in the theorem.
We claim item (a) by verifying \eqref{eq:cvg4SCF4NEPv:gl:obj-inc}. In fact,
by {\bf the NEPv Ansatz}, it follows from \eqref{eq:NEPv:eta(i):weak} that
$$
f(P^{(i+1)})\ge f(P^{(i)})+\omega\eta_i\ge \omega\,c\,\eta_{*i}
$$
verifying \eqref{eq:cvg4SCF4NEPv:gl:obj-inc} with $\kappa=\omega\,c>0$.

We now prove item~(b), let $\{\cR(P^{(i)})\}_{i\in\bbI_1}$ be a subsequence that converges to $\cR(P_*)$.
Since $\{P^{(i)}\}_{i\in\bbI_1}$ is a bounded sequence, it has a subsequence $\{P^{(i)}\}_{i\in\bbI_2}$
that converges to some $\check P_*$ where $\bbI_2\subseteq\bbI_1$. Clearly,
$\cR(\check P_*)=\cR(P_*)$,
implying
\begin{subequations}\label{eq:cvg4SCF4NEPv':pf-1}
\begin{alignat}{2}
&\lambda_k(H(\check P_*))-\lambda_{k+1}(H(\check P_*))>0\quad
             &&\mbox{by \eqref{eq:eig-gap:assume}, and}, \label{eq:cvg4SCF4NEPv':pf-1a}\\
&H(\check P_*)\check P_*=\check P_*\big[\check P_*^{\T}H(\check P_*)\check P_*\big]\quad &&\mbox{by \Cref{thm:cvg4SCF4NEPv:gl}(iii)}.
             \label{eq:cvg4SCF4NEPv':pf-1b}
\end{alignat}
\end{subequations}
Next, $\{P^{(i+1)}\}_{i\in\bbI_2}$ has a convergent subsequence $\{P^{(i+1)}\}_{i\in\bbI_3}$, say
converging to $\wtd P_*$, where $\bbI_3\subseteq\bbI_2$. As a result of \Cref{thm:cvg4SCF4NEPv:gl}(ii), we have
$$
\lim_{\bbI_3\ni i\to\infty} f(P^{(i)})=\lim_{\bbI_3\ni i\to\infty} f(P^{(i+1)})=\lim_{i\to\infty} f(P^{(i)}),
$$
and, thus, $f(\check P_*)=f(\wtd P_*)$.
It follows from
\eqref{eq:NEPv:eta(i):weak} that
\begin{multline}\label{eq:thm:cvg4SCF4NEPv:strong:pf-4}
\tr_{\max,k}\big(H(P^{(i)})\big)-\tr\big([\what P^{(i)}]^{\T}H(P^{(i)})\what P^{(i)}\big) \\
   \le (1-c)\Big[\tr_{\max,k}\big(H(P^{(i)})\big)-\tr\big([P^{(i)}]^{\T}H(P^{(i)})P^{(i)}\big)\Big].
\end{multline}
Now invoke \Cref{lm:maxtrace-proj} to get
$$
\|\sin\Theta(\cR(P^{(i+1)}),\cP_{\max,k}^{(i)})\|_{\F}
   \le\sqrt{2(1-c)\frac {\tr_{\max,k}\big(H(P^{(i)})\big)-\tr\big([P^{(i)}]^{\T}H(P^{(i)})P^{(i)}\big)}
                        {\delta_i}}
$$
since $\cR(P^{(i+1)})=\cR(\what P^{(i)})$ because $P^{(i+1)}=\what P^{(i)}Q_i$. Letting $\bbI_3\ni i\to\infty$ yields
\linebreak
$\|\sin\Theta(\cR(\wtd P_*),\cR(\check P_*)\|_{\F}=0$, implying
$\cR(\wtd P_*)=\cR(\check P_*)$.
Therefore, as $\bbI_3\ni i\to\infty$, by \eqref{eq:cvg4SCF4NEPv':pf-0} we have
$$
\dist_2(\cR(P^{(i)}),\cR(P^{(i+1)})\le\dist_2(\cR(P^{(i)},\cR(\check P_*))+\dist_2(\cR(\wtd P_*),\cR(P^{(i+1)})\to 0.
$$
Now use Lemma~\ref{lm:isolatedconvg} to conclude that the entire sequence
$\{\cR(P^{(i)})\}_{i=0}^{\infty}$ converges to $\cR(P_*)$. This completes the proof of item~(b).

Moments ago, in order to use Lemma~\ref{lm:isolatedconvg} to prove item (b), we start with a subsequence of $\{\cR(P^{(i)})\}_{i=0}^{\infty}$ that converges to $\cR(P_*)$, which led to a subsequence of
$\{P^{(i)}\}_{i=0}^{\infty}$ that converges to $\check P_*$ with $\cR(\check P_*)=\cR(P_*)$,
but not knowing whether $\check P_*= P_*$ or not.
However for proving item (c), we only need to start with a subsequence of $\{P^{(i)}\}_{i=0}^{\infty}$ that does converge
to $P_*$. Let $\{P^{(i)}\}_{i\in\bbI_2}$ be any subsequence that converges to $P_*$.
The portion of the proof of item (b) up to
before invoking \eqref{eq:eig-gap:assume} is valid with $\check P_*$ replaced by $P_*$, at which point, we shall invoke
\eqref{eq:eig-gap:assume'} instead to conclude $\cR(\wtd P_*)=\cR(P_*)$.
We claim $\wtd P_*=P_*$, which follows from the assumption
that $f(P_*)>f(P)$ for any $P\ne P_*$ and $\cR(P)=\cR(P_*)$ because $\cR(\wtd P_*)=\cR(P_*)$ and $f(\wtd P_*)=f(P_*)$.
Finally, we invoke Lemma~\ref{lm:isolatedconvg} to conclude our proof of item (c).

Finally, consider item (d). Without loss of generality, we may assume $Q_i=I_k$ for $i\ge 0$ in Algorithm~\ref{alg:SCF4NEPv}.
Let $\wtd P^{(i)}=P^{(i)}V_i$ for $i\ge 0$. Again portion of the proof of item (b) can be used upon proper modifications, i.e.,
replacing all $P^{(i)}$ with $\wtd P^{(i)}$.
By assumption, $P_*$ is an isolated accumulation point of
$\{\wtd P^{(i)}\}_{i=0}^{\infty}$. Let $\{\wtd P^{(i)}\}_{i\in\bbI_2}$ be a subsequence
that converges to $P_*$ where $\bbI_2\subseteq\bbI_1$, and let $\{\wtd P^{(i+1)}\}_{i\in\bbI_3}$ where
$\bbI_3\subseteq \bbI_2$ be a convergent subsequence:
\begin{equation}\label{eq:Z-right-UI-pf}
\wtd P_*:=\lim_{\bbI_3\ni i\to\infty}\wtd P^{(i+1)}\in \STM{k}{n}.
\end{equation}
Using the fact that $H$ is right-unitarily invariant, we will eventually come up with $\cR(\wtd P_*)=\cR(P_*)$.
Denote by
$$
\Theta^{(i+1)}:=\Theta(\cR(\wtd P^{(i+1)},\cR(P_*))=\Theta(\cR(\wtd P^{(i+1)},\cR(\wtd P_*)),
$$
where the last equality is due to $\cR(\wtd P_*)=\cR(P_*)$. We know that $\Theta^{(i+1)}\to 0$ as $\bbI_2\supseteq\bbI_3\ni i\to\infty$,
because of \eqref{eq:Z-right-UI-pf}.
By \eqref{eq:basis-align}
$$
\|\wtd P^{(i+1)}-P_*\|_{\F}=2\|\sin(\Theta^{(i+1)}/2)\|_{\F},
$$
which then goes to $0$ as $\bbI_2\supseteq\bbI_3\ni i\to\infty$, too.
Hence as $\bbI_3\ni i\to\infty$
$$
\|\wtd P^{(i)}-\wtd P^{(i+1)}\|_{\F}\le\|\wtd P^{(i)}- P_*\|_{\F}+\| P_*-\wtd P^{(i+1)}\|_{\F}\to 0.
$$
Now use Lemma~\ref{lm:isolatedconvg} to conclude that the entire sequence
$\{\wtd P^{(i)}\}_{i=0}^{\infty}$ converges to $P_*$, as needed.
\end{proof}

\begin{proof}[Proof of Theorem~\ref{thm:cvg4SCFvLOCG:NEPv}]
We start by verifying \eqref{eq:cvg4SCF4NEPv:gl:obj-inc}.
We still have \eqref{eq:LOCG-init} at $Z^{(0)}$.
We note that the first SCF iteration in Algorithm~\ref{alg:SCF4NEPv} for computing $Z_{\opt}$ is:
$Z^{(1)}=\what Z^{(0)}\wtd Q_1$ and with $\what Z^{(0)}$ satisfying
\begin{multline*}
\tr\big([\what Z^{(0)}]^{\T}\wtd H(Z^{(0)})\what Z^{(0)}\big)-\tr\big([Z^{(0)}]^{\T}\wtd H(Z^{(0)}) Z^{(0)}\big) \\
   \ge c\Big[\tr_{\max,k}\big(\wtd H(Z^{(0)})\big)-\tr\big([Z^{(0)}]^{\T}\wtd H(Z^{(0)}) Z^{(0)}\big)\Big],
\end{multline*}
where $\wtd H(\cdot)$ is given as in \eqref{eq:cond:KKT=NEPv:reduced} and $\wtd Q_1\in\STM{k}{k}$ is
determined according to the inherited {\bf NEPv Ansatz} as guaranteed by Lemma~\ref{lm:NEPvAssump.-reduced}.
Regardless,
\begin{equation}\label{eq:cvg4SCFvLOCG:NEPv:pf-1}
\wtd f(Z^{(1)})
  \ge\wtd f(Z^{(0)})
    +c\,\omega\Big[\underbrace{\tr_{\max,k}\big(\wtd H(Z^{(0)})\big)
       -\tr\big([Z^{(0)}]^{\T}\wtd H(Z^{(0)}) Z^{(0)}\big)}_{=:\wtd\eta_{*i}}\Big].
\end{equation}
Keeping in mind \eqref{eq:LOCGsub} and \eqref{eq:KKT-reduced}, we get
$$
\wtd H(Z^{(0)})=W_i^{\T}H(P^{(i)})W_i, \quad
[Z^{(0)}]^{\T}\wtd H(Z^{(0)}) Z^{(0)}=P_i^{\T}H(P^{(i)})P_i,
$$
yielding $\wtd\eta_{*i}=\tr_{\max,k}\big(W_i^{\T}H(P^{(i)})W_i\big)-\tr\big([P^{(i)}]^{\T}H(P^{(i)})P^{(i)}\big)$.
By \eqref{eq:cvg4SCFvLOCG:NEPv:assume}, eventually
\begin{equation}\label{eq:cvg4SCFvLOCG:NEPv:pf-2}
\wtd\eta_{*i}\ge c'\Big[\underbrace{\tr_{\max,k}\big(H(P^{(i)})\big)-\tr\big([P^{(i)}]^{\T}H(P^{(i)})P^{(i)}\big)}_{=:\eta_{*i}}\Big]\ge 0.
\end{equation}
Notice $\wtd f(Z^{(1)})=f(W_iZ^{(1)})$. Because at least one SCF iteration is carried out, we have
\begin{equation}\label{eq:cvg4SCFvLOCG:NEPv:pf-3}
f(P^{(i+1)})=\wtd f(Z_{\opt})\ge \wtd f(Z^{(1)})\ge f(P^{(i)})+c\,c'\,\omega\eta_{*i},
\end{equation}
upon putting all together, where $Z_{\opt}$ is the computed approximation of the exact optimizer to \eqref{eq:LOCGsub-1}.
This verifies \eqref{eq:cvg4SCF4NEPv:gl:obj-inc} with $\kappa=c\,c'\,\omega>0$.
Therefore
all conclusions in \Cref{thm:cvg4SCF4NEPv:gl} are valid.
\end{proof}

\section{The $M$-inner Product}\label{sec:M-IN-Prod}
The developments we have so far can be extended to the case of the $M$-inner product $\langle \bx,\by\rangle_M=\by^{\T}M\bx$ where $M\succ 0$.
Instead of \eqref{eq:main-opt}, we face with,
\begin{equation}\label{eq:main-opt:M}
\max_{P^{\T}MP=I_k}f(P).
\end{equation}
To proceed, we let  the Cholesky decomposition of $M$ be $M=R^{\T}R$ and set $Z=RP$.
It can be verified that
$P^{\T}MP=I_k$ if and only if $Z^{\T}Z=I_k$. Hence $Z=RP$ and $P=R^{-1}Z$ establish an one-one mapping between $Z\in\STM{k}{n}$
and $\{P\in\bbR^{n\times k}\,:\,P^{\T}MP=I_k\}$.
Optimization problem \eqref{eq:main-opt:M} is equivalently turned into
\begin{equation}\label{eq:main-opt:M'}
\max_{Z^{\T}Z=I_k}\hat f(Z):=f(R^{-1}Z),
\end{equation}
which is in the form of \eqref{eq:main-opt}. The results we have obtained so far are applicable to \eqref{eq:main-opt:M'}.
For best outcomes,
it is suggested to write down related results symbolically in $Z$ and then translate them back into ones in original matrix variable $P$
of \eqref{eq:main-opt:M}. In what follows, we will outline the basic idea.

By the result in section~\ref{sec:KKT}, the KKT condition of \eqref{eq:main-opt:M'} can be stated as
\begin{equation}\label{eq:KKT-M:Z}
\what\scrH(Z):=\frac {\partial \hat f(Z)}{\partial Z}=Z\Lambda
\quad \mbox{with} \quad
\Lambda^{\T}=\Lambda\in\bbR^{k\times k},\quad Z\in\STM{k}{n}.
\end{equation}
It can be verified that
\begin{equation}\label{eq:part:Z<->P}
\frac {\partial \hat f(Z)}{\partial Z}=R^{-\T}\frac {\partial f(P)}{\partial P}=: R^{-\T}\scrH(P).
\end{equation}
Combine \eqref{eq:KKT-M:Z} and \eqref{eq:part:Z<->P} to yield, after simplifications, the KKT condition of \eqref{eq:main-opt:M} as
\begin{equation}\label{eq:KKT-M:P}
\scrH(P)=MP\Lambda
\quad \mbox{with} \quad
\Lambda^{\T}=\Lambda\in\bbR^{k\times k},\quad P^{\T}MP=I_k.
\end{equation}
In principle, an NPDo approach can be established by translating what we have in Part~I for solving \eqref{eq:KKT-M:Z} for $Z$
to solving \eqref{eq:KKT-M:P} for $P$.

As for the NEPv approach in Part~II, we can create a counterpart as well. Suppose we have a symmetric matrix-valued function $\what H(Z)$
for optimization problem \eqref{eq:main-opt:M'}
such that
\begin{equation}\label{eq:KKT=NEPv:Z-M}
\what H(Z)Z-\what\scrH(Z)=Z\what\scrM(Z),
\end{equation}
which ensures the equivalency between the KKT condition \eqref{eq:KKT-M:Z} and NEPv
\begin{equation}\label{eq:NEPv:Z-M}
\what H(Z)Z=Z\Omega, \quad \Omega^{\T}=\Omega, \quad Z\in\STM{k}{n},
\end{equation}
according to Theorem~\ref{thm:H(P)-eligibility}.
Plug in $Z=RP$ to \eqref{eq:KKT=NEPv:Z-M} and use \eqref{eq:part:Z<->P} to get
$$
\what H(RP)RP-R^{-\T}\scrH(P)=RP\scrM(RP),
$$
and hence
\begin{equation}\label{eq:KKT=NEPv:Z-P}
\underbrace{R^{\T}\what H(RP)R}_{=:H(P)}\,P-\scrH(P)=MP\,\underbrace{\what\scrM(RP)}_{=:\scrM(P)},
\end{equation}
a condition that ensures the equivalency between the KKT condition \eqref{eq:KKT-M:P} and generalized NEPv
\begin{equation}\label{eq:NEPv:P-M}
H(P)P=MP\Omega, \quad \Omega^{\T}=\Omega, \quad P^{\T}MP=I_k.
\end{equation}

\begin{theorem}\label{thm:H(P)-eligibility:M}
Let $H(P)\in\bbR^{n\times n}$ be a symmetric matrix-valued function on $\STM{k}{n}$, satisfying
\begin{equation}\label{eq:cond:KKT=NEPv:M}
H(P)P-\frac{\partial f(P)}{\partial P}=MP\,\scrM(P)\quad\mbox{for $P\in\STM{k}{n}$},
\end{equation}
where $\scrM(P)\in\bbR^{k\times k}$ is some matrix-valued function.
$P\in\STM{k}{n}$ is a solution to the KKT condition \eqref{eq:KKT-M:P} if and only if
it is a solution to NEPv \eqref{eq:NEPv:P-M}
and $\scrM(P)$ is  symmetric.
\end{theorem}

\begin{proof}
If $P$ is a solution to the KKT condition \eqref{eq:KKT-M:P}.
Then, by \eqref{eq:cond:KKT=NEPv:M},
$$
H(P)P=MP\Lambda+MP\scrM(P)=MP(\Lambda+\scrM(P))=:MP\Omega,
$$
where $\Omega=\Lambda+\scrM(P)$ is symmetric because alternatively $\Omega=P^{\T}H(P)P$ which is symmetric, and hence $\scrM(P)=\Omega-\Lambda$ is also symmetric.
On the other hand, if $P$ is a solution to NEPv \eqref{eq:NEPv:P-M} such that $\scrM(P)$ is  symmetric, then again
by  \eqref{eq:cond:KKT=NEPv:M}
$$
\scrH(P)=MP\Omega-MP\scrM(P)=MP\big(\Omega-\scrM(P)\big)=:MP\Lambda,
$$
where $\Lambda=\Omega-\scrM(P)$ is symmetric because $\scrM(P)$ is assumed symmetric.
\end{proof}

Consider, for example, the $\Theta$TR problem but with  constraint $P^{\T}MP=I_k$ instead, for which
$$
f(P)=\frac {\tr(P^{\T}AP+P^{\T}D)}{[\tr(P^{\T}BP)]^{\theta}}, \quad
\hat f(Z)=\frac {\tr(Z^{\T}\wtd AZ+Z^{\T}\wtd D)}{[\tr(Z^{\T}\wtd BZ)]^{\theta}},
$$
where $\wtd A=R^{-\T}AR^{-1}$, $\wtd B=R^{-\T}BR^{-1}$, and $\wtd D=R^{-\T}D$. According to the discussion
at the beginning of section~\ref{sec:AF-NEPv}, $\what H(Z)$ to use in \eqref{eq:NEPv:Z-M} is
$$
\what H(Z)=\frac 2{[\tr(Z^{\T}\wtd BZ)]^{\theta}}
     \left(\wtd A+\frac {\wtd DZ^{\T}+Z\wtd D^{\T}}2
           -\theta\,\frac {\tr(Z^{\T}\wtd AZ+Z^{\T}\wtd D)}{\tr(Z^{\T}\wtd BZ)}\,\wtd B\right)
$$
Finally, $H(P)$, as defined in \eqref{eq:KKT=NEPv:Z-P},  to use in \eqref{eq:NEPv:P-M} is given by
$$
H(P)=\frac 2{[\tr(P^{\T}BP)]^{\theta}}
     \left(A+\frac {DP^{\T}M+MPD^{\T}}2-\theta\,\frac {\tr(P^{\T}AP+P^{\T}D)}{\tr(P^{\T}BP)}\,B\right),
$$
for which
$$
H(P)P-\scrH(P)=MP\left(\frac 1{[\tr(P^{\T}BP)]^{\theta}}\, {D^{\T}P}\right),
$$
and hence any solution to the resulting NEPv \eqref{eq:NEPv:P-M} such that $D^{\T}P$ is symmetric is a KKT
point of $\Theta$TR with  constraint $P^{\T}MP=I_k$ and vice versa.


\section{Proof of Inequality \eqref{eq:f-inc-theta-TR:refined}}\label{sec:f-inc-theta-TR:pf}
In this appendix, we will refine the argument in \cite{wazl:2023} that led to \cite[Theorem 2.2]{wazl:2023}.  Let
$$
g(P)=\frac {\tr(P^{\T}AP}{[\tr(P^{\T}BP)]^{\theta}}, \quad
f(P)=\frac {\tr(P^{\T}AP+P^{\T}D)}{[\tr(P^{\T}BP)]^{\theta}}=g(P)+\frac {\tr(P^{\T}D)}{[\tr(P^{\T}BP)]^{\theta}},
$$
and recall
\begin{align*}
\scrH(P):=\frac{\partial f(P)}{\partial P}
    &=\frac 2{[\tr(P^{\T}BP)]^{\theta}}
     \left(A+\frac D2-\theta\,\frac {\tr(P^{\T}AP+P^{\T}D)}{\tr(P^{\T}BP)}\,B\right), \\
H(P)&=\frac 2{[\tr(P^{\T}BP)]^{\theta}}
     \left(A+\frac {DP^{\T}+PD^{\T}}2-\theta\,\frac {\tr(P^{\T}AP+P^{\T}D)}{\tr(P^{\T}BP)}\,B\right).
\end{align*}
Throughout this section, $0\le\theta\le 1$, $B\succeq 0$ and $\rank(B)>n-k$ which ensures $\tr(P^{\T}BP)\ge \tr_{\min,k}(B)>0$ for any $P\in\STM{k}{n}$.

The next lemma
is a refinement of \cite[Lemma 2.1]{wazl:2023}.

\def\newalpha{a}
\def\newbeta{b}
\def\newdelta{d}
\def\newgamma{h}

\begin{lemma}\label{lm:mono:wazl2022}
For $P,\,\what P\in\STM{k}{n}$, let
\begin{equation}\nonumber 
\newalpha=\tr ( P^{\T}AP),\,\,
\newdelta=\tr ( P^{\T}D),\,\,
\newbeta=\tr ( P^{\T} B P),\,\,
\hat\newbeta=\tr (\what P^{\T} B\what P).
\end{equation}
If
\begin{equation}\label{eq:tr-ineq}
\tr(\what P^{\T} H(P)\what P)\ge\tr ( P^{\T} H(P) P)+\eta,
\end{equation}
then
\begin{equation}\label{eq:obj-ineq-detail}
f(P)+\newgamma+\frac 12\frac {\newbeta^{\theta}}{\hat\newbeta^{\theta}}\eta\le g(\what P)
             +\frac {\tr(\what P^{\T} DP^{\T}\what P)}{[\tr(\what P^{\T} B\what P)]^{\theta}},
\end{equation}
where
\begin{equation}\label{eq:gamma}
\newgamma=\frac {\newalpha+\newdelta}{\hat\newbeta^{\theta}\newbeta}
     \Big[(1-\theta)\newbeta+\theta\hat\newbeta-\newbeta^{1-\theta}\hat\newbeta^{\theta}\Big].
\end{equation}
\end{lemma}

\begin{proof}
It can be verified that
$$
\tr ( P^{\T} H(P) P)
            =2(1-\theta)f(P).
$$
Let $\hat\newalpha=\tr (\what P^{\T}A\what P)$.
By assumption \eqref{eq:tr-ineq}, we have
\begin{align}
\eta+2(1-\theta)f(P)&\le\tr(\what P^{\T} H(P)\what P) \nonumber\\
   &\le \frac 2{\newbeta^{\theta}}\big[\hat\newalpha+\tr(\what P^{\T} DP^{\T}\what P)-\theta f_1(P)\hat\newbeta\big],\nonumber \\
\frac 12\newbeta^{\theta}\eta+(1-\theta)f(P)\newbeta^{\theta}
   &\le\hat\newalpha+\tr(\what P^{\T} DP^{\T}\what P)-\theta f_1(P)\hat\newbeta,\nonumber \\
\frac 12\frac {\newbeta^{\theta}}{\hat\newbeta^{\theta}}\eta+(1-\theta)f(P)\frac {\newbeta^{\theta}}{\hat\newbeta^{\theta}}
   &\le g(\what P)
       +\frac {\tr(\what P^{\T} DP^{\T}\what P)}{\hat\newbeta^{\theta}}-\theta  f(P)\frac {\hat\newbeta^{1-\theta}}{\newbeta^{1-\theta}}, \nonumber
\end{align}
implying
$$
g(\what P)+\frac {\tr(\what P^{\T} DP^{\T}\what P)}{\hat\newbeta^{\theta}}
  \ge f( P)+\newgamma+\frac 12\frac {\newbeta^{\theta}}{\hat\newbeta^{\theta}}\eta,
$$
where
\begin{align}
\newgamma&=(1-\theta)f(P)\frac {\newbeta^{\theta}}{\hat\newbeta^{\theta}}
           +\theta  f(P)\frac {\hat\newbeta^{1-\theta}}{\newbeta^{1-\theta}}-f(P) \nonumber \\
  &=(1-\theta)\frac {\newalpha+\newdelta}{\hat\newbeta^{\theta}}
           +\theta\, \frac {\newalpha+\newdelta}{\newbeta}\,\hat\newbeta^{1-\theta}-\frac {\newalpha+\newdelta}{\newbeta^{\theta}} \nonumber \\
  &=\frac {\newalpha+\newdelta}{\hat\newbeta^{\theta}\newbeta}
    \Big[(1-\theta)\newbeta+\theta\hat\newbeta-\newbeta^{1-\theta}\hat\newbeta^{\theta}\Big]. \nonumber
\end{align}
This proves  inequality  \eqref{eq:obj-ineq-detail}, and it is strict if inequality \eqref{eq:tr-ineq} is strict.
\end{proof}

The next theorem
is a refinement of \cite[Theorem 2.2]{wazl:2023}.

\begin{theorem}\label{thm:mono:wazl2022}
For $P,\what P\in\STM{k}{n}$, suppose either $\theta\in\{0,1\}$ or $\tr (P^{\T}AP+P^{\T}D)\ge 0$, and let
$\wtd P=\what PQ$ where $Q\in\STM{k}{k}$ is an orthonormal polar factor of $\what P^{\T} D$.
If \eqref{eq:tr-ineq} holds, then
\begin{equation}\label{eq:obj-ineq}
\frac {\newbeta^{\theta}}{\hat\newbeta^{\theta}}\eta+f(P)
    \le g(\what P)+\frac {\tr(\what P^{\T} DP^{\T}\what P)}{[\tr(\what P^{\T} B\what P)]^{\theta}}
\end{equation}
yielding
\begin{align}
f(\wtd P)&\ge f(P)+\frac 12\frac {\newbeta^{\theta}}{\hat\newbeta^{\theta}}\eta+
    \frac {\|\what P^{\T} D\|_{\tr}-\tr(\what P^{\T} DP^{\T}\what P)}{\hat\newbeta^{\theta}} \label{eq:f-inc-theta-TR:refined-1}\\
      &\ge f(P)+\omega\,\eta
       +[\tr_{\max,k}(B)]^{-\theta}\,\Big[\|\what P^{\T}D\|_{\tr}-\tr(\what P^{\T}DP^{\T}\what P)\Big], \label{eq:f-inc-theta-TR:refined-2}
\end{align}
where $\newbeta=\tr ( P^{\T} B P)$, $\hat\newbeta=\tr (\what P^{\T} B\what P)$, and
$$
\omega=
\begin{cases}
\frac 12\left(\frac {\tr_{\min,k}(B)}{\tr_{\max,k}(B)}\right)^{\theta},\quad&\mbox{if $\eta\ge 0$}, \\
\frac 12\left(\frac {\tr_{\max,k}(B)}{\tr_{\min,k}(B)}\right)^{\theta},\quad&\mbox{if $\eta< 0$},
\end{cases}
$$
$\tr_{\min,k}(B)$ and $\tr_{\max,k}(B)$ are the sum of the $k$ smallest eigenvalues and that of the $k$ largest eigenvalues of $B$, respectively.
\end{theorem}

\begin{proof}
In Lemma~\ref{lm:mono:wazl2022}, we note $\newgamma\equiv 0$
in the case $\theta\in\{0,1\}$, and $\newgamma\ge 0$ in the case $\newalpha+\newdelta=\tr (P^{\T}AP+P^{\T}D)\ge 0$,
and hence we have \eqref{eq:obj-ineq} which yields \eqref{eq:f-inc-theta-TR:refined-1} upon noticing
$\tr (\what P^{\T} A\what P)=\tr (\wtd P^{\T} A\wtd P)$,
$\tr (\what P^{\T} B\what P)=\tr (\wtd P^{\T} B\wtd P)$, and writing
\begin{align*}
\tr(\what P^{\T} DP^{\T}\what P)&=\|\what P^{\T} D\|_{\tr}-\Big[\|\what P^{\T}D\|_{\tr}-\tr(\what P^{\T}DP^{\T}\what P)\Big] \\
   &=\tr(\wtd P^{\T} D)-\Big[\|\what P^{\T}D\|_{\tr}-\tr(\what P^{\T}DP^{\T}\what P)\Big],
\end{align*}
where we have used $\|\what P^{\T} D\|_{\tr}=\tr(\wtd P^{\T} D)$.
Note that $\|\what P^{\T}D\|_{\tr}-\tr(\what P^{\T}DP^{\T}\what P)\ge 0$ because
$$
\tr(\what P^{\T} DP^{\T}\what P)
  \le \|\what P^{\T} DP^{\T}\what P\|_{\tr}
  \le \|\what P^{\T} D\|_{\tr}\|P^{\T}\what P\|_2
  \le \|\what P^{\T} D\|_{\tr}.
$$
Finally use $0<\tr_{\min,k}(B)\le\newbeta\le \tr_{\max,k}(B)$
and $0<\tr_{\min,k}(B)\le\hat\newbeta\le \tr_{\max,k}(B)$ to claim \eqref{eq:f-inc-theta-TR:refined-2} from
\eqref{eq:f-inc-theta-TR:refined-1}.
\end{proof}

\section{Convexity of $\tr\big((\Gamma P^{\T}AP)^m\big)$}\label{apx:convex-NPDo:AF}
The main result of this section is to prove that $\tr\big((\Gamma P^{\T}AP)^m\big)$ is convex in $P\in\bbR^{n\times k}$
as stated in  the next theorem. It is not clear if this result is entirely new for $m\ge 2$.

\begin{theorem}\label{thm:convex-NPDo:AF}
Let $A\in\bbR^{n\times n}$ and $\Gamma\in\bbR^{k\times k}$ and suppose that both are positive semidefinite. Then
$\tr\big((\Gamma P^{\T}AP)^m\big)$ is convex in $P\in\bbR^{n\times k}$ where $m\ge 1$ is an integer.
\end{theorem}

Our proof is based on the fact: {\em any twice differentiable function $h\,:\,\bbR^d\mapsto\bbR$ is convex
if and only if for every $\bp_0, \bw\in\bbR^d$, the scalar function $\psi(t) := h(\bp_0 + t\bw)$ satisfies
$\psi''(t) \ge 0$ for all $t$ (equivalently, the Hessian of $h$ is positive semidefinite
everywhere).
}

\begin{proof}[Proof of \Cref{thm:convex-NPDo:AF}]
First we consider the case $\Gamma=I_k$ and prove that $\tr\big((P^{\T}AP)^m\big)$ is convex in $P$.
That $\tr(P^{\T}AP)$, i.e., in the case $m=1$, is convex is well known because
$$
\tr(P^{\T}AP)=\sum_{i=1}^k\bp_i^{\T}A\bp_i
$$
and each $\bp_i^{\T}A\bp_i$ is convex in $\bp_i\in\bbR^n$, the $i$th column of $P$.

Suppose $m\ge 2$ and $\Gamma=I_k$.
Fix arbitrary $P_0, W \in \mathbb{R}^{n\times k}$, set
$P(t) = P_0 + tW$. It suffices to show that
\[
\psi(t) := \tr\big([S(t)]^m\big) \quad\mbox{with}\quad S(t) := P(t)^{\T}AP(t),
\]
satisfies $\psi''(t) \ge 0$ for all $t$.
Note that $S(t)\in\bbR^{k\times k}$ is positive semidefinite because of $A\succeq 0$.
For ease of presentation, we write $S\equiv S(t)$, i.e., suppress the dependency on $t$ in writing. We have
\begin{align}
\psi'(t) &=m\,\tr(S^{m-1}S'), \nonumber\\
\psi''(t)&= m\sum_{\ell=0}^{m-2}\tr\big(S^{\ell}S'S^{m-2-\ell}S'\big)+ m\,\tr\big(S^{m-1}S''\big). \label{eq:convex-NPDo:AF:pf-1}
\end{align}
Now with $P(t) = P_0 + tW$, we get
\[
S'(t) = W^{\T}AP(t) + P(t)^{\T}AW, \qquad S''(t) = 2W^{\T}AW.
\]
Both are symmetric, and also $W^{\T}AW \succeq 0$. The second term in \eqref{eq:convex-NPDo:AF:pf-1}
is nonnegative because $S^{m-1}\succeq 0$ and
$$
\tr\big(S^{m-1}S''\big)=\tr\big([S'']^{1/2}S^{m-1}[S'']^{1/2}\big)\ge 0.
$$
We claim that the first term in \eqref{eq:convex-NPDo:AF:pf-1} is also nonnegative.
Fix $t$ and let $S = S(t) = U\Lambda U^{\T}$
be its eigendecomoposition, where
$\Lambda = \diag(\lambda_1,\dots,\lambda_k)$ with all $\lambda_i \ge 0$.
Let $\wtd{S'} = U^{\T}S'U \equiv [\tilde{s}_{ij}]$ which is symmetric. Then
\begin{equation}\label{eq:convex-NPDo:AF:pf-2}
\tr(S^{\ell}S'S^{m-2-\ell}S') = \tr\big(\Lambda^\ell\wtd{S'}\Lambda^{m-2-\ell}\wtd{S'}\big)
= \sum_{i,j} \lambda_i^{\,\ell}\lambda_j^{\,m-2-\ell}\,\tilde{s}_{ij}^2.
\end{equation}
Summing \eqref{eq:convex-NPDo:AF:pf-2} over $\ell = 0,\dots,m-2$ and using $\lambda_i, \lambda_j \ge 0$ (so
every $\lambda_i^\ell\lambda_j^{m-2-\ell} \ge 0$), we get
\[
\sum_{\ell=0}^{m-2}\tr(S^{\ell}S'S^{m-2-\ell}S')
= \sum_{i,j}\Big(\underbrace{\textstyle\sum_{\ell=0}^{m-2}
\lambda_i^{\,\ell}\lambda_j^{\,m-2-\ell}}_{\ge 0}\Big)\tilde{s}_{ij}^2 \ge 0,
\]
as was to be shown. Finally we arrive at $\psi''(t)\ge 0$ for all $t\in\bbR$ and hence $\tr\big((P^{\T}AP)^m\big)$ is convex in $P$.

For any general $\Gamma\succeq 0$, define $\what P=P\Gamma^{1/2}$. It can be seen that
$$
f(P):=\tr\big((\Gamma P^{\T}AP)^m\big)=\tr((\what P^{\T}A\what P)^m)=:g(\what P).
$$
For $P_1,P_2$ and $\lambda\in[0,1]$, let $\what P_i=P_i\Gamma^{1/2}$ for $i=1,2$.
By the convexity of $g(\what P)$ in $\what P$ we proved moments ago, we have
\begin{align*}
f(\lambda P_1+(1-\lambda)P_2)
  &=g\big(\lambda \what P_1+(1-\lambda)\what P_2\big) \\
  &\le \lambda g(\what P_1)+(1-\lambda)g(\what P_2) \\
  &=\lambda f(P_1)+(1-\lambda)f(P_2),
\end{align*}
i.e.,
$f(P)=\tr\big((\Gamma P^{\T}AP)^m\big)$ is convex in $P$.
\end{proof}

\clearpage

{\small
\def\noopsort#1{}\def\l{\char32l}\def\v#1{{\accent20 #1}}
  \let\^^_=\v\def\hbk{hardback}\def\pbk{paperback}

}

\end{document}